\documentclass{article}

\usepackage[utf8]{inputenc}

\usepackage{geometry}

\usepackage{amsmath, amsfonts, amssymb, amsthm, amscd,booktabs}

\usepackage{textcomp}

\usepackage{bm}

\usepackage{bbm}

\usepackage{mathtools}
\usepackage{mathabx}

\usepackage{wrapfig} %

\usepackage[usenames,svgnames,dvipsnames]{xcolor}

\usepackage{url}

\usepackage[ruled,vlined]{algorithm2e}

\usepackage{float} %

\usepackage{graphicx,enumitem}

\usepackage{hyperref}
\hypersetup{
    colorlinks,
    linkcolor={blue!80!black},
    citecolor={green!40!black},
    urlcolor={red!50!black}
}

\usepackage{cleveref}

\usepackage{natbib}

\newtheorem{theorem}{Theorem}[section]
\newtheorem{lemma}[theorem]{Lemma}
\newtheorem{corollary}[theorem]{Corollary}
\newtheorem{proposition}[theorem]{Proposition}

\theoremstyle{definition}
\newtheorem{definition}[theorem]{Definition}
\newtheorem{remark}[theorem]{Remark}
\newtheorem{example}[theorem]{Example}

\usepackage{epstopdf}

\usepackage{microtype}
\usepackage{longtable}
\usepackage{tocloft}
\newcommand{\eps}{\varepsilon}

\newcommand{\EE}{\mathbb{E}}

\newcommand{\PP}{\mathbb{P}}
\newcommand{\CC}{\mathbb{C}}

\newcommand{\RR}{\mathbb{R}}

\newcommand{\NN}{\mathbb{N}}
\newcommand{\QQ}{\mathbb{Q}}

\newcommand{\cA}{\mathcal{A}}
\newcommand{\cB}{\mathcal{B}}

\newcommand{\cD}{\mathcal{D}}

\newcommand{\cG}{\mathcal{G}}

\newcommand{\cI}{\mathcal{I}}

\newcommand{\cL}{\mathcal{L}}
\newcommand{\cM}{\mathcal{M}}
\newcommand{\cN}{\mathcal{N}}

\newcommand{\cT}{\mathcal{T}}

\newcommand{\bX}{\bm{X}}
\newcommand{\bY}{\bm{Y}}
\newcommand{\bZ}{\bm{Z}}
\newcommand{\bU}{\bm{U}}
\newcommand{\bV}{\bm{V}}
\newcommand{\bS}{\bm{S}}
\newcommand{\bA}{\bm{A}}
\newcommand{\bB}{\bm{B}}
\newcommand{\bC}{\bm{C}}
\newcommand{\bD}{\bm{D}}

\newcommand{\beps}{\bm{\varepsilon}}

\newcommand{\hbeta}{{\widehat{\beta}}}
\newcommand{\hSigma}{{\widehat{\bm{\Sigma}}}}

\newcommand{\hsigma}{\widehat{\sigma}}
\newcommand{\hkappa}{\widehat{\kappa}}
\newcommand{\hxi}{{\widehat{\xi}}}
\newcommand{\hmu}{\widehat{\mu}}
\newcommand{\hf}{{\widehat{f}}}

\newcommand{\hR}{{\widehat{R}}}
\newcommand{\hQ}{{\widehat{Q}}}

\newcommand{\tf}{{\widetilde{f}}}
\newcommand{\tbeta}{{\widetilde{\beta}}}

\newcommand{\tsigma}{{\widetilde{\sigma}}}
\newcommand{\tv}{\widetilde{v}}

\newcommand{\argmin}{\mathop{\mathrm{arg\,min}}}

\newcommand{\tr}{\mathop{\mathrm{tr}}}
\newcommand{\Var}{\mathop{\mathrm{Var}}}

\newcommand{\op}{\mathop{op}}

\newcommand{\asto}{\xrightarrow{\text{a.s.}}}
\newcommand{\pto}{\xrightarrow{\text{p}}}
\newcommand{\dto}{\xrightarrow{\text{d}}}

\def\1{\mathbbm{1}}

\newcommand{\onestep}{\mathrm{os}}

\newcommand{\zerostep}{\mathrm{zs}}
\newcommand{\mnls}{\mathrm{mn2}}

\newcommand{\ridge}{\mathrm{ridge}}
\newcommand{\lasso}{\mathrm{lasso}}

\newcommand{\MOM}{\texttt{MOM}}
\newcommand{\test}{\mathrm{te}}
\newcommand{\train}{\mathrm{tr}}
\newcommand{\add}{\mathrm{add}}
\newcommand{\mul}{\mathrm{mul}}
\newcommand{\cv}{\mathrm{cv}}
\newcommand{\deter}{\mathrm{det}}
\newcommand{\mnla}{\mathrm{mn1}}
\newcommand{\init}{\mathrm{init}}
\newcommand{\ols}{\mathrm{ols}}

\newcommand{\CEN}{\texttt{CEN}}
\newcommand{\AVG}{\texttt{AVG}}

\definecolor{yuting}{RGB}{255,69,0}

\renewcommand{\Re}{\operatorname{Re}}
\renewcommand{\Im}{\operatorname{Im}}

\newcommand{\asympequi}{\simeq}

\DeclareMathOperator{\sign}{sgn}

\newcommand\ddfrac[2]{\frac{\displaystyle #1}{\displaystyle #2}}

\allowdisplaybreaks

\graphicspath{ {./figures/risk-monotonization/} }

\title{
\texorpdfstring{\vspace{-2em}}{}
Mitigating multiple descents: \\ 
A model-agnostic framework for risk monotonization}

\newcommand{\footremember}[2]{%
    \footnote{#2}
    \newcounter{#1}
    \setcounter{#1}{\value{footnote}}%
}
\newcommand{\footrecall}[1]{%
    \footnotemark[\value{#1}]%
} 

\author{
Pratik Patil\footremember{cmustats}{Department of Statistics and Data Science, 
Carnegie Mellon University, Pittsburgh, PA 15213, USA.}\footremember{cmumld}{Machine Learning Department, Carnegie Mellon University,
Pittsburgh, PA 15213, USA.}
\and Arun Kumar Kuchibhotla\footrecall{cmustats}
\and Yuting Wei\footnote{Department of Statistics and Data Science,
The Wharton School, University of Pennsylvania, Philadelphia, PA 19104, USA.}
\and Alessandro Rinaldo\footrecall{cmustats}
}
\date{\today \vspace{-12pt}}

\begin{document}

\cftsetindents{section}{1em}{2em}
\cftsetindents{subsection}{1em}{3em}

\maketitle

\begin{abstract}
    Recent empirical and theoretical analyses 
    of several commonly used prediction procedures
    reveal a peculiar risk behavior in high dimensions, 
    referred to as double/multiple descent, 
    in which the asymptotic risk is a non-monotonic function 
    of the limiting aspect ratio of the number of features 
    or parameters to the sample size. 
    To mitigate this undesirable behavior, 
    we develop a general framework for risk monotonization 
    based on cross-validation that takes as input 
    a generic prediction procedure and returns a modified procedure
    whose out-of-sample prediction risk is, asymptotically,
    monotonic in the limiting aspect ratio.
    As part of our framework,
    we propose two data-driven methodologies, 
    namely zero- and one-step, that are akin to bagging and boosting,
    respectively, and show that, under very mild assumptions, 
    they provably achieve monotonic asymptotic risk behavior.
    Our results are applicable to a broad variety of
    prediction procedures and loss functions, 
    and do not require a well-specified (parametric) model. 
    We exemplify our framework with concrete analyses of 
    the minimum   $\ell_2$, $\ell_1$-norm least squares prediction procedures. 
    As one of the ingredients in our analysis,
    we also derive novel additive and multiplicative forms of 
    oracle risk inequalities  for split cross-validation 
    that are of independent interest.
\end{abstract}


\setcounter{tocdepth}{2}

\tableofcontents

\section{Introduction}
\label{sec:introduction}

Modern machine learning models deploy a large number of parameters
relative to 
the number of observations.
Even though such overparameterized models typically have the capacity
to (nearly) interpolate noisy training data, they often generalize well
on unseen test data in practice \citep{zhang_bengio_hardt_recht_vinyals_2016,zhang2021understanding}.
The striking and widespread successes of interpolating models
has been a topic of growing interest in the recent mathematical statistics  
literature~\citep[see, e.g.,][]{belkin2019reconciling,belkin_hsu_mitra_2018,belkin2019does,bartlett_long_lugosi_tsigler_2020}, as it seemingly defies the widely-accepted statistical wisdom 
that interpolation will generally lead to over-fitting 
and poor generalization~\citep[Figure 2.11]{hastie2009elements}. 
A body of recent work has both empirically and theoretically investigated 
this surprising phenomenon for different models,
including linear regression \citep{hastie_montanari_rosset_tibshirani_2019,muthukumar_vodrahalli_subramanian_sahai_2020,belkin_hsu_xu_2020,bartlett_long_lugosi_tsigler_2020},
kernel regression \citep{liang2020just},
nearest neighbor methods \citep{xing_song_cheng_2018,xing_song_cheng_2022},
boosting algorithms \citep{liang2020precise},
among others.
See the survey papers by \cite{bartlett_montanari_rakhlin_2021}
and \cite{dar_muthukumar_baraniuk_2021} for more related references.

A closely related and equally striking feature of overparameterized models
is the so-called ``double/multiple descent'' behavior
in the generalization error curve
when plotted against the number of parameters
or as a function of the aspect ratio
 of the number of parameters to the sample size. 
In a typical double descent scenario, 
the generalization or test error initially increases as a function of the aspect ratio.
It peaks and in some cases explodes as this ratio crosses
the {\it interpolation threshold},
where the learning algorithm achieves a degree of complexity 
that allows for perfect interpolation of the data. 
Past the interpolation threshold, the test error tapers down
as the complexity of the algorithm increases relative to the sample size. 
Furthermore, for some algorithms and settings, 
e.g., the lasso and the minimum $\ell_{1}$-norm least square~\citep[e.g.,][]{li_wei_2021} 
or  various structures of the design matrix \citep{adlam2020neural,chen2020multiple},
multiple descents may occur. 
Double and multiple descent phenomena have been first demonstrated empirically,
e.g., for decision trees, random features and two-layer and deep neural networks,
and some of these findings have now been corroborated by rigorous theories 
in a growing body of work: 
see, e.g.,~\cite{neyshabur2014search,nakkiran2019deep,belkin2018understand,belkin2019reconciling,Mei2019generalization,adlam2020neural,chen2020multiple,li_wei_2021},
among others.
However, in general, the shape and number of local minima
associated with a non-monotonic risk  profile due to double descent depend non-trivially on the learning problem, the algorithm deployed, and to an extent, 
the properties of the data generating distribution in ways that are only partially understood. 

The non-monotonic behavior of the generalization error as a function of the aspect ratio in the over-parameterized settings 
suggests the jarring conclusion that, in high dimensions, 
increasing the sample size might actually yield
a worse generalization error. 
In contrast, it is highly desirable to rely on prediction procedures 
that are guaranteed to deliver, at least asymptotically, a risk profile 
that is monotonically increasing in the aspect ratio, 
over a large class of data generating distributions. 
(Note that increasing in aspect ratio is same as decreasing in sample size for a given number of features.) To that effect, 
some authors have considered ridge-regularized estimators; 
see \cite{nakkiran_venkat_Kakade_Ma-2020,hastie_montanari_rosset_tibshirani_2019}. 
In those cases, under fairly restrictive settings and distributional assumptions, 
a monotonic risk profile can be assured. 
However, in general settings and for any given procedure, 
it is unclear how to determine whether the associated risk profile 
is at least approximately non-monotonic and, if so, how to mitigate it. 
The ubiquity of the double and multiple descent phenomenon 
in over-parameterized settings begs the question: 
\begin{center}
\emph{Is it possible to modify any given prediction procedure 
in order to achieve a monotonic risk behavior?} 
\end{center}

In this paper, we answer this question in the affirmative. 
More specifically, we develop a simple, general-purpose framework 
that takes as input an arbitrary learning algorithm and returns a modified version whose out-of-sample risk will be asymptotically no larger than the  smallest risk achievable
beyond the aspect ratio for the problem at hand. In particular, the asymptotic risk of the returned procedure, as a function of the aspect ratio, will stay below the ``monotonized'' asymptotic risk profile of the original procedure corresponding to its largest non-decreasing minorant (see \Cref{figure:illustrate-intro} for an illustration).
As a result, when the risk function of the original procedure exhibits double or multiple descents,
our modification will guarantee, asymptotically, a far smaller out-of-sample risk near the peaks of the risk function. Our approach is applicable to a large class of data generating distributions and learning problems, with mild to no assumptions on the learning algorithm of choice.

To illustrate the type of guarantees obtained in this paper, we provide a preview of one of our main results from~\Cref{sec:zerostep-overparameterized} and comment on its implication. Adopting a standard regression framework, we assume that the data $\cD_n = \{ (X_1, Y_1), \ldots, (X_n,Y_n) \}$ are comprised of $n$ i.i.d.\ pairs of  a $p$-dimensional covariate and a response variable from an unknown distribution. 
Using $\cD_n$,  suppose one fits a predictor $\hf$ --- a random function that maps $x \in \mathbb{R}^p \mapsto \hf(x) \in \mathbb{R}$. Given a loss function $\ell \colon \mathbb{R} \times \mathbb{R} \to \mathbb{R}_{\ge 0}$, we evaluate the performance of $\hf$ by its conditional predictive risk given the data, defined by $R(\hf; \cD_{n}) = \EE[\ell(Y_0, \hf(X_0)) \mid \cD_n]$, where $(X_0, Y_0)$ is an unseen data point, drawn independently from the data generating distribution. Note the risk is a random variable, as it depends on the data $\cD_n$. 
We are interested in the limiting behavior of the risk under the proportional asymptotic regime 
in which $n, p \to \infty$
with the aspect ratio $p / n$ converging to a constant $\gamma \in (0,\infty)$.
As noted above, in such regime the asymptotic risk profile of  $\hf$
has been recently shown to be non-monotonic 
for a wide variety of problems and procedures. 
In order to mitigate such behavior, we devise a modification of the original procedure $\hf$ that results into a new procedure $\hf^\zerostep$, called zero-step procedure (described in \Cref{alg:zero-step}),
 whose asymptotic risk profile is provably monotonic in $\gamma$.
The following informal result
can be derived
as a consequence of 
results in \Cref{sec:zerostep-overparameterized}.

\begin{theorem}
    [Informal monotonization result]
    \label{thm:asymptotic-risk-tuned-zero-step-informal}
    Suppose
    there exists a deterministic function $R^\deter(\cdot; \hf) : (0, \infty] \to [0, \infty]$
    such that
    for any $\phi \in (0, \infty]$
    for any dataset $\cD$ consisting of $m$ i.i.d.\ observations
    with $p_m$ features,
    $ R(\hf; \cD) \pto R^\deter(\phi; \hf)$,
    whenever $m, p_m \to \infty$
    and $p_m / m \to \phi$.
    Then, 
    under mild assumptions on $R^\deter$, the loss function $\ell$, and the data generating distribution, the zero-step  procedure $\hf^\zerostep$ satisfies 
     \[
         \Big|
            R(\hf^\zerostep;\cD_n)
            - \min_{\zeta \ge \gamma} R^\deter(\zeta; \hf)
        \Big|
        ~\pto~ 0
    \]
    as $n, p \to \infty$  and $p / n \to \gamma \in (0, \infty)$.
\end{theorem}

\begin{figure}[!t]
\centering
    \includegraphics[width=0.45\textwidth]{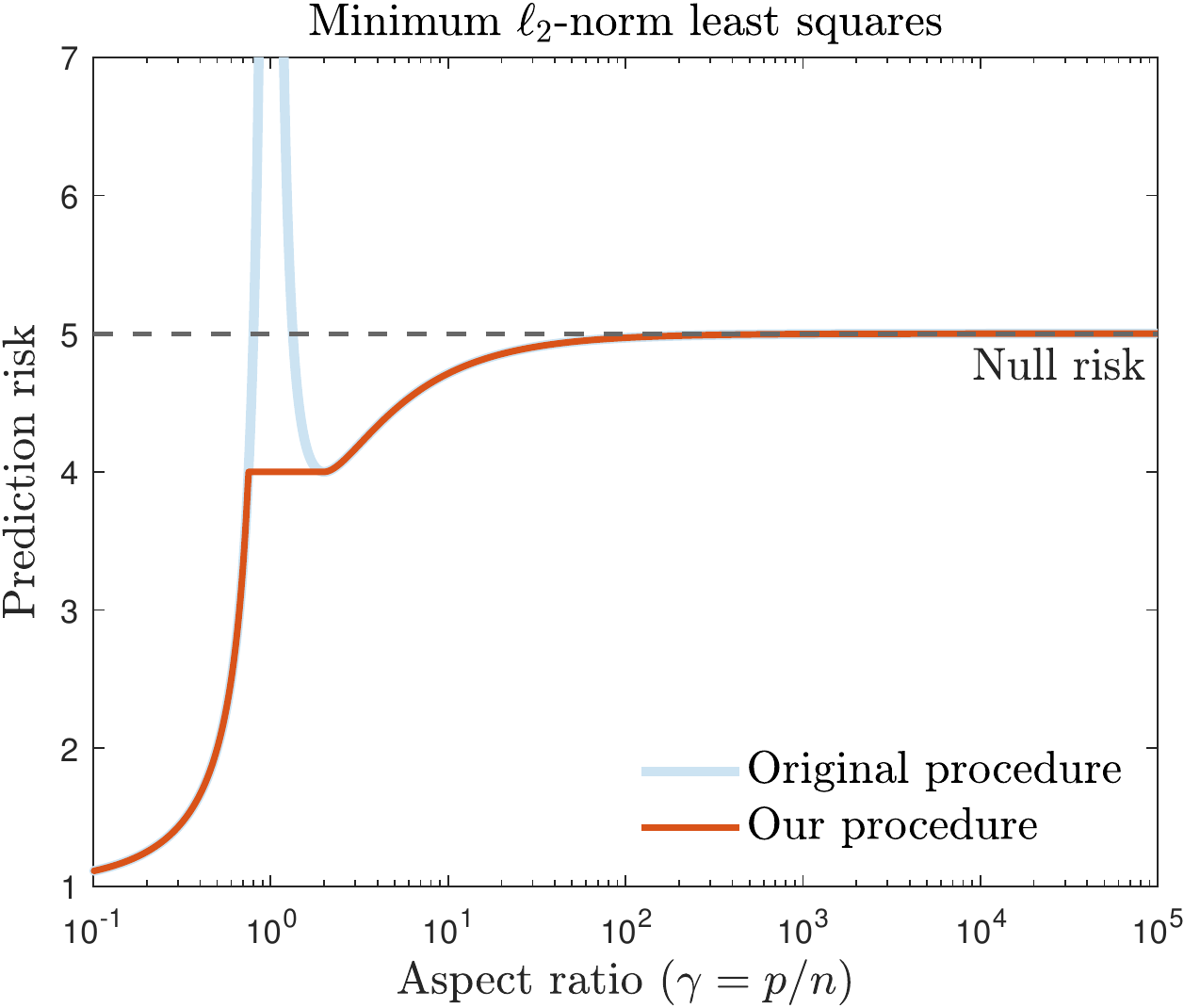}
    \quad
    \includegraphics[width=0.45\textwidth]{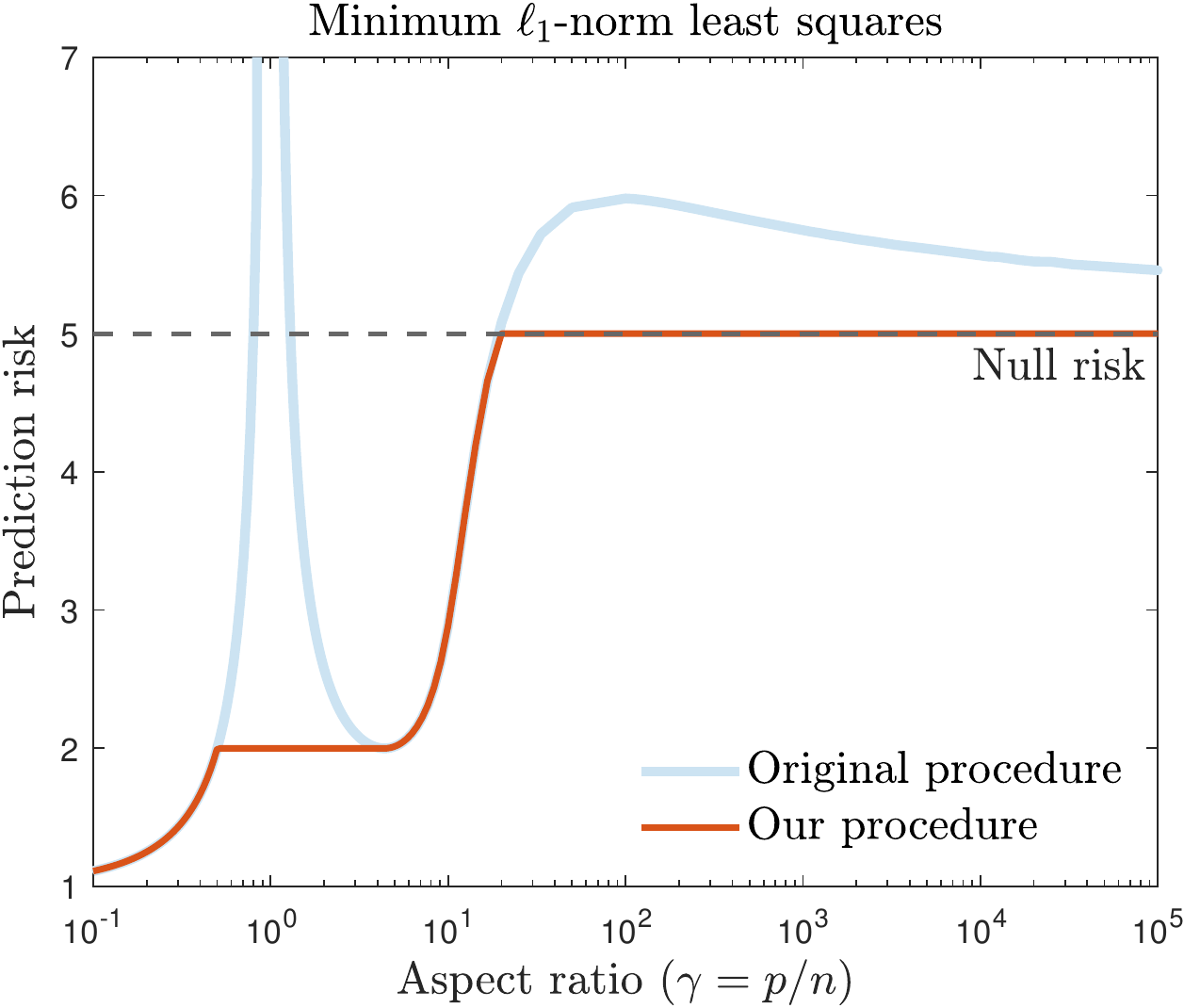}
    \caption{Monotonized asymptotic conditional prediction risk of the zero-step procedure (described in \Cref{alg:zero-step}) for 
    the minimum $\ell_2$-norm and $\ell_1$-norm
    least squares procedures.
    The figure in the left panel follows the setup of 
    Figure 2 of \cite{hastie_montanari_rosset_tibshirani_2019},
    and the figure in the right panel follows the setup of
    Figure 3 of \cite{li_wei_2021} (at sparsity level = 0.01). 
    Both settings assume isotropic features and a linear model with noise variance $\sigma^2 = 1$ and linear coefficients of squared Euclidean norm  $\rho^2 = 4$.
    Note that the risk is lower bounded by 
    $\sigma^2 = 1$
    and the risk of the null predictor (null risk) is 
    $\rho^2 + \sigma^2 = 5$. 
    }
\label{figure:illustrate-intro}
\end{figure}

\Cref{figure:illustrate-intro}
illustrates the above result for 
the minimum $\ell_2$-norm least squares estimator
\citep{hastie_montanari_rosset_tibshirani_2019}
and 
the minimum $\ell_1$-norm least squares estimator
\citep{li_wei_2021}. The light-blue lines show the asymptotic risk profiles of the two procedures, which are non-monotonic as they diverge to infinity around the interpolation threshold of $1$, at which the sample size and the number of features are equal. The red lines depict the risk profiles of the zero-step procedure $\hf^\zerostep$, which corresponds to the map
\begin{equation}\label{eq:monotonized}
\gamma \in (0,\infty) ~\mapsto~ \min_{\zeta \ge \gamma} R^\deter(\zeta; \hf).
\end{equation}
The function \eqref{eq:monotonized} is a monotonically non-decreasing function of $\gamma$, regardless of whether $\gamma \mapsto R^\deter(\gamma; \hf)$ is non-monotonic. Furthermore, since 
\[
\smash{\min_{\zeta \ge \gamma} R^\deter(\zeta; \hf)
\le R^\deter(\gamma; \hf)},   \mbox{ for all } \gamma > 0,
\]
the asymptotic risk of $\hf^{\zerostep}$
is no worse than that of $\hf$.
We refer to the function described in  \eqref{eq:monotonized} as the {\it monotonized risk of the base procedure} $\hf$.

The assumptions required in
\Cref{thm:asymptotic-risk-tuned-zero-step-informal} 
are very mild, and apply to a broad range of procedures and settings. 
Indeed, as remarked above, 
the risk profile $R^\deter(\cdot; \hf)$ of several estimators have been recently identified under proportional asymptotics
regime; see \Cref{rem:prop_asymptotics_risk_examples}.
The requirements on the loss functions are also mild
and can be verified for common loss functions.
In fact,
our results do not require
proportional asymptotics
and hold more generally.

We also develop a more sophisticated methodology whose asymptotic risk profile is not only monotonic in the aspect ratio but can be strictly smaller than the monotonized risk profile \eqref{eq:monotonized}, a fact that we again verify for the  minimum $\ell_2$, $\ell_1$-norm least squares procedures. See \Cref{sec:one-step}.

\paragraph{Core idea: the zero-step procedure.}
Our methodology is conceptually straightforward, as it relies on a combination of sample splitting, sub-sampling, and cross-validation. The core principle is as follows. Starting off  with an aspect ratio of $p/n$, if the risk were to be lower at, say, twice this aspect ratio $2p/n$, then we could just use half the data to evaluate the predictor, enjoying a smaller risk than the one obtained when training with the entire data. To decide whether the out-of-sample error is lower at any larger aspect ratio, we use cross-validation to ``glean at'' the values of the risk function at all aspect ratios larger than the one for the full data.
To elaborate, we next give an informal description of one of our main methods, the zero-step procedure that we study in \Cref{sec:zero-step}. 

We initially  split the data into  a training and a validation set in such a way that the size of the validation set is a vanishing proportion of that of the training set.
In the first step, we 
compute a collection of predictors,
each resulting from applying the same base prediction procedure on a
sub-sample of size $k_n$ varying over a grid of values in $\mathcal{K}_n$. Depending on the size of the sub-sample, we are able to mimic the behavior of the risk at larger aspect ratios ($p/k_n$, $k_n\in\mathcal{K}_n$). 
In the second step, we estimate the out-of-sample risk of each of these predictors using the validation set. 
With $\{p/k_n: k_n\in\mathcal{K}_n\}$ approximating the set $[p/n, \infty]$, these estimated out-of-sample risks act as proxies for the true generalization error at larger aspect ratios.
In the final step, we  perform model selection by minimizing the estimated test error across the candidate aspect ratios.
In order to make full use of the data, one can use more than one sub-sample for each $k_n\in\mathcal{K}_n$, a practice that closely resembles bagging.
To prove the ``correctness'' of the split-sample cross-validation,
we develop novel oracle inequalities in additive and multiplicative forms
that are of independent interest.

Because the core components of our approach are sub-sampling and cross-validation, our methodology is applicable to virtually any algorithm -- even the black-box type -- and its validity holds under minimal assumptions on the data generating distribution.

\subsection{Summary of results}
\label{sec:summary-contributions}

Below we summarize the main contributions of this paper.

\begin{itemize}
    \item 
    \textbf{Novel guarantees for split-sample cross-validation.} At its core, our methodology performs model selection of arbitrary learning procedures built over sub-samples of different sizes, with the size of the sub-samples treated as a tuning parameter to optimize. Towards that goal, we rely on split-sample cross-validation, which we analyze in \Cref{sec:general-crossvalidation-modelselection}.  In \Cref{prop:general-model-selection-guarantee}, we provide deterministic inequalities for
    the risk of 
    split cross-validated predictors 
    in both 
    additive and multiplicative form. We remark that multiplicative oracle inequalities allow for the possibility of unbounded oracle risk values, and are therefore well suited to incorporate prediction procedures exhibiting the double descent phenomena around the interpolating threshold. Leveraging
 concentration inequalities for both the  mean estimator of the prediction risk and the median-of-means estimator, in \Cref{sec:common-loss-functions}, we show how these bounds
    imply finite-sample oracle inequalities
    for split-sample cross-validation
    that are applicable to a broad range of loss functions and under minimal assumptions on the learning procedure. In particular, our results do not require well-specified (parametric) models. We exemplify our bounds on various loss functions for both regression and classification, and in \Cref{thm:oracle-bound-linear-predictor-squared-error}, we give a general  multiplicative oracle inequality for arbitrary linear predictors under mild distributional assumptions.
    
    \item
    \textbf{Zero-step procedure.} 
    Using oracle inequalities for split-sample cross-validation, we put forth a general methodology that takes as input an arbitrary prediction procedure and minimizes the prediction risk of its bagged version over a grid of 
    sub-sample sizes.
    We call this the ``zero-step'' prediction procedure.
    We analyze the asymptotic risk behavior of the zero-step
    procedure  under proportional asymptotics,
    in which the number of features grows proportionally with the number
    of observations.
    In  \Cref{thm:asymptotic-risk-tuned-zero-step}, we prove that
    the risk of predictor returned by the zero-step procedure
    is upper bounded
    by the monotonized risk given in \eqref{eq:monotonized}. Unlike most contributions in the literature on over-parameterized learning,
 our results do not depend on well-specified (parametric) models and only require the existence of a sufficiently well-behaved asymptotic risk profile.  
    
    \item
    \textbf{One-step procedure.}
In \Cref{sec:one-step},  we further generalize the zero-step procedure by considering an adjustment of the original predictor that is  inspired by the one-step estimation method used in parametric statistics to improve efficiency~\citep[Section 5.7]{vanderVaart_2000}. This modification, which can be thought of as a single-iterate boosting of the baseline procedure, is shown, both in theory and in simulations, to produce an asymptotic monotonized risk that is smaller than the monotonized risk of the zero-step procedure; see \Cref{thm:asymptotic-risk-tuned-one-step}. We derive explicit expressions of the asymptotic risk profile of the one-step procedure for the minimum $\ell_2$, $\ell_1$-norm least squares prediction procedures.
The main insight we draw from the minimum $\ell_2$-norm least squares
example is that the one-step procedure in addition
to changing the aspect ratio of the predictor
also reduces the signal energy
leading to a smaller asymptotic risk;
see \Cref{rem:zerostep-vs-onestep-isotropic}.

    \item 
    \textbf{Risk profiles}.
    In our study of the performance of the
    zero-step and one-step
    procedures,
    we derive several auxiliary results
    that might of independent interest.
    Specifically, we provide a systematic
    way to certify the continuity or lower semicontinuity
    of the asymptotic risk profile of any prediction procedure, 
    assuming only point-wise convergence of the conditional prediction risk under proportional asymptotics; see~\Cref{prop:continuity-from-continuous-convergence-rdet}. This is often hard to prove directly from the asymptotic risk profiles as they are usually defined implicitly via one or more fixed-point equations.
    Also of independent interest
    is a representation that we prove,
    for the conditional prediction risk
    of an arbitrary linear predictor with a one-iterate
    boosting with minimum $\ell_2$-norm least squares,
    using the recent tools from random matrix theory.
    This, in particular, involves
    deriving deterministic equivalents for
    the generalized bias and variance
    of the ridgeless predictor which may be of independent interest; see~\Cref{lem:conditional-deterministic-convergence-onestep,lem:one-step-predrisk-decomposition}.
    
\end{itemize}

We corroborate our theoretical results with several illustrative simulations. An intriguing finding emerging from our numerical studies is the fact that bagging, i.e., aggregation over sub-sample, appears to have a significant positive impact on the asymptotic risk profile of both the zero- and one-step procedure: averaging over an increasing number of sub-samples results in a downward shift of the risk asymptotic profile, especially around the interpolation threshold: see, e.g., \Cref{fig:gaufeat_isocov_sparsesig_linmod_n500_nv80_nruns250_zerostep_mnla_gamma100,fig:gaufeat_isocov_isosig_linmod_n1000_nv100_nruns100_zerostep_mnls_gamma10}.  Though we do not provide a theoretical justification for this interesting phenomenon, we offer some conjectures in the discussion section; see~\Cref{sec:discussion}. 

\subsection{Other related work}
\label{sec:related-work}

In this section, we review some related work on risk non-monotonicity, cross-validation, as well as exact asymptotic risk characterization. 
Explicit references to these works, when appropriate, are also made in the main sections of the paper.

\paragraph{Non-monotonicity of generalization performance.}
The study of non-monotone risk behavior is largely motivated by empirical evidence in standard 
statistical learning tasks such as classification and prediction, where instances of non-monotonic risk profiles were 
originally  discovered and reported. See \cite{trunk1979problem,duin1995small,opper1996statistical} and \cite{loog_viering_mey_krijthe_tax_2020} for some earlier findings on the double descent
risk behavior.
Recently, it has garnered growing interest due to the remarkable successes of neural networks  where similar non-monotonic behavior has also been observed; see \cite{lecun1990second,geiger2019jamming,zhang_bengio_hardt_recht_vinyals_2016,zhang2021understanding} and references therein. 
The non-monotonic behavior of the test error as a function of the model size in general context was brought up by \cite{belkin2019reconciling} and has since been theoretically established for many other classical estimators such as linear/kernel regression, ridge regression, logistic regression, and under stylized models such as linear model or random features model.
Besides the work discussed in our main sections, see also \cite{kini2020analytic,Mei2019generalization,mitra2019understanding,derezinski2020exact,frei2022benign} and the survey paper \cite{bartlett_montanari_rakhlin_2021}.
When it comes to the sample-wise non-monotonic performances, 
a recent line of work asks and provides partial answers to the question: 
given additional observation points, when and to what extend 
will the generalization performance improve \citep{viering2019open,nakkiran2019more,nakkiran_venkat_Kakade_Ma-2020,mhammedi2021risk}.
In particular, \cite{nakkiran_venkat_Kakade_Ma-2020} investigates the role of optimal tuning in the context of ridge regression, and for a class of linear models, demonstrated that the optimally-tuned $\ell_{2}$ regularization achieves monotonic generalization performance. 

\paragraph{Data-splitting and cross-validation.}
The framework developed in the current paper crucially depends on split-sample cross-validation, which compares different predictors trained on one part of the sample using out-of-sample risk estimates from the remaining part. The split-sample cross-validation is a well-known methodology studied in several works (e.g.,~\cite{stone1974cross,gyorfi2002distribution,yang2007consistency,arlot2010survey}). Split-sample cross-validation is theoretically easier to analyze compared to the $k$-fold cross-validation and is shown to yield optimal rates in the context of non-parametric regression~\citep{yang2007consistency,van2007super,van2006oracle}. These works have derived oracle inequalities that show that split-sample cross-validation based predictor has asymptotically the smallest risk among the collection of predictors up to an additive error (that converges to zero). 
The oracle inequalities are either called exact or inexact depending on whether the constant multiplying the smallest risk is 1 or $1 + \delta$ (for an arbitrarily $\delta$); see, e.g.,~\cite{lecue2012general}. All these works have used split-sample cross-validation for the purpose of choosing predictors with good prediction risk,
and the existing oracle inequalities are all additive in nature. 

Application of cross-validation for over-parameterized learning is more recent and here special care is required in choosing the split sizes because splitting in half would change the aspect ratios in the proportional asymptotics regime. In contrast to the low dimensional or non-parametric setting, it is well-known that the classical $k$-fold cross-validation framework suffers from severe bias and thus requires careful modification or a diverging choice of $k$ (see, e.g.,~\cite{MRRK-2021,rad2020scalable}). 
In particular, when $k$ is taken to be $n$, the resulting procedure is also known as leave-one-out cross-validation (LOOCV), 
which mitigates these bias issue 
and has proven to be effective in a variety of settings; 
see ~\cite{beirami2017optimal,wang2018approximate,giordano2019swiss,stephenson2020approximate,wilson2020approximate,austern_zhou_2020,xu2021consistent,patil2021uniform,patil2022estimating} and references therein.

Our use of cross-validation is slightly different: the goal is to choose the ``optimal'' sub-sample size for a single prediction procedure. Furthermore, supplementing the existing oracle inequalities for cross-validation, we also provide a multiplicative oracle inequality which shows that the split-sample cross-validated predictor attains the smallest risk in the collection up to a factor converging to $1$ with the sample size. This multiplicative version is crucial for our study, allowing us to consider ingredient predictors whose risk might diverge with sample size.

\paragraph{Risk characterization.}
In developing our zero-step and one-step procedures, we assume existence of a deterministic risk profile function for every aspect ratio. 
As discussed, the exact formulas for the risk profile functions have been obtained for various estimators in both classification and regression settings.
In the past decade, several distinct techniques and tools have been developed to explicitly describe and analyze these risk functions. 
Prominent examples include the leave-one-out type perturbation analysis (e.g.,~\cite{karoui2013asymptotic,karoui2018impact}), 
the approximate message passing machinery (e.g.,~\cite{donoho2009message,donoho2016high,bayati2011dynamics}),
and the convex Gaussian min-max theorem (e.g.,~\cite{stojnic2013framework,thrampoulidis2015regularized,thrampoulidis2018precise}). 
These techniques 
rely critically upon a well-specified model, as well as the assumption that the entries of the design matrix are drawn i.i.d.\ from standard normal distribution, while some restricted universality results are developed in \cite{bayati2015universality,montanari2017universality,chen2021universality,hu2020universality}.
In this work, however, we take a more direct approach and develop some non-asymptotic oracle risk inequalities.
Leveraging upon these oracle inequalities, our results do not require well-specified models, and only assume the existence of a relatively well-behaved risk profile, which presumably allows for weaker distributional assumptions. 

\subsection{Organization and notation}

\paragraph{Organization.}
The rest of the paper is organized as follows.
\begin{itemize}
    \item 
    In \Cref{sec:general-crossvalidation-modelselection},
    we describe the general cross-validation and model selection algorithm,
    derive associated oracle risk inequalities,
    and provide probabilistic bounds on the error terms.
    We then obtain concrete results 
    for a variety of classification and regression loss functions.
    \item
    In \Cref{sec:zero-step}, 
    we describe the zero-step prediction procedure,
    and provide its risk monotonization guarantee.
    We then explicitly verify the related assumptions
    for the ridgeless and lassoless prediction procedures,
    and show corresponding numerical illustrations.
    \item
    In \Cref{sec:one-step},
    we describe the one-step prediction procedure,
    and provide its risk monotonization guarantee.
    We then explicitly verify assumptions
    for arbitrary linear predictors,
    the special cases of ridgeless and lassoless prediction procedures,
    and show corresponding numerical illustrations.
    \item
    In \Cref{sec:discussion},
    we conclude the paper 
    and provide three concrete directions for future work.
\end{itemize}

Nearly all the proofs in the paper are deferred to the Supplementary Material.
The sections and the equation numbers in the Supplementary Material
are prefixed with the letters ``S'' and ``E'', respectively.

\paragraph{Notation.}

We use $\NN$ to denote the set of natural numbers,
$\RR$ to denote the set of real numbers,
$\RR_{\ge 0}$ to denote the set of non-negative real numbers,
$\RR_{> 0}$ to denote the set of positive real numbers,
and  $\overline{\RR}$ to denote the extended real number system,
i.e., $\overline{\RR} = \RR \cup \{ -\infty, +\infty \}$.
For a real number $a$,
$(a)_+$ denotes its positive part,
$\lfloor a \rfloor$ denotes its floor, 
$\lceil a \rceil$ denotes its ceiling.
For a set $\cA$,
we use $\smash{\1_{\cA}}$ to denote its indicator function.
We denote 
convergence in probability by $\smash{\pto}$, 
almost sure convergence by $\smash{\asto}$,
and weak convergence by $\smash{\dto}$.
We use generic letters $C, C_1, C_2, \dots$ to denote constants
whose values may change from line to line.

For a comprehensive list of notation used in the paper,
see \Cref{sec:notation}.

\section{General cross-validation and model selection}
\label{sec:general-crossvalidation-modelselection}

The primary focus of this paper is 
to develop a framework to improve upon 
prediction procedures in the overparameterized regime
in which the number of features $p$ is comparable to
and often exceeds the number of observations $n$,
and where the predictive risk may be non-monotonic in the aspect ratio $p/n$.
As discussed in \Cref{sec:introduction}, 
a fundamental component of our methodology is 
the selection of an optimal size of the sub-samples 
through cross-validation.  
To that effect, we begin by deriving some general, 
non-asymptotic oracle risk inequalities
for split-sample cross-validation, as described in
\Cref{alg:general-cross-validation-model-selection},
that hold under minimal assumptions.
While our bounds apply to a wide range of learning problems
and may be of independent interest, they are crucial
in demonstrating the risk monotonization properties of the procedures
presented in \Cref{sec:zero-step,sec:one-step}.

Though cross-validation is a well-known and well-studied procedure
~\citep[see, e.g.,][]{van2007super,gyorfi2002distribution,yang2007consistency},
our work extends the previous results on cross-validation
in a couple of ways:
(1) We derive two forms of oracle risk inequalities:
the additive form that is better suited for bounded loss functions 
(especially classification losses), 
and the multiplicative form that is better suited unbounded loss functions
(especially regression losses); 
(2) In addition to common sample mean based
estimation of the prediction risk,
we also analyze the median-of-means based
estimation of the prediction risk
that proves to be useful
in relaxing strong moment 
assumption on the predictors.

\begin{algorithm}
    \caption{General cross-validation and model selection procedure}
    \label{alg:general-cross-validation-model-selection}
    \textbf{Inputs}:\\
    \begin{itemize}[noitemsep]
        \item[--] a dataset $\cD_n = \{ (X_i, Y_i) \in \RR^{p} \times \RR : 1 \le i \le n \}$; 
        \item[--] a positive integer $n_\test < n$; 
        \item[--] an index set $\Xi$; 
        \item[--] a set of prediction procedures $\{ \hf^\xi$: $\xi \in \Xi \}$; 
        \item[--] a loss function $\ell : \RR \times \RR \to \RR_{\ge 0}$; 
        \item[--] a centering procedure $\CEN \in \{ \AVG, \MOM \}$;
        \item[--] a real number $\eta > 0$ if $\CEN$ is $\MOM$. 
    \end{itemize}
    \textbf{Output:}\\
        \vspace{-1em}
        \begin{itemize}[noitemsep]
            \item[--] a predictor $\hf^\cv(\cdot; \cD_n) : \RR^{p} \to \RR$.
        \end{itemize}
    \textbf{Procedure:}
    \begin{enumerate}
        \item
        Randomly split the index set $\cI_n = \{ 1, \dots, n \}$
        into two disjoint sets
        $\mathcal{I}_\train$ and $\mathcal{I}_\test$ such that
        $| \mathcal{I}_\train | = n - n_\test$ (which we denote by $n_\train$),
        $| \mathcal{I}_\test | = n_\test$.
        Denote the corresponding splitting of the dataset $\cD_n$ by
        $\cD_\train = \{ (X_i, Y_i) : i \in \mathcal{I}_\train \}$ (for training)
        and
        $\cD_\test = \{ (X_j, Y_j) : j \in \mathcal{I}_\test \}$ (for testing).
        \item
        For each $\xi \in \Xi$,
        fit the prediction procedure $\hf^\xi$ on $\cD_\train$
        to obtain the predictor $\hf^\xi(\cdot; \cD_\train) : \RR^{p} \to \RR$.
        \item
        For each $\xi \in \Xi$,\\
        \begin{itemize}
            \item
            if $\CEN~=~\AVG$,
            estimate the conditional prediction risk of $\hf^\xi$ using
            \begin{align}
            \label{eqn:avg-risk-pre}
                \hR(\hf^\xi(\cdot; \cD_\train))
                = \frac{1}{|\cD_\test|} \sum_{j \in \mathcal{I}_\test}
                \ell(Y_j, \hf^\xi(X_j; \cD_\train)).
            \end{align}
                 
            \item 
            if $\CEN~=~\MOM$,
            estimate the conditional prediction risk of $\hf^\xi$ using
            \begin{align}
            \label{eqn:mom-risk-pre}
              \hR(\hf^\xi(\cdot; \cD_\train)) 
                = \MOM
                \big(
                    \big\{ \ell(Y_j, \hf^\xi(X_j; \cD_\train)), \, j \in \mathcal{I}_\test \big\},
                    \, \eta
                \big).  
            \end{align}
            See discussion after \Cref{lem:mom-concentration}
            for the definition of $\MOM(\cdot, \cdot)$.
        \end{itemize}
        \item Set $\widehat{\xi}\in\Xi$ to be the index
        that minimizes the estimated prediction risk using
        \begin{align}
        \label{eqn:model-selection-xi}
            \widehat{\xi}
            \in \argmin_{\xi \in \Xi} \hR(\hf^\xi(\cdot; \cD_\train)).
        \end{align}
        Note that $\widehat{\xi}$ need not be unique
        (hence the set notation)
        and any choice
        that leads to the minimum estimated risk
        enjoys the subsequent theoretical guarantees in the paper.
        
        \item Return the predictor $\hf^\cv(\cdot; \cD_n) = \hf^{\widehat{\xi}}(\cdot; \cD_\train)$.
    \end{enumerate}
\end{algorithm}

\subsection{Oracle risk inequalities}
\label{sec:oracle-risk-inequalities}

Setting the stage, suppose we are given $n$ samples of labeled data
\smash{$\cD_n = \{ (X_1, Y_1), (X_2, Y_2), \dots, (X_n, Y_n) \}$},
where $X_i \in \RR^p$ is a $p$-dimensional feature vector
and $Y_i \in \RR$ is a scalar response variable
for $i = 1, \dots, n$.
Let $\hf$ be a prediction procedure that maps $\cD_n$
to a predictor $\hf(\cdot; \cD_n) : \RR^{p} \to \RR$
(a measurable function of the data $\cD_n$).
For any predictor 
$\hf(\cdot; \cD_n)$,
trained on the data set $\cD_n$,
that takes in a feature vector $x \in \RR^{p}$
and outputs a real-valued prediction $\hf(x; \cD_n)$,
we measure its predictive accuracy 
via a non-negative loss function
$\ell : \RR \times \RR \to \RR_{\ge 0}$.
Given a new feature vector $X_0 \in \RR^{p}$
with associated response variable $Y_0 \in \RR$
so that $(X_0, Y_0)$ is independent of $\cD_n$,\footnote{We will reserve the notation $(X_0, Y_0)$
to denote a random variable that is drawn independent
of $\cD_n$.}
the prediction error or out-of-sample error incurred by $\hf(\cdot; \cD_n)$ is $\ell(Y_0, \hf(X_0; \cD_n))$. Note that the prediction error $\ell(Y_0, \hf(X_0; \cD_n))$
is a random variable that is a function of both $\cD_n$ and $(X_0, Y_0)$.

We will quantify the performance of $\hf(\cdot; \cD_n)$
using the conditional expected prediction loss.
The conditional expected prediction loss
given the data $\cD_n$, or the conditional prediction risk for short,
of $\hf(\cdot; \cD_n)$ is defined as
\begin{align}
\label{eq:prediction-risk}
    R(\hf(\cdot; \cD_n))
    ~:=~ \EE_{X_0, Y_0} [ \ell(Y_0, \hf(X_0; \cD_n)) \mid \cD_n ]
    ~=~ \int \ell(y, \hf(x; \cD_n)) \; \mathrm{d}P(x, y),
\end{align}
where $P$ denotes the joint probability distribution of $(X_0, Y_0)$.
Note that $  R(\hf(\cdot; \cD_n))$ is a random variable
that depends on $\cD_n$.
An empirical estimator of $R(\hf(\cdot; \cD_n))$
is denoted by $\hR(\hf(\cdot; \cD_n))$. In this paper, we mainly consider two such estimators: the average estimator and the median-of-means estimator 
as defined in \eqref{eqn:avg-risk-pre} and \eqref{eqn:mom-risk-pre}, respectively.

Consider any prescribed index set $\Xi$, where each $\xi \in \Xi$ corresponds
to a specific model that will be clear from the context. 
Based on the training data, a predictor $\hf^\xi(\cdot; \cD_\train)$ is fitted
for each model $\xi$ and estimated risks of $\hf^\xi$, $\xi \in \Xi$ 
are compared on a validation data set as described
in \Cref{alg:general-cross-validation-model-selection}. 
Let $\hf^\cv(\cdot; \cD_n)$ be the final predictor returned by
\Cref{alg:general-cross-validation-model-selection}.
We shall consider two types of oracle inequalities: one in an additive form
and the other in a multiplicative form. More specifically, for any prescribed model set $\Xi$, 
define the additive error term and multiplicative error term respectively as follows:
\begin{subequations}
    \begin{align}
        \Delta_n^{\add}
        &:= \max_{\xi \in \Xi}
        \Big|
        \hR(\hf^\xi(\cdot; \cD_\train))
        -
         R(\hf^\xi(\cdot; \cD_\train)) 
         \Big|,
        \label{eq:Delta_n_add}
        \\
        \Delta_n^\mul
        &:= \max_{\xi \in \Xi}
        \Big\vert
            \frac{\hR(\hf^\xi(\cdot; \cD_\train))}{R(\hf^\xi(\cdot; \cD_\train))} - 1 
        \Big\vert. \label{eq:Delta_n_mul}
    \end{align}
\end{subequations}
The following proposition relates
the performance of $\hf^\cv(\cdot; \cD_n)$ 
to the ``oracle'' prediction risk 
in terms of these errors terms.

\begin{proposition}
    [Deterministic oracle risk inequalities]
    \label{prop:general-model-selection-guarantee}
    The prediction risk of $\hf^\cv(\cdot; \cD_n)$ satisfies
    the following deterministic oracle inequalities:
    \begin{enumerate}
        \item additive form: 
    \begin{equation}\label{eq:model-free-guarantee-general-model-selection}
    \begin{split}
        R(\hf^\cv(\cdot; \cD_n))
        ~&\leq~ \min_{\xi \in \Xi} R(\hf^{\xi}(\cdot; \cD_\train) + 2\Delta_n^{\add}, \\
        \EE[R(\hf^\cv(\cdot; \cD_n))]
        ~&\leq~ \min_{\xi \in \Xi} \EE[R(\hf^\xi(\cdot; \cD_\train)]
        + 2\mathbb{E}[\Delta_n^{\add}].
        \end{split}
    \end{equation}
    \item 
     multiplicative form:
    \begin{align}
        \label{eq:oracle-risk-inequality-multiplicative-form}
        R(\hf^\cv(\cdot; \cD_n))
        ~\leq~ 
        \frac{1 + \Delta_n^\mul}{(1 - \Delta_n^\mul)_{+}}
        \cdot \min_{\xi \in \Xi} R(\hf^\xi(\cdot; \cD_\train).
    \end{align}
    \end{enumerate}
\end{proposition}

\Cref{prop:general-model-selection-guarantee}
provides oracle bounds on the prediction risk of $\hf^\cv(\cdot; \cD_n)$
in terms of the error terms $\Delta_n^{\add}$ and $\Delta_n^{\mul}$.
Note that \Cref{prop:general-model-selection-guarantee} does not make any assumptions
about the underlying model of the data or the dependence structure between the observations.
Under some general conditions on the data, one can show that $\Delta_n^{\add}$
and/or $\Delta_n^{\mul}$ converge to zero in probability as $n\to\infty$.
The exact rate of convergence depends on the number of observations $n_{\test}$
in the test data and also on the tail behavior of
$\ell(Y_0, \widehat{f}^{\xi}(X_0; \cD_\train))$ conditional on 
$\widehat{f}^{\xi}(\cdot; \cD_\train)$.
For notational convenience, from now,
we will write $\hf^\cv$ and $\hf^\xi$
to  denote $\hf^\cv(\cdot; \cD_n)$
and $\hf^\xi(\cdot; \cD_\train)$, respectively.

\begin{remark}
    [Lower bound on $R(\hf^\cv)$]
    \label{rem:lower-bound-risk-hfcv}
    \Cref{prop:general-model-selection-guarantee}
    provides upper bounds on the (conditional)
    prediction risk of $\hf^\cv$
    in terms of the minimum risk of $\hf^\xi$.
    It can be readily seen that
    the risk of $\hf^\cv$ is always lower bounded
    by the minimum risk.
    More formally,
    note that $\hf^\cv = \sum_{\xi \in \Xi} \hf^\xi \1_{\hxi = \xi}$, and, therefore,
    \[
        R(\hf^\cv)
        ~=~
        \sum_{\xi \in \Xi} R(\hf^\xi) \1_{\hxi = \xi}
        ~\ge~
        \min_{\xi \in \Xi} R(\hf^\xi) \sum_{\xi \in \Xi} \1_{\hxi = \xi}
        ~=~
        \min_{\xi \in \Xi} R(\hf^\xi).
    \]
    Combined with \Cref{prop:general-model-selection-guarantee},
    we conclude that
    \[
        \min_{\xi \in \Xi} R(\hf^\xi)
        ~\le~
        R(\hf^\cv)
        ~\le~
        \begin{cases}
            \min_{\xi \in \Xi} R(\hf^\xi) + \Delta_n^\add \\
            \min_{\xi \in \Xi} R(\hf^\xi)
            \cdot
            (1 + \Delta_n^\mul) / (1 - \Delta_n^\mul)_+.
        \end{cases}
    \]
    Thus,
    convergence (in probability) of either
    $\Delta_n^\add$ or $\Delta_n^\mul$
    to $0$
    implies that
    the risk of $\hf^\cv$ is asymptotically the same
    as the minimum risk of $\hf^\xi$, $\xi \in \Xi$
    in either additive or multiplicative sense, respectively.
\end{remark}
The additive and multiplicative form of oracle inequalities have their own advantages.
Traditionally, the additive form is more common. 
The additive oracle inequality for the prediction risk
readily implies the additive oracle inequality on the excess risk.
In other words,
\[
    R(\hf^\cv) - R(f^\star)
    \le
    \min_{\xi \in \Xi} R(\hf^\xi) - R(f^\star) + \Delta_n^\add,
\]
for any predictor $f^\star$.
In particular, this will hold for the best (oracle) predictor
for the prediction risk.
This is not true of the multiplicative oracle inequality,
which instead only implies the bound
\[
    R(\hf^\cv) - R(f^\star)
    \le
    c_n
    \big\{ \min_{\xi \in \Xi} R(\hf^\xi) - R(f^\star) \big\}
    + (c_n - 1) R(f^\star),
\]
where $f^\star$ is any predictor (in particular,
the one with the best prediction risk) and
\[
    c_n = \frac{1 + \Delta_n^\mul}{(1 - \Delta_n^\mul)_+},
    \quad
    c_n - 1  = \frac{2 \Delta_n^\mul}{(1 - \Delta_n^\mul)_+}.
\]

In terms of claiming that $\hf^\cv$ has prediction risk
close to the best in the collection of predictors
$\{ \hf^\xi, \xi \in \Xi \}$,
the multiplicative form has certain advantages compared to the additive form.
In the case that $\min_{\xi \in \Xi} R(\hf^\xi)$ converges to $0$,
the additive oracle inequality~\eqref{eq:model-free-guarantee-general-model-selection} implies that the risk of the selected predictor $\widehat{f}^{\cv}$ asymptotically matches the risk of the \emph{best} predictor among the collection $\{ \widehat{f}^{\xi}, \xi\in\Xi \}$
only if $\Delta_n^{\add}$ converges to zero faster than $\min_{\xi\in\Xi}R(\widehat{f}^{\xi})$.
If, however, $\Delta_n^{\add}$ converges to zero slower than the minimum risk in the collection, then the additive oracle inequality does not imply a favorable result. In this case, a multiplicative oracle inequality helps.
As long as $\Delta_n^\mul$ converges to 0,
the multiplicative oracle inequality implies
that $\hf^\cv$ matches in risk with the best predictor in the collection, irrespective of whether the minimum risk converges to zero or not.
Note that $\Delta_n^{\add}$ only controls the additive error of the risk estimator $\widehat{R}(\widehat{f}^{\xi})$, which is easier to control than the multiplicative error; think of controlling the error of sample mean of $\mbox{Bernoulli}(p)$ random variables with $p = p_n\to0$;
See \Cref{rem:compare-Delta_n-add-vs-Delta_n-mul} for a more mathematical discussion.
Even when $\min_{\xi \in \Xi} R(\hf^\xi)$
does not converge to zero,
the multiplicative form might be advantageous compared to the additive form.
Indeed, suppose that $\hf^{\xi_0}$ is in the collection
and its risk diverges as $n \to \infty$.
Then, it may not be true that
\[
    \big|
        \hR(\hf^{\xi_0})
        - R(\hf^{\xi_0})
    \big|
    \overset{p}{\to} 0,
\]
because
both $\hR(\hf^{\xi_0})$ and $R(\hf^{\xi_0})$
are diverging.
This implies that $\Delta_n^\add$ does not converge to $0$
and in fact, might diverge.
However,
the minimum risk in the collection could still be finite,
and the additive oracle inequality fails to capture this.
On the other hand,
$\hR(\hf^{\xi_0}) / R(\hf^{\xi_0})$ can still converge to $1$
as $n \to \infty$ even if $R(\hf^{\xi_0})$ diverges to $\infty$.
In our applications
in overparameterized learning,
we will encounter this situation where the number of features
($p$)
is close to the number of observations ($n$),
i.e., $p/n \approx 1$.
See \Cref{rem:divergence-Delta_n^add} for more details.

\begin{remark}
    [From multiplicative to additive oracle inequality]
    Note that if $\Delta_n^{\mul} = o_p(1)$, then $(1 + \Delta_n^{\mul})/(1 - \Delta_n^{\mul})_+ = 1 + O_p(1)\Delta_n^{\mul} = 1 + o_p(1)$,
    then the multiplicative oracle inequality~\eqref{eq:oracle-risk-inequality-multiplicative-form} 
    yields
    \[
    R(\widehat{f}^{\cv}) ~\le~ (1 + O_p(1)\Delta_n^{\mul})\min_{\xi\in\Xi} R(\widehat{f}^{\xi}) ~=~ (1 + o_p(1))\min_{\xi\in\Xi}R(\widehat{f}^{\xi}).
    \]
    Observe that this multiplicative form can be converted into an additive form as
    \[
    R(\widehat{f}^{\cv}) ~\le~ \min_{\xi\in\Xi}R(\widehat{f}^{\xi}) ~+~ O_p(1)\Delta_n^{\mul}\min_{\xi\in\Xi}R(\widehat{f}^{\xi}),
    \]
    where the second term on the right hand side is always smaller order compared to the first term as long as $\Delta_n^{\mul}$ converges in probability to zero.
\end{remark}

From this discussion, it follows that one can choose a predictor with the \emph{best} prediction risk in a collection if either $\Delta_n^{\add}$ or $\Delta_n^{\mul}$ converges in probability to zero. The application of \Cref{alg:general-cross-validation-model-selection} for risk monotonizing procedures will be discussed in the next three sections. In the next two subsections, we provide some general sufficient conditions to verify $\Delta_n^{\add} = o_p(1)$ and $\Delta_n^{\mul} = o_p(1)$ for independent data. We also provide examples of common loss functions and show that under some mild moment assumptions, they satisfy $\Delta_n^{\add} = o_p(1)$ and $\Delta_n^{\mul} = o_p(1)$.

\subsection
[Control of error terms]
{Control of $\Delta_n^{\add}$ and $\Delta_n^{\mul}$}
\label{subsec:control-of-error-terms}

In order to characterize $R(\hf^\cv)$, 
by \Cref{prop:general-model-selection-guarantee}
it is sufficient to control 
$\Delta_n^{\add}$ and $\Delta_n^{\mul}$.  
In this section, we demonstrate that under certain assumptions
on the loss function $\ell$,
the error terms are small both in probability and in expectation, which in turn yields optimality of $\hf^\cv$ among the predictors in $\{ \widehat{f}^{\xi}, \xi\in\Xi \}$.

To facilitate our discussion, for each $\xi \in \Xi$,
define the conditional $\psi_1$-Orlicz norm
of $\ell(Y_0, \hf^\xi(X_0))$ given  $\cD_n$ as
\begin{equation}\label{eq:conditional-Orlicz}
    \| \ell(Y_0, \hf^\xi(X_0)) \|_{\psi_1 \mid \cD_n}
    := \inf
    \big\{
        C > 0: \,
        \mathbb{E}
        \big[
            \exp\big( |\ell(Y_0, \hf^\xi(X_0))| / C \big) \mid \mathcal{D}_n
        \big] 
        \le 2
    \big\}.
\end{equation}
Similarly, for $r \ge 1$, define the conditional $L_r$-norm as
\begin{equation}\label{eq:conditional-Lp}
    \|\ell(Y_0, \hf^{\xi}(X_0))\|_{L_r \mid \cD_n}
    :=
    \big(
         \mathbb{E}
         \big[
             \big|\ell(Y_0, \hf^{\xi}(X_0))\big|^r \mathrel{\big|} \cD_n
         \big]
    \big)^{1/r}.
\end{equation}
It is well-known~\citep[Proposition 2.7.1]{vershynin_2018} that 
\[
\|\ell(Y_0, \widehat{f}^{\xi}(X_0))\|_{\psi_1|\mathcal{D}_n} ~\asymp~ \sup_{r \ge 1} r^{-1}\|\ell(Y_0, \widehat{f}^{\xi}(X_0))\|_{L_r|\mathcal{D}_n},
\]
i.e.,
there are absolute constants $C_l$ and $C_u$ such that
\[
    0< C_l
    \le
    \frac
    {\| \ell(Y_0, \hf^\xi(X_0)) \|_{\psi_1 \mid \cD}}
    {\sup_{r \ge 1} r^{-1} \| \ell(Y_0, \hf^\xi(X_0)) \|_{L_r \mid \cD_n}}
    \le
    C_u<\infty.
\]

\subsubsection
[Additive form]
{Control of $\Delta_n^{\add}$}

Let $\hf^\xi$, $n_\test$,
and $\CEN$ be as defined in \Cref{alg:general-cross-validation-model-selection},
and $\Delta_n^{\add}$ be as defined in  \eqref{eq:Delta_n_add}.
    
\begin{lemma}
    [Control of $\Delta_n^{\add}$ and its expectation
    for losses with bounded conditional $\psi_1$ norm]
    \label{lem:bounded-orlitz-error-control}
    Suppose $(X_i, Y_i), i \in \mathcal{I}_\test$ are sampled i.i.d.\
    from $P$.
    Suppose the loss function $\ell$ is such that
    $$\| \ell(Y_0, \hf^\xi(X_0)) \|_{\psi_1 \mid \cD_n} \le \hsigma_\xi$$
    for $(X_0, Y_0) \sim P$ and 
    set
    $\hsigma_{\Xi} := \max_{\xi \in \Xi} \hsigma_\xi$.
    Fix any $0 < A < \infty$.
    Then,
    for $\CEN = \AVG$,
    or $\CEN = \MOM$ with $\eta = n^{-A} / | \Xi |$,
    \footnote{See \Cref{rem:restrict_A_eta}.}
    there exists an absolute constant $C_1 > 0$
    such that
    \[
        \PP
        \left(
            \Delta_n^{\add}
            \ge
            C_1
            \hsigma_{\Xi}
            \max
            \left\{
                \sqrt{\frac{\log\left(|\Xi| n^A\right)}{n_\test}},
                \frac{\log\left(|\Xi| n^A\right)}{n_\test}
            \right\}
        \right)
        \le
        n^{-A}.
    \]
    Additionally, if for some $A > 0$, there exists a $C_2 > 0$ such that $\PP(\hsigma_{\Xi} \ge C_2) \le n^{-A}$, then
    there exists an absolute constant $C_3 > 0$
    such that
    \begin{equation}\label{eq:bounded-orlicz-expectation-bound}
        \begin{split}
           \mathbb{E}[\Delta_n^{\add}]
           &\le C_1 C_2 
           \max
           \left\{
                \sqrt{\frac{\log\left(|\Xi| n^A\right)}{n_\test}},
                \frac{\log\left(|\Xi| n^A\right)}{n_\test}
            \right\}
            + C_3
            n^{-A/r}
            |\Xi|^{1/t}
            \max
            \left\{
                 \sqrt{\frac{t}{n_\test}}, \frac{t}{n_\test}
            \right\}
            \max_{\xi \in \Xi}
            \| \hsigma_\xi \|_{L_t}
       \end{split}
    \end{equation}
for every  $r, t \ge 2$ and $1/r + 1/t = 1$.
\end{lemma}

\begin{lemma}[Control of $\Delta_n^{\add}$ and its expectation
for losses with bounded conditional $L_2$ norm]
\label{lem:bounded-variance-error-control}
    Suppose $(X_i, Y_i), i \in \mathcal{I}_{\test}$ are sampled i.i.d.\
    from $P$.
    Suppose the loss function $\ell$ is such that
    $$\| \ell(Y_0, \hf^\xi(X_0)) \|_{L_2 \mid \cD_n} \le \hsigma_\xi$$
    for $(X_0, Y_0) \sim P$ and set
     $\hsigma_\Xi := \max_{\xi \in \Xi} \hsigma_\xi$.
    Fix any $0 < A < \infty$.
    Then,
    for $\CEN = \MOM$ with $\eta = n^{-A} / | \Xi |$,
    there exists an absolute constant $C_1 > 0$
    such that
    \begin{equation}\label{eq:probabilistic.bound.Delta.add}
        \PP
        \left(
            \Delta_n^{\add}
            \ge
            C_1
            \hsigma_{\Xi}
            \sqrt{\frac{\log(|\Xi| n^A)}{n_\test}}
        \right)
        \le
        n^{-A}.
    \end{equation}
    Additionally, if for some $A > 0$ there exists a $C_2 > 0$ such that $\PP(\hsigma_\Xi \ge C_2) \le n^{-A}$, then for $\CEN = \MOM$,
    \begin{equation}\label{eq:bounded-variance-expectation-bound}
        \EE\left[ \Delta_n^{\add} \right]
        \le
        C_1 C_2
        \sqrt{\frac{\log(|\Xi| n^A)}{n_\test}}
        +
        C_3
        n^{-A/2}
        |\Xi|^{1/2} 
        \sqrt{\frac{\log^2(|\Xi| n^A)}{n_\test}}
        \max_{\xi \in \Xi}
        \| \hsigma_\xi \|_{L_2}
    \end{equation}
    for some absolute constant $C_3 > 0$.
\end{lemma}

\begin{remark}
    [Comparison of assumptions for $\CEN = \AVG$ and $\CEN = \MOM$.]
    Comparing 
    \Cref{lem:bounded-orlitz-error-control,lem:bounded-variance-error-control},
    we note that the median-of-means method of risk estimation
    only requires control of
    the $L_2$ moments of the loss function
    compared to the $\psi_1$ (exponential) moments
    of the loss function.
    This is not surprising
    given that the median-of-means
    was developed as a sub-Gaussian
    estimator of the mean,
    only assuming finite variance
    (\Cref{lem:mom-concentration}).
    The $L_2$
    moment assumption
    in \Cref{lem:bounded-variance-error-control}
    can be further relaxed
    to an $L_{1+\alpha}$ moment
    assumption
    for $\alpha \in (0, 1]$
    \citep[Theorem 3]{lugosi2019mean}
    at the cost of weaker rate
    of convergence of $\Delta_n^\add$.
    One can, of course, replace
    the median-of-means estimator
    with any other sub-Gaussian
    or sub-exponential mean estimator
    \citep{catoni_2012, minsker_2015, fan_li_wang_2017}
    and obtain a similar weakening
    of the moment assumptions.
    Same remark continues to hold
    for $\Delta_n^\mul$
    discussed in \Cref{sec:control-of-Delta-mul}.
\end{remark}

\begin{remark}
    [Restriction on $A$ for $\CEN = \MOM$]
    \label{rem:restrict_A_eta}
    In \Cref{lem:bounded-orlitz-error-control,lem:bounded-variance-error-control},
    we allow for a free parameter $A$.
    However,
    in order for the choice of $\eta$
    to be feasible in the MOM construction (see, e.g., \Cref{lem:mom-concentration} in \Cref{sec:useful-concentration-results}),
    we need $B = \lceil 8 \log(1/\eta) \rceil \le n_\test$,
    which puts the following constraint on $A$:
    \[
        8 \log(n^{A} | \Xi |) \le n_\test
        \quad
        \iff
        \quad
        A \log n \le \frac{n_\test}{8} - \log(| \Xi |)
        \quad
        \iff
        \quad
        A \le \frac{n_\test}{8 \log n} - \frac{\log(| \Xi |)}{\log n}.
    \]
    For a large enough $n$,
    this allows for a large range of $A$.
    In addition, the right hand side is large enough
    to imply exponentially small probability bound
    for the event that $\Delta_n^\add$ is large.
    The same remark holds for
    \Cref{lem:bounded-orlitz-error-control-mul-form,lem:bounded-variance-error-control-mul-form} below.
\end{remark}

The key quantities that drive the tail probability 
and expectation bound on $\Delta_n^{\add}$ in both
\Cref{lem:bounded-orlitz-error-control,lem:bounded-variance-error-control}
are $\hsigma_\Xi$ and $|\Xi|$.
The following remark specifies the permissible growth rates
on $\hsigma_\Xi$ and $|\Xi|$
to ensure that $\Delta_n^{\add}$ is asymptotically small
in probability.

\begin{remark}
    [Tolerable growth rates on $\hsigma_\Xi$
    for $\Delta_n^{\add} = o_p(1)$]
    \label{rem:growth-rates-probabilistic-bound}
    Suppose $|\Xi| \le n^{S}$ for some constant $S > 0$ independent of $n, p$.
    If
    \[
        \hsigma_\Xi = o_p\left(\sqrt{\frac{n_\test}{\log n }}\right),
    \]
    then under the setting of
    \Cref{lem:bounded-orlitz-error-control,lem:bounded-variance-error-control},
    $\Delta_n^{\add} = o_p(1)$ as $n \to \infty$.
    The remark follows simply by noting that
    the dominating term in the probabilistic bound
    on $\Delta_n^{\add}$ in \eqref{eq:probabilistic.bound.Delta.add} is of order
    \[
       \hsigma_{\Xi}
            \sqrt{\frac{\log(|\Xi| n^A)}{n_\test}} \le \hsigma_\Xi \sqrt{\frac{(S+A)\log n }{n_\test}} = O\left(\hsigma_{\Xi}\sqrt{\frac{\log n}{n_{\test}}}\right).
    \]
    See 
    \Cref{sec:growth-rates-expectation-bound}
    for feasible rates for $\hsigma_\Xi$
    to ensure that $\EE[\Delta_n^\add] = o(1)$.
\end{remark}

\subsubsection
[Multiplicative form]
{Control of $\Delta_n^{\mul}$}
\label{sec:control-of-Delta-mul}

Moving on to $\Delta_n^{\mul}$, 
analogously to \Cref{lem:bounded-orlitz-error-control,lem:bounded-variance-error-control},
the following results provide high probability bounds on $\Delta_n^{\mul}$
in terms of a coefficient of variation parameter $\kappa$ which is the relative standard deviation of $\ell(Y_0, \widehat{f}^{\xi}(X_0))$ conditional on $\mathcal{D}_n$.
Let $\hf^\xi$, $n_\test$, $\CEN$ be as defined
    \Cref{alg:general-cross-validation-model-selection},
    and $\Delta_n^\mul$ be as in
    \eqref{eq:Delta_n_mul}.

\begin{lemma}
    [Control of $\Delta_n^{\mul}$ for losses with bounded conditional $\psi_1$ norm]
    \label{lem:bounded-orlitz-error-control-mul-form}
    Suppose $(X_j, Y_j)$, $j \in \cI_\test$ are sampled i.i.d.\ from $P$.
    Suppose the loss function $\ell$ is such that
$$    \| \ell(Y_0, \hf^\xi(X_0)) \|_{\psi_1 \mid \cD_n} \le \hsigma_\xi ~\text{ for } (X_0, Y_0) \sim P.$$
    Define $\hkappa_\xi = \hsigma_\xi / R(\hf^\xi)$
    and $\hkappa_\Xi = \max_{\xi \in \Xi} \hkappa_\xi$.
    Fix any $0 < A < \infty$.
    Then, for $\CEN = \AVG$,
    or $\CEN = \MOM$ with $\eta = n^{-A} / | \Xi |$,
    \[
        \PP
        \left(
            \Delta_n^{\mul}
            \ge
            C
            \hkappa_\Xi
            \max
            \left\{
            \sqrt{\frac{\log(|\Xi| n^{A})}{n_\test}},
            \frac{\log(|\Xi| n^{A})}{n_\test}
            \right\}
        \right)
        \le n^{-A}
    \]
    for a positive constant $C$.
\end{lemma}

\begin{lemma}
    [Control of $\Delta_n^{\mul}$ for losses with bounded conditional $L_2$ norm]
    \label{lem:bounded-variance-error-control-mul-form}
    Suppose $(X_j, Y_j)$, $j \in \cI_\test$ are sampled i.i.d.\ from $P$.
    Suppose the loss function $\ell$ is such that
    $$\| \ell(Y_0, \hf^\xi(X_0)) \|_{L_2 \mid \cD_n} \le \hsigma_\xi ~\text{ for }(X_0, Y_0) \sim P.$$
    Define $\hkappa_\xi := \hsigma_\xi / R(\hf^\xi)$
    and $\hkappa_\Xi := \max_{\xi \in \Xi} \hkappa_\xi$.
    Fix any $0 < A < \infty$.
    Then, for $\CEN = \MOM$ with $\eta = n^{-A} / | \Xi |$,
    \[
        \PP
        \left(
            \Delta_n^{\mul}
            \ge
            C
            \hkappa_\Xi
            \sqrt{\frac{\log(|\Xi| n^{A})}{n_\test}}
        \right)
        \le n^{-A}
    \]
    for a positive constant $C$.
\end{lemma}

\begin{remark}
    [Tolerable growth rate on $\hkappa_\Xi$ for probabilistic bound]
    \label{rem:growth-rates-probabilistic-bound-kappa}
    Suppose $|\Xi| \le n^{S}$ for some $S < \infty$. If
    \[
        \hkappa_\Xi
        = o_p\left( \sqrt{\frac{n_\test}{\log n}} \right),
    \]
    then under the setting of
    \Cref{lem:bounded-orlitz-error-control-mul-form,lem:bounded-variance-error-control-mul-form},
    $\Delta_n^{\mul} = o_p(1)$ as $n \to \infty$.
\end{remark}

\begin{remark}
    [Comparing the control of $\Delta_n^\add$
    versus $\Delta_n^\mul$]
    \label{rem:compare-Delta_n-add-vs-Delta_n-mul}
    Note that from
    \Cref{lem:bounded-orlitz-error-control,lem:bounded-orlitz-error-control-mul-form},
    controlling $\Delta_n^\add$ requires
    controlling $\hsigma_\Xi$,
    while controlling $\Delta_n^\mul$
    requires controlling $\hkappa_\Xi$.
    The former is on the scale of the standard deviation
    of the loss,
    while the latter is normalized standard deviation
    (where the normalization is with respect to the expectation of the loss).
    The advantage of the latter is that, even if the standard deviation
    diverges, the normalized standard deviation can be finite.
    This, in fact, happens for the case of minimum $\ell_2$-norm 
    least squares predictor when $\gamma \approx 1$,
    in which case the control of $\Delta_n^\mul$ is feasible.
    See also the discussion in \Cref{rem:divergence-Delta_n^add}.
\end{remark}

\begin{remark}
    [Choice of $n_\test$]
    \label{rem:ntest_choice}
    The above results hold true
    as long as $n_\test \to \infty$.
    Of course, the choice $n_\test$ restricts
    the allowable growth rate of $\hsigma_\Xi$ and $\hkappa_\Xi$
    as discussed in \Cref{rem:growth-rates-probabilistic-bound,rem:growth-rates-probabilistic-bound-kappa}.
    In our later applications in overparameterized learning, we adopt the
    proportional asymptotics framework in which  the number of covariates
    to the number of  observations converges to a non-zero constant.
    For this reason, we restrict ourselves to
    the choices of $n_\test$ such that
    $n_\test / n \to 0$ as $n \to \infty$;
    for example, one can take
    $n_\test = n^{\nu}$ for some $\nu < 1$.
    This allows
    us to have training models with the same limiting aspect ratio
    (dimension/sample size)
    as that of the original data without splitting.
    However,
    the larger the $n_\test$,
    the more accurate our estimator of the prediction risk. 
    For this reason,
    we suggest $n_\test = O(n / \log n)$
    rather than $n_\test = n^\nu$.
\end{remark}

\subsection
{Applications to loss functions}
\label{sec:common-loss-functions}

Below we consider several examples of common predictors and loss functions,
and bound the corresponding 
conditional $\hsigma$ parameters
used in~\Cref{lem:bounded-orlitz-error-control,lem:bounded-variance-error-control},
and conditional $\hkappa$ parameters
used in
\Cref{lem:bounded-orlitz-error-control-mul-form,lem:bounded-variance-error-control-mul-form}.
Recall the conditional $\psi_1$ and $L_r$ norms from~\eqref{eq:conditional-Orlicz} and~\eqref{eq:conditional-Lp}, respectively.
In addition,
let $\psi_2$ denote the $\psi_2$-Orlicz norm.

Recall $\hsigma_\Xi$
is the maximum of either
$\| \ell(Y_0, \hf^\xi(X_0)) \|_{\psi_1 \mid \cD_n}$
or
$\| \ell(Y_0, \hf^\xi(X_0)) \|_{L_2 \mid \cD_n}$
over $\xi \in \Xi$.
Also recall $\hkappa_\Xi$
is the maximum of either
$\| \ell(Y_0, \hf^\xi(X_0)) \|_{\psi_1 \mid \cD_n} / \| \ell(Y_0, \hf^\xi(X_0)) \|_{L_1 \mid \cD_n}$
or
$\| \ell(Y_0, \hf^\xi(X_0)) \|_{L_2 \mid \cD_n} / \| \ell(Y_0, \hf^\xi(X_0)) \|_{L_1 \mid \cD_n}$
over $\xi \in \Xi$.
In the following,
we control
each of these quantities for one of the predictors $\hf^\xi$, $\xi\in\Xi$,
which we denote simply by $\hf$ for brevity.

\subsubsection
[Bounded classification loss functions]
{Bounded classification loss functions}

\begin{proposition}
    [Generic classifier and 0-1 loss and hinge loss]
    \label{prop:misclassification-hinge-psi1l2-bound}
    Let $\hf$ be any predictor.
    \begin{enumerate}
        \item
        Suppose
        $\ell(Y_0, \hf(X_0)) = \max\big\{ 0, 1 - Y_0 \hf(X_0) \big\}$
        is the hinge loss.
        Assume $|Y_0| \le 1$ and $|\hf(X_0)| \le 1$.
        Then,
        \[
            \|\ell(Y_0, \hf(X_0))\|_{\psi_1 | \cD_n}
            \le 2,
            \quad
            \text{ and }
            \quad
            \| \ell(Y_0, \hf(X_0)) \|_{L_2 \mid \cD_n}
            \le 2.
        \]
        \item
        Suppose
        $\ell(Y_0, \hf(X_0)) 
        = \mathbbm{1}\{Y_0 \ne \hf(X_0)\}$
        is the 0-1 loss.
        Then,
        \begin{equation}
        \label{eq:bounded-loss-bound}
            \| \ell(Y_0, \hf(X_0)) \|_{\psi_1 \mid \cD_n}
            \le 1,
            \quad
            \text{ and }
            \quad
            \| \ell(Y_0, \hf(X_0)) \|_{L_2 \mid \cD_n}
            \le 1.
        \end{equation}
    \end{enumerate}
    More generally,
    any loss function that is bounded by $1$
    satisfies \eqref{eq:bounded-loss-bound}.
\end{proposition}

\Cref{prop:misclassification-hinge-psi1l2-bound}
implies that
the parameter $\hsigma_\Xi$
is bounded by $1$ (with probability $1$) for any collection of bounded
classifiers $\{ \hf^\xi, \xi \in \Xi \}$.
Hence,
\Cref{lem:bounded-orlitz-error-control,lem:bounded-variance-error-control}
imply
that
$\Delta_n^\add = O_p(\sqrt{\log(|\Xi|) / n_\test})$.
Therefore,
the additive form of oracle inequality from
\Cref{prop:general-model-selection-guarantee}
can be used to conclude
the following result.

\begin{theorem}
    [Oracle inequality for arbitrary classifiers]
    \label{thm:oracle-bound-classifier-miss-hinge=error}
    For any collection of classifiers
    $\{ \hf^\xi, \xi \in \Xi \}$
    with $\log(| \Xi |) = o(n_\test)$
    and the loss being the mis-classification or hinge loss
    with bounded response and predictor,
    \[
        \Big|
        R(\hf^\cv)
        -
        \min_{\xi \in \Xi}
        R(\hf^\xi)
        \Big|
        ~=~
        O_p\left(\sqrt{\frac{\log(| \Xi |)}{n_\test}}\right).
    \]
\end{theorem}

\Cref{thm:oracle-bound-classifier-miss-hinge=error}
can be used to argue
that tuning of hyperparameters in an arbitrary classifier
using \Cref{alg:general-cross-validation-model-selection}
leads to an ``optimal'' classifier
under the $0-1$ or hinge loss.
Moreover,
\Cref{prop:misclassification-hinge-psi1l2-bound}
extends
to arbitrary bounded loss functions.

For logistic or the cross-entropy loss,
being unbounded,
is not covered by \Cref{prop:misclassification-hinge-psi1l2-bound}.
However, we can use the multiplicative form of the oracle risk inequality 
\eqref{eq:oracle-risk-inequality-multiplicative-form}
as done in the next section in \Cref{prop:linear-predictor-logistic-loss}.

\subsubsection
[Unbounded regression loss functions]
{Unbounded regression loss functions}

\begin{proposition}
    [Linear predictor and square loss]\label{prop:subexp-ex-squared}
    Let $\hf$ be a linear predictor, i.e.,
    for any $x_0 \in \RR^{p}$, $\hf(x_0) = x_0^\top \hbeta$
    for some estimator $\hbeta \in \RR^{p}$ fitted on $\cD_n$.
    Suppose $\ell(Y_0, \hf(X_0)) = (Y_0 - \hf(X_0))^2$
    is the square loss.
    Let $(X_0, Y_0) \sim P$.
    Assume $\EE[X_0] = 0_p$ and let $\Sigma := \EE[X_0 X_0^\top]$.
    Then, the following statements hold:
    \begin{enumerate}
        \item If $(X_0, Y_0) \in \RR^{p} \times \RR$ 
        satisfies
        $\psi_2-L_2$ equivalence,
        i.e., $\| a Y_0 + b^\top X_0 \|_{\psi_2} 
        \le \tau \| a Y_0 + b^\top X_0 \|_{L_2}$
        for all $a \in \RR$ and $b \in \RR^{p}$,
        then
        \begin{align}
        \label{eqn:linear-case-one}
            \| \ell(Y_0, \hf(X_0)) \|_{\psi_1|\cD_n}
            &\le
            \tau^2 \inf_{\beta \in \RR^{p}}
            (\| Y_0 - X_0^\top \beta \|_{\psi_2} 
            + \| \hbeta - \beta \|_{\Sigma})^2,
            \quad
            \text{and}
            \quad
                \frac
                {\| \ell(Y_0, \hf(X_0)) \|_{\psi_1 \mid \cD_n}}
                {\EE[\ell(Y_0, \hf(X_0)) \mid \cD_n]}
                \le
                \tau^2.
        \end{align}

        \item If $(X_0, Y_0)$ satisfies the $L_4-L_2$ equivalence, i.e., 
        $\| a Y_0 + b^\top X_0 \|_{L_4}
        \le \tau \| a Y_0 + b^\top X_0 \|_{L_2}$
        for all $a \in \RR$ and $b \in \RR^{p}$,
        then
        \begin{align}
         \label{eqn:linear-case-two}
            \| \ell(Y_0, \hf(X_0)) \|_{L_2 | \cD_n}
            \le
            \tau^2 \inf_{\beta \in \RR^{p}}
            (\| Y_0 - X_0^\top \beta \|_{L_2} 
            + \| \hbeta - \beta \|_{\Sigma})^2,
            \quad
            \text{and}
            \quad
            \frac
            {\| \ell(Y_0, \hf(X_0)) \|_{L_2 \mid \cD_n}}
            {\EE[\ell(Y_0, \hf(X_0)) \mid \cD_n]}
            \le
            \tau^2.
        \end{align}
    \end{enumerate}
\end{proposition}

\begin{proposition}
    [Linear predictor and absolute loss]
    \label{prop:linear-predictor-absolute-loss}
    Let $\hf$ be a linear predictor corresponding
    to estimator $\hbeta$ fitted on $\cD_n$.
    Suppose $\ell(Y_0, \hf(X_0)) = | Y_0 - X_0^\top \hbeta |$
    is the absolute loss.
    Let $(X_0, Y_0) \sim P$.
    Assume $\EE[X_0] = 0_p$
    and let $\Sigma := \EE[X_0 X_0^\top]$.
    Then, the following statements hold:
    \begin{enumerate}
        \item
        If $(X_0, Y_0) \in \RR^{p} \times \RR$
        satisfies $\psi_1-L_1$ equivalence,
        i.e.,
        $\| a Y_0 + b^\top X_0 \|_{\psi_1}
        \le \tau \| a Y_0 + b^\top X_0 \|_{L_1}$
        for all $a \in \RR$ and $b \in \RR^{p}$,
        then
        \begin{equation}
            \label{eq:abserr-psi1l1}
            \| \ell(Y_0, \hf(X_0)) \|_{\psi_1 \mid \cD_n}
            \le \tau \inf_{\beta \in \RR^{p}}
            (\| Y_0 - X_0^\top \beta \|_{L_1} 
            + \| X_0^\top (\hbeta - \beta) \|_{L_1 \mid \cD_n}),
            \quad
            \frac
            {\| \ell(Y_0, \hf(X_0)) \|_{\psi_1 \mid \cD_n}}
            {\EE[\ell(Y_0, \hf(X_0)) \mid \cD_n]}
            \le \tau.
        \end{equation}
        \item If $(X_0, Y_0)$ satisfies $L_2-L_1$ equivalence,
        i.e.,
        $\| a Y_0 + b^\top X_0 \|_{L_2} \le \tau \| a Y_0 + b^\top X_0 \|_{L_1}$,
        for all $a \in \RR^{p}$ and $b \in \RR{p}$,
        then
        \begin{equation}
            \label{eq:abserr-l2l1}
            \| \ell(Y_0, \hf(X_0)) \|_{L_2 \mid \cD_n}
            \le \tau \inf_{\beta \in \RR^{p}} 
            (\| Y_0 - X_0^\top \beta \|_{L_1}
            + \| X_0^\top (\hbeta - \beta) \|_{L_1 \mid \cD_n}),
            \quad
            \frac
            {\| \ell(Y_0, \hf(X_0)) \|_{L_2 \mid \cD_n}}
            {\EE[\ell(Y_0, \hf(X_0)) \mid \cD_n]}
            \le \tau.
        \end{equation}
    \end{enumerate}
\end{proposition}

\begin{proposition}
    [Linear predictor and logistic loss]
    \label{prop:linear-predictor-logistic-loss}
    Let $Y_0\in[0, 1]$ almost surely.
    Let $\hf$ be a linear predictor corresponding to
    an estimator $\hbeta$ fitted on $\cD_n$.
    Suppose $\ell(Y_0, \hf(X_0))$ is the logistic or cross-entropy loss:
    \[
        \ell(Y_0, \hf(X_0))
        = - Y_0 \log \left( \frac{1}{1 + e^{- X_0^\top \hbeta}} \right) 
        - (1 - Y_0) \log \left( 1- \frac{1}{1 + e^{- X_0^\top \hbeta}} \right).
    \]
    Assume there exists $p_{\min}\in(0, 1)$ such that 
    $p_{\min} \le \mathbb{E}[Y_0\mid X_0 = x] 
    \le 1 - p_{\min}$
    for all $x$.
    Then, the following statements hold:
    \begin{enumerate}
        \item If $X_0 \in \RR^{p}$ satisfies $\psi_1-L_1$ equivalence,
        i.e., $\| b^\top X_0 \|_{\psi_1} \le \tau \| b^\top X_0 \|_{L_1}$
        for all $b \in \RR^{p}$,
        then
        \[
            \frac{\| \ell(Y_0, \hf(X_0)) \|_{\psi_1 \mid \cD_n}}{\EE[ \ell(Y_0, \hf(X_0)) \mid \cD_n]}
            \le
            2\tau p_{\min}^{-1}.
        \]
        \item If $X_0\in\mathbb{R}^p$ satisfies $L_2-L_1$ equivalence,
        i.e., $\| b^\top X_0 \|_{L_2} \le \tau \| b^\top X_0 \|_{L_1}$
        for all $b \in \RR^{p}$,
        then
        \[
            \frac{\| \ell(Y_0, \hf(X_0)) \|_{L_2 \mid \cD_n}}{\EE[\ell(Y_0, \hf(X_0)) \mid \cD_n]}
            \le
            2\tau p_{\min}^{-1}.
        \]
    \end{enumerate}
\end{proposition}

In the remarks that follow we offer  a discussion
of the different types of norm equivalences assumed in
\Crefrange{prop:subexp-ex-squared}{prop:linear-predictor-logistic-loss}.
\begin{remark}
    [Discussion of $\psi_2-L_2$ and $L_4-L_2$ equivalences]
    \label{rem:psi2l2-l4l2}
    A centered random vector $Z \in \RR^{p}$ is said to be 
    $\tau$-sub-Gaussian 
    if 
    \begin{equation}\label{def:sub-Gaussian}
        \sup_{a\in\RR^p}\frac{\| a^\top Z \|_{\psi_2}}{\| a \|_{\Sigma_Z}}
        \le \tau < \infty\quad\mbox{where}\quad \Sigma_Z := \mbox{Cov}(Z).
    \end{equation}
    See for instance Definition 1.2 and Remark 1.3
    of \cite{mendelson_zhivotovskiy_2020}
    for more details.
    The $L_4-L_2$ equivalence assumption is popular
    in robust estimation of covariance matrices.
    See, for example, 
    \cite{minsker_wei_2020,minsker_2018,mendelson_zhivotovskiy_2020}.
    This is weaker than 
    the sub-Gaussianity assumption in \eqref{def:sub-Gaussian}
    in the sense that $\psi_2-L_2$
    equivalence implies $L_4-L_2$ equivalence.
    This follows from the well-known fact that
    \[
        C_l
        \le
        \frac{\| W \|_{\psi_2}}{\sup_{r \ge 1} r^{-1/2} \| W \|_{L_r}}
        \le
        C_u
    \]
    for some universal constants $C_l$ and $C_u$;
    see \citet[Proposition 2.5.2]{vershynin_2018}.
    The $L_4-L_2$ equivalence assumption is also weaker than
    a commonly used assumption
    in the random matrix theory (RMT) literature.
    In RMT, one typically assumes
    features of the form 
    $\Sigma^{1/2} Z$,
    where $Z$ have i.i.d.\ entries
    and $\Sigma$ is feature covariance matrix.
    If the components of $Z$ are independent
    and have bounded kurtosis,
    then this typical RMT assumption implies
    $L_4-L_2$ equivalence.
\end{remark}

\begin{remark}
    [Discussion of $\psi_1-L_1$ and $L_2-L_1$ equivalences]
    \label{rem:psi1l1-l2l1}
    In \Cref{rem:psi2l2-l4l2},
    we have given examples of distributions 
    that satisfy $\psi_2-L_2$ and/or $L_4-L_2$ equivalence.
    From the fact that, for any random variable $W$,
    the function $r \mapsto \log \EE[|W|^r]$ ($r \ge 1$) is convex
    \citep[Section 9, inequality (b)]{loeve_2017},
    we can conclude that
    $\psi_2-L_2$ equivalence implies $\psi_1-L_1$ equivalence,
    and $L_4-L_2$ equivalence implies $L_2-L_1$ equivalence;
    see \Cref{prop:norm-equivalences}.
    We further note that
    distributions satisfying $\psi_1-L_2$ equivalence
    also satisfy $\psi_1-L_1$ and $L_2-L_1$ equivalence.
    See \Cref{fig:norm-equivalences-implications} for a visual summary
    of these equivalences and their proofs
    in \Cref{sec:norm-equivalence-implications}.
    
    We will now discuss other distributions
    that satisfy $\psi_1-L_2$
    equivalence
    (which implies $\psi_1-L_1$ equivalence).
    A random vector $Z \in \RR^{q}$ is log-concave if 
    for any two measurable subsets
    $A$ and $B$ of $\RR^{q}$,
    and for any $\theta \in [0, 1]$,
    \[
        \log \PP(Z \in \theta A + (1 - \theta) B)
        ~\ge~ \theta \cdot \PP(Z \in A)
        + (1 - \theta) \cdot \PP(Z \in B),
    \]
    whenever the set $\theta A + (1 - \theta) B
    = \{\theta x_1 + (1 - \theta) x_2 : x_1 \in A, x_2 \in B \}$ is measurable;
    see Definition~2.2 of \cite{adamczak_radoslaw_litvak_pajor_tomczakjaegermann_2010}.
   There exist a universal constant $C$ 
   such that all 
   log-concave random vectors $Z \in \RR^{q}$
   with mean $0$
   satisfy 
    \[
        \| a^\top Z \|_{\psi_1}
        \le C \| a^\top Z \|_{L_1}
    \]
    for all $a \in \RR^{q}$.
    This follows from the results of
    \cite{adamczak_radoslaw_litvak_pajor_tomczakjaegermann_2010}
    and \cite{latala_1999};
     see also
    \citet[Corollary 3]{nayar_oleszkiewicz_2012},
    Proposition 2.1.1 of \cite{warsaw_2003},
    and Proposition 2.14 of \cite{ledoux_2001}.
    In particular, Lemma~2.3 of \cite{adamczak_radoslaw_litvak_pajor_tomczakjaegermann_2010}
    implies that
    there exists a universal constant $C$ such that
    for all $a \in \RR^{q}$
    \[
        \| a^\top Z \|_{\psi_1} 
        \le C \| a^\top Z \|_{L_2}.
    \]
    Finally, note that since $L_4-L_2$ equivalence implies $L_2-L_1$ equivalence,
    and the RMT features as described in \Cref{rem:psi2l2-l4l2} satisfy $L_4-L_2$
    equivalence, they in turn satisfy $L_2-L_1$
    equivalence.

\end{remark}

\begin{remark}
    [Model-free nature of assumptions]
    It is worth emphasizing that
    we do not require a well-specified linear model
    for \Cref{prop:subexp-ex-squared,prop:linear-predictor-absolute-loss}.
    Hence, our results are model agnostic.
\end{remark}

\Crefrange{prop:subexp-ex-squared}{prop:linear-predictor-logistic-loss}
imply that, under the stated assumptions,
for any collection of predictors
$\{ \hf^\xi: \hf^\xi(x) = x^\top \hbeta^\xi, \xi \in \Xi \}$,
$\hkappa_\Xi$ is bounded
if
$(X_0, Y_0)$ satisfies a requisite moment equivalence assumption.
On the other hand,
the control of $\hsigma_\Xi$
depends crucially on behavior of $\max_{\xi \in \Xi} \| \hbeta^\xi - \beta_0 \|_{\Sigma}$.
Because $\hkappa_\Xi$ is bounded with probability $1$,
\Cref{lem:bounded-orlitz-error-control-mul-form,lem:bounded-variance-error-control-mul-form}
can be used to conclude
$\Delta_n^\mul = O_p(K_{X, Y} \sqrt{\log(| \Xi |) / n_\test})$,
where $K_{X, Y}$ is
the constant in the moment equivalence.
Hence,
the multiplicative form of the oracle inequality
from
\Cref{prop:general-model-selection-guarantee}
can used to conclude the following general result
for an arbitrary collection of linear predictors.
\begin{theorem}
    [Oracle inequality for arbitrary linear predictors]
    \label{thm:oracle-bound-linear-predictor-squared-error}
    Fix any collection of predictors  
    $\{ \hf^\xi: \hf^\xi(x) = x^\top \hbeta^\xi, \xi \in \Xi \}$.
    Let $\hf^\cv$ be the output of \Cref{alg:general-cross-validation-model-selection} with $\hf^\xi, \xi\in\Xi$ as the ingredient predictors.
    Suppose one of the following conditions hold:
    \begin{enumerate}
        \item The loss is squared error,
        $(X_0, Y_0)$ satisfies 
        $\psi_2-L_2$ equivalence  when $\texttt{CEN} = \texttt{AVE}$ 
        and $L_4-L_2$ equivalence when $\texttt{CEN} = \texttt{MOM}$.
        \item The loss is absolute error, 
        $(X_0, Y_0)$ satisfies 
        $\psi_1-L_2$ equivalence  when $\texttt{CEN} = \texttt{AVE}$ 
        and $L_2-L_1$ equivalence when $\texttt{CEN} = \texttt{MOM}$.        
        \item The loss is logistic error 
        and $p_{\min} \le \mathbb{E}[Y_0\mid X = x] \le 1 - p_{\min}$ for some $p_{\min} \in (0, 1)$, 
        $X_0$ satisfies 
        $\psi_1-L_1$ equivalence when $\texttt{CEN} = \texttt{AVE}$ 
        and $L_2-L_1$ equivalence when $\texttt{CEN} = \texttt{MOM}$.
    \end{enumerate}
    Then, there exists a constant $C$ depending only on the moment equivalence condition such that for any $A > 0$ and for $\hf^\cv$ returned by 
    \Cref{alg:general-cross-validation-model-selection},
    we have with probability at least $1 - n^{-A}$,
    \[
        \left|
        \frac{R(\hf^\cv)}{\min_{\xi \in \Xi} R(\hf^\xi)}
        -
        1
        \right|
        ~\le~
            C\sqrt{\frac{\log(|\Xi|n^A)}{n_\test}}.
    \]
    Here, for $\texttt{CEN} = \texttt{AVE}$, there are no restrictions on $A$. For $\texttt{CEN} = \texttt{MOM}$, we need $\eta$ to be $n^{-A}/|\Xi|$ in~\Cref{alg:general-cross-validation-model-selection}.
\end{theorem}

\Cref{thm:oracle-bound-linear-predictor-squared-error}
implies that a multiplicative form
of oracle inequality
holds true
for any collection of linear predictors
with three commonly used
loss functions -- 
square, absolute, or logistic loss -- 
under certain moment equivalence
conditions on the underlying data.
It is worth stressing that
\Cref{thm:oracle-bound-linear-predictor-squared-error}
does not require
any parametric model assumption
on the data.
The moment equivalence conditions
required are quite mild as indicated
in
\Cref{rem:psi2l2-l4l2,rem:psi1l1-l2l1}.
\Cref{thm:oracle-bound-linear-predictor-squared-error}    
can be used to argue
that tuning of hyperparameters
for an arbitrary linear predictor
using \Cref{alg:general-cross-validation-model-selection}
leads to an ``optimal'' linear predictor.
In particular,
this  includes
variable selection in linear regression,
and penalty selection in ridge regression
or lasso.

\begin{remark}
    [Divergence of $\Delta_n^\add$]
    \label{rem:divergence-Delta_n^add}
    As mentioned above,
    control of $\hsigma_\Xi$ for a collection
    of linear predictors
    depends crucially on $\max_{\xi \in \Xi} \| \hbeta^\xi - \beta_0 \|_{\Sigma}$.
    Controlling this maximum is not 
    difficult in the ``low-dimensional'' regime,
    where the number of features is asymptotically negligible
    compared to the number of observations.
    If,
    however,
    the collection of linear predictors involves 
    the least squares estimator
    with the number of features
    approximately same as the number of observations,
    then
    Corollaries 1 and 3
    of \cite{hastie_montanari_rosset_tibshirani_2019}
    implies that $\max_{\xi \in \Xi}\| \hbeta^\xi - \beta_0 \|_{\Sigma} \to \infty$ almost surely
    under some regularity assumptions.
    The case of number of features
    approximately the same as the number of observations
    can be seen in the problem of
    tuning the number of basis functions in series regression
    (see also \citet{Mei2019generalization,bartlett_montanari_rakhlin_2021}
    for similar results
    on random features regression
    and kernel regression).
    In this case,
    $\Delta_n^\add$ diverges
    while $\Delta_n^\mul$ is bounded
    hinting the advantages of the multiplicative
    form of the oracle inequality over the additive form.
\end{remark}

\subsection{Illustrative prediction procedures}

In the following two sections,
we provide concrete applications
of the results from this section
in the context of overparameterized learning.
The main motivation of our applications
is to synthesize a predictor whose prediction risk
is approximately monotonically non-increasing in the sample size. 
Although this represents the basic idea of 
``more data does not hurt,''
many commonly studied predictors
such as minimum $\ell_2$-norm least squares, minimum $\ell_1$-norm least squares
in the overparameterized regime
do not satisfy this property.
In the following sections,
we will provide two different ways to synthesize
a predictor with this property starting from any given base prediction procedure.

\begin{definition}
    [Prediction procedure]
    \label{def:pred-procedure}
    A prediction procedure,
    denoted by $\tf$ is a real-valued map,
    with two arguments:
    (1) a feature vector;
    and (2) a dataset.
    If $\cD_m = \{ (X_i, Y_i) : 1 \le i \le m \}$
    represents a dataset of size $m$,
    then $\tf(x; \cD_m)$
    represents prediction at $x$
    of the prediction procedure $\tf$
    trained on the dataset $\cD_m$.
\end{definition}

\begin{example}
[Minimum $\ell_2$-norm least squares prediction procedure]
\label{ex:mn2ls-ridge}
Suppose
$\cD_m = \{ (X_i, Y_i) \in \RR^{p} \times \RR : 1 \le i \le m \}$.
The minimum $\ell_2$-norm least squares (MN2LS) estimator 
trained on $\cD_m$ is defined as
\[
    \tbeta_\mnls(\cD_m)
    := \argmin_{\beta \in \RR^{p}}
    \bigg\{ \| \beta \|_2 : \beta \text{ is a minimizer of the function }
    \theta \mapsto \sum_{i=1}^{m} (Y_i - X_i^\top \theta)^2  \bigg\}.
\]
The estimator can be written explicitly in terms of $(X_i, Y_i)$, $i = 1, \dots, m$ as
\begin{equation}
    \label{eq:mn2ls}
    \tbeta_{\mnls}(\cD_m)
    = \left( \frac{1}{m} \sum_{i=1}^{m} X_i X_i^\top \right)^{\dagger}
    \left( \frac{1}{m} \sum_{i=1}^{m} X_i Y_i \right),
\end{equation}
where $A^\dagger$ denotes the Moore-Penrose inverse of $A$.
It is also the ``ridgeless" least squares estimator
because of the fact that
$\tbeta_\mnls(\cD_m) = \lim_{\lambda \to 0^+} \tbeta_{\ridge, \lambda}(\cD_m)$,
where $\tbeta_{\ridge, \lambda}(\cD_m)$ is the ridge estimator
at a regularization parameter $\lambda > 0$
trained on $\cD_m$:
\begin{equation}
    \label{eq:ridge}
    \tbeta_{\ridge, \lambda}(\cD_m)
    := \argmin_{\theta \in \RR^{p}} \bigg\{ \frac{1}{m} \sum_{i=1}^{m} (Y_i - X_i^\top \theta)^2 
    + \lambda \| \theta \|_2^2 \bigg\}.
\end{equation}
The MN2LS estimator
has been attracted attention in the last few years
and its risk behavior has been studied by
\cite{bartlett_long_lugosi_tsigler_2020,
belkin_hsu_xu_2020,
hastie_montanari_rosset_tibshirani_2019,
muthukumar_vodrahalli_subramanian_sahai_2020},
among others.
The MN2LS 
predictor
is now
defined as
\begin{equation}
    \label{eq:mn2ls-predictor}
    \tf_{\mnls}(x; \cD)
    := x^\top \tbeta_{\mnls}(\cD),
\end{equation}
for any vector $x \in \RR^{p}$
and dataset $\cD$ containing
random vectors from $\RR^{p} \times \RR$.
\end{example}

\begin{example}
[Minimum $\ell_1$-norm least squares prediction procedure]
\label{ex:mn1ls-lasso}
Suppose $\cD_m = \{ (X_i, Y_i) \in \RR^{p} \times \RR : 1 \le i \le m \}$.
The minimum $\ell_1$-norm least squares (MN1LS) estimator trained on $\cD_m$
is defined as
\begin{equation}
    \label{eq:mn1ls-def}
    \tbeta_{\mnla}(\cD_m)
    =
    \argmin_{\beta \in \RR^{p}}
    \bigg\{
        \| \beta \|_{1}
        : \beta \text{ is a minimizer of the function }
        \theta \mapsto \sum_{i=1}^{m} (Y_i - X_i^\top \theta)^2
    \bigg\}.
\end{equation}
It is also the ``lassoless" least squares estimator
because of the fact that
$\tbeta_\mnla(\cD_m) = \lim_{\lambda \to 0^+} \tbeta_{\lasso, \lambda}$,
where $\tbeta_{\lasso, \lambda}(\cD_m)$ is the lasso estimator
at a regularization parameter $\lambda > 0$ trained on $\cD_m$:
\begin{equation}
    \label{eq:lasso}
    \tbeta_{\lasso, \lambda}(\cD_m)
    :=
    \argmin_{\theta \in \RR^{p}}
    \bigg\{
        \frac{1}{2 m} \sum_{i=1}^{m}
        (Y_i - X_i^\top \theta)^2
        + \lambda \| \theta \|_1
    \bigg\}.
\end{equation}
The MN1LS estimator
connects naturally to the basis pursuit estimator in 
compressed sensing literature (e.g.~\cite{candes2006near,donoho2006compressed})
and its risk in the proportional regime has been recently analyzed in
\cite{mitra2019understanding,li_wei_2021}.
The MN1LS predictor is now defined as
\begin{equation}
    \label{eq:mn1ls-predictor}
    \tf_{\mnla}(x; \cD)
    := x^\top \tbeta_{\mnla}(\cD),
\end{equation}
for any vector $x \in \RR^{p}$
and dataset $\cD$ containing
random vectors from $\RR^{p} \times \RR$.
\end{example}

Note that the MN2LS and MN1LS estimators coincide
when there is a unique minimizer of the function
$\theta \mapsto \textstyle \sum_{i=1}^{m} (Y_i - X_i^\top \theta)^2$,
in which case both the estimators become the least squares estimator.

We focus mostly on the case of linear predictors
and squared error loss,
although all our results are easily extendable
to general predictors and loss functions.
(See \Cref{rem:prop_asymptotics_risk_examples}
later in the paper for more details.)

\section{Application 1: Zero-step prediction procedure}
\label{sec:zero-step}

\subsection{Motivation}
\label{sec:zerostep-motivation}

\begin{wrapfigure}[12]{r}{0.5\textwidth}
    \centering
    \vspace{-4mm}
    \includegraphics[width=0.36\columnwidth]{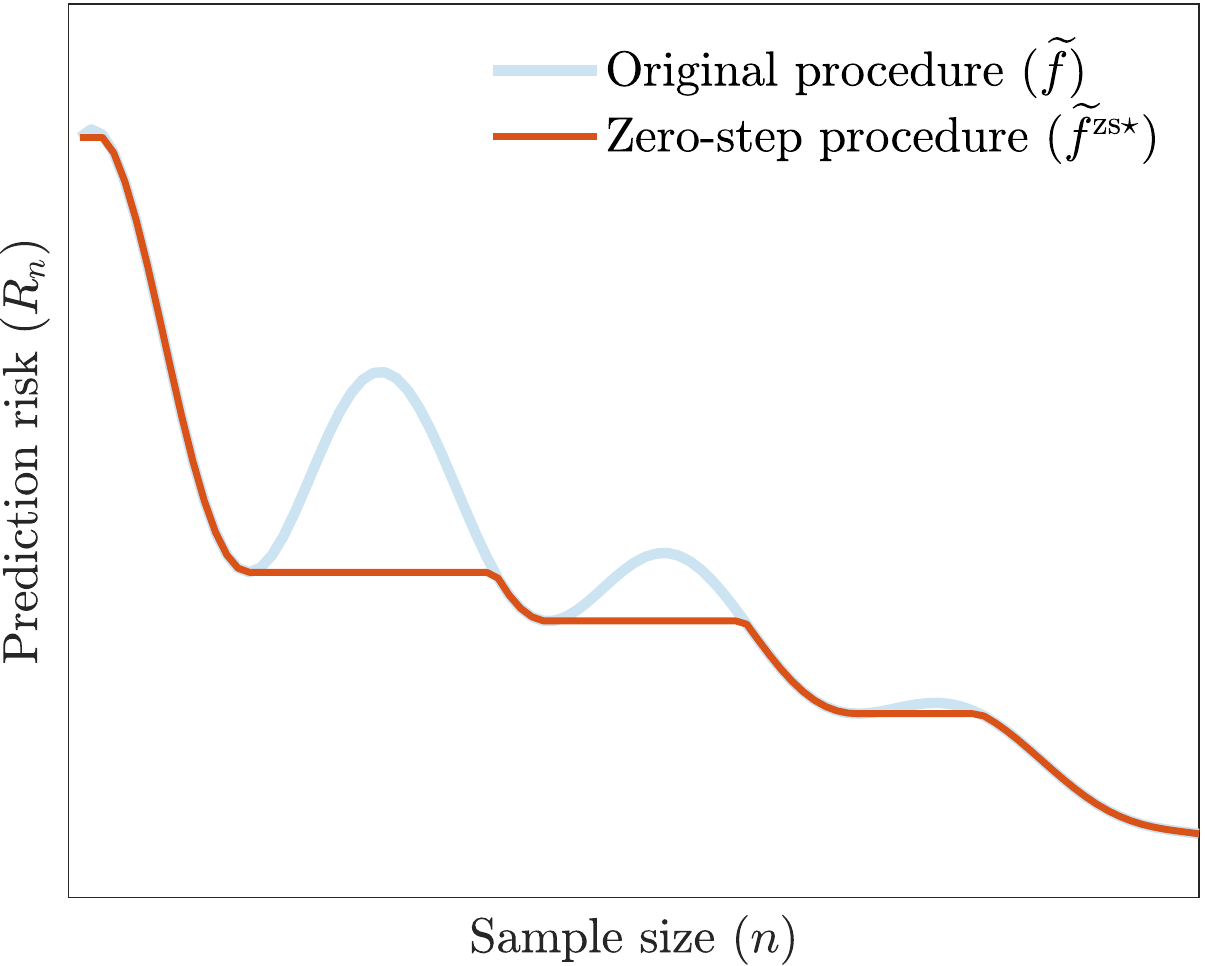}
    \vspace{-4mm}
  \caption{Illustration of risk monotonization.}
  \label{fig:risk-monotonization-illustration}
\end{wrapfigure}

Suppose $R_n$ represents
the prediction risk of a
given prediction procedure $\tf$
on a dataset containing $n$ i.i.d.\ observations.
It is desirable that
$R_n$ as a function of $n \ge 1$
is non-increasing.
As described above,
this however may not hold
for an arbitrary procedure $\tf$.
If we have access to $R_k$ for $1 \le k \le n$,
then one could just return
the predictor obtained
by applying the prediction procedure $\tf$
on a subset of $k^\star_n$ i.i.d.\ observations
where $k^\star_n = \argmin \{ R_k : 1 \le k \le n \}$.
This procedure, (denoted by, say) $\tf^{\zerostep\star}$,
essentially returns
a predictor whose risk
is the largest non-increasing function
that is below the risk of $\tf$;
see \Cref{fig:risk-monotonization-illustration} for an illustration.

It is trivially true that
the risk of the prediction procedure $\tf^{\zerostep\star}$
as a function of $n \ge 1$ is non-decreasing
and its risk at the sample size $n$ is given by
$\min_{k \le n} R_k$.
This procedure $\tf^{\zerostep\star}$ is, however,
not actionable in practice because
one seldom has access to the true risk $R_n$ of $\tf$.

The goal of this section
is to develop a prediction procedure $\hf^{\zerostep}$
starting with the base prediction procedure $\tf$
such that
the risk of $\hf^{\zerostep}$
is the largest non-increasing function
that is below the risk of $\tf$
(asymptotically).
We achieve this goal
by applying \Cref{alg:general-cross-validation-model-selection}
with the ingredient predictors being
the prediction procedure $\tf$ applied
on the subsets of the original data of varying sample sizes.

\begin{remark}
    [Conditional versus unconditional risk]
    There are two versions of the prediction risk $R_n$
    that one can consider:
    conditional (on the dataset $\cD_n$)
    and unconditional/non-stochastic.
    The conditional risk
    is not just a function of sample size,
    but also of the data $\cD_n$.
    Hence, the conditional risk $R_k$, for $k \le n$,
    is ill-defined as just a function of the sample size $k$.
    Therefore, the motivation above
    should be considered with respect to a non-stochastic
    approximation of the conditional risk.
    See \Cref{sec:risk-behavior-zerostep} for a precise definition
    of a non-stochastic approximation of the conditional risk
    which respect to which we talk of risk monotonization in the sample size.
\end{remark}

\subsection{Formal description}

Formally,
let the original dataset be denoted by
$\cD_n = \{ (X_1, Y_1), \dots, (X_n, Y_n) \}$.
As in \Cref{alg:general-cross-validation-model-selection},
consider the training and testing datasets
$\cD_\train$ and $\cD_\test$, respectively.
Note that our choice of $n_\test$
as described in \Cref{rem:ntest_choice}
satisfies $n_\test = o(n)$,
and hence,
the risk of $\tf$ trained on $\cD_\train$
is expected to be asymptotically the same
as the risk of $\tf$ trained on $\cD_n$.

To achieve the goal described in \Cref{sec:zerostep-motivation},
one can define the ingredient predictors
required in \Cref{alg:general-cross-validation-model-selection}
as follows:
Let $\cD_\train^{k}$ denote
a subset of $\cD_\train$ with  $n_\train - k$ observations
for $1 \le k \le n_\train$.
For $\Xi_n = \{ 1, 2, \dots, n_\train - 1 \}$
and $\xi \in \Xi_n$,
define $\tf^\xi(x) = \tf(x; \cD_\train^\xi)$
as the predictor obtained
by training $\tf$ on $\cD_\train^{\xi}$.
\Cref{prop:general-model-selection-guarantee}
along with
\Cref{lem:bounded-orlitz-error-control,lem:bounded-variance-error-control}
and \Cref{lem:bounded-orlitz-error-control-mul-form,lem:bounded-variance-error-control-mul-form}
can be used to imply that
$\hf^\cv$ thus obtained
has a non-increasing risk as a function of the sample size.

There are two important points to note here:
\begin{enumerate}
    \item
    The external randomness of choosing 
    a subset $\cD_\train^{\xi} \subseteq \cD_n$
    of size $\xi$.
    Observe that there are $\binom{n_\train}{\xi}$
    different subsets each with $n_\train - \xi$ i.i.d.\
    observations.
    Asymptotically,
    the prediction risk of
    $\tf$ trained on any of these subsets
    would be the same.
    To reduce such external randomness
    and make use of
    many different subsets of the same size,
    we take the ingredient predictor
    $\hf^\xi$ to be:
    \begin{equation}
        \label{eq:zerostep-averaged}
        \hf^\xi(x)
        =
        \frac{1}{M}
        \sum_{j=1}^{M}
        \tf(x; \cD_\train^{\xi, j}),
    \end{equation}
    where $\cD_\train^{\xi, j}$, $1 \le j \le M$
    are $M$ sets drawn independently
    (with replacement)
    from the collection of $\binom{n_\train}{\xi}$
    \footnote{Here, 
    $\binom{n}{r}$ denotes the binomial coefficient representing
    the number of distinct ways to pick $r$ elements from a set of $n$ elements
    for positive integers $n$ and $r$.}
    subsets of $\cD_\train$ of size $n_\train - \xi$.
    With $M = \infty$,
    $\hf^\xi$ becomes the average of
    $\tf$ trained on all possible subsets
    of $\cD_\train$ of size $n_\train - \xi$.
    This choice of $M$ removes
    any potential external randomness
    in defining $\hf^\xi$.
    The choice of $M = 1$ has the largest
    amount of external randomness.
    Based on the theory of $U$-statistics
    \citep[Chapter 5]{serfling_2009},
    we expect the choice $M = \infty$ to yield
    a predictor with the smallest variance;
    see \eqref{eq:sqauredrisk_decomposition_zerostep_bagged}. 
    Observe that the expected value
    $\hf^\xi(x)$ remains constant as $M$ changes
    because the distribution of $\cD_\train^{\xi, j}$
    remains identical across $j \ge 1$.
    However, the computation of
    $\hf^\xi$ with $M = \infty$ is infeasible,
    and hence, we use a finite $M \ge 1$.
    
    \item
    In the description above,
    we have $n_\train$ predictors
    to use in \Cref{alg:general-cross-validation-model-selection}.
    Note that the risk of a predictor trained on
    $m+1$ observations is asymptotically no different
    from that of a predictor trained on $m$ observations.
    The same comment holds true for predictors trained on
    $m + o(m)$ and $m$ observations.
    For this reason,
    we can replace $\Xi_n = \{ 1, 2, \dots, n_\train - 1 \}$
    with
    \begin{equation}
        \label{eq:indexset-zerostep}
        \Xi_n
        =
        \left\{
            1, 2, \dots,
            \left\lceil
                \frac{n_\train}{\lfloor n^\nu \rfloor}
                - 2
            \right\rceil
        \right\}
        \footnote{The subtraction of 2 in right end point in the
        definition \eqref{eq:indexset-zerostep} of $\Xi_n$ is for technical reasons.},
        \quad
        \text{ for some }
        \nu \in (0, 1),
    \end{equation}
    and consider predictors obtained by training
    $\tf$ on subsets of sizes
    $n_\train - \xi \lfloor n^\nu \rfloor$
    for $\xi \in \Xi_n$.
    This helps in reducing the computational cost
    of obtaining $\hf^\cv$ using
    \Cref{alg:general-cross-validation-model-selection}.
    This further helps in the theoretical properties
    of $\hf^\cv$
    in our application of union bound
    in the results of \Cref{sec:general-crossvalidation-modelselection}.
\end{enumerate}

Taking into account the remarks above,
with $\Xi$ as in \eqref{eq:indexset-zerostep},
for $\xi \in \Xi_n$,
we define $\hf^\xi$ as in \eqref{eq:zerostep-averaged},
but with an important change
that $\cD_\train^{\xi, j}$, $1 \le j \le M$,
now represent randomly drawn subsets
of $\cD_\train$ of size
$n_\xi = n_\train - \xi \lfloor n^\nu \rfloor$.
The ingredient predictors used in
\Cref{alg:general-cross-validation-model-selection}
are given by $\hf^\xi$, $\xi \in \Xi_n$.
We call the resulting predictor
obtained from
\Cref{alg:general-cross-validation-model-selection}
as the zero-step predictor based on $\tf$
and we denote the corresponding prediction procedure
to be $\hf^\zerostep$.
The zero-step procedure is summarized in \Cref{alg:zero-step}.

\begin{algorithm}
    \caption{Zero-step procedure}
    \label{alg:zero-step}
    \textbf{Inputs}:\\
    \begin{itemize}[noitemsep]
        \item[--] all inputs of \Cref{alg:general-cross-validation-model-selection}
        other than the index set $\Xi$;
        \item[--] a positive integer $M$.
    \end{itemize}
    \textbf{Output}:\\
    \vspace{-1em}
    \begin{itemize}
        \item[--] a predictor $\hf^\zerostep$
    \end{itemize}
    \textbf{Procedure}:
    \begin{enumerate}
        \item Let $n_\train = n - n_\test$.
        Construct an index set $\Xi_n$ per \eqref{eq:indexset-zerostep}.
        \item Construct train and test sets $\cD_\train$ and $\cD_\test$
        per Step 1 of \Cref{alg:general-cross-validation-model-selection}.
        \item Let $n_\xi = n_\train - \xi \lfloor n^\nu \rfloor$.
        For each $\xi \in \Xi_n$ and $j = 1, \dots, M$,
        draw random subsets $\cD_\train^{\xi, j}$ of size $n_\xi$ from $\cD_\train$.
        For each $\xi \in \Xi$, fit predictors $\hf^\xi$ per \eqref{eq:zerostep-averaged}
        using prediction procedure $\tf$ and $\{ \cD_\train^{\xi, j} : 1 \le j \le M \}$.
        \item Run Steps 3--5 of \Cref{alg:general-cross-validation-model-selection}
        using index set $\Xi = \Xi_n$
        and set of predictors $\{ \hf^\xi$, $\xi \in \Xi \}$.
        \item Return $\hf^\zerostep$ as the resulting $\hf^\cv$ from 
        \Cref{alg:general-cross-validation-model-selection}.
    \end{enumerate}
\end{algorithm}

\subsection[Risk behavior]{Risk behavior of $\hf^\zerostep$}
\label{sec:risk-behavior-zerostep}

As alluded to before,
in order to talk about risk monotonization,
one needs to consider a non-stochastic approximation
to the conditional risk that depends only
on the prediction procedure, the sample size,
and properties of the data distribution.
The definition below makes this precise.

\begin{definition}
    [Deterministic approximation of conditional prediction risk]
    \label{def:rn_nonstochastic-def}
    For any prediction procedure $\tf$,
    we call a map $R^\deter(\cdot; \tf) : \mathbb{N} \to \RR_{\ge 0}$
    a deterministic (or non-stochastic)
    approximation of the conditional risk
    of $\tf$
    if
    for all datasets $\cD_m$ of $m$
    i.i.d.\ random vectors,
    \begin{equation}
        \label{eq:rn_nonstochastic-def}
        \frac
        {| R(\tf(\cdot; \cD_m)) - R^\deter(m; \tf) |}
        {R^\deter(m; \tf)}
        = o_p(1),
    \end{equation}
    as $m \to \infty$.
    (Recall that
    $R(\tf(\cdot, \cD_m)) = \int \ell(y; \tf(x; \cD_m)) \mathrm{d}P(x,y)$.)
\end{definition}
It is important to recognize
that $R^\deter(m; \hf)$ is only a function
of the sample size $m$, the prediction \emph{procedure} $\tf$,
and the underlying distribution $P$,
and not the dataset $\cD_m$.
Note that we do not necessarily
require $R^\deter(m; \tf)$ to be the expected value of
$R(\tf(\cdot; \cD_m))$.
Furthermore,
a non-asymptotic approximation $R^\deter(\cdot; \tf)$
of the conditional risk may not be unique.

\begin{remark}
   [Relative convergence in
    \Cref{def:rn_nonstochastic-def}]
    In \eqref{eq:rn_nonstochastic-def},
    the division
    by $R^\deter(m; \tf)$
    ensures that
    the deterministic approximation
    to the conditional risk of
    $\tf(\cdot; \cD_m)$
    is non-trivial
    (i.e., non-zero)
    even if the conditional risk converges
    in probability to zero.
    If the conditional risk is bounded away from zero,
    asymptotically,
    then 
    \eqref{eq:rn_nonstochastic-def}
    is trivially implied by
    \[
        |
        R(\tf(\cdot; \cD_m))
        -
        R^\deter(m; \tf)
        |
        =
        o_p(1),
    \]
    as $m \to \infty$.
    In most settings of overparameterized learning,
    the conditional prediction risk is asymptotically
    bounded away from zero
    (see \eqref{eq:risk-lower-bound}, for example).
\end{remark}

Because $| \Xi_n | \le n$,
the results of
\Cref{sec:general-crossvalidation-modelselection}
imply that
with appropriate choices of
$\CEN$ and $\eta$ in
\Cref{alg:general-cross-validation-model-selection}
we obtain 
$\hf^{\zerostep}$
that satisfies the following risk bound:
\begin{equation}
    \label{eq:zerostep-nonasymp-risk-bound}
    R(\hf^{\zerostep})
    =
    \begin{cases}
        \min_{\xi \in \Xi_n}
        R(\hf^\xi)
        +
        O_p(1) \sqrt{\log n / n_\test}
        & \text{ if } \hsigma_\Xi = O_p(1) \\
        \min_{\xi \in \Xi_n} R(\hf^\xi)
        \big(1 + O_p(1) \sqrt{\log n / n_\test}\big)
        & \text{ if } \hkappa_\xi = O_p(1).
    \end{cases}
\end{equation}
Assume now
there exists a function $R^\deter: \mathbb{N} \to \RR_{\ge 0}$
such that the following holds:
\begin{equation}
    \label{eq:rn-deterministic-approximation-2}
    \tag{DET}
    \lim_{n \to \infty} \,
    \sup_{\xi_n \in \Xi_n} \,
    \PP
    \left(
        \frac{|R(\tf(\cdot; \cD_\train^{\xi_n, j})) - R^\deter(n_{\xi_n}; \tf)|}{R^\deter(n_{\xi_n}; \tf)}        
        > \epsilon
    \right)
    = 0
    \quad
    \text{for all }
    \epsilon > 0.
\end{equation}
Recall that
$\cD_\train^{\xi_n, j}$ for $1 \le j \le n$
are identically distributed,
and hence,
$\tf(\cdot, \cD_\train^{\xi_n, j})$
are also identically distributed predictors.
This implies
that assuming \eqref{eq:rn-deterministic-approximation-2}
for $j = 1$ is the same as assuming it
for all $1 \le j \le M$.
Note that \eqref{eq:rn-deterministic-approximation-2}
is essentially the same as \eqref{eq:rn_nonstochastic-def},
but with a different sequence of sample sizes
$\{ n_{\xi_n} \}_{n \ge 1}$ with $\xi_n \in \Xi_n$.
In accordance with our goal
of monotonizing the non-stochastic approximation
$R^\deter(\cdot; \tf)$ of
the prediction procedure $\tf$,
we aim to show
that the zero-step prediction procedure $\hf^\zerostep$
has its conditional prediction risk
approximated by $\min_{\xi \in \Xi_n} R^\deter(n_{\xi}; \tf)$.
For notational convenience,
set
\begin{equation}
    \label{eq:minimizer-rdeter-def}
    R^\deter_{\nearrow}(n; \tf)
    ~:=~
    \min_{\xi \in \Xi_n} R^\deter(n_\xi; \tf)
    \quad
    \text{and}
    \quad
    \xi^\star_n
    ~\in~
    \argmin_{\xi \in \Xi_n} R^\deter(n_\xi; \tf).
\end{equation}
Note the notation above is meant to reflect that
the index $\xi^\star_n$
can be chosen to be any element of the minimizing set.
If $\Xi_n = \{ 1, \dots, n_\train - 1 \}$,
and $\nu = 0$,
then $R^\deter_{\nearrow}(n; \tf) = \min \{ R^\deter(k; \tf) : 1 \le k \le n_\train - 1 \}$.
Although it might be tempting to
take $\Xi_n = \{ 1, \dots, n_\train - 1 \}$
and $\nu = 0$,
instead of the one in \eqref{eq:indexset-zerostep},
assumption \eqref{eq:rn-deterministic-approximation-2}
for all non-stochastic sequences $\{ n_{\xi_n} \}_{n \ge 1}$
with $\xi_n \in \Xi_n$
becomes almost certainly unreasonable.
To see this,
observe that
$\xi_n = n_\train - 1$
belongs to $\Xi_n$ for every $n$,
and for this choice,
$n_{\xi_n} = 1$.
Hence,
the predictor $\tf(\cdot; \cD_\train^{\xi, j})$
is computed based on one observation,
and cannot satisfy \eqref{eq:rn-deterministic-approximation-2}.
In the following calculations, however,
we only require assumption
\eqref{eq:rn-deterministic-approximation-2}
for the non-stochastic sequence
$\{ \xi^\star_n \}_{n \ge 1}$.
If $n_{\xi^\star_n}$ is known to diverge to $\infty$ and the distribution of the data stays constant,
then assumption \eqref{eq:rn-deterministic-approximation-2}
is reasonable and is exactly the same as
the existence of a deterministic approximation
to the conditional risk of $\tf$
in the sense of \Cref{def:rn_nonstochastic-def}.
In this favorable case of $n_{\xi_n^\star}$
diverging to $\infty$ with $n$,
one can take
$\Xi_n = \{ 1, \dots, n_\train - 1 \}$,
and $\nu = 0$.
Note that with $\Xi_n$ as defined in \eqref{eq:indexset-zerostep},
$n_{\xi_n} \to \infty$ for all $\xi_n \in \Xi_n$,
and thus in particular $n_{\xi_n^\star} \to \infty$
as $n \to \infty$.

It should be stressed that
\eqref{eq:rn-deterministic-approximation-2}
is an assumption on the base prediction procedure
$\tf$ and not on the ingredient predictors
$\hf^\xi$.
In general,
the risk behavior of $\tf$ does not necessarily
imply that of $\hf^\xi$ which is an average of $M$
predictors obtained from $\tf$.
However,
the risk of $\hf^\xi$
can be bounded in terms of the risk
$\tf$ for loss functions
$\ell(\cdot, \cdot)$ that are convex in the second argument.
Observe that
\begin{equation}
    \label{eq:bound-risk-hf-by-tf}
     R(\hf^\xi)
    ~=~
    R\left(\frac{1}{M}
    \sum_{j=1}^{M} \tf(\cdot; \cD_\train^{\xi, j}) \right)
    ~\le~
    \frac{1}{M}
    \sum_{j=1}^{M}
    R(\tf(\cdot; \cD_\train^{\xi, j})).
\end{equation}
The inequality \eqref{eq:bound-risk-hf-by-tf}
follows from Jensen's inequality.
It becomes an equality if $M = 1$
without the requirement that the loss function is convex.

Inequality \eqref{eq:bound-risk-hf-by-tf}
along with the non-stochastic
risk approximation
\eqref{eq:rn-deterministic-approximation-2}
can be used to control
$\min_{\xi \in \Xi_n} R(\hf^\xi)$
in \eqref{eq:zerostep-nonasymp-risk-bound}.
From \eqref{eq:minimizer-rdeter-def},
we obtain
\begin{equation}\label{eq:bound-minR-hf}
\begin{split}
    \min_{\xi \in \Xi_n} R(\hf^\xi)
    \overset{(a)}{\le}
    \min_{\xi \in \Xi_n}
    \frac{1}{M}
    \sum_{j=1}^{M}
    R(\tf(\cdot; \cD_\train^{\xi, j})) 
    &\overset{(b)}{\le}
    \frac{1}{M}
    \sum_{j=1}^{M}
    R(\tf(\cdot; \cD_\train^{\xi_n^\star, j})) \\
    &=
    R^\deter(n_{\xi_n^\star}; \tf)
    \left(
        1
        +
        \frac{1}{M}
        \sum_{j=1}^{M}
        \frac{R(\tf(\cdot, \cD_\train^{\xi_n^\star, j})) - R^\deter(n_{\xi_n^\star}; \tf)}{R^\deter(n_{\xi_n^\star}; \tf)}
    \right) \\
    &\overset{(c)}{=}
    \min_{\xi \in \Xi_n} R^\deter(n_\xi; \tf) 
    (1 + o_p(1)) \\
    &=
    R^\deter_\nearrow(n; \tf)
    (1 + o_p(1)).
\end{split}
\end{equation}
Inequality $(a)$ in \eqref{eq:bound-minR-hf}
follows from
using Jensen's inequality.
Inequality $(b)$ follows
because $\xi_n^\star \in \Xi_n$.
Equality $(c)$
follows for any fixed $M \ge 1$ from
the non-stochastic risk approximation
\eqref{eq:rn-deterministic-approximation-2};
this can be seen from the fact that
the sum of a  finite number of $o_p(1)$ random variables is $o_p(1)$.

All the inequalities in \eqref{eq:bound-minR-hf}
can be made equalities for $M = 1$,
if instead of \eqref{eq:rn-deterministic-approximation-2} 
we make the stronger assumption 
that 
\begin{equation}
    \label{eq:rn-deterministic-approximation-3}
    \tag{DET*}
    \lim_{n \to \infty} \,
    \PP
    \left(
    \sup_{\xi_n \in \Xi_n} \,
        \frac{|R(\tf(\cdot; \cD_\train^{\xi_n, j})) - R^\deter(n_{\xi_n}; \tf)|}{R^\deter(n_{\xi_n}; \tf)}        
        > \epsilon
    \right)
    = 0
    \quad
    \text{for all }
    \epsilon > 0.
\end{equation}
This is clearly a stronger assumption
than required for \eqref{eq:bound-minR-hf},
where
we only required such relative convergence
for a specific $\xi_n^\star \in \Xi_n$.
Under \eqref{eq:rn-deterministic-approximation-3},
we can write
\begin{align*}
    \min_{\xi \in \Xi_n}
    \frac{1}{M}
    \sum_{j=1}^{M}
    R(\tf(\cdot; \cD_\train^{\xi, j}))
    &=
    \min_{\xi \in \Xi_n}
    R^\deter(n_{\xi}; \tf)
    \left(
        1
        +
        \frac{1}{M}
        \sum_{j=1}^{M}
        \frac{R(\tf(\cdot; \cD_\train^{\xi, j})) - R^\deter(n_\xi; \tf)}{R^\deter(n_\xi; \tf)}
    \right) \\
    &\lessgtr
    R^\deter_\nearrow(n; \tf)
    \left(
        1
        \pm
        \frac{1}{M}
        \sum_{j=1}^{M}
        \sup_{\xi \in \Xi_n}
        \left|
        \frac{R(\tf(\cdot; \cD_\train^{\xi, j})) - R^\deter(n_\xi; \tf)}{R^\deter(n_\xi; \tf)}
        \right|
    \right) \\
    &=
    R^\deter_\nearrow(n; \tf)
    (1 + o_p(1)).
\end{align*}
We now conclude that for $M = 1$,
\begin{equation}
    \label{eq:nonasymp-bound-M=1}
    \min_{\xi \in \Xi_n} R(\hf^\xi)
    =
    \min_{\xi \in \Xi_n} R(\tf(\cdot; \cD_\train^{\xi, 1}))
    =
   R^\deter_\nearrow(n; \tf)
    (1 + o_p(1)).
\end{equation}
This proves that
all the inequalities in \eqref{eq:bound-minR-hf}
can be made equalities for $M = 1$
under the stronger assumption \eqref{eq:rn-deterministic-approximation-3}.
Combined with
\eqref{eq:zerostep-nonasymp-risk-bound},
this implies that
\begin{equation}
\label{eq:zerostep_risk-gaurantee-mindeter}
\begin{split}
    R(\hf^\zerostep)
    &=
    \begin{cases}
    R^\deter_\nearrow(n; \tf)
        (1 + o_p(1))
        + O_p(1) \sqrt{\log n / n_\test}
        & \text{ if } \hsigma_\Xi = O_p(1) \\
        R^\deter_\nearrow(n; \tf)
        (1 + o_p(1))
        & \text{ if } \hkappa_\Xi = O_p(1)
    \end{cases} \\
    &=
    R^\deter_\nearrow(n; \tf)
    \begin{cases}
        1 + o_p(1) +
       \sqrt{\log n / n_\test} /
         R^\deter_\nearrow(n; \tf)
        & \text{ if } \hsigma_\Xi = O_p(1) \\
        1 + o_p(1)
        & \text{ if } \hkappa_\Xi = O_p(1).
    \end{cases}
\end{split}
\end{equation}
As mentioned before,
assumption \eqref{eq:rn-deterministic-approximation-3}
is significantly stronger than
\eqref{eq:rn-deterministic-approximation-2}.
In the absence of
\eqref{eq:rn-deterministic-approximation-3},
inequality \eqref{eq:bound-minR-hf}
combined with
\eqref{eq:zerostep-nonasymp-risk-bound}
implies that
\eqref{eq:zerostep_risk-gaurantee-mindeter}
holds with inequalities instead of equalities.
For simplicity, denote:
\begin{enumerate}[label={\rm(O\arabic*)}]
    \item
    \label{cond:zerostep-add}
    $\hsigma_\Xi = O_p(1)$
    and
    $R^\deter_\nearrow(n; \tf) \sqrt{n_\test / \log n} \to \infty$.
    \item
    \label{cond:zerostep-mult}
    $\hkappa_\Xi = O_p(1)$.
\end{enumerate}
Hence,
we have proved the following result:
\begin{theorem}
    [Monotonization by zero-step procedure]
    \label{thm:monotonization-zerostep}
    For $M = 1$,
    if assumption~\eqref{eq:rn-deterministic-approximation-3} and either
    \ref{cond:zerostep-add}
    or
    \ref{cond:zerostep-mult}
    hold true,
    then
    $R^\deter_\nearrow(\cdot; \tf)$
    is a deterministic approximation of
    the prediction procedure $\hf^\zerostep$, i.e.,
    \[
        \frac{|R(\hf^\zerostep) - R^\deter_\nearrow(n; \tf)|}{R^\deter_\nearrow(n; \tf)}
        = o_p(1).
    \]
    For $M \ge 1$,
    if $\ell(\cdot, \cdot)$
    is convex in the second argument, assumption~\eqref{eq:rn-deterministic-approximation-2},
    and either
    \ref{cond:zerostep-add}
    or
    \ref{cond:zerostep-mult}
    hold true,
    then
    \[
        \frac{(R(\hf^\zerostep) - R^\deter_\nearrow(n; \tf))_+}{R^\deter_\nearrow(n; \tf)}
        = o_p(1).
    \]
\end{theorem}

\begin{remark}
    [Choice of $\Xi_n$]
    \label{rem:choice-of-index-set}
    All the calculations presented in this section hold for any set $\Xi_n$ with $|\Xi_n| \le n$. As long as either~\eqref{eq:rn-deterministic-approximation-2} (for $\xi_n = \xi_n^\star$ in~\eqref{eq:minimizer-rdeter-def}) or~\eqref{eq:rn-deterministic-approximation-3} holds true, then one can use $\Xi_n = \{1, 2, \ldots, n_{\train} - 1\}$ and $\nu = 0$. For this choice, $R^\deter_\nearrow(\cdot; \widehat{f})$ is the monotonized risk as illustrated in~\Cref{fig:risk-monotonization-illustration}. With the choice of $\Xi_n$ mentioned in~\eqref{eq:indexset-zerostep}, $R^\deter_{\nearrow}(\cdot; \widehat{f})$ is not a complete monotonization but it serves as an approximate monotone risk.
\end{remark}
\begin{remark}
    [Exact risk $\hf^\zerostep$]
    For $M = 1$ (under~\eqref{eq:rn-deterministic-approximation-3}),~\Cref{thm:monotonization-zerostep} essentially implies that
    the risk of the zero-step procedure
    closely tracks the monotonized deterministic approximation
    to the conditional prediction risk of $\tf$
    trained on $\cD_\train$.
    For $M \ge 1$ (under~\eqref{eq:rn-deterministic-approximation-2}),
    \Cref{thm:monotonization-zerostep}
    does not imply the risk of the zero-step predictor
    is monotonic
    or even that that a non-stochastic approximation
    of the risk exists
    in the sense of \Cref{def:rn_nonstochastic-def}.
    However,
    our simulations in limited settings
    presented in \Cref{sec:zerostep-illustration}
    suggest that 
    the risk of the zero-step prediction procedure
    is monotone even for $M \ge 1$.
\end{remark}

\begin{remark}
    [Verification of assumptions in
    \Cref{thm:monotonization-zerostep}]
    The bound on $\widehat{\sigma}_{\Xi}$ and $\hkappa_{\Xi}$
    in
    Assumptions \ref{cond:zerostep-add}
    and \ref{cond:zerostep-mult}
    can be verified for some common loss functions
    and predictors
    as discussed in \Cref{sec:common-loss-functions}.
    The verification of
    assumption \eqref{eq:rn-deterministic-approximation-2}
    or \eqref{eq:rn-deterministic-approximation-3}
    is very much tied to the exact prediction procedure.
    We verify \eqref{eq:rn-deterministic-approximation-2}
    in a specific setting
    in \Cref{sec:zerostep-overparameterized}.
\end{remark}

\subsubsection
[Overparameterized regime]
{{Risk behavior of $\hf^\zerostep$ under proportional asymptotics}}
\label{sec:zerostep-overparameterized}

In the discussion leading up to
\Cref{thm:monotonization-zerostep},
we have not made a specific reference
to the growth or non-growth
of the dimension of the features.
Technically, \Cref{thm:monotonization-zerostep}
does allow for
the dimension $p$ of the features to change
with the sample size $n$, i.e.,
one can have $p = p_n$.

Risk monotonization is 
an interesting phenomenon to study
in light of the double (or multiple)
descent results in the overparameterized setting
where $p_n / n \to \gamma$ as $n \to \infty$.
In our previous discussion of 
non-stochastic approximation of the conditional prediction
risk, we did not stress the dependence on the dimension of
features.
In the following,
we consider the implications of 
\Cref{thm:monotonization-zerostep}
in the context of overparameterized learning
and hence consider
the following setting.

Recall that
the original dataset
$\cD_n$ consists of
$n$ i.i.d.\ observations
$(X_i, Y_i) \in \RR^{p} \times \RR$,
$1 \le i \le n$ from distribution $P$.
In the following as we allow
the dimension $p$ of the features
to change with the sample size $n$
and assume that $p = p_n$ satisfies
\begin{enumerate}[label={\rm(PA($\gamma$))},leftmargin=1.35cm]
    \item
    \label{asm:prop_asymptotics}
    $p_n / n \to \gamma \in (0, \infty)$
    as $n \to \infty$.
\end{enumerate}
The above asymptotic regime, which is standard in random matrix theory \citep{bai_silverstein_2010}, is used 
in the overparameterized learning literature, where it has been referred to as
proportional asymptotics.
(see e.g.,
\cite{dobriban_wager_2018,
hastie_montanari_rosset_tibshirani_2019,Mei2019generalization,
bartlett_montanari_rakhlin_2021}).
Note that
under assumption \ref{asm:prop_asymptotics}
the underlying distribution $P$ of the observations
in $\cD_n$ should 
be indexed by the sample size $n$.
We suppress this dependence for convenience.
Under the proportional asymptotics regime
for commonly studied prediction procedures,
a deterministic approximation
to the conditional prediction risk
of a subset $\cD_m \subseteq \cD_n$
depends not on $m$ but on $p_n / m$,
among other properties of the distribution $P$.
For this reason,
in any discussion of the deterministic approximation
of the conditional prediction risk,
we write
$R^\deter(p_n / m ; \tf)$
instead of $R^\deter(m; \tf)$.
Now the goal of this subsection
is to derive the deterministic approximation
of the conditional risk of the zero-step
predictor under \ref{asm:prop_asymptotics}.

Recall that from the crucial calculation in \eqref{eq:bound-minR-hf}
leading to the risk of zero-step predictor, 
we require
\begin{equation}
\label{eq:cond-deter-nonpa}
\frac{R(\tf(\cdot, \cD_\train^{\xi_n^\star, j})) - R^\deter(n_{\xi_n^\star}; \tf)}{R^\deter(n_{\xi_n^\star}; \tf)}
= o_p(1),
\end{equation}
with $\xi^\star_n$ defined as in \eqref{eq:minimizer-rdeter-def}.
Except for \eqref{eq:cond-deter-nonpa},
all the remaining steps in 
\eqref{eq:bound-minR-hf} 
hold true
even in the overparameterized setting.
In the following, we will provide simple sufficient condition
for verification of \eqref{eq:cond-deter-nonpa} under \ref{asm:prop_asymptotics}.
As mentioned above, the deterministic risk  under \ref{asm:prop_asymptotics}
often depends not only on the sample size alone, but also on the ratio of the number of features
to the sample size.
Therefore, we find it helpful to rewrite \eqref{eq:cond-deter-nonpa}
as
\begin{equation}
    \label{eq:rn-deterministic-approximation-2-prop-asymptotics}
    \tag{DETPA-0}
    \frac
    {
        R(\tf(\cdot; \cD_\train^{\xi_n^\star, j}))
        - R^{\deter}(p_n / n_{\xi_n^\star}; \tf)
    }
    {R^{\deter}(p_n / n_{\xi_n^\star}; \tf)}
    =
    o_p(1),
    \quad
    \text{where}
    \quad
   \xi_n^\star
   \in \argmin_{\xi \in \Xi_n}
   R^\deter(p_n / n_{\xi}; \tf).
\end{equation}
Note that
assumption \ref{asm:prop_asymptotics}
does not imply that
$p_n / n_{\xi_n^\star}$ converges to a fixed limit
as $n \to \infty$.

Under assumption
\eqref{eq:rn-deterministic-approximation-2-prop-asymptotics},
\Cref{thm:monotonization-zerostep}
readily implies
the risk behavior of $\hf^\zerostep$.
However,
the possibility that
$p_n / n_{\xi_n^\star}$ does not converge to a fixed limit
necessitates a closer examination of
assumption \eqref{eq:rn-deterministic-approximation-2-prop-asymptotics}.
We provide a two-fold reduction
of assumption \eqref{eq:rn-deterministic-approximation-2-prop-asymptotics}.
Firstly,
it suffices to verify that the absolute difference
between
$R(\tf(\cdot; \cD_\train^{\xi_n^\star, j}))$
and $ R^{\deter}(p_n / n_{\xi_n^\star}; \tf)$
converges to $0$ when $R^\deter(\cdot; \tf)$
is uniformly bounded away from $0$.
This is a reasonable assumption in practice because
several loss functions under mild conditions on the response
have risk lower bounded by the unavoidable error which is strictly positive.
For example, assuming the loss $\ell$ is the squared loss
and that $\EE[(Y_0 - \EE[Y_0 \mid X_0])^2] > 0$, we have for any prediction procedure $\widetilde{f}$
and any training dataset $\cD_m$
containing $m$ observation,
\begin{equation}
\label{eq:risk-lower-bound}
R(\widetilde{f}(\cdot; \mathcal{D}_{m})) ~=~ \mathbb{E}[(Y_0 - \widetilde{f}(X_0; \mathcal{D}_{m}))^2\big|\mathcal{D}_{m}] ~\ge~ \mathbb{E}[(Y_0 - \mathbb{E}[Y_0|X_0])^2] > 0. 
\end{equation}
Hence, in this case, if there exists a deterministic function $R^\deter : (0, \infty] \to [0, \infty]$
such that under \ref{asm:prop_asymptotics},
as $n \to \infty$,
\begin{equation}
    \label{eq:rn-deterministic-approximation-prop-asymptotics-add}
        R(\tf(\cdot; \cD_\train^{\xi_n^\star, j}))
        - R^{\deter}(p_n / n_{\xi_n^\star}; \tf)
    =
    o_p(1),
    \quad
    \text{where}
    \quad
   \xi_n^\star
   \in \argmin_{\xi \in \Xi_n}
   R^\deter(p_n / n_{\xi}; \tf),
\end{equation}
then \eqref{eq:rn-deterministic-approximation-2-prop-asymptotics}
is satisfied.
Secondly,
the following lemma shows that
under \ref{asm:prop_asymptotics},
\eqref{eq:rn-deterministic-approximation-prop-asymptotics-add}
is satisfied
if there exists a deterministic approximation
for the conditional risk 
with datasets having a converging aspect ratio
(i.e., datasets for which the ratio of the number of features
to the sample size converges to a constant).

For any $\gamma > 0$, define
\[
    \cM_{\gamma}^{\zerostep}
    ~:=~ \argmin_{\zeta : \zeta \ge \gamma} R^\deter(\zeta; \tf).
\]
\begin{lemma}
    [Reduction of \eqref{eq:rn-deterministic-approximation-2-prop-asymptotics}]
    \label{lem:rn-deterministic-approximation-4-prop-asymptotics}
    Let $\cD_{k_m}$ be a dataset with $k_m$ observations and $p_m$ features.
    Consider a prediction procedure $\tf$ trained on $\cD_{k_m}$.
    Assume the loss function $\ell$ is such that 
    $R(\tf(\cdot; \cD_{k_m}))$ is uniformly bounded from below by $0$.
    Let $\gamma > 0$ be a real number.
    Suppose there exists a proper, lower semicontinuous function 
    $R^\deter(\cdot; \tf): [\gamma, \infty] \to [0, \infty]$ such that
    \begin{equation}
        \label{tag:detpar-0}
        \tag{DETPAR-0}
        R(\tf(\cdot; \cD_{k_m}))
        ~\pto~
        R^\deter(\phi; \tf),
    \end{equation}
    as $k_m, p_m \to \infty$
    and $p_m / k_m \to \phi \in \cM_{\gamma}^{\zerostep}$.
    Further suppose that $R^\deter(\cdot; \tf)$ is continuous
    on the set $\cM_{\gamma}^{\zerostep}$.
    Then, \eqref{eq:rn-deterministic-approximation-2-prop-asymptotics} 
    is satisfied.
\end{lemma}

We prove \Cref{lem:rn-deterministic-approximation-4-prop-asymptotics}
using the real analysis fact
that a sequence $\{ a_n \}_{n \ge 1}$ 
converges to $0$ if and only if
for any subsequence $\{ a_{n_k} \}_{k \ge 1}$,
there exists a further subsequence
$\{ a_{n_{k_l}} \}_{l \ge 1}$
that converges to $0$
(see, for example, Problem 12 of \cite{royden_1988};
also see \Cref{lem:subsequence-to-sequence} for a self-contained proof).
We apply this fact to the sequence
\[
    a_n(\epsilon)
    =
    \PP
    \left(
    \left|
       R(\tf(\cdot; \cD_\train^{\xi_n^\star, j}))
        - R^{\deter}(p_n / n_{\xi_n^\star}; \tf) 
    \right|
    \ge \epsilon
    \right),
\]
for every $\epsilon > 0$.
A crucial component in applying this technique
is to first produce a subsequence
$\{ n_{k_{l}} \}_{l \ge 1}$
such that $p_{n_{k_{l}}} / n_{\xi^\star_{n_{k_{l}}}}$
converges to a point in $\argmin_{\zeta \in [\gamma, \infty]} R^\deter(\zeta; \tf)$.
A few remarks on the assumptions of \Cref{lem:rn-deterministic-approximation-4-prop-asymptotics}
are in order. 

\begin{itemize}[leftmargin=*]
\item 
In most cases,
the set of minimizers of $R^\deter(\cdot; \tf)$
is a singleton set.
For such a scenario,
\Cref{lem:rn-deterministic-approximation-4-prop-asymptotics}
only requires the deterministic approximation
of the conditional prediction risk
for a single limiting aspect ratio
(i.e., \eqref{tag:detpar-0} is only required for a single $\phi$).
Several commonly studied predictors
satisfy \eqref{tag:detpar-0} as discussed below.

\item Assuming lower semicontinuity of $R^\deter(\cdot; \tf)$
is a mild assumption.
In particular, 
it does not preclude the possibility that
$R^\deter$ diverges to $\infty$ at several values in the domain
as shown in \Cref{prop:lower-semicontinuity-divergence}. 
Such risk diverging behavior 
is a common occurrence for several popular predictors
in overparameterized learning, for example, MN2LS, MN1LS, etc.
The requirement of the lower semicontinuity stems from
the goal of monotonizing $R^\deter$ from \emph{below}.
\end{itemize}

\begin{proposition}
    [Verifying lower semicontinuity for diverging risk profiles]
    \label{prop:lower-semicontinuity-divergence}
    Suppose $h : [a, c] \to \RR$ is continuous on $[a, b) \cup (b, c]$
    and $\lim_{x \to b^{-}} h(x) = \lim_{x \to b^{+}} h(x) = \infty$.
    Then, $h$ is lower semicontinuous on $[a, c]$.
\end{proposition}

\begin{itemize}[leftmargin=*]

\item[]
\Cref{prop:lower-semicontinuity-divergence}
implies that if $R^\deter$ is continuous
on a set except for a point where it diverges
to $\infty$,
then $R^\deter$ is lower semicontinuous on that set.
In this sense,
\Cref{prop:lower-semicontinuity-divergence}
relates the lower semicontinuity assumption
of \Cref{lem:rn-deterministic-approximation-4-prop-asymptotics}
to the continuity assumption of the lemma.

\item Continuity assumption 
on $R^\deter(\cdot; \tf)$ 
at the argmin set $\argmin_{\zeta \in [\gamma, \infty]} R^\deter(\zeta; \tf)$
is also mild.
\Cref{prop:continuity-from-continuous-convergence-rdet}
below shows that
\eqref{tag:detpar-0}
holding for $\phi$ in any open set $\cI$
implies continuity of $R^\deter$ on $\cI$.
In particular,
this implies continuity on the sets of the type
$\cI = (a, \infty]$.
If the set of minimizers of $R^\deter$
is a singleton set,
then \eqref{tag:detpar-0}
itself does not suffice
to guarantee the continuity of $R^\deter$
at the minimizer.
\Cref{prop:continuity-from-continuous-convergence-rdet}
in such a case requires
verifying \eqref{tag:detpar-0} on an open interval
containing the minimizer.
\end{itemize}

\begin{proposition}
    [Certifying continuity from continuous convergence]
    \label{prop:continuity-from-continuous-convergence-rdet}
    Let $\cD_{k_{m}}$ be a dataset with $k_m$ observations
    and $p_m$ features,
    and consider a prediction procedure $\tf$ trained on $\cD_{k_m}$.
    Let $\cI$ be an open set in $(0, \infty)$.
    Suppose there exists a function $R^\deter : (0, \infty] \to [0, \infty]$
    such that
    \begin{equation}
        \label{eq:pointiwise-convergence-rdet}
        R(\tf(\cdot; \cD_{k_m}))
        ~\pto~ R^\deter(\phi; \tf)
    \end{equation}
    as $k_m, p_m \to \infty$
    and $p_m / k_m \to \phi \in \cI$.
    Then, $R^\deter(\cdot; \tf)$ is continuous on $\cI$.
\end{proposition}

Combining the results and the discussion above,
the verification of \eqref{eq:rn-deterministic-approximation-2-prop-asymptotics}
under \ref{asm:prop_asymptotics}
can proceed with the following two-step program.
\begin{enumerate}[label={\rm(PRG-0-C\arabic*)},leftmargin=2cm]
    \item 
    \label{prog:zerostep-cont-conv}
    For $\phi$ such that $R^\deter(\phi; \tf) < \infty$,
    verify that
    for all datasets $\cD_{k_m}$
    with limiting aspect ratio $\phi$,
    $R(\tf(\cdot; \cD_{k_m})) \pto R^\deter(\phi; \tf)$.
    \item
    \label{prog:zerostep-cont-infty}
    Whenever $R^\deter(\phi; \tf) = \infty$,
    \[
        \lim_{\phi' \to \phi^{-}} R^\deter(\phi'; \tf)
        = \lim_{\phi' \to \phi^{+}} R^\deter(\phi'; \tf)
        = \infty.
    \]
\end{enumerate}
The continuity of $R^\deter$ at points where it is finite follows from~\ref{prog:zerostep-cont-conv} via \Cref{prop:continuity-from-continuous-convergence-rdet}.
This kind of convergence is verified
in the literature for several commonly used prediction procedures,
such as ridge regression and MN2LS \citep{hastie_montanari_rosset_tibshirani_2019},
lasso and MN1LS \citep{li_wei_2021}, etc;
see \Cref{rem:prop_asymptotics_risk_examples} for more details.
This combined with~\ref{prog:zerostep-cont-infty}
via \Cref{prop:lower-semicontinuity-divergence}
implies lower semicontinuity of $R^\deter$
on $[\gamma, \infty]$.
If there is more than one $\phi$ at which $R^\deter$ is $\infty$,
then \Cref{prop:lower-semicontinuity-divergence}
should be applied separately by splitting the domain
to only contain one point of divergence.
A more general result
of this flavour can be found in \Cref{prop:semicontinuity-metricspace}
in \Cref{sec:onestep-overparameterized}.

We will follow these steps
to verify \eqref{eq:rn-deterministic-approximation-2-prop-asymptotics}
for the ridge and lasso prediction procedures
in \Cref{sec:verification-deterministicprofile-zerostep}.
But first we will complete
the derivation of the deterministic approximation
to the conditional risk of $\hf^\zerostep$
under
\eqref{eq:rn-deterministic-approximation-2-prop-asymptotics}
following \eqref{eq:bound-minR-hf}.
\Cref{lem:rn-deterministic-approximation-4-prop-asymptotics}
combined
with 
\Cref{thm:monotonization-zerostep}
proves that the zero-step prediction procedure
approximately monotonizes the risk of 
the base prediction procedure $\tf$
as shown in the following result:

\begin{theorem}
    [Asymptotic risk profile of zero-step predictor]
    \label{thm:asymptotic-risk-tuned-zero-step}
    For any prediction procedure $\tf$,
    suppose
    \ref{asm:prop_asymptotics},
    either
    \ref{cond:zerostep-add}
    or \ref{cond:zerostep-mult},
    and the assumptions of
    \Cref{lem:rn-deterministic-approximation-4-prop-asymptotics}
    hold true.
    In addition,
    if the loss function is convex in the second argument,
    then
    for any 
    $M \ge 1$,
    \[
        \left(
            R(\hf^\zerostep; \cD_n)
            - \min_{\zeta \ge \gamma} R^\deter(\zeta; \tf)
        \right)_+
        = o_p(1).
    \]
\end{theorem}

\begin{remark}
    [Monotonicity in the limiting aspect ratio
    and improvement over base procedure]
    If we replace
    assumption \eqref{eq:rn-deterministic-approximation-2-prop-asymptotics}
    with the stronger version
    \begin{equation}
    \label{eq:rn-deterministic-approximation-2-prop-asymptotic-star}
    \tag{DETPA-0*}
    \sup_{\xi \in \Xi_n}
    \frac
    {
    |
    R(\tf(\cdot; \cD_\train^{\xi, j}))
    - R^{\deter}(p_n / n_{\xi}; \tf)
    |
    }
    {R^{\deter}(p_n / n_{\xi}; \tf)
    }
    =
    o_p(1),
\end{equation}
    as $n \to \infty$,
    then for $M = 1$,
    the conclusion of
    \Cref{thm:asymptotic-risk-tuned-zero-step}
    can be strengthened
    to
    \begin{equation}
        \label{eq:zerostep-exact-risk}
        \left|
            R(\hf^\zerostep; \cD_n)
            -
            \min_{\zeta \ge \gamma}
            R^\deter(\zeta; \tf)
        \right|
        =
        o_p(1).
    \end{equation}
    This implies
    that the risk of the zero-step
    procedure is monotonically non-decreasing in $\gamma$.
    Under the assumptions
    of \Cref{thm:asymptotic-risk-tuned-zero-step},
    one can only conclude that
    the risk of zero-step procedure
    is asymptotically bounded above
    by a monotonically non-decreasing function in $\gamma$
    in general.
    It is trivially true that
    $\min_{\zeta \le \gamma} R^\deter(\zeta; \tf)
    \le R^\deter(\gamma; \tf)$.
    Hence,
    the asymptotic risk of zero-step procedure
    is no worse than that of the base procedure.
\end{remark}

\begin{remark}
    [Finiteness of the risk of $\hf^\zerostep$]
    Predictors such the MN2LS or MN1LS
    undergo divergence in the prediction risk.
    The zero-step prediction procedure
    does not have such a divergence in the risk
    under general regularity conditions.
    In particular,
    as long as $\EE[\ell(y, 0)] < \infty$,
    then the risk of $\hf^\zerostep$ is asymptotically
    bounded by $\EE[\ell(y, 0)]$.
    Observe that $\EE[\ell(y, 0)]$
    is the risk of the null predictor
    which always returns $0$ as its prediction.
    By including the zero predictor
    in \Cref{alg:general-cross-validation-model-selection},
    the risk of $\hf^\zerostep$ will always be
    asymptotically bounded by this null risk.
\end{remark}

\subsubsection
[Verifying deterministic profiles]
{Verifying deterministic profile assumption 
\eqref{tag:detpar-0}
}
\label{sec:verification-deterministicprofile-zerostep}

In the following,
we will restrict ourselves
to the case of linear predictors
and squared error loss,
and verify assumption 
\eqref{tag:detpar-0}
for MN2LS and MN1LS base procedures.

Suppose $\cD_{k_m} = \{ (X_i, Y_i) \in \RR^{p_m} \times \RR : 1 \le i \le k_m \}$.
Recall the MN2LS and MN1LS
predictor procedures
defined in \Cref{ex:mn2ls-ridge,ex:mn1ls-lasso}.
It is now well-known that
the MN2LS and MN1LS prediction procedures
has a non-monotone risk
as a function of sample size $n$
\citep{nakkiran_venkat_Kakade_Ma-2020,hastie_montanari_rosset_tibshirani_2019, li_wei_2021}.
The following two results
verify assumption
\eqref{tag:detpar-0}
for these two procedures
under some regularity conditions
stated in
\cite{hastie_montanari_rosset_tibshirani_2019,li_wei_2021}.

\begin{proposition}
    [Verification of \eqref{tag:detpar-0}
    for MN2LS procedure]
    \label{prop:asymp-bound-ridge-main}
    Assume the setting of Theorem 3 of
    \cite{hastie_montanari_rosset_tibshirani_2019}.
    Then,
    there exists a function $R^\deter(\cdot; \tf_{\mnls}) : (0, \infty]
    \to [0, \infty]$
    such that
    \ref{prog:zerostep-cont-conv}
    holds for all $\phi \neq 1$
    and \ref{prog:zerostep-cont-infty}
    holds for $\phi = 1$.
\end{proposition}

\begin{proposition}
    [Verification of \eqref{tag:detpar-0} for MN1LS procedure]
    \label{prop:asymp-verif-mn1ls}
    Assume the setting of Theorem 2
    of \cite{li_wei_2021}.
    Then,
    there exists a function $R^\deter(\cdot; \tf_{\mnla})
    : (0, \infty] \to [0, \infty]$ such that
    \ref{prog:zerostep-cont-conv}
    holds for all $\phi \neq 1$
    and \ref{prog:zerostep-cont-infty}
    holds for $\phi = 1$.
\end{proposition}

\begin{remark}
    [Extending \Cref{prop:asymp-bound-ridge-main,prop:asymp-verif-mn1ls}
    to other predictors]
    \label{rem:prop_asymptotics_risk_examples}
    Theorem 3 of \cite{hastie_montanari_rosset_tibshirani_2019}
    only provides the asymptotic behavior
    of the prediction risk computed conditional
    only on $\{ X_i, 1 \le i \le k_m \}$.
    The proof in
    \Cref{sec:verif-asymp-profile-ridge}
    of \Cref{prop:asymp-bound-ridge-main}
    extends the calculations of
    of \cite{hastie_montanari_rosset_tibshirani_2019}
    for prediction risk conditional on $\cD_{k_m}$.
    These calculations can be further 
    extended in a straightforward manner to cover the case of $\lambda > 0$,
    i.e., the ridge regression procedure.
    See \Cref{prop:asymp-bound-ridge-main}
    for more details.
    Similar comments apply to \Cref{prop:asymp-verif-mn1ls}
    where the proposition can be easily extended to cover
    the case of $\lambda > 0$,
    i.e., the lasso prediction procedure.
    
    Additionally,
    most results in the literature
    under \ref{asm:prop_asymptotics}
    derive the risk behavior
    as $p_m / k_m \to \phi < \infty$.
    \Cref{prop:asymp-bound-ridge-main,prop:asymp-verif-mn1ls}
    also extend the existing results
    to the case when $p_m / k_m \to \infty$
    as $m \to \infty$.

    We present \Cref{prop:asymp-bound-ridge-main,prop:asymp-verif-mn1ls} as example results
    to show the verification of
    our assumptions follow rather easily
    from the existing asymptotic profile results
    in the literature.
    In the proportional asymptotic regime, the risk profiles have been characterized 
    for various other prediction procedures including, 
    high dimensional robust $M$-estimator \citep{karoui2013asymptotic,karoui2018impact,donoho2016high}, 
    the Lasso estimator \citep{miolane2021distribution,celentano2020lasso},
    and various classification procedures \citep{montanari2019generalization,liang2020precise,sur2019likelihood}.
    Our results can be suitably extended
    to verify
    \eqref{eq:rn-deterministic-approximation-2-prop-asymptotics}
    for these other predictors.
    Note that for our results, 
    we only need to know
    that the asymptotic risk exists,
    which can potentially hold true
    under weaker assumptions.
\end{remark}

\subsection{Numerical illustrations}
\label{sec:zerostep-illustration}

In this section,
we provide numerical illustration
of the risk monotonization of zero-step
prediction procedure
in the overparameterized setting,
when the base prediction procedures
are minimum $\ell_2$-norm least squares (MN2LS)
and minimum $\ell_1$-norm least squares (MN1LS).
In order to illustrate
risk monotonization as in
\Cref{thm:asymptotic-risk-tuned-zero-step},
we need to show
the risk behavior of $\hf^\zerostep$ 
at different aspect ratios.
We use the following simulation setups
for the two predictors.

\paragraph{Minimum $\ell_2$-norm least squares (MN2LS).}
We fix $n = 1000$
and vary the dimension $p$ of the features
from $100$ to $10000$
(for a total of $20$ values of $\gamma = p/n$ 
logarithmically spaced between $0.1$ to $10$).
This will show the risk behavior
of zero-step procedure
for aspect ratios between $0.1$ to $10$.
For every pair of sample size $n = 1000$
and dimension $p$,
we generate $100$ independent datasets
each with $n$ i.i.d.\ observations
from the linear model
$Y_i = X_i^\top \beta_0 + \eps_i$,
where
$X_i \sim \cN(0_p, I_p)$,
$\beta_0 \sim \cN(0_p, \rho^2 / p I_p)$
and
$\eps_i \sim \cN(0, \sigma^2)$
drawn independently of $X_i$.
The model represents
a dense signal regime
with average signal energy $\rho^2$.
We define the signal-to-noise ratio
(SNR) to be $\rho^2 / \sigma^2$.
On each dataset,
we apply the 
MN2LS
baseline procedure
as well as
the zero-step procedure.

In each run,
we additionally generate independent test datasets
each with $10000$ i.i.d.\ observations
from the same $p+1$ dimensional distribution described above
in order to approximate
the true risk of the zero-step and the base
prediction procedure.
\Cref{fig:gaufeat_isocov_isosig_linmod_n1000_nv100_nruns100_zerostep_mnls_gamma10}
shows the risks of the baseline MN2LS procedure
and the zero-step prediction procedure
for high (left, SNR = 4) and low (right, SNR = 1) SNR regimes;
we take $\sigma^2 = 1$ and $\rho^2$=SNR.
We also present the null risk
($\rho^2 + \sigma^2$),
i.e., the risk of the zero predictor
as a baseline in both the plots.
We observe from the figure that
the risk of the zero-step procedure
for every $M \ge 1$ is non-decreasing in $\gamma$.
\Cref{thm:asymptotic-risk-tuned-zero-step}
implies that the risk of the zero-step
prediction procedure for every $M \ge 1$
is \emph{asymptotically}
bounded by the risk of the base prediction procedure
at each aspect ratio $(\gamma)$.
Although this is
somewhat evident from
\Cref{fig:gaufeat_isocov_isosig_linmod_n1000_nv100_nruns100_zerostep_mnls_gamma10},
it is not satisfied for all $\gamma$,
especially for $M = 1$.
This primarily stems from
the smaller sample size at hand
and the fact that we are comparing
MN2LS trained on full data ($n = 1000$)
to the zero-step predictor computed on
the train data $(n_\train = 900)$.
With an increased sample size
(to say, $n = 2500$),
this finite-sample discrepancy vanishes.

\Cref{fig:gaufeat_isocov_isosig_linmod_n1000_nv100_nruns100_zerostep_mnls_gamma10}
shows that the zero-step procedure
with $M = 1$
attains risk monotonization
in a precise sense that its risk
is the largest non-increasing function (of $\gamma$)
below the risk of the MN2LS predictor.
For $M > 1$,
our results do not characterize the risk
of zero-step predictor,
but \Cref{fig:gaufeat_isocov_isosig_linmod_n1000_nv100_nruns100_zerostep_mnls_gamma10}
shows that averaging has a significant effect
in further reducing the risk.
As mentioned before,
this is expected from the theory of $U$-statistics
as $U$-statistics are UMVUE's of their expectations
(see, e.g., Chapter 5 of \cite{serfling_2009}).
All these comments hold
for both low and high SNR alike.

Note that the base predictor
has unbounded risk near $\gamma = 1$.
The risk of the zero-step procedure,
on the other hand,
is always bounded for all $M \ge 1$
and all $\gamma$.
In this sense,
the zero-step procedure
can also be used as a general procedure
for mitigating the surprising descent
behavior in the prediction risk.

\begin{figure}[!t]
    \centering
\includegraphics[width=0.45\columnwidth]{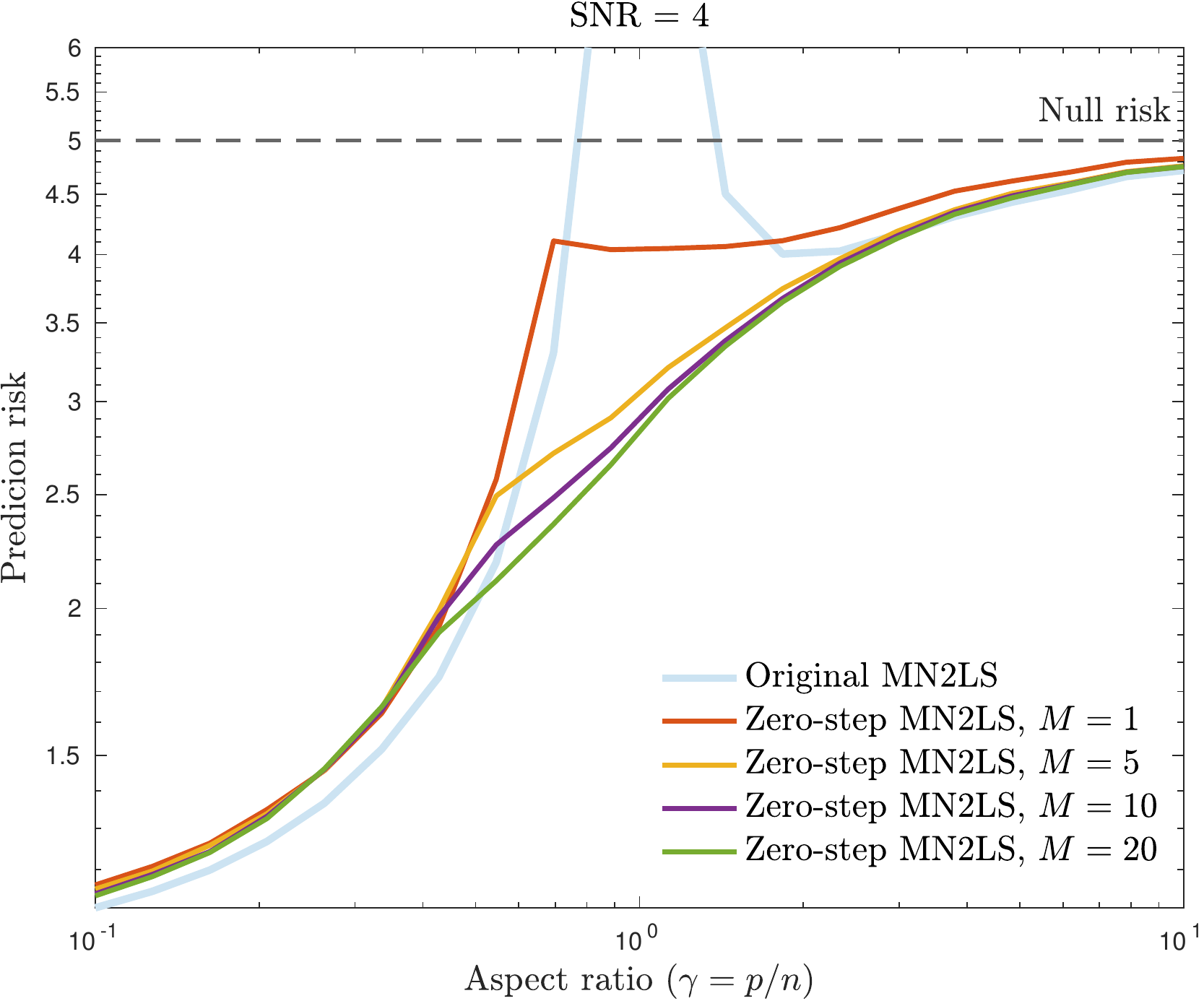}
    \quad
    \includegraphics[width=0.45\columnwidth]{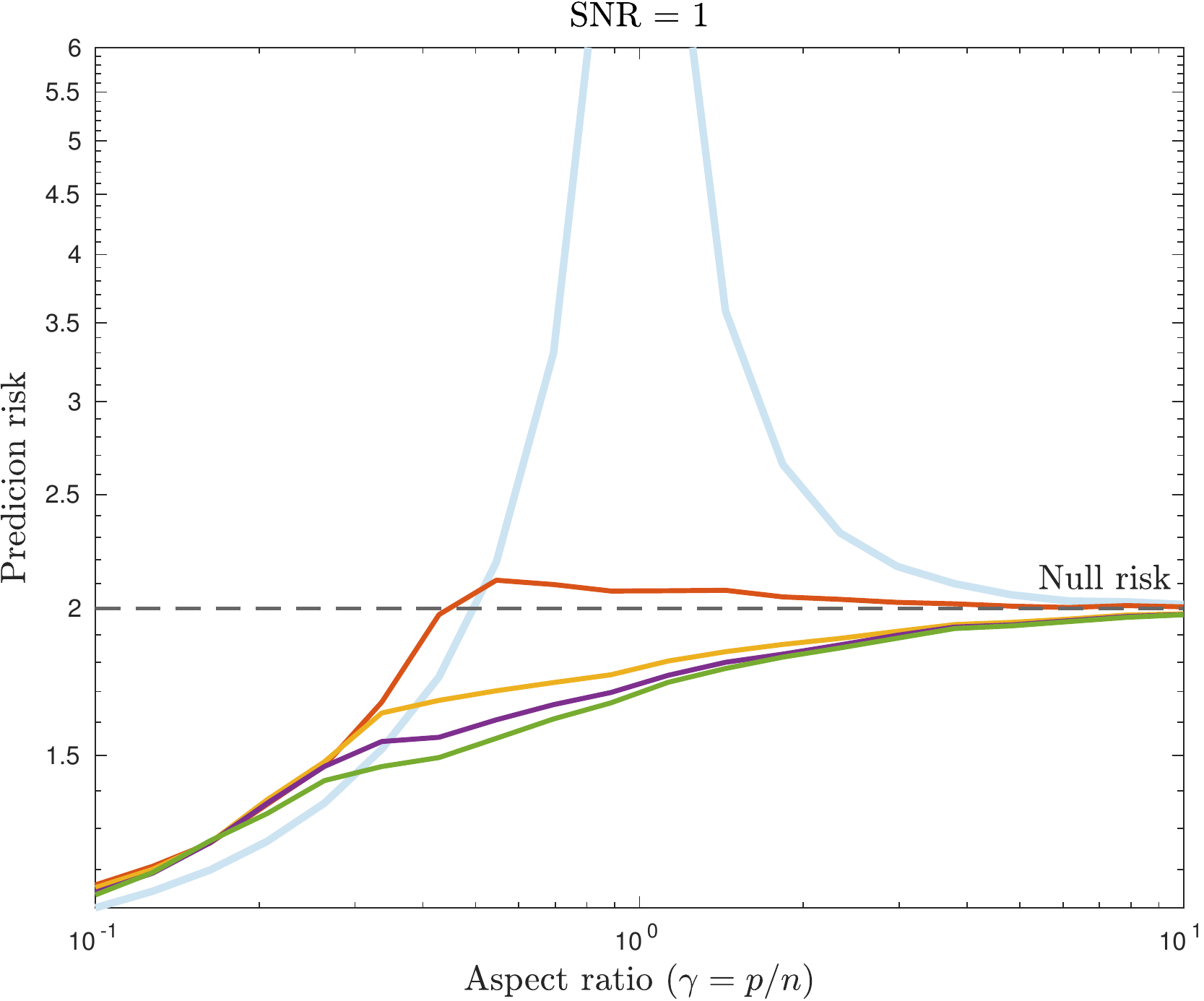}
    \caption{
    Illustration of the zero-step prediction procedure
    with MN2LS as the base predictor with varying $M$.
    The left panel shows a high SNR regime (SNR = 4),
    while the right panel shows a low SNR regime (SNR = 1).
    Here, $n = 1000$, $n_\train = 900$, $n_\test = 100$, $n^\nu = 50$.
    The features are drawn from an isotropic Gaussian distribution,
    the response follows a  linear model.
    The figure show averaged risk over 100 dataset repetitions.
    }
    \label{fig:gaufeat_isocov_isosig_linmod_n1000_nv100_nruns100_zerostep_mnls_gamma10}
\end{figure}

\paragraph{Minimum $\ell_1$-norm least squares (MN1LS).}

We fix $n = 500$
and vary the dimension $p$ of the features
from $50$ to $50000$
(for a total of $30$ values of $\gamma = p/n$
logarithmically spaced between 0.1 to 100).
This will show risk behavior of zero-step
procedure for aspect ratios between 0.1 and 100.
For every pair of sample size $n = 500$
and dimension $p$,
we generate $250$ independent dataset
each with $n$ i.i.d.\ observations
from the linear model
$Y_i = X_i^\top \beta_0 + \eps_i$,
where
$X_i \in \cN(0_p, I_p)$,
$\beta_0$ has coordinates
generated i.i.d.\ from the distribution
$B \delta_{r/\sqrt{p \pi}} + (1 - B) \delta_{0}$,
where $B \sim \textrm{Bernoulli}(\pi = 0.005)$
and $\eps_i \sim \cN(0, \sigma^2)$
is independent of $X_i$.
The model represents a sparse signal regime
(with linear sparsity level $\pi$)
with average signal energy $\rho^2$.
We again define SNR to be $\rho^2 / \sigma^2$.
On each dataset, we apply
MN1LS
baseline procedure
as well as
the zero-step procedure.

In each run,
we additionally generate independent test datasets
each with $10000$ i.i.d.\ observations
from the same $p+1$ dimensional distribution described above
in order to approximate
the true risk of the zero-step and the base
prediction procedure.
\Cref{fig:gaufeat_isocov_sparsesig_linmod_n500_nv80_nruns250_zerostep_mnla_gamma100}
shows the risks of the baseline MN1LS procedure
and the zero-step procedure for high
(left, SNR = 4)
and low (right, SNR = 1)
SNR regimes.
We take $\sigma^2 = 1$ and $\rho^2$=SNR.
We also present the null risk
($\rho^2 + \sigma^2$),
i.e., the risk of the zero predictor
as a baseline in both the plots.
We again observe that
the risk of the zero-step procedure
for every $M \ge 1$ is non-decreasing in $\gamma$.

Similar to 
\Cref{fig:gaufeat_isocov_isosig_linmod_n1000_nv100_nruns100_zerostep_mnls_gamma10},
we observe in
\Cref{fig:gaufeat_isocov_sparsesig_linmod_n500_nv80_nruns250_zerostep_mnla_gamma100}
that the zero-step procedure
with $M = 1$
attains precise risk monotonization
while
zero-step with $M > 1$ improves significantly upon
the $M = 1$ when $\gamma$ is near one.
All these comments hold
for both low and high SNR alike.

As with \Cref{fig:gaufeat_isocov_isosig_linmod_n1000_nv100_nruns100_zerostep_mnls_gamma10},
note that the base predictor MN2LS
has unbounded risk near $\gamma = 1$
in
\Cref{fig:gaufeat_isocov_sparsesig_linmod_n500_nv80_nruns250_zerostep_mnla_gamma100}.
The risk of the zero-step procedure,
on the other hand,
is always bounded for all $M \ge 1$
and all $\gamma$.

\begin{figure}[!t]
    \centering
    \includegraphics[width=0.45\columnwidth]{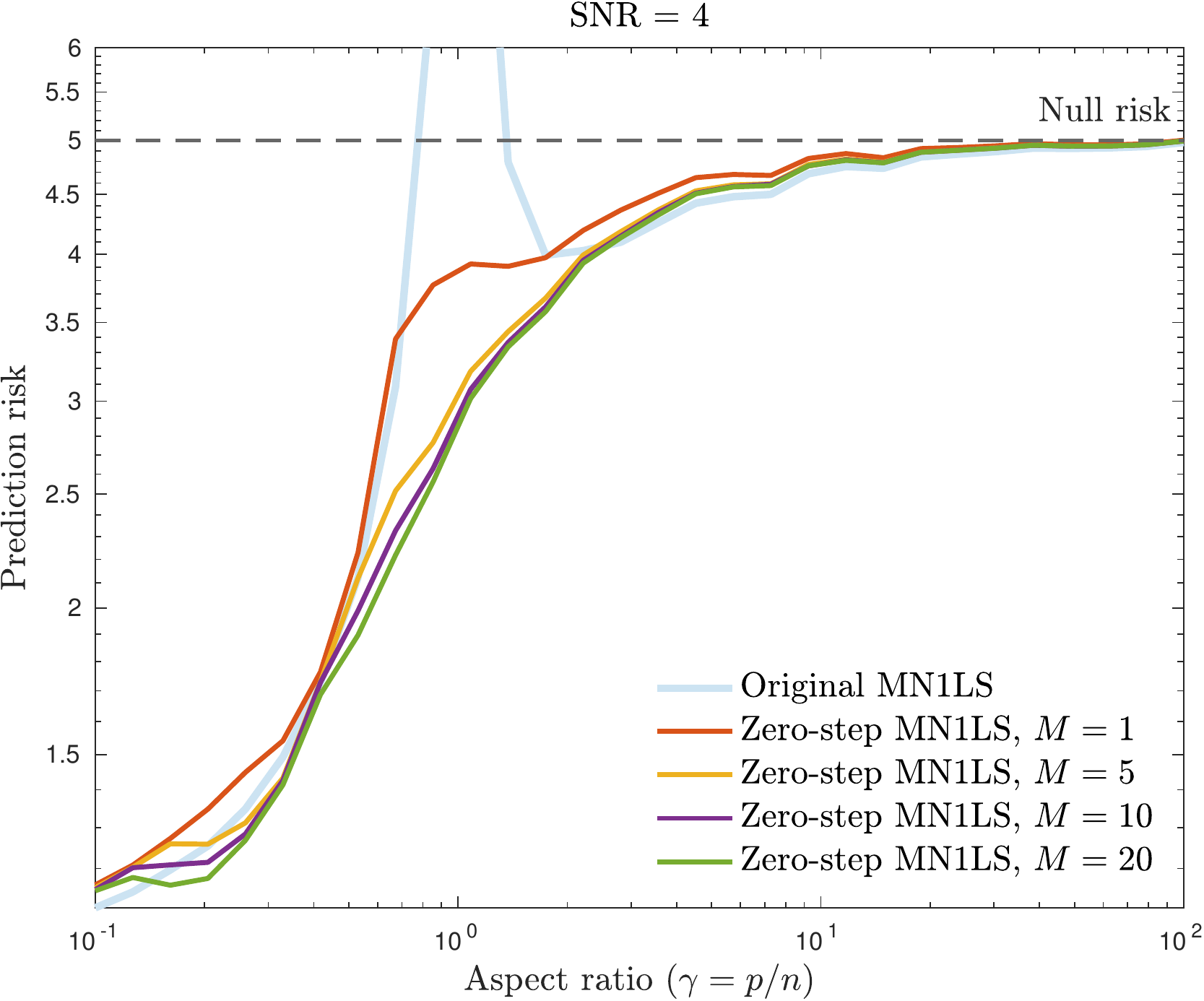}
    \quad
    \includegraphics[width=0.45\columnwidth]{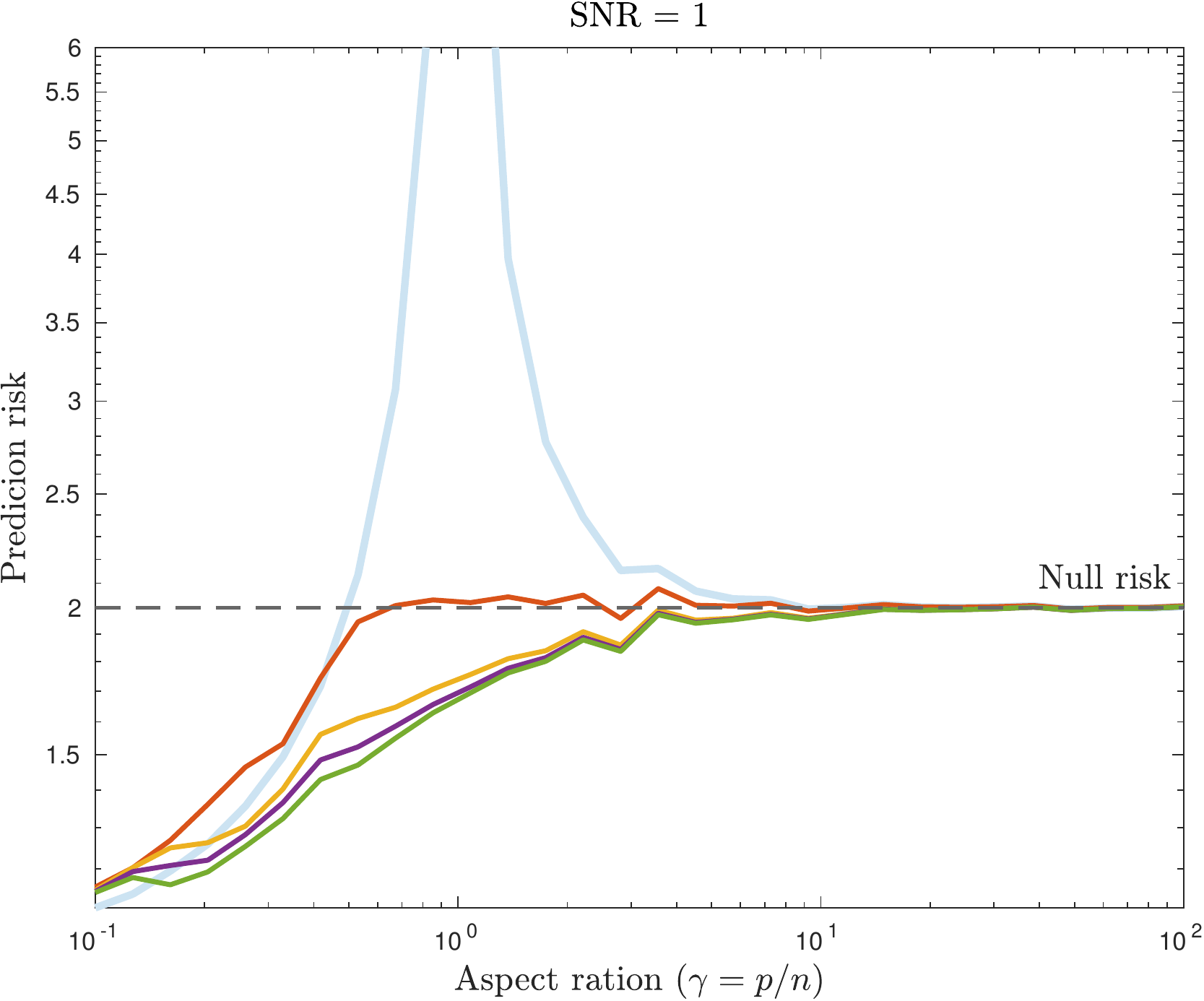}
    \caption{Illustration of 
    the zero-step prediction procedure with
    MN1LS as the base predictor with varying $M$.
    The left panel shows a high SNR regime (SNR = 4),
    while the right panel shows a low SNR regime (SNR = 1).
    Here, $n = 500$, $n_\train = 420$, $n_\test = 80$, $n^\nu = 42$.
    The features are drawn from an isotropic Gaussian distribution,
    the response follows a linear model
    with sparse signal (sparsity level = 0.005).
    The risks are averaged over 250 dataset repetitions.
    }
        \label{fig:gaufeat_isocov_sparsesig_linmod_n500_nv80_nruns250_zerostep_mnla_gamma100}
\end{figure}

\section{Application 2: One-step prediction procedure}

\label{sec:one-step}

\subsection{Motivation}

The zero-step procedure introduced in \Cref{sec:zero-step}
provides the desired asymptotic monotonization of the conditional
prediction risk under certain regularity conditions.
It takes advantage of the fact that we can train our predictors
on a smaller subset of the data when it is appropriate.
In addition, it uses repeated sampling and averaging
in order to remove the external randomness in the choice of the subset.

In this section,
we introduce a variant of the zero-step procedure
motivated by the classical statistical idea
of one-step estimation
\citep[see, e.g., Section 5.7 of][]{vanderVaart_2000}.
In the simplest case of linear regression
where the feature dimension is fixed,
the idea of one-step estimation is
that we can start with an arbitrary linear predictor
and add to it an adjustment
computed based on the residuals of the initial
linear predictor.
More precisely,
starting with any initial estimator $\tbeta^\init$
and the associated linear predictor 
$\tf(x) = x^\top \tbeta^\init$,
we have
\begin{equation}
    \label{eq:one-step-ols}
    \underbrace
    {
        \vphantom{\left( \frac{1}{n} \sum_{i=1}^{n} X_i X_i^\top \right)^{-1}} 
        X^\top \tbeta^\init
    }
    _{\text{initial predictor}}
    +
    ~
    \underbrace
    {
        X^\top
        \left(\frac{1}{n} \sum_{i=1}^{n} X_i X_i^\top\right)^{-1}
        \left(\frac{1}{n} \sum_{i=1}^{n} X_i (Y_i - X_i^\top \tbeta^\init)\right)
    }_{\text{one-step adjustment}}
    =
    X^\top \tbeta^{\ols},
\end{equation}
where the final resulting predictor
corresponds to the ordinary least squares (OLS) estimator $\tbeta^\ols$
that enjoys $n^{-1/2}$ rate and
risk optimality under a well-specified linear model.

This idea of one-step estimation
is not specific to ordinary least squares.
It can be generalized to other estimators
that are solutions to estimating
equation $\Psi_n(\beta) = 0$
where $\Psi_n : \RR^{p} \to \RR^{p}$.
The general idea is to solve
a linear approximation
to the estimating equation, i.e.,
given an initial estimator $\tbeta^\init$,
the one-step estimator is the solution
(in $\beta$) to 
the linearized estimating equation
(around $\tbeta^\init$)
\[
    \Psi_n(\tbeta^\init)
    +
    \nabla{\Psi_n}(\tbeta^\init)
    (\beta - \tbeta^{\init})
    =
    0.
\]
The solution can be expressed as
\begin{equation}
    \label{eq:one-step-general}
    \tbeta
    =
    \underbrace
    {
        \vphantom{\dot(\Psi)(\tbeta^\init)^{-1}}
        \tbeta^\init
    }
    _{\text{initial estimator}}
    -
    ~
    \underbrace
    {(\nabla{\Psi}(\tbeta^\init))^{-1} \Psi(\tbeta^\init)}_{\text{one-step adjustment}}.
\end{equation}
Here $\nabla{\Psi} : \RR^{p} \to \RR^{p} \times \RR^{p}$ denotes the Jacobian of $\Psi$.

One can also view the one-step estimator from the point of view of the Newton's algorithm.
The classical one-step estimator starts at an initial estimator $\tbeta^{\init}$
and takes a Newton's step on the empirical risk minimization problem.
For a parametric predictor $f(\cdot; \tbeta^{\init})$, 
starting with a base estimator $\tbeta^\init$,
we can define the corresponding
one-step predictor as $f(\cdot; \tbeta)$,
where $\tbeta$
is the Newton's step update
starting with $\tbeta^\init$
given by
\begin{equation}
    \label{eq:one-step-newton}
    \tbeta
    = 
    \underbrace
    {
        \vphantom{\left(\frac{1}{n} \sum_{i=1}^{n} \ddot{\ell} \right)^{-1}}
        \tbeta^{\init}
    }
    _{\text{initial estimator}}
    - 
    ~
    \underbrace
    {\left( \frac{1}{n} \sum_{i=1}^{n} \nabla^2{\ell}(Y_i, f(X_i; \tbeta^{\init})) \right)^{-1}
    \left( \frac{1}{n} \sum_{i=1}^{n} \nabla{\ell}(Y_i, f(X_i; \tbeta^{\init})) \right)
    }
    _{\text{Newton's step}}.
\end{equation}
Here, for $1 \le i \le n$,
$\nabla \ell(Y_i, f(X_i; \cdot)) : \RR^p \to \RR^p$ denotes the gradient
of the prediction loss function 
$\ell(Y_i, f(X_i; \beta))$
with respect to $\beta$,
and $\nabla^2{\ell}(Y_i, f(X_i; \cdot)) : \RR^{p} \to \RR^{p \times p}$
denotes the Hessian of the prediction loss function
with respect to $\beta$.
In the special case of a linear predictor,
where $f(x; \beta) = x^T \beta$,
the one-step estimator becomes
\[
    \tbeta
    = \tbeta^{\init}
    - \left( \frac{1}{n} \sum_{i=1}^{n} X_i X_i^T \ell''(Y_i, X_i^T \tbeta^{\init}) \right)^{-1}
    \left( \frac{1}{n} \sum_{i=1}^{n} X_i \ell'(Y_i, X_i^T \tbeta^{\init})\right),
\]
where 
$\ell': \RR \times \RR \to \RR$ is the first derivative of the loss function
$\ell(\cdot, \cdot)$ in the second coordinate,
and $\ell'': \RR \times \RR \to \RR$ is the second derivative
of the loss function in the second coordinate.

Our goal in this section
is to build upon this idea of one-step
estimation towards risk-monotonization
and improve on the zero-step procedure.
We will restrict ourselves
to one-step adjustment with respect to the square error loss
and linear predictors (per \eqref{eq:one-step-ols}).
We leave extension to a more general one-step adjustment
(per \eqref{eq:one-step-general} or \eqref{eq:one-step-newton})
for future work. 
For more discussion, see \Cref{sec:discussion}.

There are two points to note when defining \eqref{eq:one-step-ols}.
\begin{enumerate}
    \item
    The inverse 
    of the sample covariance matrix
    $\sum_{i=1}^{n} X_i X_i^\top / n$
    in \eqref{eq:one-step-ols}
    need not always exist.
    In particular,
    when the feature dimension $p > n$,
    the sample covariance matrix
    is guaranteed to be rank deficient.
    \item
    In the overparameterized regime,
    the residuals $Y_i - X_i^\top \tbeta^\init$
    for $i = 1, \dots, n$
    in \eqref{eq:one-step-ols}
    are identically zero for several commonly used estimators
    such MN2LS or MN1LS,
    if $\tbeta^\init$ and the residuals are computed on the same dataset.
\end{enumerate}

In order to overcome these two limitations,
we consider a variant of the idea of one-step estimation,
in which we make the following changes:
\begin{enumerate}[label=\arabic*$'$.]
    \item We use a Moore-Penrose pseudo-inverse
    in place of regular matrix inverse.
    Note that this is the same as adding
    a MN2LS component fitted on the residuals $Y_i - X_i^\top \tbeta^\init$.
    \item We split the training data
    and use one part to compute $\tbeta^\init$
    and use the other part to compute the residuals $Y_i - X_i^\top \tbeta^\init$.
    This ensures that the residuals are not identically zero
    in the overparameterized regime.
\end{enumerate}

In summary, to construct the one-step predictor,
we start with a base predictor computed
on a subset of data,
evaluate the residuals of this predictor
on a different subset of data,
and add to the base predictor
a MN2LS fit on the residuals.
We formalize this construction next.

\subsection{Formal description}

As before,
let the original dataset be denoted by
$\cD_n = \{ (X_1, Y_1), \dots, (X_n, Y_n) \}$
and let $\tf$ be a base prediction procedure.
As per \Cref{alg:general-cross-validation-model-selection},
let the train and test datasets be
$\cD_\train$ and $\cD_\test$, respectively.
We define the ingredient predictors
to be used in \Cref{alg:general-cross-validation-model-selection}
constructed using the one-step methodology as follows:
Define the index set $\Xi_n$ as
\[
    \Xi_n
    :=
    \Big\{
        (\xi_1, \xi_2)
        ~:~
        \xi_1 \in \{ 0, 1, \dots, n_\train - 1 \},
        \xi_2 \in \{ 0, 1, \dots, \xi_1 - 1 \}
    \Big\}.
\]
Let $\cD_\train^{\xi_1}$ and $\cD_\train^{\xi_2}$
be disjoint subsets of $\cD_\train$
with $n_\train - \xi_1$
(for $0 \le \xi_1 \le n_\train - 1$)
and $\xi_2$
(for $0 \le \xi_2 \le \xi_1$)
observations, respectively.
Let $\cI_\train^{\xi_1}$ and $\cI_\train^{\xi_2}$
denote the corresponding index sets of $\cD_\train^{\xi_1}$
and $\cD_\train^{\xi_2}$, respectively.
For each index $\xi = (\xi_1, \xi_2) \in \Xi_n$,
define the ingredient predictor $\tf^\xi$
to be used in \Cref{alg:general-cross-validation-model-selection}
in three steps:
\begin{enumerate}
    \item Fit a base prediction procedure
    $\tf$ on $\cD_\train^{\xi_1}$.
    Call this $\tf(\cdot; \cD_\train^{\xi_1})$.
    \item Compute the residuals of predictor
    $\tf(\cdot; \cD_\train^{\xi_1})$
    on $\cD_\train^{\xi_2}$,
    i.e.,
    $r_j = Y_j - \tf(X_j; \cD_\train^{\xi_1})$
    for $j \in \cI_\train^{\xi_2}$.
    \item
    Fit the MN2LS predictor on
    $\{ (X_j, r_j) : j \in \cI_\train^{\xi_2} \}$.
    This is the one-step adjustment.%
\end{enumerate}
The final ingredient predictor $\tf^\xi$
is given by
\[
    \tf^\xi(x; \cD_{\train}^{\xi_1}, \cD_{\train}^{\xi_2})
    ~:=~ \tf(x; \cD_\train^{\xi_1})
    +
    x^\top
    \left(
        \sum_{j \in \cI_\train^{\xi_2}}
        X_j X_j^\top
    \right)^{\dagger}
    \left(
        \sum_{j \in \cI_\train^{\xi_2}}
        X_j
        r_j
    \right).
\]
If $\xi_2 = 0$,
then $\cI_{\train}^{\xi_2}$ is an empty set
and there are no residuals $r_j$ computed.
In this case,
we adopt the convention that there is no one-step
adjustment.
Therefore,
the ingredient predictors for our one-step procedure
includes the ingredient predictors for the zero-step procedure.
As with the zero-step procedure,
two remarks are in order:

\begin{itemize}[leftmargin=*]
    \item There is external randomness
    in choosing subsets $\cD_\train^{\xi_1}$
    and $\cD_\train^{\xi_2}$ of sizes
    $n_\train - \xi_1$ and $\xi_2$, respectively.
    To reduce such randomness,
    we make use of many different subsets
    of the same sizes and average such different one-step predictors.
    More precisely,
    for each $\xi = (\xi_1, \xi_2) \in \Xi$,
    draw $m$ disjoint pairs of
    sets
    $(\cD_\train^{\xi_1, j}, \cD_\train^{\xi_2, j}), \dots, (\cD_\train^{\xi_1, j}, \cD_\train^{\xi_2, j})$
    from $\cD_\train$.
    Formally,
    for $1 \le j \le m$,
    we randomly draw a subset $\cD_\train^{\xi_1, j}$
    from $\cD_\train$ of size $n_\train - \xi_1$
    and a subset $\cD_\train^{\xi_2, j}$
    from $\cD_\train \setminus \cD_\train^{\xi_1, j}$
    of size $\xi_2$.
    We then fit different one-step predictors
    $\tf({\cdot; \cD_\train^{\xi_i, j}, \cD_\train^{\xi_2, j}})$
    on $(\cD_\train^{\xi_1, j}, \cD_\train^{\xi_2, j})$
    for $1 \le j \le M$,
    and take the final ingredient predictor $\hf^\xi$ to be
    the average of $M$ such predictors:
    \begin{equation}\label{eq:onestep-average}
        \hf^{\xi}(x)
        =
        \frac{1}{M}
        \sum_{j=1}^{M}
        \tf(x; \cD_\train^{\xi_1, j}, \cD_\train^{\xi_2, j}).
    \end{equation}
    As before, when $M = \infty$,
    $\hf^\xi$ becomes the average of all possible
    pairs of disjoints subsets $\cD_\train$ of sizes
    $n_\train - \xi_1$ and $\xi_2$,
    while the case of $M = 1$ has the largest amount of external randomness.
    Based on the theory of $U$-statistics,
    we again expect the choice of $M = \infty$
    to provide a predictor with the smallest variance.
    For computational reasons, we use a finite value of $M \ge 1$.
    
    \item In the description above,
    we have $n_\train (n_\train + 1) / 2$ predictors to use in \Cref{alg:general-cross-validation-model-selection}.
    Similar to the zero-step procedure,
    we replace $\Xi_n$ with
    \begin{equation}\label{eq:indexset-onestep}
        \Xi_n
        :=
        \left\{ 
            (\xi_1, \xi_2) ~:~
            \xi_1 \in
            \left\{
                2, \dots, \left\lceil \frac{n_\train}{\lfloor n^\nu \rfloor} - 2 \right\rceil
            \right\},
            \xi_2 \in
            \left\{ 
                1, \dots, \xi_1 -1
            \right\}
        \right\},
        \quad
        \text{for some }
        \nu \in (0, 1),
    \end{equation}
    and consider predictors obtained
    by training components of $\tf$
    on subsets of sizes
    $n_\train - \xi_1 \lfloor n^\nu \rfloor$
    and $\xi_2 \lfloor n^\nu \rfloor$.
    Such a change helps
    in reducing the cost
    of computing $\hf^\cv$ using 
    \Cref{alg:general-cross-validation-model-selection}.
    In addition, this also helps
    in the statistical properties of $\hf^\cv$
    when applying the union bound in the results
    of \Cref{sec:general-crossvalidation-modelselection}.
\end{itemize}

With these two modifications,
with $\Xi_n$ as defined in \eqref{eq:indexset-onestep},
for $\xi \in \Xi_n$,
we define $\hf^\xi$ as in \eqref{eq:onestep-average}
with the subsets
$\cD_\train^{\xi_1, j}$,
$\cD_\train^{\xi_2, j}$
(for $1 \le j \le M$)
now representing disjoints subsets of
sizes $n_\train - \xi_1 \lfloor n^\nu \rfloor$
and $\xi_2 \lfloor n^\nu \rfloor$, respectively.
The ingredients predictors to be used
in \Cref{alg:general-cross-validation-model-selection}
are given by $\hf^\xi$, $\xi \in \Xi_n$.
We call the resulting predictor obtained
from \Cref{alg:general-cross-validation-model-selection}
as the one-step predictor based on $\tf$,
and we denote the corresponding prediction procedure
to be $\hf^\onestep$.
The one-step procedure is summarized
in \Cref{alg:one-step}.

\begin{algorithm}
    \caption{One-step procedure}
    \label{alg:one-step}
    \textbf{Inputs}:\\
    \begin{itemize}[noitemsep]
        \item[--] all inputs of \Cref{alg:general-cross-validation-model-selection}
        other than the index set $\Xi$;
        \item[--] a positive integer $M$.
    \end{itemize}
    \textbf{Output}:\\
    \vspace{-1em}
    \begin{itemize}
        \item[--] a predictor $\hf^\onestep$
    \end{itemize}
    \textbf{Procedure}:
    \begin{enumerate}
        \item Let $n_\train = n - n_\test$.
        Construct an index set $\Xi_n$ per \eqref{eq:indexset-onestep}.
        \item Construct train and test sets $\cD_\train$ and $\cD_\test$
        per Step 1 of \Cref{alg:general-cross-validation-model-selection}.
        \item Let $n_{1,\xi_1} = n_\train - \xi_1 \lfloor n^\nu \rfloor$
        and $n_{2, \xi_2} = \xi_2 \lfloor n^\nu \rfloor$.
        For each $(\xi_1, \xi_2) \in \Xi_n$ and $j = 1, \dots, M$,
        draw random pairs of disjoint subsets
        $(\cD_\train^{\xi_1, j}, \cD_\train^{\xi_2, j})$ of sizes
        $n_{1,\xi_1}$ and $n_{2, \xi_2}$ from $\cD_\train$, respectively.
        For each $(\xi_1, \xi_2) \in \Xi_n$,
        fit predictors $\hf^\xi$ as described by \eqref{eq:onestep-average}
        using prediction procedure $\tf$ and
        $\{ (\cD_\train^{\xi_1, j}, \cD_\train^{\xi_2, j}) : 1 \le j \le M \}$.
        \item Run Steps 3--5 of \Cref{alg:general-cross-validation-model-selection}
        using index set $\Xi = \Xi_n$
        and set of predictors $\{ \hf^\xi$, $\xi \in \Xi \}$.
        \item Return $\hf^\onestep$ as the resulting $\hf^\cv$ from 
        \Cref{alg:general-cross-validation-model-selection}.
    \end{enumerate}
\end{algorithm}

\subsection[Risk behavior]{Risk behavior of $\hf^\onestep$}

In this section,
we examine the risk behavior of
one-step predictor $\hf^\onestep$.
Similar treatment
as done for the zero-step procedure
in \Cref{sec:risk-behavior-zerostep}
applies in general.
To avoid repetition,
we will primarily restrict ourselves to
overparameterized setting in this section.

\subsubsection
[Overparameterized regime]
{Risk behavior of $\hf^\onestep$ under proportional asymptotics}
\label{sec:onestep-overparameterized}

Define
$n_{1,\xi_1} = n_\train - \xi_1 \lfloor n^\nu \rfloor$
and $n_{2, \xi_2} = \xi_2 \lfloor n^\nu \rfloor$.
Assume that
there exists a deterministic profile $R^\deter(\cdot, \cdot; \tf): \RR \times \RR \to \RR$ of $\tf$
such that
the following holds:
\begin{equation}
    \label{eq:rn-deterministic-approximation-onestep-prop-asymptotics}
    \tag{DETPA-1}
    \left|
        R
        \big(
            \tf(\cdot;
            \cD_\train^{\xi_{1,n}^\star, j},
            \cD_\train^{\xi_{2,n}^\star, j})
        \big)
        -
        R^\deter
        \left(
            \frac{p}{n_{1, \xi_{1,n}^\star}},
            \frac{p}{n_{2, \xi_{2,n}^\star}};
            \tf
        \right)
    \right|
    ~=~
    o_p(1)
    R^\deter
    \left(
        \frac{p}{n_{1, \xi_{1,n}^\star}},
        \frac{p}{n_{2, \xi_{2,n}^\star}};
        \tf
    \right),
\end{equation}
where $(\xi_{1, n}^\star, \xi_{2, n}^\star)$ are the indices
that minimize the deterministic profile $R^\deter(\cdot, \cdot; \tf)$:
\begin{equation}
    \label{eq:minimizer-sequence-onestep}
    (\xi_{1, n}^\star, \xi_{2, n}^\star)
    \in
    \argmin_{(\xi_1, \xi_2) \in \Xi_n}
    R^\deter
    \left(
    \frac{p}{n_{1, \xi_1}},
    \frac{p}{n_{2, \xi_2}};
    \tf
    \right).
\end{equation}

Because $\log(|\Xi_n|) \le 2 \log(n)$,
following the arguments
in \Cref{sec:risk-behavior-zerostep},
we conclude that if
\eqref{eq:rn-deterministic-approximation-onestep-prop-asymptotics}
and
either \ref{cond:zerostep-add}\footnote{
Here, we need \ref{cond:zerostep-add}
with $R^\deter_\nearrow(n, \tf)$
replaced with the minimum appearing in
\eqref{eq:onestep-risk-gaurantee}.}
or \ref{cond:zerostep-mult}
hold,
then
\begin{equation}
    \label{eq:onestep-risk-gaurantee}
   \left( 
    R(\hf^\onestep)
    -
    \min_{(\xi_1, \xi_2) \in \Xi_n}
    R^\deter
    \left(
        \frac{p}{n_{1, \xi_1}},
        \frac{p}{n_{2, \xi_2}};
        \tf
    \right)
    \right)_+
    ~=~
    o_p(1)
    \cdot
    \min_{(\xi_1, \xi_2) \in \Xi_n}
    R^\deter
    \left(
        \frac{p}{n_{1, \xi_1}},
        \frac{p}{n_{2, \xi_2}};
        \tf
    \right).
\end{equation}

Just as we reduced verification of
\eqref{eq:rn-deterministic-approximation-2-prop-asymptotics}
to
\eqref{tag:detpar-0},
we state below a reduction of the verification of
\eqref{eq:rn-deterministic-approximation-onestep-prop-asymptotics}
that only considers non-deterministic sequences
for which the aspect ratios of the split datasets
for the constituent one-step predictors converge.

For any $\gamma > 0$, define
\[
    \cM_{\gamma}^{\onestep}
    ~:=~ \argmin_{(\zeta_1, \zeta_2) : \zeta_1^{-1} + \zeta_2^{-1} \le \gamma^{-1}} R^\deter(\zeta_1, \zeta_2; \tf).
\]

\begin{lemma}
    [Reduction of \eqref{eq:rn-deterministic-approximation-onestep-prop-asymptotics}]
    \label{lem:deterministic-approximation-reduction-onestep}
    Suppose $\cD_{k_{1,m}}$ and $\cD_{k_{2,m}}$
    are dataset with $k_{1,m}$ and $k_{2,m}$ observations
    and $p_m$ features.
    Assume the loss function $\ell$ is such that
    $R(\tf(\cdot; \cD_{k_{1, m}}, \cD_{k_{2, m}}))$
    is uniformly bounded away from $0$.
    Let $\gamma > 0$ be a real number.
    Suppose there exists a proper, lower semicontinuous function 
    $R^\deter : [\gamma, \infty] \times [\gamma, \infty] \to [0, \infty]$ such that
    the following holds true:
    \begin{equation}
        \label{eq:rn-deterministic-approximation-reduced-onestep-prop-asymptotics}
        \tag{DETPAR-1}
            R
            \big(
                \tf(\cdot; \cD_{k_{1,m}}, \cD_{k_{2,m}})
            \big)
        \pto
        R^\deter(\phi_1, \phi_2; \tf)
    \end{equation}
    as $k_{1,m}, k_{2,m}, p_m \to \infty$
    and 
    $(p_m / k_{1,m}, p_m / k_{2,m}) \to (\phi_1, \phi_2) \in \cM_{\gamma}^{\onestep}$.
    Furthermore, suppose that $R^\deter(\cdot, \cdot; \tf)$
    is continuous on the set
    $\cM_{\gamma}^{\onestep}$.
    Then,
    \eqref{eq:rn-deterministic-approximation-onestep-prop-asymptotics}
    is satisfied.
\end{lemma}

The proof of 
\Cref{lem:deterministic-approximation-reduction-onestep}
follows analogously to that of
\Cref{lem:rn-deterministic-approximation-4-prop-asymptotics}
where we show that even though
the sequence $\{ \bm{\Phi}_n = (p_n / n_{1, \xi_{1,n}^\star}, p_n / n_{2, \xi_{2,n}^\star}) \}_{n \ge 1}$ may not converge,
there exists a subsequence
$\{ \bm{\Phi}_{n_{k_{l}}} \}_{l \ge 1}$ that converges
to some $(\phi_1, \phi_2) \in \cM_{\gamma}^{\onestep}$.
Below we provide some commentary on the assumptions
of \Cref{lem:deterministic-approximation-reduction-onestep}.

\begin{itemize}[leftmargin=*]
\item 
We note that 
assuming lower semicontinuity of $R^\deter(\cdot, \cdot; \tf)$
is a mild assumption.
In particular, 
it does not preclude the possibility that
$R^\deter$ diverges to $\infty$ at several values in the domain
as shown in \Cref{prop:semicontinuity-metricspace}.
For example, the proposition
implies that if $R^\deter(\cdot, \cdot; \tf)$ is continuous
on a set 
except for when $\phi_1 = 1$ or $\phi_2 = 1$,
then $R^\deter$ is lower semicontinuous,
provided $R^\deter$ diverges to $\infty$
when either $\phi_1$ or $\phi_2$ converges to $1$.
The condition of lower semicontinuous
deterministic approximation $R^\deter(\cdot; \cdot; \tf)$
follows from the continuity of the domain
of $R^\deter(\cdot, \cdot; \tf)$
(i.e., points of finite function value).
This is similar to \Cref{prop:lower-semicontinuity-divergence}
discussed in the context of the zero-step predictor.
The formal statement for the one-step
predictor is as follows.
\end{itemize}
\begin{proposition}
    [Verifying lower semicontinuity 
    for diverging risk profiles]
    \label{prop:semicontinuity-metricspace}
    Let $(M, d)$ be a metric space.
    Let $C$ be a closed set.
    Suppose $h : M \to \overline{\RR}$
    is a function such that
    $h(x) < \infty$ for $x \in M \setminus C$,
    and $h(x) = \infty$ for $x \in C$.
    In addition,
    if $h$ restricted to $M \setminus C$
    (denoted by $h|_{M \setminus C}(\cdot)$)
    is continuous,
    and for any sequence $\{ x_n \}_{n \ge 1}$
    that converges to a point in $C$,
    $\{ h(x_n) \}_{n \ge 1}$ converges to $\infty$.
    Then,
    $h$ is lower semicontinuous on $M$.
\end{proposition}
\begin{itemize}[leftmargin=*]

\item 
Continuity assumption 
on $R^\deter(\cdot, \cdot; \tf)$
at the argmin set 
$\cM_{\gamma}^{\onestep}$
is also mild.
\Cref{prop:continuity-from-continuous-convergence-rdet-onestep}
below shows that
\eqref{tag:detpar-0}
holding for $(\phi_1, \phi_2)$ in any open set $\cI$
implies continuity of $R^\deter$ on $\cI$.

\end{itemize}
\begin{proposition}
    [Certifying continuity from continuous convergence]
    \label{prop:continuity-from-continuous-convergence-rdet-onestep}
    Let $\cD_{k_{1,m}}$ and $\cD_{k_{2,m}}$ be datasets 
    with $k_{1, m}$ and $k_{2, m}$ observations
    and $p_m$ features,
    and consider one-step ingredient prediction procedure
    $\tf$ trained on $\cD_{k_{1,m}}$ and $\cD_{k_{2,m}}$.
    Fix a open set $\cI \subseteq (0, \infty] \times (0, \infty]$.
    Suppose there exists a function $R^\deter : (0, \infty] \times (0, \infty] \to [0, \infty]$
    such that
    \begin{equation}
        \label{eq:pointiwise-convergen-rdet}
        R(\tf(\cdot; \cD_{k_{1,m}}, \cD_{k_{2,m}}))
        ~\pto~ R^\deter(\phi_1, \phi_2; \tf)
    \end{equation}
    as $k_{1,m}, k_{2,m}, p_m \to \infty$
    and $(p_m / k_{1,m}, p_m / k_{2,m}) \to (\phi_1, \phi_2) \in \cI$.
    Then, $R^\deter(\cdot, \cdot; \tf)$ is continuous on $\cI$.
\end{proposition}

Combining the results and the discussion above,
the verification of \eqref{eq:rn-deterministic-approximation-reduced-onestep-prop-asymptotics}
under \ref{asm:prop_asymptotics}
can proceed the following three-point program: 
\begin{enumerate}[label={\rm(PRG-1-C\arabic*)},leftmargin=2cm]
    \item 
    \label{prog:onestep-cont-conv}
    For $(\phi_1, \phi_2)$
    such that $R^\deter(\phi_1, \phi_2; \tf) < \infty$,
    verify that
    for all datasets $\cD_{k_{1,m}}$ and $\cD_{k_{2,m}}$
    with limiting aspect ratios $(\phi_1, \phi_2)$,
    $R(\tf(\cdot, \cdot; \cD_{k_{1,m}, \cD_{k_{2,m}}})) 
    \pto R^\deter(\phi_1, \phi_2; \tf)$.
    \item
    \label{prog:onestep-cont-infty}
    Whenever $R^\deter(\phi_1, \phi_2; \tf) = \infty$, it obeys that 
    \[
        \lim_{(\phi'_1, \phi'_2) \to (\phi_1, \phi_2)} R^\deter(\phi'_1, \phi'_2; \tf)
        = \infty.
    \]
    \item
    \label{prog:onestep-divergence-closedset}
    The set of all points where
    $R^\deter(\phi_1, \phi_2; \tf) = \infty$
    is a closed set.
\end{enumerate}

We will follow these steps
to verify \eqref{eq:rn-deterministic-approximation-reduced-onestep-prop-asymptotics}
for the MN2LS and MN1LS prediction procedures
in \Cref{sec:verifying-deterministicprofiles-onestep}.
But we will first complete
the derivation of the deterministic approximation
to the conditional risk of $\hf^\onestep$
under
\eqref{eq:rn-deterministic-approximation-reduced-onestep-prop-asymptotics}.
Following similar arguments
as those in \Cref{sec:risk-behavior-zerostep}
for the zero-step procedure,
\Cref{lem:deterministic-approximation-reduction-onestep}
along with 
\eqref{eq:onestep-risk-gaurantee}
provides the following monotonization result for the one-step
procedure:

\begin{theorem}
    [Asymptotic risk profile of one-step predictor]
    \label{thm:asymptotic-risk-tuned-one-step}
    For any prediction procedure $\tf$
    suppose 
    \ref{asm:prop_asymptotics},
    either \ref{cond:zerostep-add} or \ref{cond:zerostep-mult},
    and the assumptions of \Cref{lem:deterministic-approximation-reduction-onestep}
    hold true.
    In addition, if the loss function is convex in the second argument,
    then
    for any $M \ge 1$,
    \begin{equation}
        \label{eq:onestep-main-gaurantee}
        \left(
            R(\hf^\onestep; \cD_n)
            -
            \min_{1/\zeta_1 + 1/\zeta_2 \le 1/\gamma}
            R^\deter(\zeta_1, \zeta_2; \tf)
        \right)_{+}
        = o_p(1).
    \end{equation}
\end{theorem}

\Cref{thm:asymptotic-risk-tuned-one-step} hinges on \eqref{eq:rn-deterministic-approximation-onestep-prop-asymptotics}
and continuity of $R^\deter(\cdot, \cdot; \tf)$
which we will verify below in a specific model setting.
Before doing that,
let us briefly remark about the extensions and implications of
\eqref{eq:onestep-main-gaurantee}.

\begin{remark}
    [Exact risk of $\hf^\onestep$]
    For $M = 1$
    under \eqref{eq:rn-deterministic-approximation-onestep-prop-asymptotics},
    \eqref{eq:onestep-main-gaurantee}
    only guarantees that the risk of $\hf^\onestep$
    is bounded above by the minimum in \eqref{eq:onestep-main-gaurantee}.
    Considering a stricter version
    (DETPA-1*)
    of
    \eqref{eq:rn-deterministic-approximation-onestep-prop-asymptotics}
    that requires the $o_p(1)$
    in \eqref{eq:rn-deterministic-approximation-onestep-prop-asymptotics}
    to be uniform over all $(\xi_{1,n}, \xi_{2,n}) \in \Xi_n$,
    conclusion \eqref{eq:onestep-main-gaurantee}
    can be extended to imply for $M = 1$ that
    \begin{equation}
        \label{eq:onestep-main-gaurantee-2}
        \left|
            R(\hf^\onestep; \cD_n)
            -
            \min_{1/\zeta_1 + 1/\zeta_2 \le 1/\gamma}
            R^\deter(\zeta_1, \zeta_2; \tf)
        \right|
        = o_p(1).
    \end{equation}
    This shows that
    the risk of the one-step procedure
    with $M = 1$
    under the stricter assumption of
    (DETPA-1*) is exactly the same
    as the minimum in the display above.
    This is the characterization
    of the risk of the one-step procedure
    in the same vein as
    \eqref{eq:zerostep-exact-risk}
    is the characterization of the risk
    of the zero-step procedure.
\end{remark}

\begin{remark}
    [Monotonicity in the limiting aspect ratio]
    Observe that
    the following map
    \[
        \min_{1/\zeta_1 + 1/\zeta_2 \le 1/\gamma}
        R^\deter(\zeta_1, \zeta_2; \tf)
    \]
    is non-decreasing in $\gamma$.
    This is because
    \[
        \{ (\zeta_1, \zeta_2) : 1/\zeta_1 + 1/\zeta_2 \le 1/\gamma_u \}
        \subseteq
        \{ (\zeta_1, \zeta_2) : 1/\zeta_1 + 1/\zeta_2 \le 1/\gamma_l \}
        \quad
        \text{ for }
        \gamma_l \le \gamma_u,
    \]
    and hence the minimum can only be larger as $\gamma$ increases.
    This implies that
    the risk of the one-step procedure in asymptotically
    bounded above by a monotonically non-decreasing function in $\gamma$
    under the assumptions
    of
    \Cref{thm:asymptotic-risk-tuned-one-step}.
\end{remark}

\begin{remark}
    [Comparison with $\hf^\zerostep$]
    Observe that
    \begin{equation}
        \label{eq:onestep-vs-zerostep}
        \min_{1/\zeta_1 + 1/\zeta_2 \le 1/\gamma}
        R^\deter(\zeta_1, \zeta_2; \tf)
        ~\le~
        \min_{1/\zeta_1 \le 1/\gamma}
        R^\deter(\zeta_1; \tf),
    \end{equation}
    where 
    the left hand side is the asymptotic risk
    of $\hf^\onestep$ (with $M = 1$ and under (DETPA-1*)),
    the right hand side is the asymptotic risk of
    $\hf^\zerostep$
    (with $M = 1$ under \eqref{eq:rn-deterministic-approximation-2-prop-asymptotic-star}).
    Hence,
    under some regularity conditions,
    the one-step procedure is as good as
    the zero-step procedure if not better.
    See \Cref{rem:zerostep-vs-onestep-isotropic} for more details.
    For $M > 1$
    such a comparison is not readily
    plausible from our results.
\end{remark}

\subsubsection
[Verifying deterministic profiles]
{Verification of \eqref{eq:rn-deterministic-approximation-reduced-onestep-prop-asymptotics}}
\label{sec:verifying-deterministicprofiles-onestep}

We now verify the assumption
\eqref{eq:rn-deterministic-approximation-reduced-onestep-prop-asymptotics}
in a specific model setting
when the base prediction procedure
is either  MN2LS or MN1LS.
But first, we provide
a general result
describing the asymptotic
risk profile of
$R(\tf(\cdot; \cD_{k_{1,m}}, \cD_{k_{2,m}}))$
when the base prediction procedure is linear.

Let $\tf$ be a linear base prediction procedure
given by $\tf(x; \cD_{k_{1,m}}) = x^\top \tbeta(\cD_{k_{1,m}})$,
for some $\tbeta(\cD_{k_{1,m}}) \in \RR^{p}$
computed on $\cD_{k_{1,m}}$.
If $\cD_{k_{2,m}} = \{ (X_i, Y_i) : 1 \le i \le k_{2,m} \}$,
the ingredient predictor
$\tf(\cdot; \cD_{k_{1,m}}, \cD_{k_{2,m}})$
for the one-step prediction procedure
is given by
\begin{equation}
    \label{eq:onestep-ingredient-decomp}
    \tf(x; \cD_{k_{1,m}}, \cD_{k_{2,m}})
    = x^\top \tbeta(\cD_{k_{1,m}})
    + x^\top \tbeta_{\mnls}(\{(X_i, Y_i - X_i^\top \tbeta(\cD_{k_{1,m}})): 1 \le i \le k_{2,m} \})).
\end{equation}
The following result characterizes the conditional prediction
risk 
of $\tf(\cdot; \cD_{k_{1,m}}, \cD_{k_{2,m}})$
for the squared error loss
in terms of
the risk behavior of $\tbeta(\cD_{k_{1,m}})$.
This is possible
because the one-step adjustment
is fixed to be the MN2LS prediction procedure 
and its risk behavior can be completely
characterized as done in \Cref{sec:zerostep-overparameterized}.

Consider the setting of
\Cref{prop:asymp-bound-ridge-main}.
Let $\Sigma = W R W^\top$
denote the eigenvalue decomposition of the covariance matrix $\Sigma = \mbox{Cov}(X_0)$,
where 
$R \in \RR^{p_m \times p_m}$ is a diagonal matrix
containing eigenvalues $r_1 \ge r_2 \ge \dots \ge r_{p_m} \ge 0$,
and
$W~\in~\RR^{p_m \times p_m}$ is an orthonormal matrix
containing the corresponding eigenvectors 
$w_1, w_2, \dots, w_{p_m}~\in~\RR^{p_m}$.
In preparation for the statement to follow,
define the following (random) probability distribution 
on $\RR_{\ge 0}$:
\begin{equation}
    \label{eq:generalied-predrisk-distribution-baseprocedure}
    \widehat{Q}_n(r)
    := \frac{1}{R(\tf(\cdot; \cD_{k_{1,m}})) - \sigma^2}
    \sum_{i=1}^{p_m} ((\tbeta(\cD_{k_{1,m}}) - \beta_0)^\top w_i)^2 r_i \1\{ r_i \le r \}.
\end{equation}
Let $H_{p_m}$ denote the empirical spectral distribution
of $\Sigma$, whose value at any $r \in \RR$ is given by
\begin{equation}
    \label{eq:empirical_distribution_eigenvalues_Sigma}
    H_{p_m}(r)
    = \frac{1}{p_m} \sum_{i=1}^{p_m} \1_{\{r_i \le r\}},
\end{equation}
and let $H$ denote the corresponding limiting spectral distribution,
i.e., $H_{p_m} \dto H$ as $p_m \to \infty$.
See \ref{asm:spectrum-spectrumsignproj-conv}
in the proof of \Cref{prop:asymp-bound-ridge-main}
for more details.

\begin{lemma}
    [Continuous convergence of squared risk for one-step procedure]
    \label{lem:one-step-predrisk-decomposition}
    Let $\tf$ be any linear prediction procedure,
    and assume the setting of \Cref{prop:asymp-bound-ridge-main}.
    Let
    $k_{1,m}, k_{2,m}, p_m \to \infty$
    such that 
    $(p_m  / k_{1,m}, p_m / k_{2,m})  \to (\phi_1, \phi_2)$.
    Suppose there exists a deterministic approximation
    $R^\deter(\phi_1; \tf)$
    to the conditional squared prediction risk
    of $\tf(\cdot; \cD_{k_{1,m}})$ such that
    $R(\tf(\cdot; \cD_{k_{1,m}})) \pto R^\deter(\phi_1; \tf)$
    for $\phi_1$ 
    that satisfy
    $R^\deter(\phi_1; \tf) < \infty$.
    Assume the distribution $\widehat{Q}_n$ 
    as defined in \eqref{eq:generalied-predrisk-distribution-baseprocedure}
    converges weakly to a fixed distribution $Q$,
    in probability.
    Then,
    for 
    $\phi_2 \in (0, 1) \cup (1, \infty]$,
    we have
    $
        R(\tf(\cdot; \cD_{k_{1,m}}, \cD_{k_{2,m}}))
        \pto
        R^\deter(\phi_1, \phi_2; \tf)
    $,
    where $R^\deter(\phi_1, \phi_2; \tf)$
    is given by
    \begin{equation}
        \label{eq:detapprox-onestep-general}
        R^\deter(\phi_1, \phi_2; \tf)
        =
        \begin{dcases}
            R^\deter(\phi_1; \tf) & \text{if} ~ \phi_2 = \infty \\
            R^\deter(\phi_1; \tf) \Upsilon_b(\phi_1, \phi_2)  
            + \sigma^2 (1 - \Upsilon_b(\phi_1, \phi_2))
            + \sigma^2 \tv_g(0; \phi_2)  
            & \text{if} ~ \phi_2 \in (1, \infty) \\
            \sigma^2 \left(\frac{1}{1 - \phi_2}\right) & \text{if} ~ \phi_2 \in (0, 1).
        \end{dcases}
    \end{equation}
    Here, the scalars 
    $v(0; \phi_2)$, 
    $\tv(0; \phi_2)$, 
    $\tv_g(0; \phi_2)$,
    and
    $\Upsilon_b(\phi_1, \phi_2)$,
    for $\phi_2 \in (1, \infty)$,
    are defined as follows:
    \begin{itemize}[label={--}]
        \item 
        $v(0; \phi_2)$ is the unique solution to the fixed-point equation:
        \begin{equation}
            \label{eq:fixed-point-v-onestep-main-phi2}
            v(0; \phi_2)
            = \left( \phi_2 \int \frac{r}{v(0; \phi_2) r + 1} \, \mathrm{d}H(r)  \right)^{-1},
        \end{equation}
        \item
        $\tv(0; \phi_2)$ is defined in terms of $v(0; \phi_2)$ by the equation:
        \begin{equation}
            \label{eq:tv-onestep-main-phi2}
            \tv(0; \phi_2)
            = \left( 
             \frac{1}{v(0; \phi_2)^2} 
             - \phi_2 \int \frac{r^2}{(v(0; \phi_2) r + 1)^2} \, \mathrm{d}H(r) \right)^{-1},
        \end{equation}
        \item
        $\tv_g(0; \phi_2)$
        is defined in terms of $v(0; \phi_2)$ and $\tv(0; \phi_2)$ by the equation:
        \begin{equation}
            \label{eq:tvg-onestep-main-phi2}
            \tv_g(0; \phi_2)
            =
            \tv(0; \phi_2)
            \phi_2 \int \frac{r^2}{(v(0; \phi_2) r + 1)^2} \, \mathrm{d}H(r),
        \end{equation}
    \item
    $\Upsilon_b(\phi_1, \phi_2)$
    is defined in terms of $v(0; \phi_2)$
    and $\tv_g(0; \phi_2)$ by the equation:
    \begin{equation}
        \label{eq:def-upsilonb-upsilonv}
        \Upsilon_b(\phi_1, \phi_2) 
        =  
        (1 + \tv_g(0; \phi_2))
        \int \frac{1}{(v(0; \phi_2) r + 1)^2} \, \mathrm{d}Q(r).
    \end{equation}
    \end{itemize}
\end{lemma}

\Cref{lem:one-step-predrisk-decomposition}
provides a deterministic risk approximation
for the ingredient one-step predictor
$\tf(\cdot; \cD_{k_{1,m}}, \cD_{k_{2,m}})$
in terms of
the deterministic risk approximation
of the base prediction procedure $\tf$.
In case of isotropic covariates,
i.e., $\Sigma = I_{p_m}$,
the distribution $H$
is degenerate at $1$,
and $R^\deter(\phi_1, \phi_2; \tf)$
can be simplified
because
$\Upsilon_b(\phi_1, \phi_2) = (1 - 1 / \phi_2)$,
and $\tv_g(0; \phi_2) = 1 / (\phi_2 - 1)$.
See the proof of \Cref{prop:verif-riskprofile-mn1lsbase-mn2lsonestep}
for more details.

Note that
the assumed limiting distribution
$Q$ in general depends on $\phi_1$, $\phi_2$,
and hence $\Upsilon_b(\phi_1, \phi_2)$ is in general a function
of $\phi_1$, $\phi_2$, and the distribution
of the data.
On the other hand,
$v(0; \phi_2)$ defined in \eqref{eq:fixed-point-v-onestep-main-phi2},
is a function of $\phi_2$ alone,
and hence $\tv_g(0; \phi_2)$ is just a function of $\phi_2$.
Furthermore,
it can be verified that
$\tv_g(0; \cdot)$ is a continuous function
on $(1, \infty)$ and
$\lim_{\phi_2 \to 1^{+}} \tv_g(0; \phi_2) = \infty$;
see 
\Cref{lem:fixed-point-v-properties}~\eqref{lem:fixed-point-v-properties-item-tvg-properties}.
This implies
that $R^\deter(\phi_1, \phi_2; \tf)$
satisfies 
\ref{prog:onestep-cont-conv}--\ref{prog:onestep-divergence-closedset},
if
the base prediction procedure
satisfies
\ref{prog:zerostep-cont-infty}.
Hence,
any prediction procedure
that can be used for zero-step
can also be used for one-step
as long as
the convergence assumption on $\widehat{Q}_n$
is satisfied.
We make this precise in the following result.

\begin{corollary}
    [Verification of one-step deterministic profile program]
    \label{cor:verif-onestep-program}
    Assume the setting of 
    \Cref{lem:one-step-predrisk-decomposition}.
    In addition,
    suppose
    $R^\deter(\phi_1; \tf)$
    satisfies \ref{prog:zerostep-cont-infty}.
    Then,
    $\tf(\cdot; \cD_{k_{1,m}}, \cD_{k_{2,m}})$
    satisfies
    \ref{prog:onestep-cont-conv}--\ref{prog:onestep-divergence-closedset}
    and hence satisfies
    \eqref{eq:rn-deterministic-approximation-reduced-onestep-prop-asymptotics}.
\end{corollary}

Therefore,
the prediction procedures
mentioned in 
\Cref{rem:prop_asymptotics_risk_examples}
can be easily shown
to satisfy \eqref{eq:rn-deterministic-approximation-reduced-onestep-prop-asymptotics}.
Although we assume that $\widehat{Q}_n$
converges weakly to $Q$ in probability,
we only need in probability
convergence of
$\int f(r) \, \mathrm{d} \widehat{Q}_n(r)$
to $\int f(r) \, \mathrm{d} Q(r)$
for $f(r) = r / (v(0; \phi_2) r + 1)^2$,
which is a weaker requirement.
Intuitively,
this assumption
comes from
the representation
of $\tf(x; \cD_{k_{1,m}}, \cD_{k_{2,m}})$
in \eqref{eq:onestep-ingredient-decomp}
as $\tf(x; \cD_{k_{1,m}}, \cD_{k_{2,m}})
= x^\top \widehat{A} \tbeta(\cD_{k_{1,m}})
+ x^\top \tbeta_{\mnls}(\cD_{k_{2,m}})$
for some random matrix $\widehat{A}$;
see \Cref{lem:onestep-ingredient-simplifications}.
Hence, the risk
of $\tf$
can be written in terms of
a weighted prediction error
of $\tbeta(\cD_{k_{1,m}})$
with the weights
depending on $f(\cdot)$;
see \eqref{eq:generalized_prediction_risk_onestep}.

\begin{proposition}
    [Verification of 
    \eqref{eq:rn-deterministic-approximation-reduced-onestep-prop-asymptotics}
    for the MN2LS base procedure]
    \label{prop:verif-riskprofile-mnlsbase-mnlsonestep}
    Assume the setting of
    \Cref{prop:asymp-bound-ridge-main}.
    Then, the one-step ingredient predictor
    constructed from the MN2LS base prediction procedure
    satisfies
    \eqref{eq:rn-deterministic-approximation-reduced-onestep-prop-asymptotics}.
\end{proposition}

\begin{proposition}
    [Verification of 
    \eqref{eq:rn-deterministic-approximation-reduced-onestep-prop-asymptotics}
    for the MN1LS base procedure]
    \label{prop:verif-riskprofile-mn1lsbase-mn2lsonestep}
    Assume the setting of 
    \Cref{prop:asymp-verif-mn1ls}.
    Then, the one-step ingredient predictor
    constructed from the MN1LS base prediction
    procedure satisfies
    \eqref{eq:rn-deterministic-approximation-reduced-onestep-prop-asymptotics}.
\end{proposition}

\begin{remark}
    [Comparison of zero and one-step procedure
    for isotropic covariance]
    \label{rem:zerostep-vs-onestep-isotropic}
    In order to get an intuition
    about the risk of one-step procedure,
    consider the case of isotropic features.
    In this case,
    $R^\deter(\phi_1, \phi_2; \tf)$
    simplifies to
    \begin{equation}
        \label{eq:limiting-risk-onestep-isotropic}        
        R^\deter(\phi_1, \phi_2; \tf)
        =
        \begin{dcases}
            R^\deter(\phi_1; \tf)
            & \text{ if } \phi_2 = \infty \\
            R^\deter(\phi_1; \tf) \left(1 - \frac{1}{\phi_2}\right)
            + \sigma^2 \left( \frac{1}{\phi_2} + \frac{1}{\phi_2 - 1}\right) 
            & \text{ if } \phi_2 \in (1, \infty) \\
            \sigma^2 \left(\frac{1}{1 - \phi_2}\right)
            & \text{ if } \phi_2 \in (0, 1).
        \end{dcases}
    \end{equation}
    
    Note that $\phi_2 = \infty$
    corresponds to
    simply using the base predictor
    without any one-step residual adjustment.
    This is the same as the ingredient
    predictor used in the zero-step prediction procedure.
    The one-step prediction procedure
    would minimize 
    the expression shown in \eqref{eq:limiting-risk-onestep-isotropic},
    over $\phi_1$ and $\phi_2$
    satisfying $\phi_1^{-1} + \phi_2^{-1} \le \gamma^{-1}$.
    If the optimal $\phi_2$ turned out to be $\infty$,
    then one-step predictor and the zero-step
    predictor become the same,
    and the resulting
    limiting risk 
    is $R^\deter(\phi_1; \tf)$.
    From \eqref{eq:limiting-risk-onestep-isotropic},
    the risk for
    $\phi_2 \in (1, \infty)$
    can be decomposed as
    \[
        R^\deter(\phi_1; \tf)
        + 
        \left(
        \frac{\sigma^2}{\phi_2}
        +
        \frac{\sigma^2}{\phi_2 - 1} 
        - \frac{R^\deter(\phi_1; \tf)}{\phi_2}
        \right).
    \]
    If the quantity in the parenthesis
    is negative for some $(\phi_1, \phi_2)$
    satisfying the condition $\phi_1^{-1} + \phi_2^{-1} \le \gamma^{-1}$,
    then
    the one-step prediction procedure will yield
    a strictly better risk than the zero-step
    prediction procedure (for $M = 1$).

   One can gain more insight into how one-step procedure improves on the zero-step by considering the case of
   isotropic covariance
   and MN2LS base prediction procedure. The intriguing finding in this case is that
    the one-step prediction procedure
    with base MN2LS procedure
    is effectively the same
    as applying MN2LS
    on new data with reduced signal energy
    and with a larger limiting aspect ratio.

    Formally, under isotropic covariance with MN2LS base procedure, 
    $R^\deter$
   can be written as follows.
   Recall $\rho^2$ denotes the limit of $\| \beta_0 \|_2^2$
   and $\sigma^2$ is the noise variance. Then,
   one has
       \begin{align*}
        &R^\deter(\phi_1, \phi_2; \tf_{\mnls}) \\
        &=
        \begin{dcases}
            \left[
            \rho^2 \left(1 - \frac{1}{\phi_1}\right)
            + \sigma^2 \left(\frac{1}{\phi_1 - 1}\right)
            \right]
            \left(1 - \frac{1}{\phi_2}\right)
            + \sigma^2 \left(\frac{1}{\phi_2 - 1}\right)
            + \sigma^2
            & \text{ if } (\phi_1, \phi_2) \in (1, \infty] \times (1, \infty] \\
            \left[
            \sigma^2 \left(\frac{\phi_1}{1 - \phi_1}\right)
            \right]
            \left(1 - \frac{1}{\phi_2}\right)
            + \sigma^2 \left(\frac{1}{\phi_2 - 1}\right)
            + \sigma^2
            & \text{ if } (\phi_1, \phi_2) \in (0, 1) \times (1, \infty) \\
            \sigma^2 \left(\frac{\phi_2}{1 - \phi_2}\right)
            + \sigma^2
            & \text{ if } (\phi_1, \phi_2) \in (0, \infty) \times (0, 1).
        \end{dcases}
    \end{align*}
    Here, we treat $1/x$ and $1/(x - 1)$ to be $0$ when $x = \infty$.
    
    Let $R^\deter_{\mnls}(\phi; \rho^2, \sigma^2)$
    denote the asymptotic risk profile of the MN2LS
    predictor at aspect ratio $\phi$,
    signal energy $\rho^2$,
    and noise energy $\sigma^2$; from the proof of \Cref{prop:asymp-bound-ridge-main} \citep[see also][Theorem 1]{hastie_montanari_rosset_tibshirani_2019}, we have
    \[
    R^\deter_{\mnls}(\phi; \rho^2, \sigma^2) = \begin{cases}
         \rho^2 \left( 1 - \frac{1}{\phi}\right) + \sigma^2 \left(\frac{1}{\phi - 1}\right) +\sigma^2 &  \text{ if } \phi \in (1,\infty] \\
         \sigma^2 \left(\frac{\phi}{1 - \phi}\right) + \sigma^2 & \text{ if } \phi \in (0,1).
    \end{cases}
    \]
    Let $R^\deter_{\mnls}(\phi_1, \phi_2; \rho^2, \sigma^2)$
    denote the asymptotic risk profile of the one-step
    ingredient predictor with MN2LS base predictor with signal and noise energy $\rho^2$ and $\sigma^2$, respectively -- which above we have denoted with $R^\deter(\phi_1, \phi_2; \tf_{\mnls})$.
    Then, we can write
    \begin{equation}
        \label{eq:rdet_mnls_onestep_iterated_formula}
        R^\deter_{\mnls}(\phi_1, \phi_2; \rho^2, \sigma^2)
        =
        R^\deter_{\mnls}
        (\phi_2; R^\deter_{\mnls}
        (\phi_1; \rho^2, \sigma^2) - \sigma^2, \sigma^2).
    \end{equation}
    Thus,
    the limiting risk of the one-step predictor
    computed on
    a data with limiting aspect ratio $\gamma$
    is given by
    \begin{equation}
        \label{eq:onestep-isotropic-dynamic-mn2ls}
        R^\deter_{\mnls}
        (\phi_2(\gamma); R^\deter_{\mnls}(\phi_1(\gamma); \rho^2, \sigma^2) - \sigma^2, \sigma^2),
    \end{equation}
    where $(\phi_1(\gamma), \phi_2(\gamma))$
    represents the minimizer
    of $R^\deter_{\mnls}(\zeta_1, \zeta_2; \rho^2, \sigma^2)$
    over $\zeta_1^{-1} + \zeta_2^{-1} \le \gamma^{-1}$.
    Now the risk expression
    \eqref{eq:onestep-isotropic-dynamic-mn2ls}
    can be interpreted as follows:
    The one-step prediction procedure
    with base MN2LS procedure
    is effectively the same
    as applying MN2LS
    on new data with reduced signal energy
    (because $R^\deter_{\mnls}(\phi_1(\gamma); \rho^2, \sigma^2) < \rho^2 + \sigma^2$)
    and with a larger limiting aspect ratio
    $\phi_2(\gamma) > \gamma$.
    Note that reducing the signal energy
    reduces the risk for MN2LS due to a reduction in the estimation bias;
    see 
    \Cref{fig:isotropic_risk}
    and
    \Cref{lem:riskprofile_overparam}~\eqref{lem:riskprofile_overparam-item-fixed-gamma-increasing-in-snr}.
    Recall that the effect
    of the zero-step procedure
    would just be 
    applying MN2LS
    on a data set with a large limiting aspect
    ratio, but with the original
    signal energy $\rho^2$.
    Hence, the improvement of the one-step procedure
    over the zero-step procedure (which only takes place in the overparametrized regime)
    essentially stems from
    reducing the signal energy and thus the bias,
    which ``boosts'' the asymptotic risk.

    In this case,
    we can also explicitly carry out the optimization
    of minimizing $R^\deter(\zeta_1, \zeta_2; \tf)$
    subject to the constraint $\zeta_1^{-1} + \zeta_2^{-1}
    \le \gamma^{-1}$.
    See \Cref{sec:onestep-risk-analysis-isotropic-features} for the details.
    See \Cref{fig:onestep-vs-zerostep-M=1}
    for an illustration
    of the comparison
    the limiting risk
    of the one-step prediction procedure
    with the
    the zero-step prediction procedure.
    
    Finally, we comment that for base predictors other than the MN2LS, the risk of one-step procedure may not have as nice an interpretation as ``boosting'' the asymptotic risk by reducing the signal energy in addition to increasing aspect ratio. However, the message is that the one-step procedure adds another knob to the zero-step procedure which leads to an improved risk.
\end{remark}

\begin{figure}[!t]
    \centering
    \includegraphics[width=0.31\columnwidth]{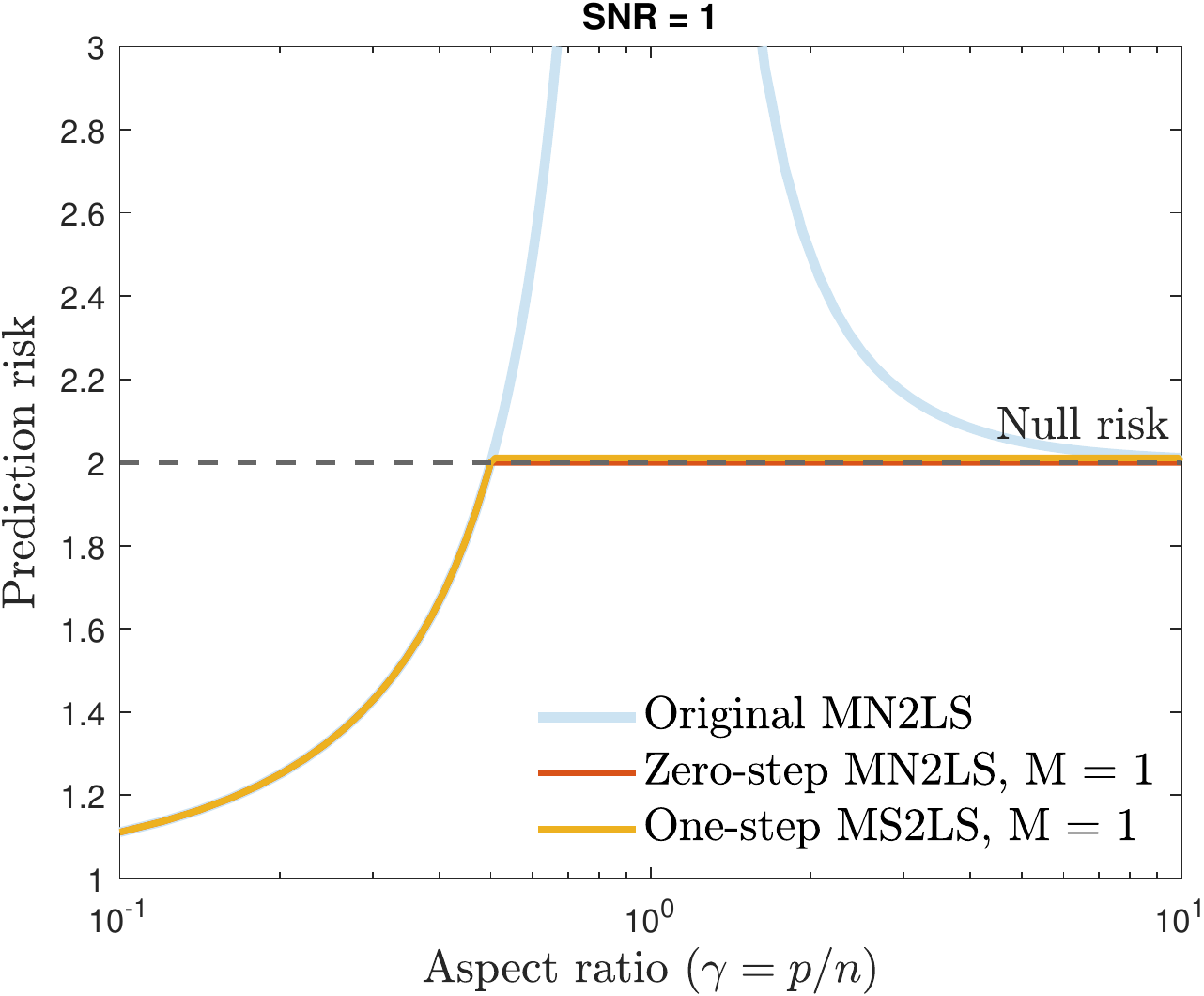}
    ~
    \includegraphics[width=0.31\columnwidth]{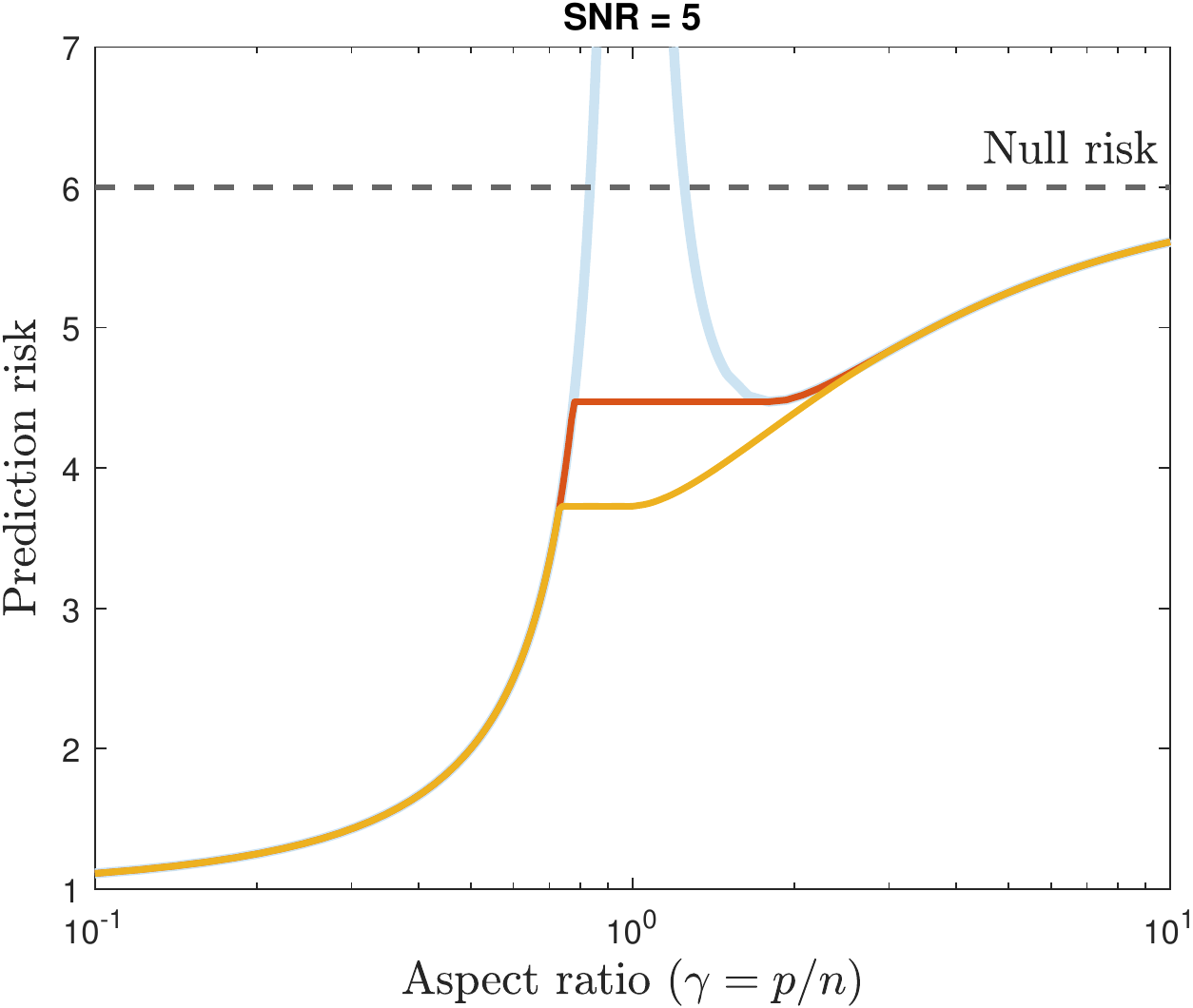}
    ~
    \includegraphics[width=0.31\columnwidth]{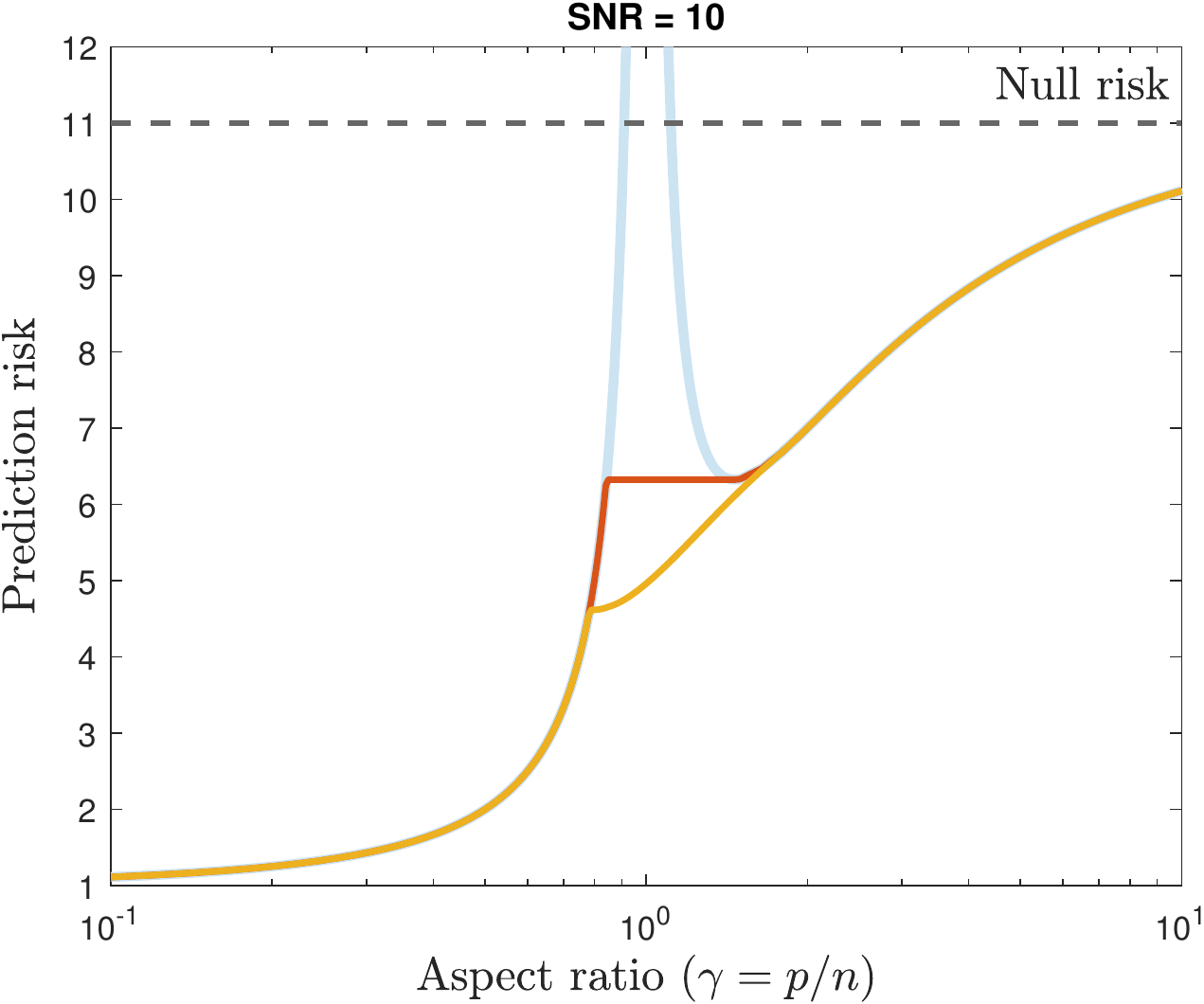}
    \caption{
    Comparison of zero-step and one-step procedures
    with MN2LS base procedures under isotropic feature covariance, 
    and low, moderate, and high SNR regimes.
    Observe that
    for SNR = 1,
    zero-step and one-step both have
    the same risk profile with $M = 1$.
    This holds true even for SNR $\le $ 1,
    as shown in \Cref{thm:esterrlim_optimized_onestep}.
    For SNR $> 1$,
    there exists a range of $\gamma$
    for which one-step is strictly better than zero-step.
    See \Cref{thm:esterrlim_optimized_onestep}
    for more details.
    }
    \label{fig:onestep-vs-zerostep-M=1}
\end{figure}

\subsection{Numerical illustrations}

In this section,
we provide numerical illustration
of the risk monotonization of one-step
prediction procedure
in the proportional asymptotic regime,
when the base prediction procedures
are MN2LS and MN1LS prediction procedures,
and the one-step adjustment is \emph{always} performed
via MN2LS.
In order to illustrate
risk monotonization as in
\Cref{thm:asymptotic-risk-tuned-one-step},
we need to show
the risk behavior of $\hf^\onestep$ 
at different aspect ratios.
We use the same simulation settings
used for the illustration
of the zero-step procedure
in \Cref{sec:zerostep-illustration}.
\Cref{fig:gaufeat_isocov_isosig_linmod_n1000_nv100_nruns100_onestep_mnls_gamma10,fig:gaufeat_isocov_sparsesig_linmod_n500_nv80_nruns100_onestep_mnla_gamma100}
present our simulation results.
The conclusions are essentially
the same as those stated
for the zero-step procedure
in \Cref{sec:zerostep-illustration}.

\paragraph{Minimum $\ell_2$-norm least squares (MN2LS).}

\Cref{fig:gaufeat_isocov_isosig_linmod_n1000_nv100_nruns100_onestep_mnls_gamma10}
shows the risks of the baseline MN2LS procedure
and the one-step prediction procedure
with MN2LS as the base prediction procedure
for high and low SNR regimes (left: SNR = 4; right: SNR = 1);
we take $\sigma^2 = 1$, so that $\rho^2$=SNR.
We also present the null risk
($\rho^2 + \sigma^2$),
i.e., the risk of the zero predictor
as a baseline in both the plots.

Similar to the behavior
of the zero-step procedure
we observe that
the risk of the one-step procedure
is non-decreasing in $\gamma$
for every $M \ge 1$.
Although
the risk of the one-step procedure
is close to being below
the risk of the base procedure,
Figure 6
shows the effects of
working with a finite sample.
(The risk of one-step
for $M = 1$
is sometimes above
the risk of the base procedure.)

\Cref{fig:gaufeat_isocov_isosig_linmod_n1000_nv100_nruns100_onestep_mnls_gamma10}
also shows
that the one-step prediction procedure
can be strictly better
than the zero-step prediction procedure.
In particular,
the left panel of Figure 6
shows that around the interpolation
threshold of $1$,
the risk of one-step prediction
procedure is not flat.
It is strictly increasing.
The risk of one-step procedure
for $M > 1$
is once again seen to be
a strict improvement over $M = 1$.

\begin{figure}[!thb]
    \centering
    \includegraphics[width=0.45\columnwidth]{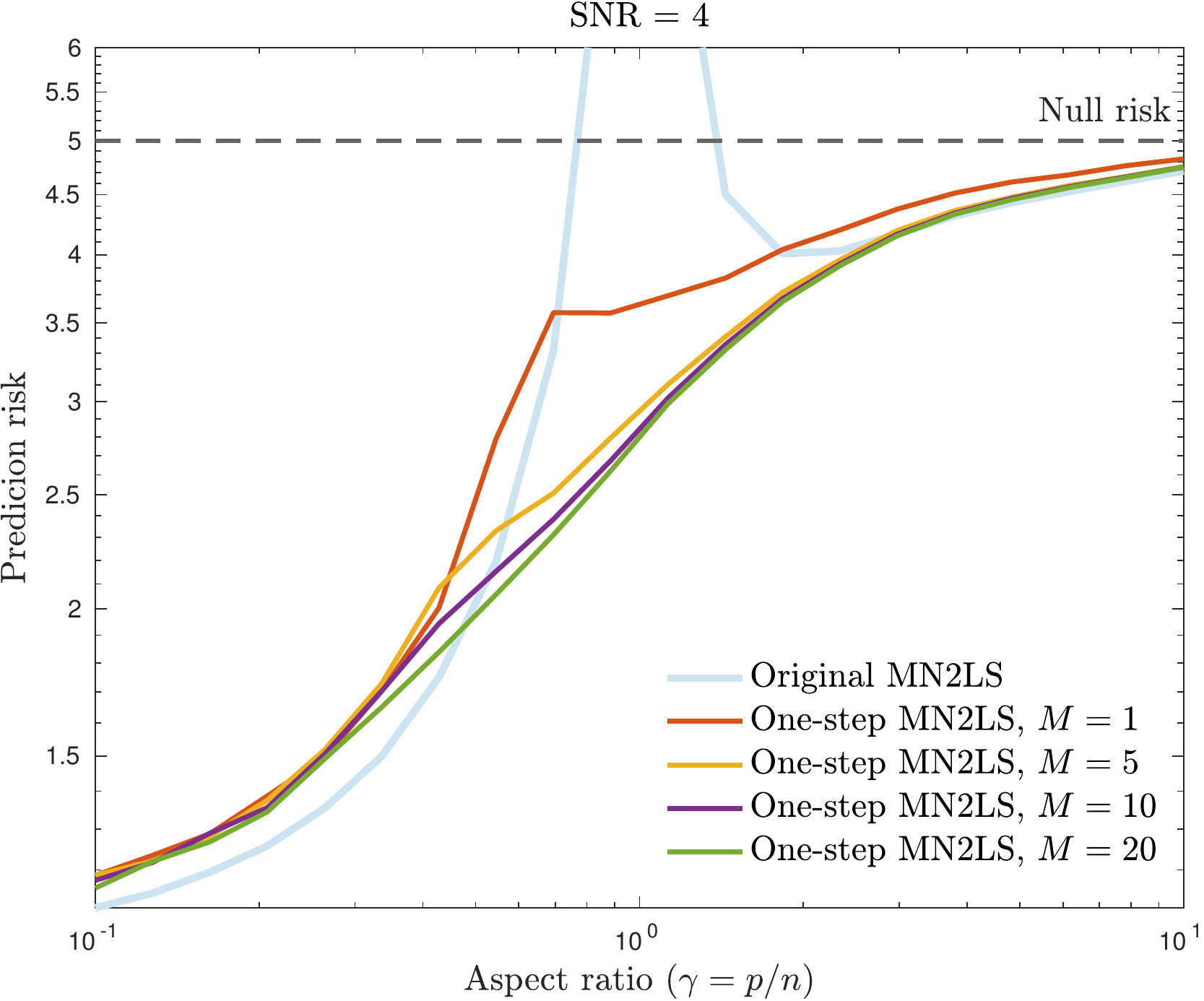}
    \quad
    \includegraphics[width=0.45\columnwidth]{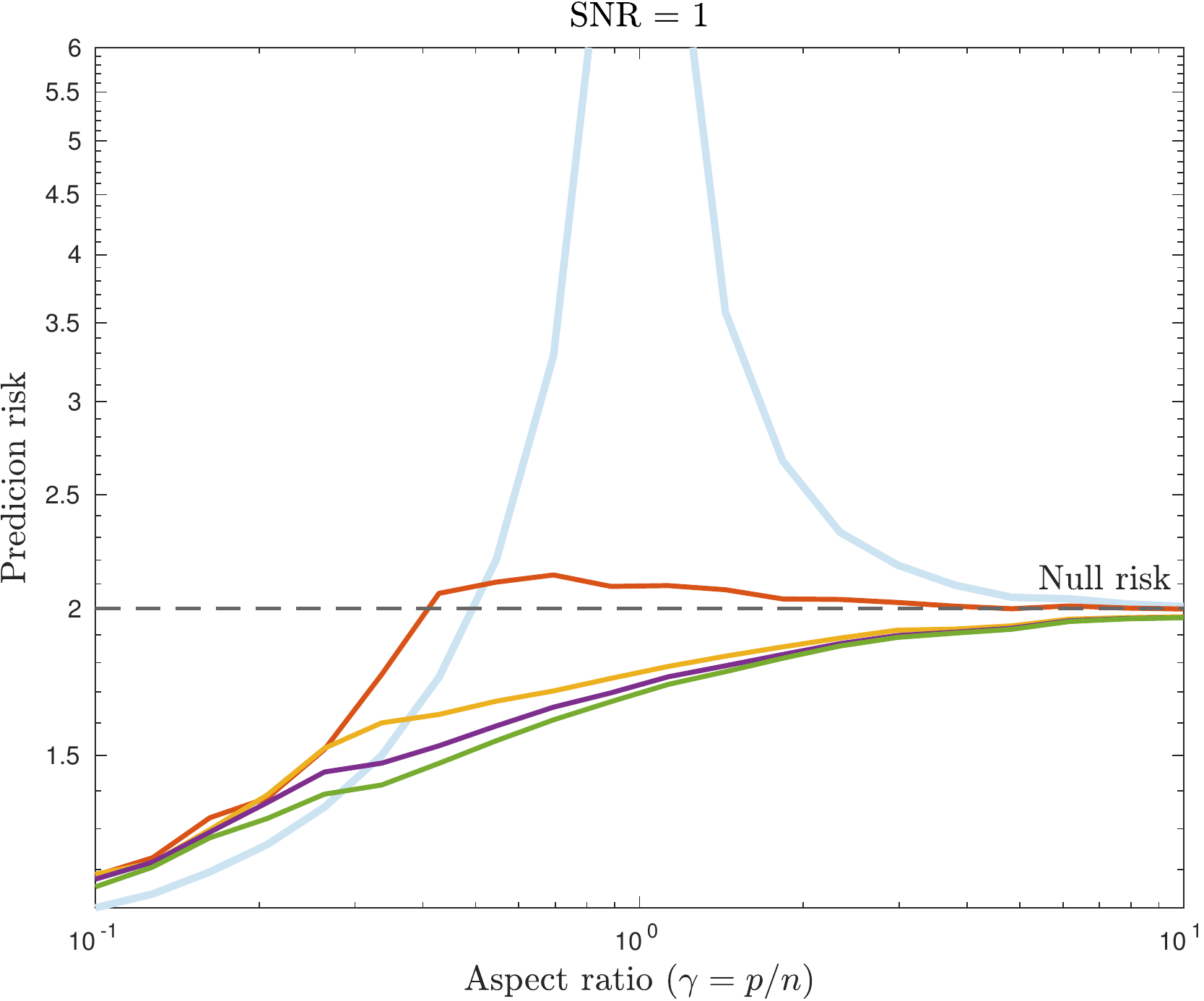}
    \caption{
    Illustration of the one-step procedure with the MN2LS as the base predictor
    and MN2LS one-step adjustment with varying $M$.
    The left panel shows a high SNR setting (SNR = 4),
    while the right panel shows a low SNR setting (SNR = 1).
    The setup has
    $n = 1000$, $n_\train = 900$, $n_\test = 100$, $n^\nu = 50$.
    The features are drawn from an isotropic Gaussian distribution, 
    the response follows a linear model with dense signal.
    The risks are averaged over 100 dataset repetitions.
    }
    \label{fig:gaufeat_isocov_isosig_linmod_n1000_nv100_nruns100_onestep_mnls_gamma10}
\end{figure}

\paragraph{Minimum $\ell_1$-norm least squares (MN1LS).}

\Cref{fig:gaufeat_isocov_sparsesig_linmod_n500_nv80_nruns100_onestep_mnla_gamma100}
shows the risks of the baseline MN1LS procedure
and the one-step procedure with MN1LS as the base
prediction procedure
for high
(left, SNR = 4)
and low (right, SNR = 1)
SNR regimes.
We take $\sigma^2 = 1$ and $\rho^2=$ SNR.
We also present the null risk
($\rho^2 + \sigma^2$),
i.e., the risk of the zero predictor
as a baseline in both the plots.
We again observe that
the risk of the one-step procedure
for every $M \ge 1$ is non-decreasing in $\gamma$.
As before,
once again
we observe in
\Cref{fig:gaufeat_isocov_sparsesig_linmod_n500_nv80_nruns100_onestep_mnla_gamma100}
that the one-step procedure
with $M = 1$
attains precise risk monotonization
while
zero-step with $M > 1$ improves significantly upon
the $M = 1$ case when $\gamma$ is near one.
All these comments hold
for both low and high SNR regimes.

\begin{figure}[!t]
    \centering
    \includegraphics[width=0.45\columnwidth]{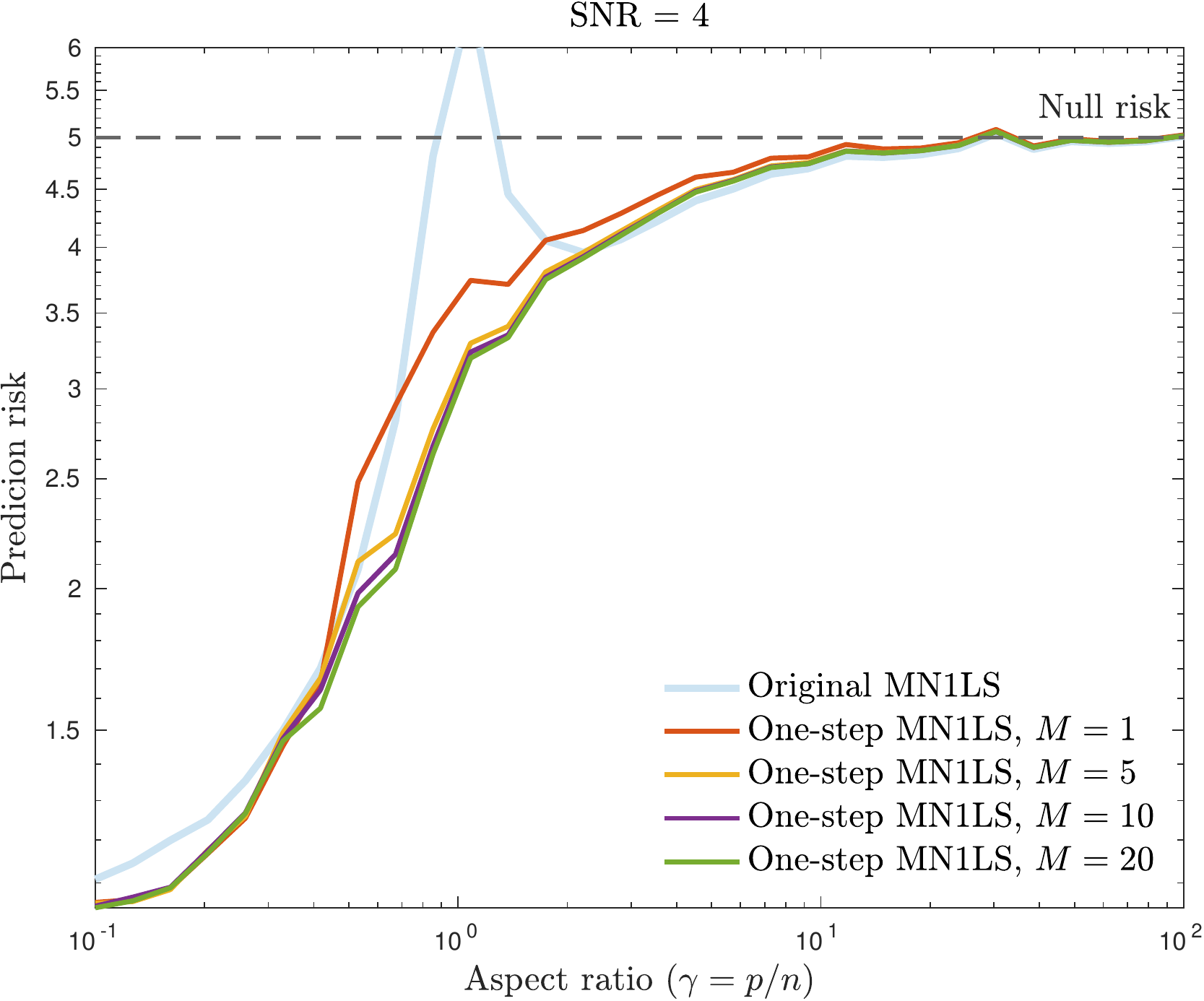}
    \quad
    \includegraphics[width=0.45\columnwidth]{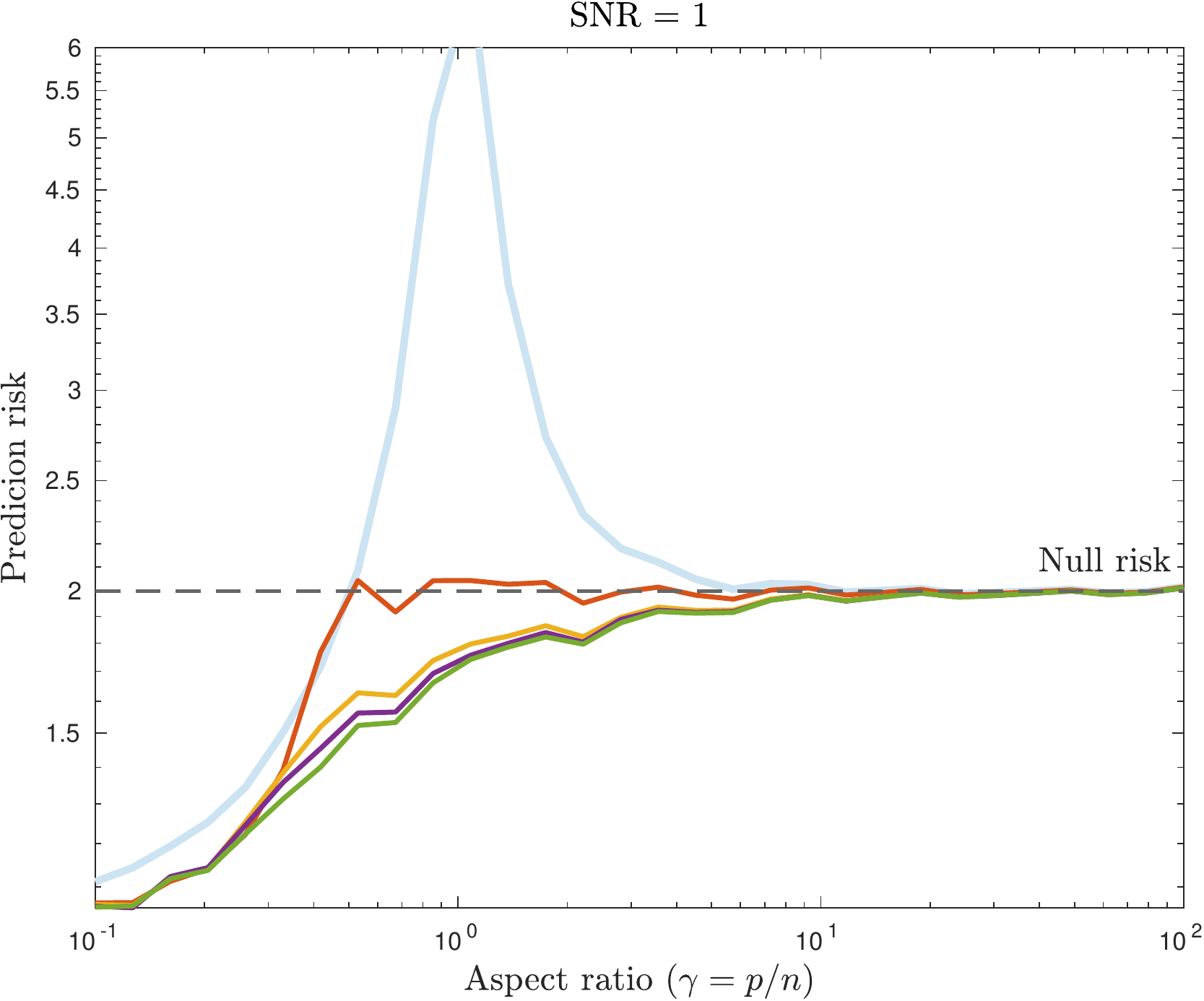}
    \caption{
    Illustration of the one-step procedure with MN1LS as the base procedure
    and MN2LS one-step adjustment with varying $M$.
    The left panel shows a high SNR setting (SNR = 4),
    while the right panel shows a low SNR setting (SNR = 1).
    In the setup,
    $n = 500$, $n_\train = 420$, $n_\test = 80$, $n^\nu = 42$.
    The features are drawn from an isotropic Gaussian distribution,
    the response follows a linear model with sparse signal (sparsity level = 0.0005).
    The risks are averaged over 100 dataset repetitions.
    }
    \label{fig:gaufeat_isocov_sparsesig_linmod_n500_nv80_nruns100_onestep_mnla_gamma100}
\end{figure}

\section{Discussion}
\label{sec:discussion}

In this paper, 
we have proposed 
a generic cross-validation framework
to monotonize any given prediction procedure
in terms of the sample size.
We studied two concrete methodologies:
zero-step and one-step prediction procedures.
The ingredient predictors
for the zero-step prediction procedure
is the base procedure applied on a subset
of the data.
The ingredient predictor
for the one-step prediction procedure
can be thought of as
boosting applied to the base procedure
learned on a subset of data (\cite{schapire2013boosting}).
In both cases,
we also introduced 
averaging over the subsets of the data
(via the parameter $M$).
This particular averaging step
can be seen as bagging,
which is known to have a variance reduction effect.

We have analyzed the properties
of zero-step and one-step prediction procedures
in a model-free setting under mild regularity assumptions.
This is in contrast to many other works
in this literature that require
strong distributional assumptions.
In part this is possible
because we assume
the existence of the limiting risk
and monotonize it
(in a data-driven way)
without requiring the knowledge/form
of the risk.

Monotonization of asymptotic risk also has implications for minimax risk. If the base prediction procedure
has a finite asymptotic risk $\underline{R}$
and $\overline{R}$, respectively,
at the limiting aspect
ratios of $0$ and $\infty$,
then both zero-step and one-step prediction procedures
applied to such a base procedure yield 
predictors whose asymptotic risk 
lies between $[\underline{R}, \overline{R}]$
for all limiting aspect ratios.
For example,
for the squared error loss
and a linear model,
the MN1LS and MN2LS predictors
have $\underline{R} = \sigma^2$
and $\overline{R} = \| \beta_0 \|_{\Sigma}^2 + \sigma^2$,
where $\sigma^2$ is the noise energy,
which is also the unavoidable prediction risk,
and $\| \beta_0 \|_{\Sigma}^2$
is the effective signal energy.
Because $\sigma^2$ is
the unavoidable prediction risk,
and hence a minimax lower bound,
the zero-step and one-step predictors
based on MN1LS and MN2LS
are minimax optimal up to a multiplicative factor
of $1~+~\text{SNR} = 1~+~\| \beta_0 \|_{\Sigma}^2 / {\sigma^2}$
over all aspect ratios ranging from $0$ to $\infty$.
Any base prediction procedure
that leads to the null predictor
(i.e., $\hf(x) = 0$ for all $x$)
for the limiting aspect ratio of $\infty$
also has the same property.
(Most reasonable prediction procedures
would yield the null predictor
as the limiting aspect ratio
tends to $\infty$.)
Furthermore,
for every procedure,
there exists another procedure
(such as the zero-step)
whose risk is at least as good 
and is monotone.
Thus, the minimax risk is a monotone function
of the limiting aspect ratio.
To our knowledge,
the minimax risk in the proportional asymptotics
regime
under generic signal structure
is not available in the literature.

Although the focus of the current paper is exclusively
on choosing optimal sample size,
one could apply the cross-validation
framework proposed
for selecting optimal 
predictors from any collection.
In particular,
one can use our methodology
to find optimal penalty parameter
for ridge regression or lasso.
It can also be used
to select the number of random features
in random features regression
or kernel features in kernel regression,
or more generally, the number of parameters
in a neural network.
In the latter case,
our procedures will yield
model-wise monotonicity
\citep{nakkiran2019deep}.

There are several interesting future directions
that one can pursue.
We will discuss three specific directions below.

\paragraph{Theoretical characterization of the effect of bagging.}

We have only characterized the risk of
the zero-step and one-step with $M = 1$
in terms of the limiting risk of the base procedure.
In this sense,
we did not fully analyze the effect of bagging
($M > 1$) for both zero-step and one-step procedures.
It is of interest
to characterize the effect of bagging:
\begin{center}
    What is the limiting risk of the zero-step and one-step
    procedures when $M > 1$?
\end{center}

From the theory of $U$-statistics,
it is expected that
the risk for $M > 1$
is non-increasing in $M$.
It is hard to however
argue that the risk of
zero/one-step predictors
is monotone in the limiting
aspect ratio when $M > 1$.
The main difficulty
lies in proving
that the ingredient predictors
for the zero-step procedure
have an asymptotic risk profile
for $M \ge 1$.
Once this is guaranteed,
the theory developed
in \Cref{sec:zerostep-overparameterized}
will readily imply
that the zero-step procedure
with $M > 1$
has an asymptotic monotonic risk profile.
We now briefly mention the difficulty
in proving the existence 
of the asymptotic risk profile for the ingredient
predictor when $M > 1$.

For concreteness,
consider the ingredient predictor
of the zero-step prediction procedure
with $M > 1$
that uses $k_{n} \le n$ observations.
This is given by
\[
    \tf_M(x)
    =
    \frac{1}{M}
    \sum_{j=1}^{M}
    \tf(x; \cD_{\train}^{j})
\quad
\text{ with } 
\quad
|\cD_{\train}^{j}| = k_n.
\]
Note that we take subsets $\mathcal{D}_{\train}^{j}$ as independent and identically distributed subsets of size $k_n$ from the data and hence for $M = \infty$, we get
\begin{equation}
\label{eq:predictor_M_infty}
\tf_{\infty}(x; \mathcal{D}_{\train}) = \frac{1}{\binom{n}{k_n}}\sum_{1\le i_1 < \ldots < i_{k_n} \le n_{\train}} \tf(x; \{(X_{i_j}, Y_{i_j}): 1\le j\le k_n\}).
\end{equation}
This is a $U$-statistics of order $k_n$ for every fixed $x$ in terms of the training data.
If $R(\tf(\cdot; \cD_{\train}^j)) \pto R^\deter(\phi)$
whenever $p / k_n \to \phi$,
then from the theory developed
in \Cref{sec:zerostep-overparameterized},
it follows that
$R(\hf^\zerostep_M) \pto \min_{\zeta \ge \gamma} R^\deter(\zeta)$ under~\ref{asm:prop_asymptotics}.
Hence,
the main difficulty in characterizing
the effect of bagging
lies in proving the existence of limit
of $R(\tf)$.
For the squared error loss,
it can be proved that
(see \Cref{sec:squaredrisk_decomposition_zerostep_bagged})
\begin{equation}
    \label{eq:sqauredrisk_decomposition_zerostep_bagged}
        R(\tf_M)
    =
    R(\tf_\infty(\cdot; \mathcal{D}_{\train}))
    +
    \frac{1}{M}
        \frac{1}{\binom{n}{k_n}}
        \sum_{i_1, \dots, i_{k_n}}
    \int
        \left(
            \tf(x; \{ (X_{i_j}, Y_{i_j}): 1 \le j \le k_n \})
            -
            \tf_\infty(x; \mathcal{D}_{\train})
        \right)^2
        \,
        \mathrm{d}P_{X_0}(x).
\end{equation}
It is interesting to note
that the risk of $\tf_M$
only depends on $M$
as a linear function of $1/M$.
If the base predictor $\tf$
is non-zero almost surely,
then the risk of $\tf_M$
is a strictly decreasing function of $M$.
Observe that
(53)
holds true
even for $M = 1$
and from our results,
we know that
the right hand side with $M = 1$
has a finite deterministic approximation.
This implies that
each of the components in (53)
is asymptotically bounded.
Hence,
as $M \to \infty$,
we can conclude that
$R(\tf_M) - R(\tf_\infty) \pto 0$.

Because $k_n \to \infty$
and $p / k_n \to \phi$,
the second term in~\eqref{eq:sqauredrisk_decomposition_zerostep_bagged} above
could be analyzed
using deterministic representation
for $\tf(X_0; \{ (X_{i_j}, Y_{i_j}) : 1 \le j \le k_n \})$
(e.g., Theorem 1 of \cite{liu_dobriban_2019} for ridge regression)
and the theory of $U$-statistics.
On the other hand,
$R(\tf_\infty)$
could also be similarly analyzed
using deterministic representations
and the theory of $U$-statistics.
We leave this for future work.

\paragraph{Other variants of boosting.}
In our empirical studies, we found that the one-step predictor (for $M = 1$) which is a boosted version of the subsampled predictor has a much better performance than the zero-step predictor (with $M = 1$), especially around the interpolation threshold. For reasons unclear to us currently, the performance of one-step predictor (for $M = 1$) can be matched, at least in shape, by a zero-step predictor with some $M > 1$. In this sense, the effect of one iterate boosting can be matched by the effect of multi-subsample bagging. Furthermore, as $M$ increases, both zero-step and one-step seem to approach the same limit in our empirical studies. The interesting aspect is that the work done by $M$ subsample bagging is achieved  by one boosting iterate. 
This begs the question: is there a better boosting mechanism that can match zero-step predictors performance at $M = \infty$. In particular:
\begin{center}
    What are the other choices of one-step residual adjustments? 
    And what is the ``best'' choice?
\end{center}
We have only analyzed the one-step residual adjustment
done via MN2LS. Other choices are certainly possible: for instance, one could do MN1LS or minimum $\ell_p$-norm least squares or minimum $\ell_2$ robust least squares in the context of linear regression. It seems cumbersome to analyze each one of these residuals adjustments case-by-case and find the best choice. For general models, one can think of the residuals adjustment we proposed as a variant of Newton's step for the squared error loss under homoscedasticity as mentioned in~\eqref{eq:one-step-general}. The discussion of the ``best'' choice of the residual adjustment very much hinges on the question of what is the best predictor in a given model in the proportional asymptotics regime. Although we do not know the answer to this question, one can potentially target the question of deriving a residual adjustment that yields an asymptotic risk performance similar to that of the zero-step predictor with $M = \infty$. 
For any given predictor, is there a one iterate boosted version (i.e., one-step predictor with $M = 1$) that achieves the same asymptotic performance as the $M$-subsample bagging with $M = \infty$?

Similar to the one-step predictor one can develop a $k$-step predictor by splitting the data into potentially $(k+1)$ batches and optimizing over the number of observations in each batch. This is analogues to $k$-iterate boosting as our one-step procedure (with $M = 1$) is analogues to the one iterate boosting. This gets computationally intensive very quickly as $k$ increases. Furthermore, we believe that $k$-step predictor combined with bagging would yield the same asymptotic risk profile as the zero- and one-step predictors with $M = \infty$. In this sense, it seems a worth problem to investigate a better one iterate booster than to investigate the $k$-step predictor precisely.

\paragraph{Comparison with other regularization strategies.}

On the surface,
zero-step and one-step procedures
might seem to use
only a subset of the data,
and hence might appear sub-optimal.
Along the same lines,
one might also wonder
why not employ
regularization techniques
and optimize over the regularization parameter.
To the first point,
note that
we make use of the whole data
in estimating the risk
and comparing predictors at different sample sizes,
and hence make use of the full data.
To the second point,
it is somewhat surprising to report
that optimally-regularized  procedures
such as ridge regression with optimal choice of penalty
need not have monotone risk (in the limiting aspect ratio);
see, for example, Figure 1 of \cite{hastie_montanari_rosset_tibshirani_2019}.
But our procedure
will always lead to a monotone risk
and hence makes better use of the data
compared to optimum regularization procedures in general.
Irrespective,
it is still interesting to consider
the relation between zero-step and one-step,
and the optimum regularization procedures
in cases where
the latter has a monotone risk.
In our empirical studies
we found that in a well-specified
linear model,
zero-step and one-step procedures
(with the MN2LS base procedure)
with a large enough $M$
have asymptotic risk very close
to the risk of the optimum ridge
regression procedure.
See the left panel of 
\Cref{fig:mnla_vs_optzerostep_vs_optonestep_vs_optridge_eps005_snr4_M20}.
In a sparse linear regression model,
zero-step and one-step procedures
(with the MN1LS base procedure)
with a large enough $M$
has asymptotic risk very close
to the risk of the optimum lasso regression.
It is also interesting to observe
that the risk is monotone for optimally tuned lasso.
See the right panel of 
\Cref{fig:mnla_vs_optzerostep_vs_optonestep_vs_optridge_eps005_snr4_M20}.
The effect of both bagging and boosting with large $M$
in this case appears to be similar.
In other words, thinking of the base procedures MN2LS and MN1LS as ridge and lasso, respectively, with zero penalty parameter, the zero- and one-step predictors with $M$ large attaining the same asymptotic risk as optimum ridge or lasso can be considered as finding optimal regularization for these procedures. Without  explicitly formalizing the regularization predictor, zero- and one-step perform ``optimal'' implicit regularization.  
To what extent such similarity extends
to other settings is an interesting future direction:
\begin{quote}
    Under what conditions, do zero- and one-step predictors with MN2LS/MN1LS base predictor match the asymptotic risk profile of optimized regularization of ridge/lasso regression? What other base predictors (and corresponding classes of regularized predictors) does this phenomenon extend to? 
\end{quote}

\begin{figure}[!t]
    \centering
    \includegraphics[width=0.45\columnwidth]{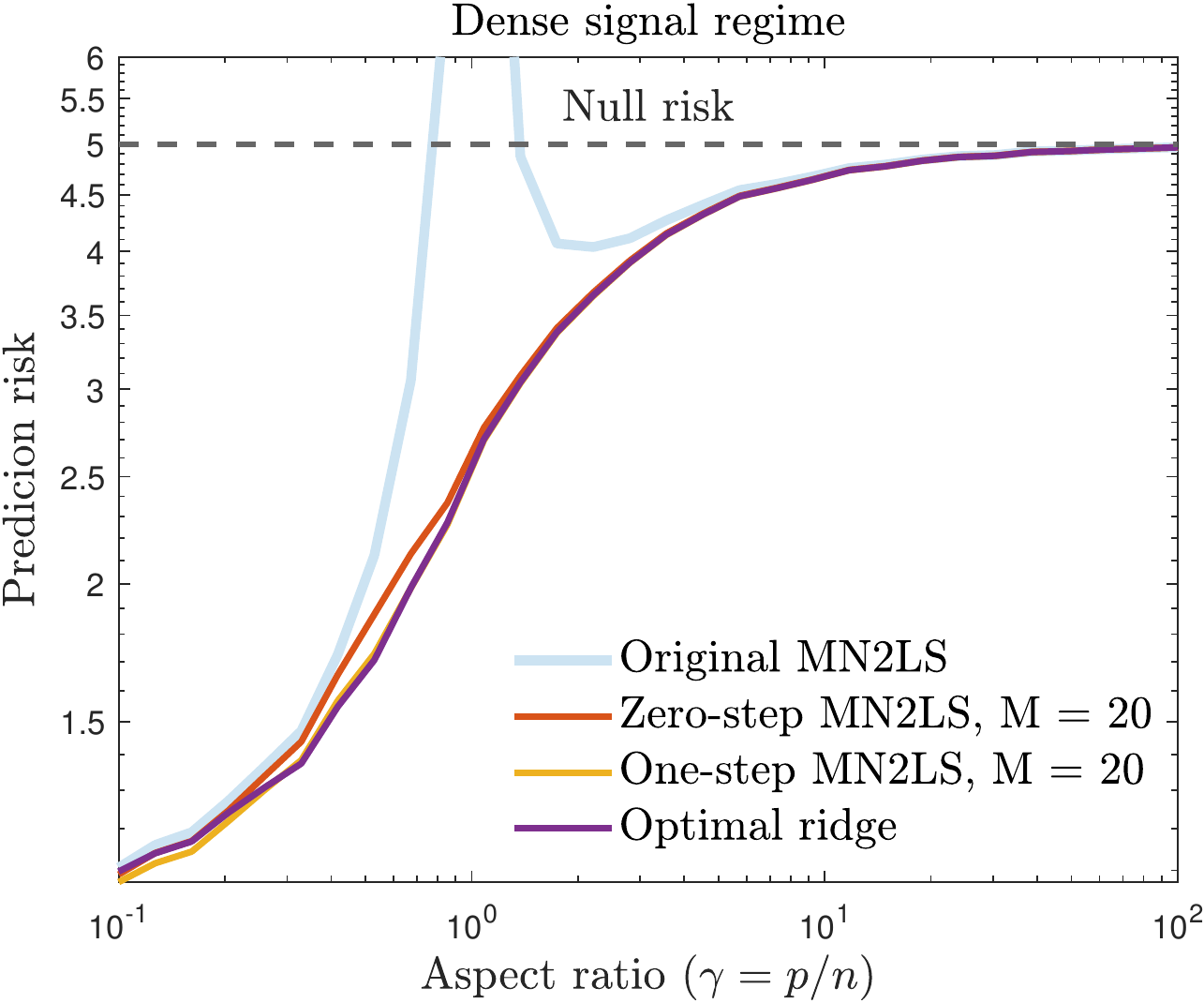}
    \quad
    \includegraphics[width=0.45\columnwidth]{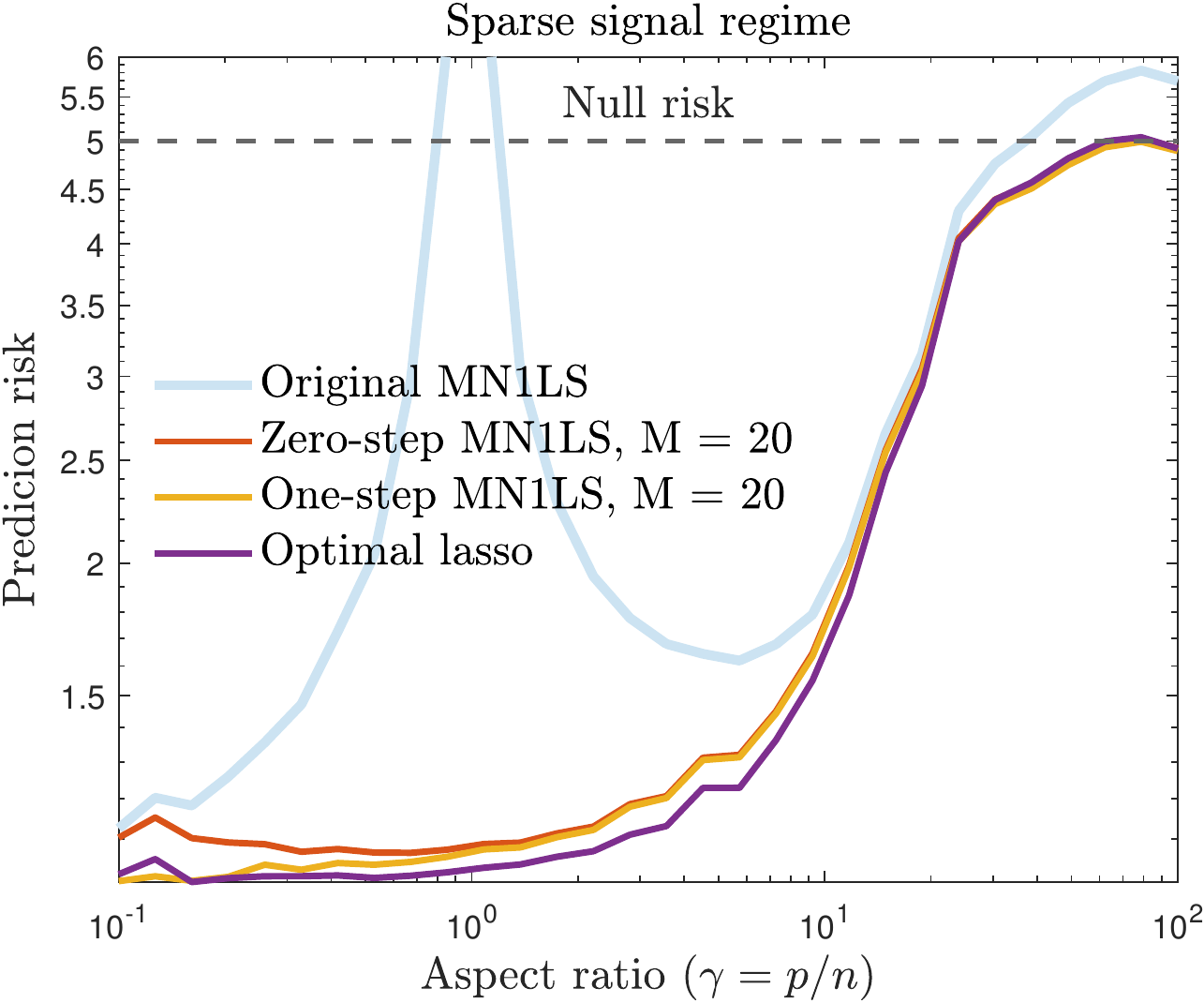}
    \caption{
    Comparison of different regularization strategies
    of zero-step, one-step, optimal ridge, and optimal lasso.
    The left panel shows a dense signal regime 
    and the right panel shows a sparse signal regime.
    The setup has
    $n=100$,
    SNR = 4.
    The features are drawn from an isotropic Gaussian distribution,
    the response follows a linear model with dense (left panel) 
    and sparse signal (right panel, sparsity level = $0.0005$).
    The risks are averaged over 100 dataset repetitions.
    }
    \label{fig:mnla_vs_optzerostep_vs_optonestep_vs_optridge_eps005_snr4_M20}
\end{figure}

\section*{Acknowledgements}

We are grateful to Ryan J.\ Tibshirani
for encouraging us to pursue this direction
and his technical help and advice throughout the process.
We thank Matey Neykov for numerous insightful discussions on this work.
We thank the participants of the Theory of Overparameterized Learning Workshoop (TOPML) 2022,
the Annual Conference on Information Sciences and Systems (CISS) 2022,
and the Deep Learning ONR MURI seminar series,
in particular Thomas Goldstein, Rob Nowak, Daniel LeJeune, Yehuda Dar,
for helpful discussions and feedback on the work.

P.\ Patil was partially supported by ONR grant N00014-20-1-2787.
Y.\ Wei was partially supported by NSF grants DMS 2147546/2015447 and CAREER award DMS-2143215.

\bibliographystyle{apalike}

\newpage
\newpage
\setcounter{section}{0}
\setcounter{equation}{0}
\setcounter{figure}{0}
\setcounter{section}{0}
\setcounter{equation}{0}
\setcounter{figure}{0}
\renewcommand{\thesection}{S.\arabic{section}}
\renewcommand{\theequation}{E.\arabic{equation}}
\renewcommand{\thefigure}{S.\arabic{figure}}
\begin{center}
\Large {\bf Supplement to ``Mitigating multiple descents: \\
A model-agnostic framework for risk monotonization''}
\end{center}

This document serves as a supplement
to the paper ``Mitigating multiple descents:
A model-agnostic framework for risk monotonization.''
The section and equation numbers in this document begin with the letters ``S'' and ``E''
to differentiate them from those in the main paper.
The content of the document is organized as follows.
\begin{itemize}
    \item 
    In \Cref{sec:proofs-crossvalidation-modelselection},
    we present proofs of results related to
    general cross-validation and model selection
    from \Crefrange{sec:oracle-risk-inequalities}{sec:common-loss-functions}.
    \item
    In \Cref{sec:proofs-riskmonotonization-zerostep},
    we present proofs of results related to 
    risk monotonization behavior of the zero-step procedure
    from 
    \Cref{sec:risk-behavior-zerostep}.
    \item
    In \Cref{sec:verif-asymp-profile-ridge},
    we present proofs for the verification of 
    the deterministic risk profile assumption
    for the MN2LS and MN1LS prediction procedures 
    from \Cref{sec:verification-deterministicprofile-zerostep}.
    \item
    In \Cref{sec:proofs-riskmonotonization-onestep},
    we present proofs of results related to
    risk monotonization behavior of the one-step procedure
    from \Cref{sec:onestep-overparameterized}.
    \item
    In \Cref{sec:verif-riskprofile-mnlsbase-mnlsonestep},
    we present proofs for the verification of
    the deterministic risk profile assumption
    for arbitrary linear prediction procedures,
    and the MN2LS and MN1LS prediction procedures
    from \Cref{sec:verifying-deterministicprofiles-onestep}.
    \item
    In \Cref{sec:useful-lemmas},
    we collect various technical helper lemmas and their proofs
    that are used in proofs in \Cref{sec:proofs-riskmonotonization-zerostep,sec:verif-asymp-profile-ridge,sec:proofs-riskmonotonization-onestep,sec:verif-riskprofile-mnlsbase-mnlsonestep},
    and other miscellaneous details.
    \item
    In \Cref{sec:calculus_deterministic_equivalents},
    we list calculus rules for a certain notion of 
    asymptotic equivalence of sequences of matrices
    that are used in proofs in \Cref{sec:verif-asymp-profile-ridge,sec:verif-riskprofile-mnlsbase-mnlsonestep}.
    \item
    In \Cref{sec:useful-concentration-results},
    we record statements of useful concentration results
    available in the literature
    that are used in proofs in \Cref{sec:proofs-crossvalidation-modelselection,sec:verif-asymp-profile-ridge,sec:verif-riskprofile-mnlsbase-mnlsonestep}.
    \item
    In \Cref{sec:notation},
    we list some of the main notation used in the paper.
\end{itemize}

\section{Proofs related to general cross-validation and model selection}
\label{sec:proofs-crossvalidation-modelselection}

\subsection{Proof of \Cref{prop:general-model-selection-guarantee}}
    \paragraph{Additive form.}
   We will first prove the oracle risk inequalities 
   \eqref{eq:model-free-guarantee-general-model-selection}
   in additive form.
    Recall
    \Cref{alg:general-cross-validation-model-selection}
    returns $\hf^\cv = \hf^\hxi$.
    Adding and subtracting
    $\min_{\xi \in \Xi} R(\hf^\xi)$
    and $\min_{\xi \in \Xi} \hR(\hf^\xi)$
    to $R(\hf^\cv)$,
    we can break $R(\hf^\cv)$ into the following additive form:
    \begin{align}
        R(\hf^\cv)
        &= \min_{\xi \in \Xi} R(\hf^\xi)
        + \min_{\xi \in \Xi} \hR(\hf^\xi)
        - \min_{\xi \in \Xi} R(\hf^\xi)
        - \min_{\xi \in \Xi} \hR(\hf^\xi)
        + R(\hf^\hxi).
    \end{align}
    An application of triangle inequality then
    lets us upper bound $R(\hf^\cv)$ into sum of three terms:
    \begin{equation}
        \label{eq:oracle-inequality-additive-form-triangle}
       R(\hf^\cv)
       \le \min_{\xi \in \Xi} R(\hf^\xi)
       + \underbrace{\Big| \min_{\xi \in \Xi} \hR(\hf^\xi)
       - \min_{\xi \in \Xi} R(\hf^\xi) \Big|}_{(a)}
       + \underbrace{\Big| R(\hf^\hxi) - \min_{\xi \in \Xi} \hR(\hf^\xi) \Big|}_{(b)}.
    \end{equation}
    We will next upper bound both terms (a) and (b) by $\Delta_n^\add$
    to finish the first inequality of \eqref{eq:model-free-guarantee-general-model-selection}.
        
        By definition \eqref{eq:Delta_n_add} of $\Delta_n^\add$,
        for every $\xi \in \Xi$,
        we can write
        \begin{equation}
            \label{eq:pointwise-bound-Delta_n^add}
            R(\hf^\xi) \le \hR(\hf^\xi) + \Delta_n^\add
            \quad
            \text{and}
            \quad
            \hR(\hf^\xi) \le R(\hf^\xi) + \Delta_n^\add.
        \end{equation}
        Taking minimum on both sides
        of the inequalities
        in \eqref{eq:pointwise-bound-Delta_n^add}
        then yields
        \[
            \min_{\xi \in \Xi} \hR(\hf^\xi)
            \le \min_{\xi \in \Xi} R(\hf^\xi)
            + \Delta_n^\add
            \quad
            \text{and}
            \quad
           \min_{\xi \in \Xi} R(\hf^\xi)
           \le \min_{\xi \in \Xi} \hR(\hf^\xi)
           + \Delta_n^\add.
        \]
        Combining the two inequalities,
        we arrive at the desired bound for term (a):
       \begin{equation}
            \label{eq:oracle-inequality-additive-form-term-a}
            \Big|
                \min_{\xi \in \Xi}
                \hR(\hf^\xi)
                -
                \min_{\xi \in \Xi}
                R(\hf^\xi)
            \Big|
            \le \Delta_n^\add.
       \end{equation}
        Since
        $\hxi \in \argmin_{\xi \in \Xi} \hR(\hf^\xi)$,
        we can obtain the following upper bound for term (b):
        \begin{equation}
        \label{eq:oracle-inequality-additive-form-term-b}
            \Big|
                R(\hf^\hxi) - \min_{\xi \in \Xi} \hR(\hf^\xi)
            \Big|
            =
            \left|
                R(\hf^\hxi) - \hR(\hf^\hxi) 
            \right|
            \le \Delta_n^\add,
        \end{equation}
    where the inequality follows from the definition of $\Delta_n^\add$.
    
    Substituting the bounds
    \eqref{eq:oracle-inequality-additive-form-term-a}
    and
    \eqref{eq:oracle-inequality-additive-form-term-b}
    into
    \eqref{eq:oracle-inequality-additive-form-triangle},
    we conclude that
    \begin{equation}
        \Big|
            R(\hf^\cv) - \min_{\xi\in\Xi} R(\widehat{f}^\xi)
        \Big|
        \le 2 \Delta_n^\add.
    \end{equation}
    This implies the first inequality
    of~\eqref{eq:model-free-guarantee-general-model-selection}.
    Taking expectations on the both sides of the first inequality 
    of~\eqref{eq:model-free-guarantee-general-model-selection},
    we obtain
    \begin{equation}
        \EE\big[R(\hf^\cv)\big]
        \le \mathbb{E}\big[\min_{\xi \in \Xi}
        R(\widehat{f}^\xi)\big]
        + 2\mathbb{E}\big[\Delta_n^\add\big].
    \end{equation}
    It is clear that
    the first term on the right hand side
    is bounded above by
    $\min_{\xi \in \Xi}\mathbb{E}[R(\widehat{f}^\xi)]$,
    and thus
    we obtain the second inequality
    of~\eqref{eq:model-free-guarantee-general-model-selection}.
    This completes the proof
    of the oracle risk inequalities
    in additive form.

\paragraph{Multiplicative form.}
We now turn to prove the oracle risk inequality
\eqref{eq:oracle-risk-inequality-multiplicative-form}
in multiplicative form.
Recall again that \Cref{alg:general-cross-validation-model-selection}
returns $\hf^\cv = \hf^\hxi$.
In contrast to the proof of \Cref{prop:general-model-selection-guarantee},
we now 
break $R(\hf^\cv)$ into the following multiplicative form:
\begin{align}
    R(\hf^\cv)
    = \frac{R(\hf^\cv)}{\widehat{R}(\hf^\cv)} \cdot \widehat{R}(\hf^\cv) 
    ~&=~ \frac{R(\hf^\cv)}{\widehat{R}(\hf^\cv)} \cdot \widehat{R}(\hf^\hxi)  \nonumber \\
    ~&\overset{(i)}{=}
    \frac{R(\hf^\cv)}{\hR(\hf^\cv)}
    \cdot \min_{\xi \in \Xi} \hR(\hf^\xi) \nonumber \\
    ~&=~ \frac{R(\hf^\cv)}{\hR(\hf^\cv)} 
    \cdot \min_{\xi\in\Xi}
    \left[
    \frac{\hR(\hf^{\xi})}{R(\hf^\xi)} \cdot R(\hf^\xi) \right] \nonumber \\
    ~&\overset{(ii)}{\le}~
    \frac{R(\hf^\cv)}{\hR(\hf^\cv)}
    \cdot \min_{\xi \in \Xi} 
    \left[
    \left(\max_{\rho \in \Xi} \frac{\hR(\hf^\rho)}{R(\hf^\rho)}\right) \cdot R(\hf^\xi) 
    \right]
    \nonumber \\
    ~&{\le}~
    \frac{R(\hf^\cv)}{\hR(\hf^\cv)}
    \cdot \left(\max_{\xi\in\Xi}\frac{\hR(\hf^\xi)}{R(\hf^\xi)}\right)
    \cdot \min_{\xi\in\Xi} R(\hf^\xi) \nonumber \\
    ~&\overset{(iii)}{\le}~ 
    \ddfrac{1}{\min_{\xi \in \Xi} \frac{\hR(\hf^\xi)}{R(\hf^\xi)}} 
    \cdot \left(\max_{\xi \in \Xi} \frac{\hR(\hf^\xi)}{R(\hf^\xi)}\right)
    \cdot \min_{\xi \in \Xi} R(\hf^\xi)
    \nonumber \\
    ~&=~
    \ddfrac{\max_{\xi \in \Xi} \frac{\hR(\hf^\xi)}{R(\hf^\xi)}}{\min_{\xi \in \Xi} \frac{\hR(\hf^\xi)}{R(\hf^\xi)}}
    \cdot \min_{\xi \in \Xi} R(\hf^\xi).
    \label{eq:oracle-inequality-multiplivative-form-decomposition}
\end{align}
In the chain above,
equality~$(i)$
follows from the definition of $\hxi$ in \Cref{alg:general-cross-validation-model-selection},
inequality~$(ii)$
follows from the inequality $a_ib_i \le (\max_{j} a_j)b_i$ for any two sequences $a_i, b_i, 1\le i\le m$,
and
inequality~$(iii)$
follows by noting that
\[
\frac{R(\hf^\cv)}{\hR(\hf^\cv)} 
= \ddfrac{1}{\frac{\hR(\hf^\cv)}{R(\hf^\cv)}} 
= \ddfrac{1}{\frac{\hR(\hf^\hxi)}{R(\hf^\hxi)}} 
\le \ddfrac{1}{\min_{\xi \in \Xi} \frac{\hR(\hf^\xi)}{R(\hf^\xi)}}.
\]

Now, from the definition of $\Delta_n^{\mul}$, 
for all $\xi \in \Xi$, we have
\[
1 - \Delta_n^{\mul} \le \frac{\hR(\hf^\xi)}{R(\hf^\xi)} \le 1 + \Delta_n^{\mul}.
\]
In addition, 
since the loss function is assumed to be non-negative,
both $R(\hf^\xi)$ and $\hR(\hf^\xi)$ are non-negative for all $\xi$.
Hence, we can bound
\begin{equation}
\label{eq:oracle-inequality-multiplivative-form-bounds}
(1 - \Delta_n^{\mul})_{+} \le \min_{\xi\in\Xi}\frac{\hR(\hf^\xi)}{R(\hf^\xi)} \le \max_{\xi\in\Xi} \frac{\hR(\hf^\xi)}{R(\hf^\xi)} \le 1 + \Delta_n^{\mul}.
\end{equation}
Using \eqref{eq:oracle-inequality-multiplivative-form-bounds}
in \eqref{eq:oracle-inequality-multiplivative-form-decomposition}
then implies the desired upper bound:
\[
R(\widehat{f}^\cv)
\le \frac{1 + \Delta_n^{\mul}}{(1 - \Delta_n^{\mul})_{+}}
\cdot \min_{\xi\in\Xi} R(\hf^\xi).
\]
This completes the proof
of the oracle risk inequality in multiplicative form.

\subsection
{Proof of \Cref{lem:bounded-orlitz-error-control}}
    \paragraph{Tail bound.}
    We begin by applying the Bernstein inequality
    (see Lemma~\ref{lem:bernstein-ineq} for the exact statement)
    on the random variables $\ell(Y_j, \hf^\xi(X_j)), j \in \cI_\test$
    with mean $R(\hf^\xi)$
    conditionally on $\cD_\train$.
    (Note that the random variables are i.i.d.\ conditionally on $\cD_\train$.)
    For any $0 < \eta < 1$ and $\xi \in \Xi$,
    we have the tail bound
    \begin{equation}\label{eq:conditional-probability-Bernstein-application}
        \PP
        \left(
            \left|
                \frac{1}{|\cD_\test|} \sum_{j \in \cI_\test}
                \ell(Y_j, \hf^\xi(X_j)) - R(\hf^\xi)
            \right|
            \ge C_1
            \max
            \left\{
                \sqrt{\hsigma_\xi^2 \frac{\log\left(2/\eta\right)}{|\cD_\test|}},
                \hsigma_\xi \frac{\log\left(2/\eta\right)}{|\cD_\test|}
            \right\}
            \, \mathrel{\Bigg|} \mathcal{D}_\train
        \right)
        \le \eta.
    \end{equation}
    Taking expectation on both sides,
    we get that the unconditional probability is also bounded by $\eta$.
    Denoting the prediction risk estimate by $\hR(\hf^\xi)$,
    and choosing $\eta = {\eta}/{|\Xi|}$,
    for any $\xi \in \Xi$,
    we can equivalently write the bound as
    \[
        \PP
        \left(
            \left| \hR(\hf^\xi) - R(\hf^\xi) \right|
            \ge C_1 \hsigma_\xi 
            \max
            \left\{
                 \sqrt{\frac{\log\left({2|\Xi|}/{\eta}\right)}{n_\test}},
                 \frac{\log\left({2|\Xi|}/{\eta}\right)}{n_\test}
            \right\}
        \right)
        \le \frac{\eta}{|\Xi|}.
    \]
    Applying union bound over $\xi \in \Xi$, for any $0 < \eta < 1/|\Xi|$,
    we get uniform bound
    \[
        \PP
        \left(
            \max_{\xi \in \Xi}
            \left| \hR(\hf^\xi) - R(\hf^\xi) \right|
            \ge C_1
            \max_{\xi \in \Xi} \hsigma_\xi
            \max
            \left\{
                 \sqrt{\frac{\log\left({2|\Xi|}/{\eta}\right)}{n_\test}},
                 \frac{\log\left({2|\Xi|}/{\eta}\right)}{n_\test}
            \right\}
        \right)
        \le \eta.
    \]
    Using the definition of $\Delta_n^\add$,
    and setting $\hsigma_\Xi := \max_{k \in \Xi} \hsigma_\xi$,
    so far we have that
    \begin{equation}\label{eq:maximum-tail-bound-Delta_n}
        \PP
        \left(
            \Delta_n^\add
            \ge C_1 \hsigma_\Xi
            \max
            \left\{
                \sqrt{\frac{\log\left({2|\Xi|}/{\eta}\right)}{n_\test}},
                \frac{\log\left({2|\Xi|}/{\eta}\right)}{n_\test}
            \right\}
        \right)
        \le \eta.
    \end{equation}
    Choosing $\eta = n^{-A}$ for $A > 0$ provides the desired tail bound
    (for a modified constant $C_1 > 0$)
    \[
        \PP
        \left(
            \Delta_n^\add
            \ge
            C_1
            \hsigma_{\Xi}
            \max
            \left\{
                \sqrt{\frac{\log\left(|\Xi| n^A\right)}{n_\test}},
                \frac{\log\left(|\Xi| n^A\right)}{n_\test}
            \right\}
        \right)
        \le
        n^{-A}.
    \]
    
   \paragraph{Expectation bound.}
   We now turn to bounding $\EE[\Delta_n^\add]$.
   Define the event
   \[
       \mathcal{B}_n^\complement
       :=
       \left\{
           \Delta_n^\add \ge C_1 C_2
           \max
           \left\{
                \sqrt{\frac{\log\left(|\Xi| n^A\right)}{n_\test}},
                \frac{\log\left(|\Xi| n^A\right)}{n_\test}
            \right\}
        \right\}.
   \]
   Since $\mathbb{P}(\hsigma_n \ge C_2) \le n^{-A}$,
   combining this with~\eqref{eq:maximum-tail-bound-Delta_n},
   we conclude that $\mathbb{P}(\mathcal{B}_n^\complement) \le 2n^{-A}$.
   For the case of $\CEN = \MOM$,
   the proof follows from that of
   \Cref{lem:bounded-variance-error-control}.
   This follows because
   bounded $\psi_1$ norm
   implies bounded $L_2$ norm.
   
   We can bound $\mathbb{E}[\Delta_n^\add]$ by breaking the expected value as
   \begin{equation}\label{eq:Delta_n-first-split-Orlicz-1}
       \begin{split}
           \mathbb{E}[\Delta_n^\add]
           &= \mathbb{E}[\Delta_n^\add\mathbbm{1}_{\mathcal{B}_n}]
           + \mathbb{E}[\Delta_n^\add\mathbbm{1}_{\mathcal{B}_n^\complement}] \\
           &\le C_1 C_2 
           \max
           \left\{
                \sqrt{\frac{\log\left(|\Xi| n^A\right)}{n_\test}},
                \frac{\log\left(|\Xi| n^A\right)}{n_\test}
           \right\}
            + \left(\mathbb{E}[(\Delta_n^\add)^t]\right)^{1/t}(\mathbb{P}(\mathcal{B}_n^c))^{1/r} \\
           &\le C_1 C_2 
           \max
           \left\{
            \sqrt{\frac{\log\left(|\Xi| n^A\right)}{n_\test}},
            \frac{\log\left(|\Xi| n^A\right)}{n_\test}
           \right\}
            + \left(\mathbb{E}[(\Delta_n^\add)^t]\right)^{1/t}(2n^{-A})^{1/r},
       \end{split}
   \end{equation}
   for H{\"o}lder conjugates $t, r \ge 2$ satisfying $1/t + 1/r = 1$.
   Observe now that 
   \begin{align*}
       \mathbb{E}[(\Delta_n^\add)^t]
       &\le
       |\Xi|
       \max_{\xi \in \Xi}
       \mathbb{E}
       \left[
             \big|\widehat{R}(\widehat{f}^\xi)
                - R(\widehat{f}^\xi)\big|^t
       \right] \\
       &\le |\Xi|
       \max_{\xi \in \Xi}
       \mathbb{E}
       \left[
         \mathbb{E}
         \left[
             \big|\widehat{R}(\widehat{f}^\xi)
                - R(\widehat{f}^\xi)\big|^t \mathrel{\big|} \mathcal{D}_\train
         \right]
       \right]\\
       &\le C_3 |\Xi|
       \max_{\xi \in \Xi}
       \mathbb{E}
       \left[
             \widehat{\sigma}_\xi^t
             \max
             \left\{
                 \left(\frac{t}{n_\test}\right)^{t/2},
                 \left(\frac{t}{n_\test}\right)^t
             \right\}
       \right],
   \end{align*}
   where the last inequality follows from
   integrating the quantile bound
   in~\eqref{eq:conditional-probability-Bernstein-application}
   and $C_3$ is a constant potentially larger than $C_1$.
   Substituting this bound in~\eqref{eq:Delta_n-first-split-Orlicz-1},
   we obtain the desired expectation bound
    \[
        \begin{split}
           \mathbb{E}[\Delta_n^\add]
           &\le C_1 C_2
           \max
           \left\{
                \sqrt{\frac{\log\left(|\Xi| n^A\right)}{n_\test}},
                \frac{\log\left(|\Xi| n^A\right)}{n_\test}
            \right\}
            + C_3
            n^{-A/r}
            |\Xi|^{1/t}
            \max
            \left\{
                 \sqrt{\frac{t}{n_\test}}, \frac{t}{n_\test}
            \right\}
            \max_{\xi \in \Xi}
            \left(\mathbb{E}[\widehat{\sigma}_\xi^t]\right)^{1/t}.
       \end{split}
    \]
    for $t, r \ge 2$ such that $1/r + 1/t = 1$.
    This completes the proof.

\subsection
{Proof of \Cref{lem:bounded-variance-error-control}}
    \paragraph{Tail bound.}
    The proof is similar to the proof of
    \Cref{lem:bounded-orlitz-error-control}.
    Our main workhorse is going to be
    \Cref{lem:mom-concentration}.
    We use
    $\eta = \big(|\Xi| n^A\big)^{-1}$
    in \Cref{alg:general-cross-validation-model-selection}.
    Applying the lemma with such $\eta$
    on the random variables
    $\ell(Y_j, \hf^\xi(X_j)), j \in \cI_\test$
    conditionally on $\cD_\train$,
    for each $\xi \in \Xi$
    we get the tail bound
    \[
        \PP
        \left(
            \left|
                \frac{1}{|\cD_\test|}
                \sum_{j \in \cI_\test}
                \ell(Y_j, \hf^\xi(X_j))
                - R(\hf^\xi)
            \right|
            \ge
            C_1 \hsigma_\xi
            \sqrt{\frac{\log(|\Xi| n^A)}{|\cD_\test|}}
            \, \mathrel{\Bigg|} \cD_\train
        \right)
        \le 
        \frac{n^{-A}}{|\Xi|}
    \]
    for some absolute constant $C_1 > 0$.
    In other words,
    \[
        \PP
        \left(
            \left|
                \hR(\hf^\xi) - R(\hf^\xi)
            \right|
            \ge
            C_1
            \hsigma_\xi
            \sqrt{\frac{\log(|\Xi| n^A)}{n_\test}}
            \mathrel{\Bigg |} \cD_\train
        \right)
        \le
        \frac{n^{-A}}{|\Xi|}.
    \]
    Integrating out $\cD_\train$
    and applying union bound over $\xi \in \Xi$
    then leads to the uniform bound 
    \begin{equation}\label{eq:tail-bound-mom}
        \PP
        \left(
            \max_{\xi \in \Xi}
            \left|
               \hR(\hf^\xi) - R(\hf^\xi)
            \right|
            \ge C_1 \max_{\xi \in \Xi} \hsigma_\xi
            \sqrt{\frac{\log(|\Xi| n^A)}{n_\test}}
        \right)
        \le n^{-A}.
    \end{equation}
    Substituting for the definitions of $\Delta_n^\add$ and $\hsigma_\Xi$
    gives the desired tail bound
    \begin{equation}\label{eq:final-tail-bound-mom}
        \PP
        \left(
            \Delta_n^\add
            \ge
            C_1
            \hsigma_{\Xi}
            \sqrt{\frac{\log(|\Xi| n^A)}{n_\test}}
        \right)
        \le
        n^{-A}.
    \end{equation}
    
    \paragraph{Expectation bound.}
    For bounding $\EE[\Delta_n^\add]$, we again follow similar strategy
    as in the proof of
    \Cref{lem:bounded-orlitz-error-control}.
    In order to bound certain expectations,
    we begin by extending the tail bound \eqref{eq:final-tail-bound-mom}.
    From the assumption,
    $\PP(\hsigma_\Xi \ge C_2) \le n^{-A}$
    for a constant $C_2 > 0$.
    For such a constant, consider the event
    \[
        \mathcal{B}_n^\complement
        :=
        \left\{ 
            \Delta_n^\add
            \ge
            C_1
            C_2
            \sqrt{\frac{\log(|\Xi| n^A)}{n_\test}}
        \right\}.
    \]
    Conditioning on the event $\{ \hsigma_\Xi \ge C_2 \}$,
    we can bound the probability of $\cB_n^\complement$ as follows:
    \begin{align*}
        \PP(\cB_n^\complement)
        &= \PP
        \left(
           \Delta_n^\add
           \ge 
            C_1 C_2
           \sqrt{\frac{\log(|\Xi| n^A)}{n_\test}},
           \hsigma_\Xi \le 
           C_2
        \right)
        + \PP
        \left(
           \Delta_n^\add \ge 
            C_1 C_2
           \sqrt{\frac{\log(|\Xi| n^A)}{n_\test}},
           \hsigma_\Xi \ge 
            C_2
        \right) \\
        &\le \PP
        \left(
            \Delta_n^\add \ge C_1 \hsigma_\Xi \sqrt{\frac{\log(|\Xi| n^A)}{n_\test}}
        \right)
        + \PP \left( \hsigma_n \ge C_2 \right) \le \frac{2}{n^A},
    \end{align*}
    where we used the bound from \eqref{eq:final-tail-bound-mom}.
    We are now ready to bound $\EE[\Delta_n^\add]$
    by splitting using the event $\cB_n^\complement$.
    We have
    \begin{align}
        \EE\left[\Delta_n^\add\right]
        &= \EE\left[\Delta_n^\add \mathbbm{1}_{\mathcal{B}_n}\right]
        + \EE\left[\Delta_n^\add \mathbbm{1}_{\mathcal{B}_n^\complement}\right] \nonumber \\
       &\le 
        C_1 C_2
       \sqrt{\frac{\log(|\Xi| n^A)}{n_\test}}
       + \left(\PP(\mathcal{B}_n^\complement)\right)^{1/2}
       \left(\EE[|\Delta_n^\add|^2]\right)^{1/2} \nonumber \\
       &\le 
        C_1 C_2
       \sqrt{\frac{\log(|\Xi| n^A)}{n_\test}}
       + \left(2 n^{-A}\right)^{1/2}
       \left(\EE[|\Delta_n^\add|^2]\right)^{1/2} \label{eq:expection-bound-mom-intermediate}
    \end{align}
    where in the first inequality,
    we used Cauchy-Schwartz inequality for the second term.
    It remains to bound $\EE[|\Delta_n^\add|^2]$, which we do below.
    We have
    \begin{align*}
        \EE[\left|\Delta_n^\add\right|^2]
        &= \EE
        \left[
                \max_{\xi \in \Xi}
                \left|\hR(\hf^\xi) - R(\hf^\xi)\right|^2
        \right] \le |\Xi| \max_{\xi \in \Xi}
        \EE\left[|\hR(\hf^\xi) - R(\hf^\xi)|^2\right].
    \end{align*}
    For bounding the second term,
    recall that the $\MOM$ procedure computes
    $\hR(\hf^\xi)$ as the median of empirical means computed on
    $B$ partitions of the test data.
    For each of the $B$ partitions,
    the variance of the empirical mean is
    $\hsigma_\xi^2 / (n_{\test} / B)$.
    To bound the variance of the median of means on $B$ partitions,
    we invoke Theorem 1 of 
    \cite{gribkova_2020}
    (with $k = 2$, $\rho = 1$, and $i$ corresponding to the median position).
    Note that each of the $B$ empirical means are
    independent and identically distributed.
    This provides
    \begin{align*}
        \EE
        \left[
            \left| \hR(\hf^\xi) - R(\hf^\xi) \right|^2
            \mathrel{\Big |}
            \cD_\train
        \right]
        &\le
        C
        \left( \frac{\hsigma_\xi^2}{n_\test / B} \right)
        \le C \frac{B \hsigma_\xi^2}{n_\test}.
    \end{align*}
    for some absolute constant $C$.
    Thus,
    \begin{align*}
        \left(
            \EE\left[|\Delta_n^\add|^2\right]
        \right)^{1/2}
        &\le
        C
        \left(
            |\Xi| \frac{B}{n_\test}
            \max_{\xi \in \Xi} \EE[\hsigma_\xi^2]
        \right)^{1/2} \\
        &\le
        C
        |\Xi|^{1/2}
        \sqrt{\frac{B}{n_\test}}
        \max_{\xi \in \Xi}
        \Big(\EE[\hsigma_\xi^2]\Big)^{1/2} \\
    \end{align*}
    Recalling $B = \lceil 8 \log(|\Xi| n^A) \rceil$
    and combining this bound with \eqref{eq:expection-bound-mom-intermediate},
    we finally have the desired expectation bound
    \[
        \EE\left[ \Delta_n^\add \right]
        \le
        C_1 C_2 
        \sqrt{\frac{\log(|\Xi| n^A)}{n_\test}}
        +
        C_3
        n^{-A/2}
        |\Xi|^{1/2}
        \sqrt{\frac{\log(|\Xi| n^A)}{n_\test}}
        \max_{\xi \in \Xi}
        \Big(\EE[\hsigma_\xi^2]\Big)^{1/2}.
    \]
    for some absolute constant $C_3 > 0$.
    This completes the proof.
    
\subsection
{Proof of \Cref{lem:bounded-orlitz-error-control-mul-form}}

As argued in the proof of \Cref{lem:bounded-orlitz-error-control},
using \Cref{lem:bernstein-ineq},
for any $A > 0$,
we have the tail bound:
\[
    \PP
    \left(
        \left|
            \hR(\hf^\xi) - R(\hf^\xi)
        \right|
        \ge
        C \hsigma_\xi
        \max
        \left\{
            \sqrt{\frac{\log(|\Xi| n^A)}{|\cD_\test|}},
            \frac{\log(|\Xi| n^{A})}{|\cD_\test|}
        \right\}
        \mathrel{\Bigg |} \cD_\train
    \right)
    \le
    \frac{n^{-A}}{|\Xi|}
\]
for some universal constant $C > 0$.
By diving $R(\hf^\xi)$ on the both side of error event,
and denoting $\hsigma_\xi / R(\hf^\xi)$ by $\hkappa_\xi$,
equivalently we have 
\[
    \PP
    \left(
        \left|
            \frac{\hR(\hf^\xi)}{R(\hf^\xi)} - 1
        \right|
        \ge
        C
        \hkappa_\xi
        \max
        \left\{
            \sqrt{\frac{\log(|\Xi| n^A)}{|\cD_\test|}},
            \frac{\log(|\Xi| n^A)}{|\cD_\test|}
        \right\}
        \mathrel{\Bigg |} \cD_\train
    \right)
    \le
    \frac{n^{-A}}{|\Xi|}.
\]
Integrating over randomness in $\cD_\train$,
and applying union bound over $\xi \in \Xi$,
we obtain
\[
    \PP
    \left(
        \max_{\xi \in \Xi}
        \bigg|
            \frac{\hR(\hf^\xi)}{R(\hf^\xi)} - 1
        \bigg|
        \ge
        C
        \max_{\xi \in \Xi} \hkappa_\xi
        \max
        \left\{
            \sqrt{\frac{\log(|\Xi| n^A)}{n_\test}},
            \frac{\log(|\Xi| n^{A})}{n_\test}
        \right\}
    \right)
    \le n^{-A}.
\]
In other words, in terms $\Delta_n^{\mul}$ and $\hkappa_\Xi$,
we have
\[
    \PP
    \left(
        \Delta_n^{\mul}
        \ge
        C
        \hkappa_\Xi
        \max
        \left\{
            \sqrt{\frac{\log(|\Xi| n^A)}{n_\test}},
            \frac{\log(|\Xi| n^A)}{n_\test}
        \right\}
    \right)
    \le n^{-A},
\]
as desired.
This completes the proof.

\subsection
{Proof of \Cref{lem:bounded-variance-error-control-mul-form}}

As argued in the proof of \Cref{lem:bounded-variance-error-control},
using \Cref{lem:mom-concentration},
for any $A > 0$,
we have the following tail bound:
\[
    \PP
    \left(
        \left|
            \hR(\hf^\xi) - R(\hf^\xi)
        \right|
        \ge
        C \hsigma_\xi
        \sqrt{\frac{\log(|\Xi| n^A)}{|\cD_\test|}}
        \mathrel{\Bigg |} \cD_\train
    \right)
    \le
    \frac{n^{-A}}{|\Xi|}
\]
for some universal constant $C > 0$.
By diving $R(\hf^\xi)$ on the both side of error event,
and denoting $\hsigma_\xi / R(\hf^\xi)$ by $\hkappa_\xi$,
we obtain
\[
    \PP
    \left(
        \left|
            \frac{\hR(\hf^\xi)}{R(\hf^\xi)} - 1
        \right|
        \ge
        C
        \hkappa_\xi
        \sqrt{\frac{\log(|\Xi| n^A)}{|\cD_\test|}}
        \mathrel{\Bigg |} \cD_\train
    \right)
    \le
    \frac{n^{-A}}{|\Xi|}.
\]
Integrating over randomness in $\cD_\train$,
and applying union bound over $\xi \in \Xi$,
this implies that
\[
    \PP
    \left(
        \max_{\xi \in \Xi}
        \bigg|
            \frac{\hR(\hf^\xi)}{R(\hf^\xi)} - 1
        \bigg|
        \ge
        C
        \max_{\xi \in \Xi} \hkappa_\xi
        \sqrt{\frac{\log(|\Xi| n^A)}{n_\test}}
    \right)
    \le n^{-A}.
\]
Writing in terms $\Delta_n^{\mul}$ and $\hkappa_\Xi$,
we arrive at the desired bound:
\[
    \PP
    \left(
        \Delta_n^{\mul}
        \ge
        C
        \hkappa_\Xi
        \sqrt{\frac{\log(|\Xi| n^A)}{n_\test}}
    \right)
    \le n^{-A}.
\]
This finishes the proof.

\subsection
{Proof of \Cref{prop:misclassification-hinge-psi1l2-bound}}

    \paragraph{Part 1.}
    For the first part,
    observe that
    $| \ell(Y_0, \hf(X_0)) | = \max\{ 0, 1 - Y_0 \hf(X_0) \} \le 2$
    assuming $|Y_0| \le 1$ and $|\hf(X_0)| \le 1$.
    For a bounded random variable $Z$,
    $\| Z \|_{\psi_2} \lesssim \| Z \|_{\infty}$
    (see, e.g., Example 2.5.8 of \cite{vershynin_2018}).
    Thus, the random variable $\ell(Y_0, \hf(X_0))$
    is conditionally sub-Gaussian 
    with sub-Gaussian norm $2$ (up to constants),
    and consequently sub-exponential
    with the same sub-exponential norm upper bound.
    The conditional $L_2$ norm bound follows similarly.
    
    \paragraph{Part 2.}
    The second part follows in the same vein by noting that
    $\ell(Y_0, \hf(X_0)) = \mathbbm{1}_{Y_0 \ne \hf(X_0)}$
    only takes values $0$ or $1$, and Bernoulli random variables are sub-Gaussian
    with sub-Gaussian norm $1$ (up to constants)
    and hence sub-exponential with the same sub-exponential norm
    upper bound.
    The bound on the conditional $L_2$ norm follows analogously.

\subsection
{Proof of \Cref{thm:oracle-bound-classifier-miss-hinge=error}}

An outline for the proof is already provided in \Cref{sec:common-loss-functions}.
The theorem follows by combining
the additive form of the oracle inequality from \Cref{prop:general-model-selection-guarantee},
along with
the probabilistic bounds on $\Delta^\add$ from
\Cref{lem:bounded-variance-error-control,lem:bounded-orlitz-error-control},
and the bounds on conditional $\psi_1$ and $L_2$ norm bounds from
\Cref{prop:misclassification-hinge-psi1l2-bound}.

\subsection
{Proof of \Cref{prop:subexp-ex-squared}}

\paragraph{Part 1.}
For the first part,
we bound the $\psi_1$ norm of the squared error
by the squared $\psi_2$ norm of the error to get
\begin{equation}
    \label{eq:squared-error-psi1-norm}
    \| \ell(Y_0, \hf(X_0)) \|_{\psi_1 \mid \cD_n}
    = \| (Y_0 - X_0^\top \hbeta)^2 \|_{\psi_1 \mid \cD_n}
    \le \| Y_0 - X_0^\top \hbeta \|_{\psi_2 \mid \cD_n}^2,
\end{equation}
where the inequality follows by Lemma 2.7.7 of \cite{vershynin_2018}.
Note that for any $\beta \in \RR^{p}$, we have
\begin{equation}
    \label{eq:error-decomposition-wrt-arbitrary-vector}
    (Y_0 - X_0^\top \hbeta)
    = (Y_0 - X_0^\top \beta) + X_0^\top (\beta - \hbeta).
\end{equation}
Because $\| Z_1 + Z_2 \|_{\psi_2} \le \| Z_1 \|_{\psi_2} + \| Z \|_{\psi_2}$
we can bound
\begin{equation}
    \label{eq:prederr-psi2-triangle}
    \| Y_0 - X_0^\top \hbeta \|_{\psi_2 \mid \cD_n}
    \le \| Y_0 - X_0^\top \beta \|_{\psi_2} 
    + \| X_0^\top (\beta - \hbeta) \|_{\psi_2 \mid \cD_n}.
\end{equation}
Noting that $Y_0 - X_0^\top \beta = (Y_0, X_0)^\top (1, -\beta)$
and $(\beta - \hbeta)$ is a fixed vector conditioned on $\cD_n$,
by using $\psi_2-L_2$ equivalence on $(X_0, Y_0)$,
we have
\begin{equation}
    \label{eq:prederr-terms-psi2l2}
    \| Y_0 - X_0^\top \beta \|_{\psi_2}
    \le \tau \| Y_0 - X_0^\top \beta \|_{L_2}
    \quad
    \text{and}
    \quad
    \| X_0^\top (\beta - \hbeta) \|_{\psi_2 \mid \cD_n}
    \le \tau \| X_0^\top (\beta - \hbeta) \|_{L_2 \mid \cD_n}
    = \tau \| \hbeta - \beta \|_{\Sigma},
\end{equation}
where in the last inequality we used the fact that
$\EE[X_0] = 0$ and $\EE[X_0 X_0^\top] = \Sigma$.
Thus, combining 
\eqref{eq:squared-error-psi1-norm},
\eqref{eq:prederr-psi2-triangle},
and
\eqref{eq:prederr-terms-psi2l2},
for $\beta \in \RR^{p}$, we have
\[
    \| \ell(Y_0 - X_0^\top \hbeta) \|_{\psi_1 \mid \cD_n}
    \le (\| Y_0 - X_0^\top \beta \|_{\psi_2} + \| \hbeta  - \beta \|_{\Sigma})^2.
\]
Taking infimum over $\beta$, we have that for squared loss
\[
    \| \ell(Y_0, \hf(X_0)) \|_{\psi_1 \mid \cD_n}
    \le \tau^2 \inf_{\beta \in \RR^{p}}
    (\| Y_0 - X_0^\top \beta \|_{\psi_2} + \| \hbeta - \beta \|_{\Sigma})^2,
\]
as desired. This completes the proof of the first inequality
in \eqref{eqn:linear-case-one}.
For the second inequality in \eqref{eqn:linear-case-one},
using the $\psi_2-L_2$ equivalence on the vector $(X_0, Y_0)$,
observe that
\begin{equation}
    \label{eq:squared-error-expectation}
    \EE[\ell(Y_0, \hf(X_0)) \mid \cD_n]
    = \EE[(Y_0 - X_0^\top \hbeta)^2 \mid \cD_n]
    = \| Y_0 - X_0^\top \|_{L_2 \mid \cD_n}^2.
\end{equation}
Hence, from 
\eqref{eq:squared-error-psi1-norm} 
and
\eqref{eq:squared-error-expectation},
we have
\[
    \frac
    {\| \ell(Y_0, \hf(X_0)) \|_{\psi_1 \mid \cD_n}}
    {\EE[ \ell(Y_0, \hf(X_0)) \mid \cD_n]}
    \le 
    \frac
    {\| Y_0 - X_0^\top \hbeta \|^2_{\psi_2 \mid \cD_n}}
    {\| Y_0 - X_0^\top \hbeta \|^2_{L_2 \mid \cD_n}}
    =
    \left(
        \frac
        {\| (Y_0, X_0) (1, -\hbeta) \|_{\psi_2 \mid \cD_n}}
        {\| (Y_0, X_0) (1, -\hbeta) \|_{L_2 \mid \cD_n}}
    \right)^2
    \le \tau^2,
\]
as desired.
This completes the proof of the first part.

\paragraph{Part 2.}

We now turn to the second part to bound the conditional $L_2$ norm of the square loss.
For the square loss, note that
\begin{equation}
    \label{eq:squared-error-l2-squared}
    \| \ell(Y_0, \hf(X_0)) \|_{L_2 \mid \cD_n}^2
    = \EE[(Y_0 - \hf(X_0))^4 \mid \cD_n].
\end{equation}
Using the decomposition \eqref{eq:error-decomposition-wrt-arbitrary-vector} 
and triangle inequality with respect to the $L_4$ norm,
we have
\begin{equation}
    \label{eq:prederr-l4-triangle}
    \EE[ (Y_0 - X_0^\top \hbeta)^4  \mid \cD_n]^{1/4}
    \le 
    \EE[ (Y_0 - X_0^\top \beta)^4 \mid \cD_n]^{1/4}
    + \EE[ X_0^\top (\beta - \hbeta)^4 \mid \cD_n]^{1/4} 
\end{equation}
Using the $L_4-L_2$ equivalence for $(Y_0, X_0)$,
we can bound
\begin{equation}
    \label{eq:prederr-terms-l4l2}
    \| Y_0 - X_0^\top \beta \|_{L_4}
    \le \tau \| Y_0 - X_0^\top \beta \|_{L_2}
    \quad
    \text{and}
    \quad
    \| X_0^\top (\beta - \hbeta) \|_{L_4 \mid \cD_n}
    \le \tau \| X_0^\top (\beta - \hbeta) \|_{L_2 \mid \cD_n}.
\end{equation}
Thus, combining 
\eqref{eq:squared-error-l2-squared},
\eqref{eq:prederr-l4-triangle},
and
\eqref{eq:prederr-terms-l4l2},
we have for any $\beta \in \RR^{p}$,
\[
    \| (Y_0, \hf(X_0)) \|_{L_2 \mid \cD_n}
    \le (\tau \| Y_0 - X_0^\top \beta \|_{L_2} + \tau \| \hbeta - \beta \|_{\Sigma})^2
    \le \tau^2 
    (\| Y_0 - X_0^\top \beta \|_{L_2} + \| \hbeta - \beta \|_{\Sigma})^2.
\]
This completes the proof of first inequality in \eqref{eqn:linear-case-two}.
For the second inequality of \eqref{eqn:linear-case-two},
note that
\[
    \frac
    {\| \ell(Y_0, \hf(X_0)) \|_{L_2 \mid \cD_n}}
    {\EE[\ell(Y_0, \hf(X_0)) \mid \cD_n]}
    \le
    \frac
    {\| Y_0 - \hf(X_0) \|_{L_4 \mid \cD_n}^2}
    {\| Y_0 - \hf(X_0) \|_{L_2 \mid \cD_n}^2}
    =
    \left(
        \frac
        {\| (Y_0, X_0) (1, -\hbeta) \|_{L_4 \mid \cD_n}}
        {\| (Y_0, X_0) (1, -\hbeta) \|_{L_2 \mid \cD_n}}
    \right)^2
    \le \tau^2.
\]
This concludes the proof of the second part.

\subsection
{Proof of \Cref{prop:linear-predictor-absolute-loss}}

The proof is similar to that of \Cref{prop:subexp-ex-squared}.

\paragraph{Part 1.}

From the decomposition \eqref{eq:error-decomposition-wrt-arbitrary-vector}
and the triangle inequality on $\psi_1$ norm,
we have for any $\beta \in \RR^{p}$,
\begin{equation}
    \label{eq:prederr-psi1-triangle}
    \| Y_0 - X_0^\top \hbeta \|_{\psi_1 \mid \cD_n}
    \le
    \| Y_0 - X_0^\top \beta \|_{\psi_1}
    + \| X_0^\top (\beta - \hbeta) \|_{\psi_1 \mid \cD_n}.
\end{equation}
Using the $\psi_1-L_1$ equivalence of $(X_0, Y_0)$,
note that
\begin{equation}
    \label{eq:prederr-terms-psi1l1}
    \| Y_0 - X_0^\top \beta \|_{\psi_1}
    \le \tau \| Y_0 - X_0^\top \beta \|_{L_1}
    \quad
    \text{and}
    \quad
    \| X_0^\top (\beta - \hbeta) \|_{\psi_1 \mid \cD_n}
    \le \tau \| X_0^\top (\beta - \hbeta) \|_{\psi_1 \mid \cD_n}.
\end{equation}
Thus, 
from 
\eqref{eq:prederr-psi1-triangle} 
and
\eqref{eq:prederr-terms-psi1l1},
for any $\beta \in \RR^{p}$,
we have
\[
    \| Y_0 - X_0^\top \hbeta \|_{\psi_1 \mid \cD_n}
    \le 
    \tau 
    ( \| Y_0 - X_0^\top \beta \|_{L_1} +  \| X_0^\top (\hbeta - \beta) \|_{L_1 \mid \cD_n} ).
\]
Now taking infimum over $\beta \in \RR^{p}$ yields the first inequality
of \eqref{eq:abserr-l2l1}.
To show the second inequality,
observe that
\[
    \frac
    {\| \ell(Y_0, \hf(X_0)) \|_{\psi_1 \mid \cD_n}}
    {\EE[ \ell(Y_0, \hf(X_0)) \mid \cD_n ]}
    \le
    \frac
    {\| Y_0 - X_0^\top \hbeta \|_{\psi_1 \mid \cD_n}}
    {\| Y_0 - X_0^\top \hbeta \|_{L_1 \mid \cD_n}}
    \le
    \tau,
\]
as desired. This finishes the proof.

\paragraph{Part 2.}

The second part follows analogously to the first part
by using the $L_2-L_1$ equivalence on $(X_0, Y_0)$.

\subsection
{Proof of \Cref{prop:linear-predictor-logistic-loss}}
\label{sec:bounds-norms-ratiosofnorms-logistic}

We start by writing the loss as
\begin{align*}
    \ell(Y_0, \hf(X_0))
    &= Y_0 \log(1 + e^{-X_0^\top \hbeta}) + (1 - Y_0) \log(1 + e^{X_0^\top \hbeta})\\
    &= \mathrm{KL}(Y_0,\,(1 + \exp(-X_0^{\top}\hbeta))^{-1}).
\end{align*}
Observe that the loss is non-negative
since $\log(1 + e^{t}) \ge 0$ for all $t$.

\paragraph{Upper bounds on $\psi_1$ and $L_2$ norms.}
We will first obtain an upper on the loss
and consequently on the $\psi_1$ and $L_2$ norms of the loss.
Because $Y_0$ takes values $0$ or $1$,
we have that
\begin{align*}
    \ell(Y_0, \hf(X_0))
    &\le \max \big\{ \log(1 + e^{-X_0^\top \hbeta}), \log(1 + e^{X_0^\top \hbeta}) \big\} \\
    &\le \log(1 + e^{|X_0^\top \hbeta|}),
\end{align*}
where the second inequality follows since $t \mapsto e^{t}$ is monotonically increasing in $t$.
Now 
using the following bound on $\log(1 + e^{|t|})$:
\[
    \log(1 + e^{|t|})
    \le
    \begin{cases}
        \log 2 & \text{ if } e^{|t|} \le 1 \\
        \log(2 e^{|t|}) = \log 2 + | t | & \text{ otherwise},
    \end{cases}
\]
we can upper bound the loss by
\[
    \ell(Y_0, \hf(X_0))
    \le | X_0^\top \hbeta | + \log 2.
\]
Hence,
we can upper bound the $\psi_1$ and $L_2$ norm of the loss as follows:
\begin{align}
\|\ell(Y_0, \widehat{f}(X_0))\|_{\psi_1 \mid \cD_n} &\le \log(2) + \|X_0^{\top}\hbeta\|_{\psi_1 \mid \cD_n}, 
\label{eq:logistic-psi1-ub} \\
(\mathbb{E}[\ell^2(Y_0, \widehat{f}(X_0)) \mid \cD_n])^{1/2} &\le \log(2) 
+ (\mathbb{E}[|X_0^{\top}\hbeta|^2 \mid \cD_n])^{1/2}.
\label{eq:logistic-l2-ub}
\end{align}

\paragraph{Lower bound on expectation.}
Next we obtain a lower bound on $\EE[\ell(Y_0, \hf(X_0)) \mid \cD_n]$.
Setting $p(x) = \mathbb{E}[Y_0|X_0 = x]$, it is clear that
\[
\mathbb{E}[\ell(Y_0, \hf(X_0)) \mid \cD_n, X_0] = p(X_0)\log(1 + \exp(-X_0^{\top}\hbeta)) + (1 - p(X_0))\log(1 + \exp(X_0^{\top}\hbeta)).
\]
Because $0 < p_{\min} \le \min\{p(x), 1 - p(x)\}$
for all $x$,
we have
\begin{align}
\mathbb{E}[\ell(Y_0, \hf(X_0)) \mid \cD_n] 
&\ge 
p_{\min}
\, \mathbb{E}[\max\{\log(1 + \exp(-X_0^{\top}\hbeta)),\, \log(1 + \exp(X_0^{\top}\hbeta))\} \mid \cD_n]
\nonumber \\ 
&= 
p_{\min}
\, \mathbb{E}[\log(1 + \exp(|X_0^{\top}\hbeta|)) \mid \cD_n] 
\nonumber \\
&\ge 
\frac{p_{\min}}{2}
\, \mathbb{E}[\log(2) + |X_0^{\top}\hbeta| \mid \cD_n] 
    = \frac{p_{\min}}{2}(\log(2) + \mathbb{E}|X^{\top}_0\hbeta|),
\label{eq:logistic-l1-lb}
\end{align}
where 
the second equality follows since $t \mapsto e^{t}$ is monotonically increasing in $t \in \RR$,
and the last inequality follows from the fact that $1/2 \le \log(1 + \exp(x))/(\log(2) + x) \le 1$ for all $x \ge 0$.

Using \eqref{eq:logistic-psi1-ub} and \eqref{eq:logistic-l1-lb},
we have
\[
    \frac
    {\| \ell(Y_0, \hf(X_0)) \|_{\psi_1 \mid \cD_n}}
    {\EE[\ell(Y_0, \hf(X_0)) \mid \cD_n]}
    \le 
    \frac
    {\| X_0^\top \hbeta \|_{\psi_1 \mid \cD_n} + \log(2)}
    {p_{\min}(\EE[|X_0^\top \hbeta| \mid \cD_n] + \log(2)) / 2}
    \le
    \frac
    {\tau \| X_0^\top \hbeta \|_{L_1 \mid \cD_n} + \log(2)}
    {p_{\min}(\tau \| X_0^\top \hbeta \|_{L_1 \mid \cD_n} + \log(2)) / 2}
    =
    2 \tau p_{\min}^{-1}.
\]
This proves the first part of \Cref{prop:linear-predictor-logistic-loss}.
A similar bound holds for the second inequality of \Cref{prop:linear-predictor-logistic-loss}
using upper bound from \eqref{eq:logistic-l2-ub} 
and lower bound \eqref{eq:logistic-l1-lb}.
This completes the proof.

\subsection
{Proof of \Cref{thm:oracle-bound-linear-predictor-squared-error}}

An outline for the proof is provided in \Cref{sec:common-loss-functions}.
The theorem follows by combining
the multiplicative form of the oracle inequality
from \Cref{prop:general-model-selection-guarantee},
along with probabilistic bounds on $\Delta^\mul$ from
\Cref{lem:bounded-orlitz-error-control-mul-form,lem:bounded-variance-error-control-mul-form},
and the bounds on ratio of conditional $\psi_1$ and $L_1$ norms,
and $L_2$ and $L_1$ norms from
\Cref{prop:subexp-ex-squared}.

\section{Proofs related to risk monotonization for zero-step procedure}
\label{sec:proofs-riskmonotonization-zerostep}

\subsection
{Proof of \Cref{thm:monotonization-zerostep}}

An outline for the proof is already provided in 
\Cref{sec:risk-behavior-zerostep}.
For the sake of completeness,
we briefly summarize the main steps below.

The deterministic additive and multiplicative oracle risk inequalities
from \Cref{prop:general-model-selection-guarantee},
along with probabilistic bounds from
\Cref{lem:bounded-orlitz-error-control,lem:bounded-variance-error-control,lem:bounded-orlitz-error-control-mul-form,lem:bounded-variance-error-control-mul-form},
provide the following bound on the
risk of the zero-step predictor 
\begin{equation}
    \label{eq:zerostep-nonasymp-risk-bound-proof}
    R(\hf^{\zerostep})
    =
    \begin{cases}
        \min_{\xi \in \Xi_n}
        R(\hf^\xi)
        +
        O_p(1) \sqrt{\log n / n_\test}
        & \text{ if } \hsigma_\Xi = O_p(1), \\
        \min_{\xi \in \Xi_n} R(\hf^\xi)
        \big(1 + O_p(1) \sqrt{\log n / n_\test}\big)
        & \text{ if } \hkappa_\xi = O_p(1).
    \end{cases}
\end{equation}
Depending on the value of $M$,
we now bound the term $\min_{\xi \in \Xi_n} R(\hf^\xi)$
under the assumptions
\eqref{eq:rn-deterministic-approximation-3}
or
\eqref{eq:rn-deterministic-approximation-2}.

\paragraph{Case of $M = 1$.}

Under \eqref{eq:rn-deterministic-approximation-3},
we have
from \eqref{eq:nonasymp-bound-M=1},
\begin{equation}
    \label{eq:nonasymp-bound-M=1-proof}
    \min_{\xi \in \Xi_n} R(\hf^\xi)
    =
    \min_{\xi \in \Xi_n} R(\tf(\cdot; \cD_\train^{\xi, 1}))
    =
   R^\deter_\nearrow(n; \tf)
    (1 + o_p(1)).
\end{equation}
Combining
\eqref{eq:nonasymp-bound-M=1-proof}
with
\eqref{eq:zerostep-nonasymp-risk-bound-proof}
yields
\begin{equation}
\label{eq:zerostep_risk-gaurantee-mindeter-proof}
\begin{split}
    R(\hf^\zerostep)
    &=
    \begin{cases}
    R^\deter_\nearrow(n; \tf)
        (1 + o_p(1))
        + O_p(1) \sqrt{\log n / n_\test}
        & \text{ if } \hsigma_\Xi = O_p(1) \\
        R^\deter_\nearrow(n; \tf)
        (1 + o_p(1))
        & \text{ if } \hkappa_\Xi = O_p(1)
    \end{cases} \\
    &=
    R^\deter_\nearrow(n; \tf)
    \begin{cases}
        1 + o_p(1) +
       \sqrt{\log n / n_\test} /
         R^\deter_\nearrow(n; \tf)
        & \text{ if } \hsigma_\Xi = O_p(1) \\
        1 + o_p(1)
        & \text{ if } \hkappa_\Xi = O_p(1).
    \end{cases}
\end{split}
\end{equation}
Thus,
under
\ref{cond:zerostep-add}
or 
\ref{cond:zerostep-mult},
we have
$
    {|R(\hf^\zerostep) - R^\deter_\nearrow(n; \tf)|}/{R^\deter_\nearrow(n; \tf)}
    = o_p(1)
$
as desired.

\paragraph{Case of $M > 1$.}

Under \eqref{eq:rn-deterministic-approximation-2},
we have from \eqref{eq:bound-minR-hf},
\begin{equation}
\label{eq:bound-minR-hf-proof}
\begin{split}
    \min_{\xi \in \Xi_n} R(\hf^\xi)
    \le
    R^\deter_\nearrow(n; \tf)
    (1 + o_p(1)).
\end{split}
\end{equation}
Now similar to the case of $M = 1$,
combining
\eqref{eq:bound-minR-hf-proof}
with
\eqref{eq:zerostep-nonasymp-risk-bound-proof},
and under 
\ref{cond:zerostep-add}
or 
\ref{cond:zerostep-mult},
we have that
$
    {(R(\hf^\zerostep) - R^\deter_\nearrow(n; \tf))_+}/{R^\deter_\nearrow(n; \tf)}
    = o_p(1)
$
as claimed.
This finishes the proof.

\subsection
{Proof of \Cref{lem:rn-deterministic-approximation-4-prop-asymptotics}}

Our goal is to verify \eqref{eq:rn-deterministic-approximation-2-prop-asymptotics},
i.e.,
existence of a deterministic profile $R^\deter(\cdot; \tf)$ such that
for all non-stochastic sequences 
$\xi^\star_n \in \argmin_{\xi \in \Xi_n} R^\deter(p_n / n_{\xi}; \tf)$
and $1 \le j \le M$,
\[
    \frac
    {
            R(\tf(\cdot; \cD_\train^{\xi_n^\star, j}))
            - R^\deter(p_n / n_{\xi_n^\star}; \tf)
    }
    {
        R^\deter(p_n / n_{\xi_n^\star}; \tf)
    }
    \pto 0,
\]
as $n \to \infty$ under \ref{asm:prop_asymptotics}.
Recall here $\tf(\cdot; \cD_\train^{\xi_n^\star, j})$, $1 \le j \le M$,
is a predictor trained on the dataset $\cD_\train^{\xi_n^\star, j}$
of sample size $n_{\xi_n^\star} = n_\train - \xi_n^\star \lfloor n^\nu \rfloor$ and feature dimension $p_n$.
We will make a series of reductions
to verify \eqref{eq:rn-deterministic-approximation-2-prop-asymptotics}
from the assumptions of \Cref{lem:rn-deterministic-approximation-4-prop-asymptotics}.

First, note that $R(\tf(\cdot; \cD_\train^{\xi_n, j}))$
for $1 \le j \le M$ are identically distributed.
It thus suffices to pick $j = 1$,
which we will do below and drop the index for notational brevity.
Second, since $R(\tf(\cdot; \cD_{k_m})) > 0$ for all $k_m$,
it suffices to show that 
as $n \to \infty$ under \ref{asm:prop_asymptotics},
\[
            R(\tf(\cdot; \cD_\train^{\xi_n^\star}))
            - R^\deter(p_n / n_{\xi_n^\star}; \tf)
    \pto 0,
    \quad
    \text{where}
    \quad
    \xi_n^\star \in \argmin_{\xi \in \Xi_n} R^\deter(p_n / n_{\xi}; \tf).
\]
More explicitly, that for all $\epsilon > 0$,
it suffices to verify that
as $n \to \infty$
under \ref{asm:prop_asymptotics},
\[
    \PP
    \big(
        {
            | R(\tf(\cdot; \cD_\train^{\xi_n^\star})) 
            - R^\deter(p_n/n_{\xi_n^\star}; \tf) |
        }
        \ge \epsilon
    \big)
    \to 0,
    \quad
    \text{where}
    \quad
    \xi_n^\star \in \argmin_{\xi \in \Xi_n} R^\deter(p_n / n_\xi; \tf).
\]
Now, we will do our final reduction.
Fix $\epsilon > 0$.
Define a sequence $\{ h_n(\epsilon) \}_{n \ge 1}$ as follows:
\[
    h_n(\epsilon)
    :=
    \PP
    \big(
        {| R(\tf(\cdot; \cD_\train^{\xi_n^\star})) 
        - R^\deter(p_n / n_{\xi_n^\star}; \tf)|}
        \ge \epsilon
    \big).
\]
From the discussion in \Cref{sec:zerostep-overparameterized},
we know that $p_n / n_{\xi_n^\star}$ may not necessarily converge as $n \to \infty$.
But applying \Cref{lem:subsequence-to-sequence}
on the sequence $\{ h_n(\epsilon) \}_{n \ge 1}$,
in order to verify that $h_n(\epsilon) \to 0$ as $n \to \infty$,
it suffices to show that
for any index subsequence $\{ n_{k} \}_{k \ge 1}$,
there exists a further subsequence $\{ n_{k_{l}} \}_{l \ge 1}$
such that $h_{n_{k_{l}}}(\epsilon) \to 0$ as $l \to 0$.
Towards that goal,
fix an arbitrary index subsequence $\{ n_{k} \}_{k \ge 1}$.
We will appeal to \Cref{lem:limit-argmins-metricspace}
to construct the desired subsequence $\{ n_{k_{l}} \}_{l \ge 1}$
along which we will argue that $h_{n_{k_{l}}} \to 0$
provided the assumptions of \Cref{lem:rn-deterministic-approximation-4-prop-asymptotics}
are satisfied.
In particular, 
from \Cref{lem:spacefilling-grid-zerostep-general},
note that since $n_\train / n \to 1$ as $n \to \infty$,
we have $\Pi_{\Xi_n}(\zeta) \to \zeta$ for any $\zeta \in [\gamma, \infty]$
as $n \to \infty$.
Now applying \Cref{lem:limit-argmins-metricspace} on $R^\deter(\cdot; \tf)$
and the grid $\Xi_n$
guarantees that 
for any subsequence $\{ p_{n_{k}} / n_{\xi_{n_k}^\star} \}_{k \ge 1}$,
there exists a subsequence $\{ p_{n_{k_{l}}} / n_{\xi_{n_{k_{l}}}^\star} \}_{l \ge 1}$
such that as $l \to \infty$,
\begin{equation}
    \label{eq:subsequence-limit-gaurantee}
    \frac{p_n}{n_{\xi^\star_{n_{k_{l}}}}}
    \to \phi \in \argmin_{\zeta \in [\gamma, \infty]} R^\deter(\zeta; \tf).
\end{equation}
We will now show that
$h_{n_{k_{l}}}(\epsilon) \to 0$ as $l \to \infty$
if the profile convergence assumption 
\eqref{tag:detpar-0}
of 
\Cref{lem:rn-deterministic-approximation-4-prop-asymptotics}
is satisfied,
i.e., for a dataset $\cD_{k_{m}}$ with $k_m$ observations
and $p_m$ features,
there exists $R^\deter(\cdot; \tf)$ such that
\begin{equation}
    \label{eq:hypothesis-detpar0}
    R(\tf(\cdot; \cD_{k_m}))
    \pto R^\deter(\phi; \tf)
\quad
\text{whenever}
\quad
    \frac{p_m}{k_m}
    \to \phi \in \argmin_{\zeta \in [\gamma, \infty]} R^\deter(\zeta; \tf).
\end{equation}
This follows easily because the profile convergence condition \eqref{eq:hypothesis-detpar0} implies that
as $l \to \infty$,
\[
    \PP
    \left(
        \left|
        R(\tf(\cdot; \cD_{\tr}^{\xi^\star_{n_{k_{l}}}} ))
        - R^\deter(\phi; \tf)
        \right|
        \ge \epsilon
    \right)
    \to 0
\quad
\text{whenever}
\quad
    \frac{p_n}{n_{\xi^\star_{n_{k_{l}}}}}
    \to \phi \in \argmin_{\zeta \in [\gamma, \infty]} R^\deter(\zeta; \tf).
\]
But since $R^\deter(\cdot; \tf)$ is continuous at $\phi$,
and $p_n / n_{\xi^\star_{n_{k_{l}}}} \to \phi \in \argmin_{\zeta \in [\gamma, \infty]} R^\deter(\zeta; \tf)$ 
as $l \to \infty$ from \eqref{eq:subsequence-limit-gaurantee}
this implies that, as $l \to \infty$,
\[
    \PP
    \left(
        \left|
        R(\tf(\cdot; \cD_{\tr}^{\xi^\star_{n_{k_{l}}}} ))
        - R^\deter(p_{n} / n_{\xi^\star_{n_{k_{l}}}} ; \tf)
        \right|
        \ge \epsilon
    \right)
    = h(n_{k_{l}})
    \to 0.
\]
This concludes the proof.

\subsection
{Proof of \Cref{prop:lower-semicontinuity-divergence}}

    In order to verify lower semicontinuity of $h$, 
    if suffices to show that for any $t \in \RR_{\ge 0}$,
    the set $\{ x : h(x) \le t \}$ is closed.
    Because $\lim_{x \to b^{-}} h(x) = \infty$
    and $h$ continuous on $[a, b)$,
    there exists $b_{-}(t) < b$
    such that $h(x) > t$ for all $x > b_{-}(t)$.
    Similarly, there exists $b_{+}(t) > b$ such that
    $h(x) > t$ for all $x < b_{+}(t)$.
    Note that
    \[
        \{ x : h(x) \le t \}
        =
        \{ x : h|_{[a, b_{-}(t)]}(x) \le t \}
        \cup
        \{ x : h|_{[b_{+}(t), c]}(x) \le t \}.
    \]
    Because $h$ is continuous on $[a, b_{-}(t)]$
    and $[b_{+}(t), c]$, it is also lower semicontinuous
    on these intervals, and hence the corresponding level sets are closed.
    Because the intersection of two closed sets is closed,
    the statement follows.

\subsection
{Proof of \Cref{prop:continuity-from-continuous-convergence-rdet}}

The proof builds on similar idea as that in the proof of
\Cref{lem:continuity-from-continuous-convergence-random-functions}
and employs a proof by contradiction.
However, since the random functions in this case (which are conditional prediction risks)
are not simply indexed by $n$ (but also by other properties of the data distributions), 
we will need to do a bit more work.

We wish to show that $R^\deter(\cdot; \tf)$ is continuous on $\cI \in (0, \infty)$.
We will first show that $R^\deter(\cdot; \tf)$ is $\QQ$-continuous
(see \Cref{def:q-continuity})
on $\cI$
and use \Cref{lem:rational-continuity-implies-real-continuity}
to lift $\QQ$-continuity to $\RR$-continuity.
Towards showing $\QQ$-continuity, for the sake of contradiction,
suppose $R^\deter(\cdot; \tf)$ is $\QQ$-discontinuous at some point $\phi_\infty \in \cI$.
This implies that there exists a sequence $\{ \phi_r \}_{r \ge 1}$
in $\QQ_{> 0}$
such that $\phi_r \to \phi_\infty$,
but for some $\epsilon > 0$
and all $r \ge 1$,
\begin{equation}
    \label{eq:qqdiscontinuity-Rdet-sequence-existence-implication}
    R^\deter(\phi_r; \tf)
    \notin [R^\deter(\phi_\infty; \tf) - 2 \epsilon, R^\deter(\phi_\infty; \tf) + 2 \epsilon].
\end{equation}
(Note that $R^\deter(\phi_r; \tf) \not\to  R^\deter(\phi_\infty; \tf)$
as $\phi_r \to \phi_\infty$.)
The proof strategy is now to construct a sequence of datasets 
$\{ \cD'_{k_m} \}_{m \ge 1}$
whose aspects ratios $p_m / k_m$
converge to $\phi_\infty$,
but the conditional prediction risks 
$R(\tf(\cdot; \cD'_{k_m}))$
of 
predictors $\tf(\cdot; \cD'_{k_m})$ 
trained on these datasets
do not converge to $R^\deter(\phi_\infty; \tf)$,
thereby supplying a contradiction to the hypothesis 
of continuous convergence of $R(\tf(\cdot; \cD'_{k_m}))$
to $R^\deter(\phi_\infty; \tf)$.
We will construct such a sequence of datasets below.

For every $r \ge 1$,
construct a sequence of datasets
$\{ \cD_{k_m}^{\phi_r} \}_{m \ge 1}$
with $k_m$ observations
and $p_m = \phi_i k_m$ features.
(Since $\phi_r \in \QQ_{> 0}$, the resulting $p_m$ is a positive integer.)
See \Cref{fig:continuity-from-continuous-convergence-zerostep} for a visual illustration.
For every $r \ge 1$,
from the assumption of \Cref{prop:continuity-from-continuous-convergence-rdet},
we have that
\begin{equation}
    \label{eq:contradict-proof-prob-converge}
    R(\tf(\cdot; \cD^{\phi_r}_{k_m}))
    \pto
    R^\deter(\phi_r; \tf)
\end{equation}
as $k_m, p_m \to \infty$
because
$p_m / k_m \to \phi_r$
as $m \to \infty$.
Now, fix $p \in (0, 1)$.
For $r = 1$,
the convergence in \eqref{eq:contradict-proof-prob-converge}
guarantees that
there exists an integer $m_1 \ge 1$ such that
the event
\begin{equation}
    \label{eq:phi1-prob-convergence-implication}
    \Omega_{m_1}
    := \{ |R(\tf(\cdot; \cD^{\phi_1}_{k_{m_1}})) - R^\deter(\phi_1; \tf)| \le \epsilon \}
\end{equation}
has probability at least $p$.
In addition, on the event $\Omega_{m_1}$,
by the triangle inequality
we have that
\begin{equation}
    \label{eq:triangle-RtoRdet-phi1}
    |R(\tf(\cdot; \cD^{\phi_1}_{k_{m_1}})) - R^\deter(\phi_\infty; \tf) |
    \ge
    |R^\deter(\phi_1; \tf) - R^\deter(\phi_\infty; \tf)|
    -
    |R(\tf(\cdot; \cD^{\phi_1}_{k_{m_1}})) - R^\deter(\phi_1; \tf)| 
    > \epsilon,
\end{equation}
where the second inequality follows by using
\eqref{eq:qqdiscontinuity-Rdet-sequence-existence-implication}
and
\eqref{eq:phi1-prob-convergence-implication}.
Next, for $r \ge 2$,
let $m_r > m_{r - 1}$
be an integer such that
the event
\begin{equation}
    \label{eq:phi2-prob-convergence-implication}
    \Omega_{m_r}
    := \{ |R(\tf(\cdot; \cD^{\phi_r}_{k_{m_r}})) - R^\deter(\phi_r; \tf)| \le \epsilon \}
\end{equation}
has probability at least $p$.
Such sequence of integers $\{ m_r \}_{r \ge 2}$ and the associated events $\{ \Omega_{m_r} \}_{r \ge 2}$
indeed exist as a consequence of the convergence in \eqref{eq:contradict-proof-prob-converge}
for $r \ge 2$.
On each $\Omega_{m_r}$
\[
    |R(\tf(\cdot; \cD^{\phi_r}_{k_{m_r}})) - R^\deter(\phi_\infty; \tf)| > \epsilon
\]
by similar reasoning as that for \eqref{eq:triangle-RtoRdet-phi1}
using 
\eqref{eq:qqdiscontinuity-Rdet-sequence-existence-implication}
and 
\eqref{eq:phi2-prob-convergence-implication}
for $r \ge 2$.
Moreover, note that since $m_r > m$, $m_r \to \infty$ 
as $r \to \infty$.

Consider now a sequence of datasets $\{ \cD'_{k_{m}}\}_{m \ge 1}$ such that:
\begin{enumerate}
    \item The first $m_1$ datasets 
    are $\{ \cD^{\phi_1}_{k_{m}} \}_{m = 1}^{m_1}$
    that have $k_{m}$ number of observations 
    and $p_{m} = \phi_1 k_m$ number of features
    for $m = 1, \dots, m_1$.
    \item The next $m_2 - m_1$ datasets
    are $\{ \cD^{\phi_2}_{k_{m}} \}_{m = m_1 + 1}^{m_2}$
    that have $k_{m}$ number of observations
    and $p_{m} = \phi_2 k_m$ number of features
    for $m = m_1 + 1, \dots, m_2$.
    \item The next $m_3 - m_2$ datasets
    are $\{ \cD_{k_m}^{\phi_3} \}_{m = m_2 + 1}^{m_3}$
    that have $k_{m}$ number of observations
    and $p_{m} = \phi_3 k_m$ number of features
    for $m = m_2 + 1, \dots, m_3$.
    \item And so on ...
\end{enumerate}
We will argue now that the sequence of datasets 
$\{ \cD'_{k_m} \}_{m \ge 1}$ works for our promised contradiction. 
Observe that in the construction above
the aspect ratios $p_{m} / k_{m} \to \phi_\infty$
because $\phi_r \to \phi_\infty$.
However,
we have that
for all $r \ge 1$,
\[
   \PP( | R(\tf(\cdot; \cD'_{k_{m_r}})) - R^\deter(\phi_\infty; \tf) | > \epsilon)
   = 
   \PP( | R(\tf(\cdot; \cD_{k_{m_r}})) - R^\deter(\phi_\infty; \tf) | > \epsilon)
   \ge p.
\]
Therefore,
there exists an $\epsilon > 0$
for which there is no $M \ge 1$
such that for $m \ge M$,
\[
    \PP
    (
        |
            R(\tf(\cdot; \cD'_{k_m}))
            - R^\deter(\phi_\infty; \tf)
        |
        > \epsilon
    )
    < p / 2.
\]
Hence, we get the desired contraction that
\[
    R(\tf(\cdot; \cD_{k_m}')) \not\pto R^\deter(\phi_\infty, \tf)
\]
as $k_m, p_m \to \infty$ and $p_m / k_m \to \phi_\infty$.
This completes the proof.

It is worth pointing out
that the proof above bears similarity
to the proof of \Cref{lem:rational-continuity-implies-real-continuity}.
It is possible to combine the two and not have to go
through the route of $\QQ$-continuity.
We, however, find it easier to break them
so that the main ideas are easier to digest
even though it leads to some repetition
of overall proof strategies.

\begin{figure}[!ht]
    \centering
    \includegraphics[width=\columnwidth]{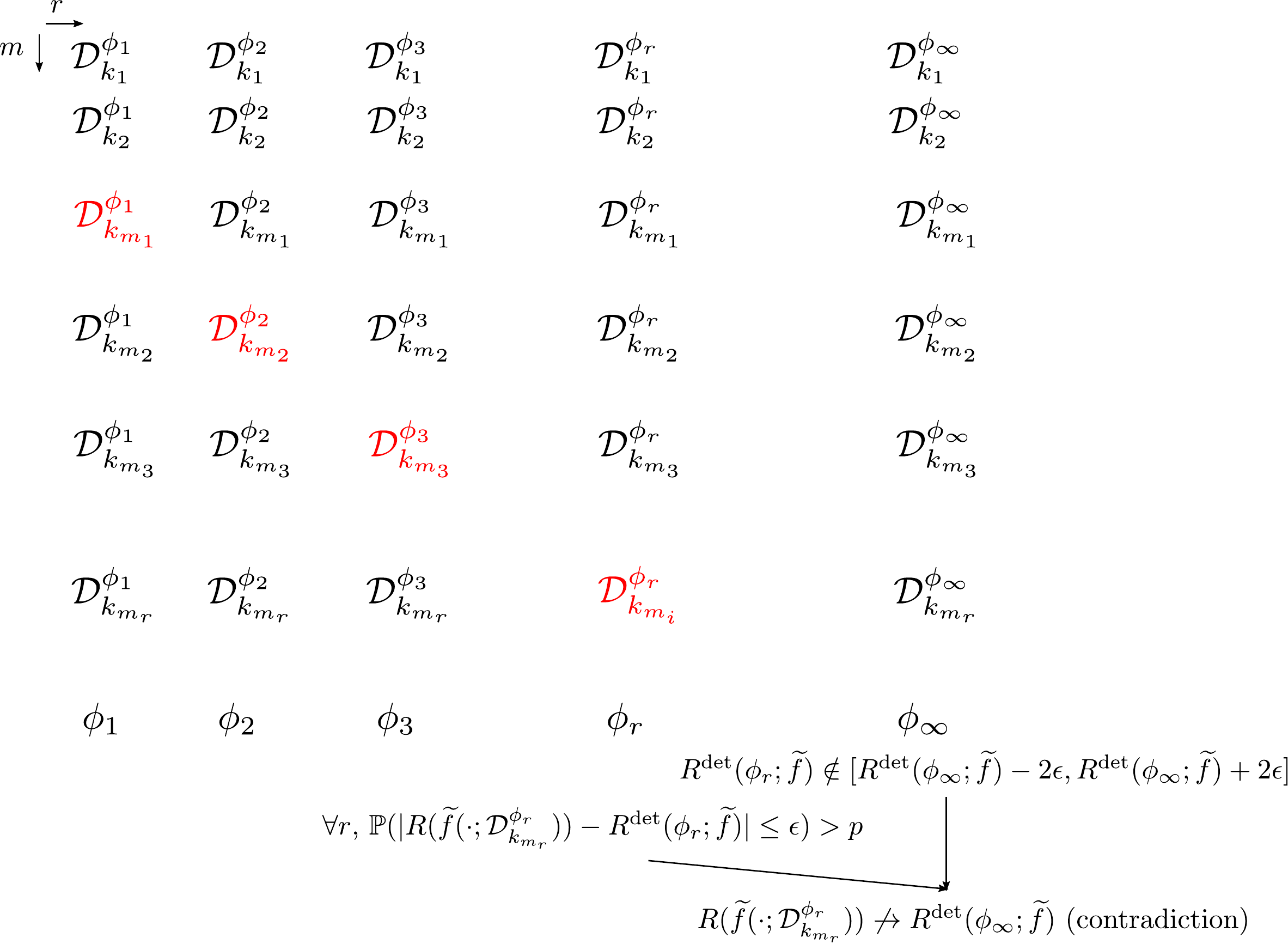}
    \caption{Illustration of construction of grid of datasets used in the proof of \Cref{prop:continuity-from-continuous-convergence-rdet}.
    (Side note: as can be seen from the figure, the argument bears
    similarity to the standard diagonalization argument.)}
    \label{fig:continuity-from-continuous-convergence-zerostep}
\end{figure}

\subsection
{Proof of \Cref{thm:asymptotic-risk-tuned-zero-step}}

We will split the proof depending on the value of $M$.

    \paragraph{Case of $M = 1$.}
    Consider first the case when $M = 1$.
    In this case,
    for every $\xi \in \Xi$,
    $\hf^\xi = \tf_1^\xi$
    (and thus, $\tf^\star = \hf^\cv$),
    which we denote by $\tf^\xi$ for simplicity of notation.
    To bound the desired difference,
    we break it into three terms:
    \begin{equation}\label{eq:fhat-star-zero-step-m=1}
    \begin{split}
            \left(R(\hf^\cv)
            - \min_{\zeta \ge p/n} R^\deter(\tf; \zeta)\right)_+ 
        &=
            \left(
            R(\hf^\cv)
            - \min_{\xi \in \Xi} R(\tf^\xi)
    \right)_+
    \\        &\quad +
            \left(
            \min_{\xi \in \Xi} R(\tf^\xi)
            - \min_{\xi \in \Xi}
            R^\deter
            \left(
                \tf; \frac{p_n}{n_\xi}
            \right)
            \right)_+
      \\&\quad  +
        \left(
            \min_{\xi \in \Xi}
            R^\deter
            \left(
                \tf; \frac{p_n}{n_\xi}
            \right)
            - \min_{\zeta \ge p/n} R^\deter(\tf; \zeta)
            \right)_+.
        \end{split}
    \end{equation}
    This inequality follows from the fact that $(a + b + c)_+ \le (a)_+ + (b)_+ + (c)_+$ for any $a, b, c\in\mathbb{R}$.
    We show below that each of the three terms
    asymptotically vanish in probability as
    $n \to \infty$ with $p/n \le \Gamma$.\\

    \underline{Term 1:}
        Because $| \Xi | \le n^{1-\nu} \le n$,
        and $\hsigma_\Xi = o_p(\sqrt{n^\nu/\log(n)})$,
        following
        \Cref{rem:growth-rates-probabilistic-bound},
        under the assumptions of
        \Cref{lem:bounded-orlitz-error-control}
        or \Cref{lem:bounded-variance-error-control},
        we have
        \begin{equation}\label{eq:first-term-zero-step-m=1}
            \left|
                R(\tf^\cv)
                - \min_{\xi \in \Xi} R(\tf^\xi)
            \right|
            = o_p(1),
        \end{equation}
        which proves that the first term on the right hand side of~\eqref{eq:fhat-star-zero-step-m=1} converges to zero in probability.
        
        \underline{Term 2:}
        To deal with the second term on the right hand side of~\eqref{eq:fhat-star-zero-step-m=1}, define
        \[
        \xi_n^\star ~\in~ \argmin_{\xi\in\Xi}\,R^\deter\left(\widetilde{f};\,\frac{p_n}{n_\xi}\right).
        \]
        Because $R^\deter(\cdot; \cdot)$ is a non-stochastic function, $\{\xi_n^\star\}_{n\ge1}$ is a non-stochastic sequence and further, trivially, $\xi_i^\star\in\Xi$ for all $n\ge1$. Observe now that
        \begin{equation}
            \begin{split}
                \min_{\xi\in\Xi} R(\widetilde{f}^{\xi}) &\le R(\widetilde{f}^{\xi_n^\star})\\ 
                &= R(\widetilde{f}^{\xi_n^\star}) - R^\deter\left(\widetilde{f};\,\frac{p_n}{n_{\xi_n^\star}}\right) + \min_{\xi\in\Xi}\,R^\deter\left(\widetilde{f};\,\frac{p_n}{n_\xi}\right).
            \end{split}
        \end{equation}
        Hence, assumption~\eqref{eq:rn-deterministic-approximation-2-prop-asymptotics}
        implies that
        \begin{equation}\label{eq:second-term-zero-step-m=1}
        \left(\min_{\xi\in\Xi} R(\widetilde{f}^{\xi}) - \min_{\xi\in\Xi}\,R^\deter\left(\widetilde{f};\,\frac{p_n}{n_\xi}\right)\right)_+ = o_p(1),
        \end{equation}
        as $n\to\infty$.
        
        \underline{Term 3:}
        Finally,
        because the risk profile $\zeta \mapsto R^\deter(\tf; \zeta)$
        is assumed to be continuous at $\zeta^\star$,
        \Cref{lem:spacefilling-grid-zerostep-general}
        with the grid $\Xi$ yields
        \begin{equation}\label{eq:third-term-zero-step-m=1}
            \left|
                \min_{\xi \in \Xi}
                R^\deter\left(\tf; \frac{p_n}{n_\xi}\right)
                - \inf_{\zeta \ge \gamma} R^\deter(\tf; \zeta)
            \right|
            = o(1).
        \end{equation}
    Combining~\eqref{eq:first-term-zero-step-m=1},~\eqref{eq:second-term-zero-step-m=1}, and~\eqref{eq:third-term-zero-step-m=1},
    we have the desired result that
    \[
        \left|
            R(\hf^\cv)
            - \min_{\zeta \ge \gamma} R^\deter(\tf; \zeta)
        \right|
        \pto 0.
    \]
    
    \paragraph{Case of $M > 1$.}
    Consider now the case when $M > 1$. 
    Note that $(x + y)_+ \le (x)_+ + (y)_+$
    since $\max\{z, 0\}$ is a convex function of $z$.
    Thus, we can break and bound the desired difference as:
    \begin{align*}
        &\left(
            R(\hf^\cv)
            - \min_{\zeta \ge p/n} R^\deter(\tf; \zeta)
        \right)_+ \\
        &\le
        \left(
            R(\hf^\cv)
            - \min_{\xi \in \Xi} R(\hf^\xi)
        \right)_+
        +
        \left(
            \min_{\xi \in \Xi} R(\hf^\xi)
            - \min_{\xi \in \Xi}\frac{1}{M} \sum_{j=1}^{M}  R(\tf_j^\xi)
        \right)_+ \\
        &\quad +
        \left(
            \min_{\xi \in \Xi}\frac{1}{M} \sum_{j=1}^{M}
             R(\tf_j^\xi)
            - \min_{\xi \in \Xi}
            R^\deter
            \left(\tf^\xi; \frac{p_n}{n_\xi}\right)
        \right)_+ \\
        &\quad +
        \left(
            \min_{\xi \in \Xi}
            R^\deter
            \left(\tf; \frac{p_n}{n_\xi}\right)
            - \min_{\zeta \ge \gamma} R^\deter(\tf; \zeta)
        \right)_+.
    \end{align*}
    As before,
    we show below that each of these terms
    are asymptotically vanishing in probability.
    
    \underline{Term 1:}
        Note that $\hsigma_\Xi \le \tsigma_\Xi$
        (from the triangle inequality for $L_2$ and $\psi_1$ norms).
        Thus, as argued above for the case of $m = 1$,
        the first term is $o_p(1)$.
        
    \underline{Term 2:}
        For the second term, observe that, for all $\xi \in \Xi$,
        \begin{align*}
           R\left(\hf^\xi\right)
           = R
           \left(
               \frac{1}{M} \sum_{j=1}^{M}
               \tf_j^\xi
           \right) 
           &= \EE
           \left[
                \ell\left(Y_0, \frac{1}{M} \sum_{i = 1}^{M} \tf_j^\xi(X_0)\right)
                \mathrel{\Big |} \cD_1
           \right] \\
           &\le \frac{1}{M} \sum_{j=1}^{M}
           \EE
           \left[
                \ell(Y_0, \tf_j^\xi(X_0))
                \mathrel{\big |} \cD_1
           \right] \\
           &\le \frac{1}{M} \sum_{j=1}^{M} R(\tf_j^\xi).
        \end{align*}
        Therefore,
        we have
        \[
            \min_{\xi \in \Xi} R(\hf^\xi)
            \le \min_{\xi\in\Xi}\frac{1}{M} \sum_{j=1}^{M} R(\tf_j^\xi)
        \]
        and the second term is 0. 
        
        \underline{Term 3:}
        For the third term, as before, note that
        \[
        \left(
                \min_{\xi \in \Xi}\frac{1}{M} \sum_{j=1}^{M}
                 R(\tf_j^\xi)
                - \min_{\xi \in \Xi}
                R^\deter
                \left(\tf^\xi; \frac{p}{n_\xi}\right)
            \right)_+ \le \left(
                \frac{1}{M} \sum_{j=1}^{M}
                 R(\tf_j^{\xi_n^\star})
                - R^\deter
                \left(\tf; \frac{p_n}{n_{\xi_n^\star}}\right)
            \right)_+,
        \]
        with the right hand side being $o_p(1)$
        because of~\eqref{eq:rn-deterministic-approximation-2-prop-asymptotics}.
        
        \underline{Term 4:}
        Analogous to the argument for the $m = 1$ case,
        the fourth term is $o(1)$.
    
    Combined together, we have the final result.
    This completes the proof.
    For an overview, a schematic for the proof of \Cref{thm:asymptotic-risk-tuned-zero-step}
    is provided in \Cref{fig:monotonization-theorem-illustration}.
    
\begin{figure}[!ht]
    \centering
    \includegraphics[width=0.5\columnwidth]{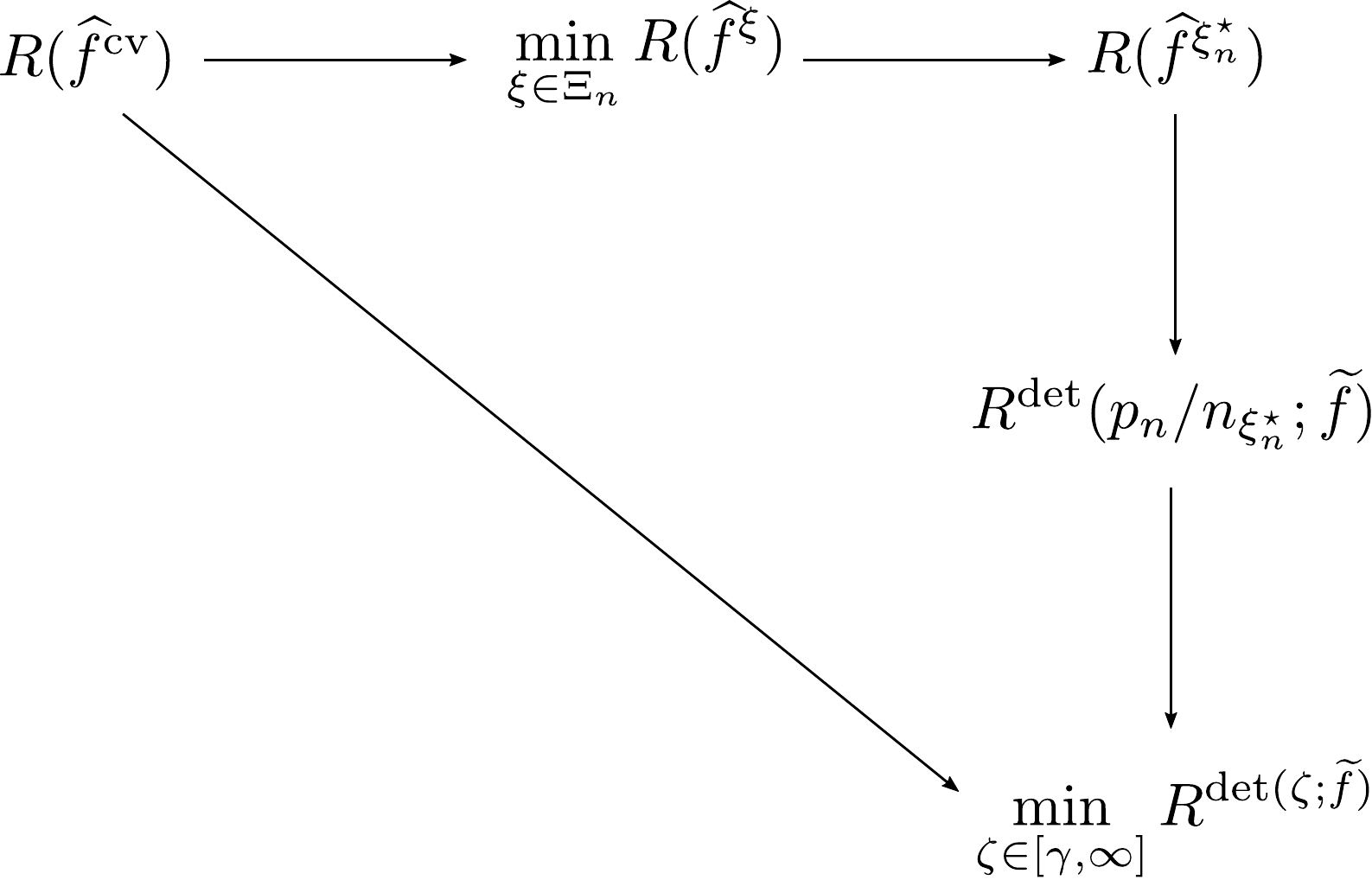}
    \caption{Schematic of the proof of \Cref{thm:asymptotic-risk-tuned-zero-step}.}
    \label{fig:monotonization-theorem-illustration}
\end{figure}

\section{Proofs related to deterministic profile verification for zero-step procedure}
\label{sec:verif-asymp-profile-ridge}

In this section, we verify the assumption 
\eqref{tag:detpar-0}
for the MN2LS and MN1LS prediction procedures.

\subsection
{Proof of \Cref{prop:asymp-bound-ridge-main}}
\label{sec:verify-profile-l2}

Recall $\cD_{k_{m}}$ is a dataset
with $k_m$ observations and $p_m$ features.
Theorem 3 of \cite{hastie_montanari_rosset_tibshirani_2019}
assumes the following 
distributional assumptions on the dataset $\cD_{k_m}$.

\begin{enumerate}[label={\rm($\ell_2$A\arabic*)}]
    \item
    \label{asm:lin-mod}
    The observations $(X_i, Y_i)$, $1 \le i \le k_m$,
    are sampled i.i.d.\ from the model
    $Y_i = X_i^\top \beta_0 + \eps_i$
    for some (deterministic) unknown signal vector $\beta_0 \in \RR^{p_m}$
    and (random) unobserved error $\eps_i$,
    assumed to be independent of $X_i \in \RR^{p_m}$, 
    with 
    mean $0$,
    variance $\sigma^2$,
    and bounded moment of order $4 + \delta$
    for some $\delta > 0$.
    \item
    \label{asm:rmt-feat}
    The feature vector $X_i$, $1 \le i \le k_m$, 
    decomposes as $X_i = \Sigma^{1/2} Z_i$,
    where $\Sigma \in \RR^{p_m \times p_m}$ is a positive semidefinite (covariance) matrix
    and $Z_i \in \RR^{p_m \times 1}$ is a random vector
    containing i.i.d.\ entries with
    mean $0$, variance $1$,
    and bounded moment of order $4 + \delta$
    for some $\delta > 0$.
    \item
    \label{asm:signal-bounded-norm}
    The norm of the signal vector 
    $\| \beta_0 \|_2$
    is uniformly bounded in $p$,
    and $\lim_{p_m \to \infty} \| \beta_0 \|_2^2 = \rho^2 < \infty$.
    \item
    \label{asm:covariance-bounded-eigvals}
    There exist real numbers $r_{\min}$ and $r_{\max}$
    with $0 < r_{\min} \le r_{\max} < \infty$
    such that
    $r_{\min} I_{p_m}~\preceq~\Sigma~\preceq~r_{\max} I_{p_m}$.
    \item
    \label{asm:spectrum-spectrumsignproj-conv}
    Let $\Sigma = W R W^\top$
    denote the eigenvalue decomposition of the covariance matrix $\Sigma$,
    where 
    $R \in \RR^{p_m \times p_m}$ is a diagonal matrix
    containing eigenvalues (in non-increasing order) 
    $r_1 \ge r_2 \ge \dots \ge r_{p_m} \ge 0$,
    and
    $W~\in~\RR^{p_m \times p_m}$ is an orthonormal matrix
    containing the associated eigenvectors 
    $w_1, w_2, \dots, w_{p_m}~\in~\RR^{p_m}$.
    Let $H_{p_m}$ denote the empirical spectral distribution
    of $\Sigma$
    (supposed on $\RR_{> 0}$)
    whose value at any $r \in \RR$ is given by
    \[
        H_{p_m}(r)
        = \frac{1}{p_m} \sum_{i=1}^{p_m} \1_{\{r_i \le r\}}.
    \]
    Let $G_{p_m}$ denote a certain distribution (supported on $\RR_{> 0}$)
    that encodes the components of the signal vector $\beta_0$ in the eigenbasis of $\Sigma$
    via the distribution of (squared) projection of $\beta_0$
    along the eigenvectors $w_j, 1 \le j \le p_m$,
    whose value any $r \in \RR$ is given by
    \[
        G_{p_m}(r)
        = \frac{1}{\| \beta_0 \|_2^2} \sum_{i = 1}^{p_m} (\beta_0^\top w_i)^2 \, \1_{\{ r_i \le r \}}.
    \]
    Assume there exist fixed distributions $H$ and $G$
    (supported on $\RR_{> 0}$)
    such that $H_{p_m} \overset{d}{\to} H$
    and $G_{p_m} \overset{d}{\to} G$
    as $p_m  \to \infty$.
\end{enumerate}

Under assumptions \ref{asm:lin-mod}--\ref{asm:spectrum-spectrumsignproj-conv},
we will verify that,
for the MN2LS base prediction procedure $\tf_\mnls$,
there exists a deterministic risk approximation 
$R^\deter(\cdot; \tf_{\mnls}) : (0, \infty] \to [0, \infty]$
that satisfy the two conditions stated in \Cref{prop:asymp-bound-ridge-main}.
In particular, we will show that the function 
$R^\deter(\cdot; \tf_{\mnls})$ defined below
satisfies the required conditions:
\begin{equation}
    \label{eq:detapprox-formula-zerostep-mn2ls}
    R^\deter(\phi; \tf_{\mnls})
    =
    \begin{dcases}
        \sigma^2 \frac{1}{1 - \phi} 
        & \text{ if } \phi \in (0, 1) \\
        \infty
        & \text{ if } \phi = 1 \\
        \rho^2 (1 + \tv_g(0; \phi)) 
        \int \frac{r}{(1 + v(0; \phi) r)^2} \, \mathrm{d}G(r)  \\
        \quad
        +~ \sigma^2
       \left( 
            \phi
            \tv(0; \phi)
            \int \frac{r^2}{(1 + v(0; \phi) r)^2}
            \, \mathrm{d} H(r)
            +
            1
       \right)
        & \text{ if } \phi = (1, \infty) \\
        \rho^2 \int r \, \mathrm{d}G(r) + \sigma^2
        & \text{ if } \phi = \infty,
    \end{dcases}
\end{equation}
where the scalars $v(0; \phi)$, $\tv(0; \phi)$, and $\tv_g(0; \phi)$,
for $\phi \in (1, \infty)$, are defined as follows:
\begin{itemize}
    \item 
    $v(0; \phi)$
    is the unique solution to
    the fixed-point equation:
    \begin{equation}
        \label{eq:fixed-point-v-mn2ls-1}
        \frac{1}{\phi}
       = \int \frac{v(0; \phi) r}{1 + v(0; \phi) r} \, \mathrm{d}H(r),
    \end{equation}
    \item
    $\tv(0; \phi)$ is defined
    through $v(0; \phi)$ by the equation:
    \begin{equation}
        \label{eq:fixed-point-v'-mn2ls-1}
        \tv(0; \phi)
        = 
        \left(
        \frac{1}{v(0; \phi)^2}
        - \phi \int \frac{r^2}{(1 + v(0; \phi) r)^2} \, \mathrm{d}H(r)
        \right)^{-1},
    \end{equation}
    \item
    $\tv_g(0; \phi)$ is defined through $v(0; \phi)$ and $\tv(0; \phi)$
    by the equation:
    \begin{equation}
        \label{eq:def-tvg-mn2ls-1}
        \tv_g(0; \phi)
        =
        \tv(0; \phi)
            \phi
            \int \frac{r^2}{(1 + v(0; \phi) r)^2} \, \mathrm{d}H(r).
    \end{equation}
\end{itemize}

We will verify the two conditions of \Cref{prop:asymp-bound-ridge-main} below.

The limiting risk for the MN2LS predictor
provided in \eqref{eq:detapprox-formula-zerostep-mn2ls},
although in a different notation,
matches the one obtained in Theorem 3 of \cite{hastie_montanari_rosset_tibshirani_2019}.
We believe our notation makes the subsequent analysis
for the one-step procedure easy to follow for the reader.
It is worth mentioning, however, that
\cite{hastie_montanari_rosset_tibshirani_2019}
only explicitly consider $\phi \in (0, 1) \cup (1, \infty)$.
We extend the analysis to show that
the risk continuously diverges to $\infty$
as $\phi \to 1$
and also continuously converges to the null risk
as $\phi \to \infty$.
In addition,
as mentioned in 
\Cref{rem:prop_asymptotics_risk_examples},
we analyze the prediction risk
conditioned on both $(\bX, \bY)$
as opposed to only on
$\bX$
as done in \cite{hastie_montanari_rosset_tibshirani_2019}.
Furthermore,
we also establish continuity properties
of the 
deterministic risk approximation
in the aspect ratio
that is needed for our analysis.

\subsubsection*{\underline{Condition 1}: Continuous convergence of conditional risk 
over $\phi \in (0, 1) \cup (1, \infty]$.}

Let $\bX \in \RR^{k_m \times p_m}$ denote the design matrix
and $\bY \in \RR^{k_m}$ denote the response vector
associated with the dataset $\cD_{k_{m}}$.
Let $\beps \in \RR^{k_m}$
denote the error vector
containing errors $\eps_i$, $1 \le i \le k_m$.
Write
the data model
from assumption \ref{asm:lin-mod}
as
$\bY = \bX^\top \beta_0 + \beps$,
and the MN2LS estimator \eqref{eq:mn2ls} as
\begin{equation}
    \label{eq:mn2ls-matrix-form}
    \tbeta_{\mnls}(\cD_{k_{m}})
    = (\bX^\top \bX / k_m)^{\dagger} \bX^\top \bY / k_m.
\end{equation} 
The associated predictor $\tf_{\mnls}(\cdot; \cD_{k_m})$
is given by \eqref{eq:mn2ls-predictor}.
Recall the prediction risk $R_{\bX, \bY}(\tf_{\mnls}(\cdot; \cD_{k_m}))$
(where we use the subscripts $\bX, \bY$
to explicitly indicate the dependence
of $R(\tf_{\mnls}(\cdot; \cD_{k_m}))$ 
on the training data $(\bX, \bY)$)
under the squared error loss is given by
\begin{equation}
    \label{eq:mn2ls-squaredrisk}
    R_{\bX, \bY}(\tf_{\mnls}(\cdot; \cD_{k_m}))
    = \EE[(Y_0 - \tf_{\mnls}(X_0; \cD_{k_{m}}))^2 \mid \bX, \bY],
\end{equation}
where $(X_0, Y_0)$ is sampled independently
from the same distribution as the training data $(\bX, \bY)$.

Our goal is to show that as $k_m, p_m \to \infty$,
if $p_m / k_m \to \phi \in (0, 1) \cup (1, \infty]$,
$R_{\bX, \bY}(\tf_{\mnls}(\cdot; \cD_{k_m}))
\asto R^\deter(\phi; \tf_{\mnls})$.
The proof follows by combining
\Cref{prop:cond-conv-mn2ls,prop:limits-risk-functionals-mn2ls,prop:limits-infty-mn2ls}.
Specifically:
\begin{enumerate}
    \item 
    \Cref{prop:cond-conv-mn2ls,prop:limits-risk-functionals-mn2ls}
    combined together imply that
    $R_{\bX, \bY}(\tf_{\mnls}(\cdot; \cD_{k_m})) \asto R^\deter(\phi; \tf_{\mnls})$
    as $p_m, k_m \to \infty$ and $p_m / k_m \to \phi \in (0, 1) \cup (1, \infty)$.
    \item
    \Cref{prop:limits-infty-mn2ls} imply that
    $R_{\bX, \bY}(\tf_{\mnls}(\cdot; \cD_{k_{m}})) \asto R^\deter(\infty; \tf_{\mnls})$
    as $p_m, k_m \to \infty$ and $p_m / k_m \to \infty$.
\end{enumerate}
Below we prove \Crefrange{prop:cond-conv-mn2ls}{prop:limits-infty-mn2ls}.

In preparation for the statements to follow,
denote by $\hSigma := \bX^\top \bX / k_m$ the sample covariance matrix.
Let the singular value decomposition of
$\bX / \sqrt{k_m}$
be $\bX / \sqrt{k_m} = \bU \bS \bV^\top$,
where $\bU \in \RR^{k_m \times k_m}$
and $\bV \in \RR^{p_m \times p_m}$ are orthonormal matrices,
and $\bS \in \RR^{k_m \times p}$
is a diagonal matrix containing singular values
in non-increasing order $s_1 \ge s_2 \ge \dots$.

The proposition below provides conditional convergence
for the prediction risk \eqref{eq:mn2ls-squaredrisk}
when $p_m / k_m \to \phi \in (0, 1) \cup (1, \infty)$
as $p_m, k_m \to \infty$.

\begin{proposition}
    [Conditional convergence of squared prediction risk of MN2LS predictor]
    \label{prop:cond-conv-mn2ls}
    Suppose assumptions \ref{asm:lin-mod}--\ref{asm:covariance-bounded-eigvals} hold.
    Then,
    as $k_m, p_m \to \infty$,
    if $p_m/k_m \to \phi \in (0, 1) \cup (1, \infty)$,
    then
    \begin{equation}
        \label{eq:mn2ls-bias-var-functionals}
        R_{\bX, \bY}(\tf_{\mnls}(\cdot; \cD_{k_{m}}))
        - \beta_0^\top (I_{p_m} - \hSigma^\dagger \hSigma) \Sigma (I_{p_m} - \hSigma^\dagger \hSigma) \beta_0
        - \sigma^2 \tr[\hSigma^{\dagger} \Sigma] / k_{m}
        - \sigma^2
        \asto 0.
    \end{equation}
\end{proposition}
\begin{proof}
    Under assumption \ref{asm:lin-mod},
    the squared prediction risk \eqref{eq:mn2ls-squaredrisk}
    decomposes into
    \begin{equation}
        \label{eq:predrisk_mn2ls}
        R_{\bX, \bY}(\tf_\mnls(\cdot; \cD_{k_{m}}))
        = (\tbeta_\mnls(\cD_{k_{m}}) - \beta_0)^\top \Sigma (\tbeta_\mnls(\cD_{k_{m}}) - \beta_0) + \sigma^2.
    \end{equation}
    Similarly, under assumption \ref{asm:lin-mod}, 
    the estimator \eqref{eq:mn2ls-matrix-form} decomposes into
    \begin{align*}
        \tbeta_\mnls(\cD_{k_{m}})
        = (\bX^\top \bX / k_m)^{\dagger} \bX^\top \bX / k_m \, \beta_0
        + (\bX^\top \bX / k_m)^{\dagger} \bX^\top \beps / k_m.
    \end{align*}
    Consequently, 
    the difference between the estimator and the true parameter decomposes as
    \begin{align}
        \label{eq:estim-diff-mnls-cond-conv}
        \tbeta_\mnls(\cD_{k_{m}}) - \beta_0
        = \big\{ (\bX^\top \bX / k_m)^{\dagger} \bX^\top \bX / k_m - I_{p_m} \big\} \beta_0
        + (\bX^\top \bX / k_m)^{\dagger} \bX^\top \beps / k_m.
    \end{align}
    Substituting \eqref{eq:estim-diff-mnls-cond-conv} into \eqref{eq:predrisk_mn2ls},
    we can split the first term on the right hand side of \eqref{eq:predrisk_mn2ls} into
    three component terms:
    \begin{align*}
        (\tbeta_\mnls(\cD_{k_{m}}) - \beta_0)^\top \Sigma (\tbeta_\mnls(\cD_{k_{m}}) - \beta_0)
        &= \bB_0 + \bV_0 + \bC_0,
    \end{align*}
    where the component terms are given by:
    \begin{align*}
       \bB_{0}
       &=
       \beta_0^\top
       \big\{
       (\bX^\top \bX / k_{m})^{\dagger} \bX^\top \bX / k_{m} - I_{p_m}
       \big\}
       \Sigma
       \big\{
       (\bX^\top \bX / k_{m})^{\dagger} \bX^\top \bX / k_{m} - I_{p_m}
       \big\}
       \beta_0 
       = 
       \beta_0^\top 
       (I_{p_m} - \hSigma^\dagger \hSigma) 
       \Sigma 
       (I_{p_m} - \hSigma^\dagger \hSigma) 
       \beta_0,
      \\
       \bC_{0}
       &=
       \beta_0^\top
       \big\{
       (\bX^\top \bX / k_{m})^{\dagger} \bX^\top \bX / k_{m} - I_{p_m}
       \big\}
       \Sigma
       (\bX^\top \bX / k_{m})^{\dagger}
       \bX^\top \beps / k_{m}  
       = 
       - 
       \beta_0^\top (I_{p_m} - \hSigma^{\dagger} \hSigma) 
       \Sigma 
       \hSigma^{\dagger} \bX^\top \beps / k_{m},
       \\
       \bV_{0}
       &=
       \beps^\top \bX / k_{m}
       (\bX^\top \bX / k_{m})^{\dagger}
       \Sigma
       (\bX^\top \bX / k_{m})^{\dagger}
       \bX^\top \beps / k_{m}
       = \beps^\top (\bX \hSigma^\dagger  \Sigma \hSigma^\dagger \bX^\top / k_{m}) \beps / k_{m}.
    \end{align*}
    To finish the proof,
    we will show concentration of the terms $\bC_0$ and $\bV_0$ below.

    \underline{Term $\bC_0$}:
    We will show that
    $\bC_0 \asto 0$
    as $k_{m}, p_{m}  \to \infty$ such that $p_{m} / k_{m} \to \phi \in (0, 1) \cup (1, \infty)$.
    Note that
    \begin{align}
        \| \bX \hSigma^{\dagger} \Sigma (I_{p_m} - \hSigma^{\dagger} \hSigma) \beta_0 \|_2^2 / k_{m}
        &= 
        \beta_0^\top 
        (I_{p_m} - \hSigma^{\dagger} \hSigma)
        \Sigma
        \hSigma^{\dagger}
        \bX^\top
        \bX
        \hSigma^{\dagger}
        \Sigma
        (I_{p_m} - \hSigma^{\dagger} \hSigma)
        \beta_0 / k_m \nonumber \\
        &\le \| \beta_0 \|_2^2 \| \| (I_{p_m} - \hSigma^{\dagger} \hSigma) \Sigma
        \hSigma^{\dagger} \hSigma \hSigma^{\dagger} \Sigma (I_{p_m} - \hSigma^{\dagger} \hSigma) \|_{\mathrm{op}} \nonumber \\
        &\le 
        \| \beta_0 \|_2^2 \|  
        \cdot
        r_{\max}^2 
        \cdot
        \| \hSigma^{\dagger} \|_{\mathrm{op}},
        \label{eq:bound-ell2normsquare-crossterm-zerostep}
    \end{align}
    where in the last inequality \eqref{eq:bound-ell2normsquare-crossterm-zerostep},
    we used the fact that $\| I_{p_m} - \hSigma^{\dagger} \hSigma \|_{\mathrm{op}} \le 1$,
    $\| \Sigma \|_{\mathrm{op}} \le r_{\max}$,
    and that $\hSigma^{\dagger} \hSigma \hSigma^{\dagger} = \hSigma^{\dagger}$, along with the submultiplicativity of the operator norm.
    Now, note that $\liminf \min_{1 \le i \le p} s_i^2 \ge r_{\min} (1- \sqrt{\phi})^{2}$
    almost surely from \cite{bai_silverstein_2010}
    for $\phi \in (0, 1) \cup (1, \infty)$.
    Therefore, $\limsup \| \hSigma^{\dagger} \|_{\mathrm{op}} \le C$ 
    for some constant $C < \infty$ almost surely.
    Applying \Cref{lem:concen-linform}, we thus have that $\bC_0 \asto 0$.
    
    \underline{Term $\bV_0$}:
    We will show that $\bV_0 - \tr[\hSigma^{+} \Sigma] / k_m \asto 0$
    as $k_m, p_m \to \infty$ such that $p_m / k_m \to \phi \in (0, 1) \cup (1, \infty)$.
    Observe that
    \begin{equation}
        \label{eq:bound-operatornorm-varianceterm-zerostep}
        \| \bX \hSigma^{\dagger} \Sigma \hSigma^{\dagger} \bX^\top / k_{m} \|_{\mathrm{op}}
        \le r_{\max} \| \hSigma \|_{\mathrm{op}} \| \hSigma^{\dagger} \|_{\mathrm{op}}^2.
    \end{equation}
    Now, note that
    $
        \limsup \| \hSigma \|_{\mathrm{op}} 
        \le \limsup \max_{1 \le i \le p} s_i^2
        \le r_{\max} (1 + \sqrt{\phi})^2,
    $
    almost surely for $\phi \in (0, 1) \cup (1, \infty)$
    from \cite{bai_silverstein_2010}.
    In addition, as argued above,
    $\| \hSigma^{\dagger} \|_{\mathrm{op}} \le C$ almost surely
    for some constant $C < \infty$.
    Thus, using \Cref{lem:concen-quadform},
    it follows that
    $
        \bV_0
        -
        \sigma^2
        \tr[\bX \hSigma^+ \Sigma \hSigma^+ \bX^\top] / k_{m}
        \asto 0.
    $
    Finally, since 
    $\tr[\bX \hSigma^{+} \Sigma \hSigma^{+} \bX^\top] / k_m
    = \tr[\hSigma^\dagger \hSigma \hSigma^\dagger \Sigma] / k_m
    = \tr[\hSigma^\dagger \Sigma] / k_m$,
    we obtain that
    $\bV_0 - \sigma^2 \tr[\hSigma^\dagger \Sigma] / k_m \asto 0$.
   
\end{proof}

The next proposition provides deterministic limits of the
conditional risk functionals in \Cref{prop:cond-conv-mn2ls}
when $p_m / k_m \to \phi \in (0, 1) \cup (1, \infty)$
as $k_m, p_m \to \infty$.

\begin{proposition}
    [Limits of conditional risk functionals over $\phi \in (0, 1) \cup (1, \infty)$]
    \label{prop:limits-risk-functionals-mn2ls}
    Suppose assumptions \ref{asm:rmt-feat}--\ref{asm:spectrum-spectrumsignproj-conv} hold.
    Then,
    as $k_m, p_m \to \infty$, 
    and $p_m / k_m \to \phi \in (0, 1) \cup (1, \infty)$,
    the following holds:
    \begin{itemize}
        \item Bias functional:
        \[
            \beta_0^\top
            (I_{p_m} - \hSigma^{\dagger} \hSigma)
            \Sigma
            (I_{p_m} - \hSigma^{\dagger} \hSigma)
            \beta_0
            \asto
            \begin{dcases}
                0
                & \text{ if } \phi \in (0, 1) \\ 
                \rho^2
                (1 + \tv_g(0; \phi))
                \int \frac{r}{(1 + v(0; \phi) r)^2} \, \mathrm{d}G(r)
                & \text{ if } \phi \in (1, \infty),
            \end{dcases}
        \]
        \item Variance functional:
        \[
            \sigma^2
            \tr[\hSigma^{\dagger} \Sigma] / k_{m}
            \asto
            \begin{dcases}
                \sigma^2 \frac{\phi}{1 - \phi} & \text{ if } \phi \in (0, 1) \\
                \sigma^2
                \phi
                \tv(0; \phi)
                \int \frac{r^2}{(1 + v(0; \phi) r)^2} \, \mathrm{d}H(r)
                & \text{ if } \phi \in (1, \infty),
            \end{dcases}
        \]
    \end{itemize}
    where $v(0; \phi)$, $\tv(0; \phi)$, and $\tv_g(0; \phi)$
    are as defined in \eqref{eq:fixed-point-v-mn2ls-1},
    \eqref{eq:fixed-point-v'-mn2ls-1},
    and \eqref{eq:def-tvg-mn2ls-1},
    respectively.
\end{proposition}
\begin{proof}
    We will consider the bias and functionals separately below.
    
    \paragraph{Bias functional.}
    Consider first the bias functional
    $\beta_0^\top (I_{p_m} - \hSigma^{\dagger} \hSigma) \Sigma (I_{p_m} - \hSigma^{\dagger} \hSigma) \beta_0$.
    Since $r_{\min} > 0$,
    the smallest eigenvalue of $\hSigma^{\dagger}$
    is almost surely positive,
    and
    the matrix $\hSigma$ is almost surely invertible
    as $k_m, p_m \to \infty$ and $p_m / k_m \to \phi \in (0, 1)$.
    Therefore, in this case, $\hSigma^{\dagger} \hSigma =  I_{p_{m}}$ almost surely,
    and $\beta_0^\top (I_{p_m} - \hSigma^{\dagger} \hSigma) \Sigma (I_{p_m} - \hSigma^{\dagger} \hSigma) \beta_0 \asto 0$.
    For the case when $k_m, p_m \to \infty$
    and $p_m / k_m \to \phi \in (1, \infty)$,
    from the second part of \Cref{cor:limiting-resolvents-mn2ls}
    by taking $f(\Sigma) = \Sigma$,
    we have
    \[
        (I_{p_m} - \hSigma^{\dagger} \hSigma) 
        \Sigma 
        (I_{p_m} - \hSigma^{\dagger} \hSigma) 
        \asympequi
        (1 + \tv_g(0; \phi))
        (v(0; \phi) \Sigma + I_{p_m})^{-1}
        \Sigma
        (v(0; \phi) \Sigma + I_{p_m})^{-1},
    \]
    where $v(0; \phi)$ and $\tv_g(0)$
    are as defined by
    \eqref{eq:fixed-point-v-mn2ls-1}
    and
    \eqref{eq:def-tvg-mn2ls-1},
    respectively.
    Note that
    from 
    \Cref{lem:fixed-point-v-properties}~\eqref{lem:fixed-point-v-properties-item-v-properties}
    $v(0; \phi)$ 
    is bounded
    for $\phi \in (1, \infty)$,
    and
    the function
    $
        r \mapsto 
        r / (1 + r v(0; \phi))^2
    $
    is continuous.
    Hence, under 
    \ref{asm:signal-bounded-norm}
    and
    \ref{asm:spectrum-spectrumsignproj-conv},
    using 
    \Cref{lem:calculus-detequi}~\eqref{lem:calculus-detequi-item-trace},
    we have
    \begin{align*}
        \beta_0^\top
        (I_{p_m} - \hSigma^{\dagger} \hSigma)
        \Sigma
        (I_{p_m} - \hSigma^{\dagger} \hSigma)
        \beta_0
        &\asto
        \lim_{p_m \to \infty}
        \sum_{i = 1}^{p_m}
        (1 + \tv_g(0; \phi))
        \frac{r_i}{(1 + r_i v(0; \phi))^2}
        (\beta_0^\top w_i)^2 \\
        &= 
        \lim_{p_m \to \infty}
        \| \beta_0 \|_2^2
        (1 + \tv_g(0; \phi))
        \int \frac{r}{(1 + r v(0; \phi))^2} \, \mathrm{d}G_{p_m}(r) \\
        &=
        \rho^2
        (1 + \tv_g(0; \phi))
        \int \frac{r}{ (1 + r v(0; \phi))^2} \, \mathrm{d}G(r),
    \end{align*}
    where in the last line we used the fact
    that $G_{p_m}$ and $G$ have compact supports,
    and $\lim_{p_m \to \infty} \| \beta_0 \|_2^2 = \rho^2$.
    This completes the proof of the first part.
    
    \paragraph{Variance functional.}
    Consider next the variance functional $\tr[\hSigma^{\dagger} \Sigma] / k_m$.
    As $k_m, p_m \to \infty$ and $p_m / k_m \to \phi \in (0, 1)$,
    $\hSigma$ is almost surely invertible as explained above.
    In this case, 
    $\tr[\hSigma^{\dagger} \Sigma] / k_m - \tr[(\bZ^\top \bZ / k_m)^{-1}] / k_m
    \asto 0$, where $\bZ \in \RR^{k_m \times p_m}$ is matrix with rows 
    $Z_i$, $1 \le i \le k_m$.
    From the proof of Proposition 2 of \cite{hastie_montanari_rosset_tibshirani_2019},
    this limit is given by $\phi / (1 - \phi)$.
    In the case when $k_m, p_m \to \infty$ and $p_m / k_m \to \phi \in (1, \infty)$,
    from \Cref{cor:limiting-resolvents-mn2ls},
    we have
    \[
        \hSigma^{\dagger} \Sigma
        \asympequi \tv(0; \phi) (v(0; \phi) \Sigma + I_p)^{-2} \Sigma^2.
    \]
    Along the same lines as above,
    from \Cref{lem:fixed-point-v-properties}~\eqref{lem:fixed-point-v-properties-item-v-properties},
    $v(0; \phi)$ is bounded for $\phi \in (1, \infty)$,
    and the
    the function
    $
        r \mapsto
        {r^2} / {(1 + v(0; \phi) r)^2}
    $
    is continuous.
    Thus, under 
    \ref{asm:spectrum-spectrumsignproj-conv}, 
    using 
    \Cref{lem:calculus-detequi}~\eqref{lem:calculus-detequi-item-trace},
    we have
    \begin{align*}
        \sigma^2
        \tr[\hSigma^{\dagger} \Sigma] / k_m
        &\asto
        \lim_{p_m \to \infty}
        \frac{p_m}{k_m}
        \frac{1}{p_m}
        \tv(0; \phi)
        \sum_{i = 1}^{p_m}
        \frac{r_i^2}{(1 + v(0; \phi) r_i)^2} \\
        &=
        \lim_{p_m \to \infty}
        \frac{p_m}{k_m}
        \tv(0; \phi)
        \int
        \frac{r^2}{(1 + v(0; \phi) r)^2}
        \, \mathrm{d}H(r) \\
        &=
        \phi
        \tv(0; \phi)
        \int
        \frac{r^2}{(1 + v(0; \phi) r)^2}
        \, \mathrm{d}H(r).
    \end{align*}
    This completes the proof of the second part.
\end{proof}

We remark that 
\Cref{cor:limiting-resolvents-mn2ls}
used in the proof of
\Cref{prop:limits-risk-functionals-mn2ls}
assumes existence of moments of order $8 + \alpha$
for some $\alpha > 0$
on the entries of $Z_i$, $1 \le i \le k_m$,
mentioned in assumption \ref{asm:lin-mod}.
As done in the proof
of Theorem 6 of \cite{hastie_montanari_rosset_tibshirani_2019}
(in Appendix A.1.4 therein),
this can be relaxed to only requiring 
existence of moments of order $4 + \alpha$.
This being a simple truncation argument,
we omit the details
and refer the readers to \cite{hastie_montanari_rosset_tibshirani_2019}.

The proposition below covers the case when $p_m / k_m \to \infty$
as $p_m, k_m \to \infty$.

\begin{proposition}
    [Limits of risk and deterministic risk approximation as $\phi \to \infty$]
    \label{prop:limits-infty-mn2ls}
    Suppose assumptions 
    \ref{asm:lin-mod}--\ref{asm:spectrum-spectrumsignproj-conv} hold.
    Then, as $k_m, p_m \to \infty$ and $p_m / k_m \to \infty$,
    we have
    \[
        R_{\bX, \bY}(\tf_\mnls(\cdot; \cD_{k_m}))
        - \beta_0^\top \Sigma \beta_0 - \sigma^2
        \asto 0.
    \]
    In addition,
    \[
        \lim_{\phi \to \infty}
        R^\deter(\cdot; \tf_\mnls)
        = \lim_{p_m \to \infty} \beta_0 \Sigma \beta_0 + \sigma^2
        = \rho^2 \int r \, \mathrm{d}G(r) + \sigma^2.
    \]
\end{proposition}

\begin{proof}

From \eqref{eq:predrisk_mn2ls},
note that
\begin{align*}
    R_{\bX, \bY}(\tf_\mnls(\cdot; \cD_{k_m}))
    - (\| \beta_0 \|_{\Sigma}^2 + \sigma^2)
    &= \| \tbeta_\mnls(\cD_{k_m}) \|_{\Sigma}^2
    - 2 \tbeta_\mnls(\cD_{k_m})^\top \Sigma \beta_0 \\
    &\le r_{\min}^{-1} \| \tbeta_\mnls \|_2^2
    + 2 \| \tbeta_\mnls(\cD_{k_m}) \|_2 \| \Sigma \beta_0 \|_2 \\
    &\le r_{\min}^{-1} \| \tbeta_\mnls(\cD_{k_m}) \|_2^2
    + 2 r_{\max} r \| \tbeta_\mnls(\cD_{k_m}) \|_2,
\end{align*}
where the first inequality follows
by using the lower bound $r_{\min}$ on the smallest eigenvalue of $\Sigma$,
and the Cauchy-Schwarz inequality,
and the second inequality follows
by using the upper bound $r_{\max}$ on the largest eigenvalue of $\Sigma$.
Thus,
for the first part it suffices to show that $\| \tbeta_\mnls \|_2 \to 0$
as $k_{m}, p \to 0$ and $p / k_{m} \to \infty$.
Towards that end,
note that
\begin{align*}
    \| \tbeta_\mnls(\cD_{k_m}) \|_2
    &= \| (\bX^\top \bX / k_{m})^{\dagger} \bX^\top \bY / k_{m} \|_2 \\
    &\le \| (\bX^\top \bX / k_{m})^{\dagger} \bX / \sqrt{k_{m}} \|_{\mathrm{\op}} 
    \| \bY / \sqrt{k_{m}}\|_2 \\
    &\le C \| (\bX^\top \bX / k_{m})^{\dagger} \bX / \sqrt{k_{m}} \|_{\mathrm{op}} 
    \sqrt{\rho^2 + \sigma^2},
\end{align*}
where the last inequality holds eventually almost surely
since \ref{asm:lin-mod} and \ref{asm:signal-bounded-norm}
imply that the entries of $\bY$ have bounded 4-th
moment,
and thus from the strong law of large numbers,
$\| \bY / \sqrt{k_{m}} \|_2$ is eventually almost surely
bounded above by $\sqrt{\EE[Y^2]} = \sqrt{\rho^2 + \sigma^2}$.
Observe that
operator norm of the matrix $(\bX^\top \bX / k_{m})^{\dagger} \bX / \sqrt{k_{m}}$
is upper bounded by the inverse of the 
smallest non-zero singular value $s_{\min}$ of $\bX$.
As $k_{m}, p_m \to \infty$ such that $p_m / k_{m} \to \infty$,
$s_{\min} \to \infty$ almost surely 
(e.g., from results in \cite{bloemendal_knowles_yau_yin_2016})
and therefore, $\| \beta \|_2 \to 0$ almost surely.
This completes the proof of first part.

Now, 
from
\Cref{lem:fixed-point-v-properties}~\eqref{lem:fixed-point-v-properties-item-v-properties}
$\lim_{\phi \to \infty} v(0; \phi) = 0$,
and
from
\Cref{lem:fixed-point-v-properties}~\eqref{lem:fixed-point-v-properties-item-tvg-properties}
$\lim_{\phi \to \infty} \tv_g(0; \phi) = 0$.
Thus,
\[
    \lim_{\phi \to \infty}
    \rho^2
    (1 + \tv_g(0; \phi) )
    \int \frac{r}{(1 + v(0; \phi) r)^2)} \, \mathrm{d}G(r)
    = \rho^2 \int r \, \mathrm{d}G(r).
\]
On the other hand,
from 
\Cref{lem:fixed-point-v-properties}~\eqref{lem:fixed-point-v-properties-item-tvg-properties},
\[
    \lim_{\phi \to \infty}
    \sigma^2 \phi \tv(0; \phi)
    \int \frac{r}{(1 + v(0; \phi) r)^2} \, \mathrm{d}H(r)
    = 0.
\]
This proves the second part,
and finishes the proof.
\end{proof}

\subsubsection*{\underline{Condition 2}: Left and right limits of deterministic risk approximation as $\phi \to 1$.}

Next we verify that
$\lim_{\phi \to 1} R^\deter(\phi; \tf_{\mnls}) = \infty$.
First note that $\lim_{\phi \to 1^{-} }R^\deter(\phi; \tf_{\mnls}) = 
\lim_{\phi \to 1^{-}} 1 / (1 - \phi) = \infty$.
Now, 
from
\Cref{lem:fixed-point-v-properties}~\eqref{lem:fixed-point-v-properties-item-tvg-properties},
observe that
\[
    \lim_{\phi \to 1^{+}}
    \phi
    \tv(0; \phi)
    \int \frac{r^2}{(1 + v(0; \phi) r)^2}
    \, \mathrm{d}H(r)
    = \infty.
\]
Since $\lim_{\phi \to 1^{-} }R^\deter(\phi) 
= \lim_{\phi \to 1^{+}} R^\deter(\phi) = \infty$,
we have that $\lim_{\phi \to 1} R^\deter(\phi) = \infty$,
as claimed.
This finishes the verification.

\subsection
{Proof of \Cref{prop:asymp-verif-mn1ls}}

Recall that $\cD_{k_m}$ is a dataset with $k_m$ observations
and $p_m$ features.
\cite{li_wei_2021} makes the following distributional
assumptions on the dataset $\cD_{k_m}$.
We adapt the scalings of \cite{li_wei_2021} to match the current paper for easy comparisons. 
\begin{enumerate}[label={\rm($\ell_1$A\arabic*)}]
    \item
    \label{asm:lin-mod-mn1ls}
    $(X_i, Y_i)$ for $1 \le i \le k_m$
    are i.i.d.\ observations from the model:
    $Y = X^\top \beta_0 + \eps$
    for some fixed unknown vector $\beta_0 \in \RR^{p_m \times 1}$
    and unobserved error $\eps$ where $\eps_{i}\stackrel{\textsf{i.i.d.}}{\sim} \mathcal{N}(0,\sigma^2)$
    independent of $X$.

    \item 
    \label{asm:iso-gau-feat-mn1ls}
    Each design vector is independently drawn by $X_{i}\stackrel{\textsf{i.i.d.}}{\sim} \mathcal{N}(0, I_p)$. 

    \item 
    \label{asm:sig-gen-mn1ls}
    The signal vector $\beta_0$ is random such that
    the scaled coordinates $\{\sqrt{p_m}\cdot \beta_{0}^{i}\}_{i=1}^{p_m}$ converge weakly to a probability measure $P_{\Theta}$, where $\mathbb{E}[\Theta^2] < \infty$ and $\mathbb{P}(\Theta \neq 0) > 0.$ 
\end{enumerate}

Under these assumptions, Theorem 2 of \cite{li_wei_2021} demonstrates that the prediction risk of the MN1LS estimator obeys \footnote{\cite{li_wei_2021} assumes $p/n = \phi$ for simplicity, but the proof goes through literatim as $p/n \to \phi$.}
\begin{align}
\label{eqn:convergence-mn1ls}
    \lim_{\substack{p/n 
    \to \phi \\ n,\, p\to\infty}} R(\tf_{\mnla}(\cdot; \cD_{k_m}))
    ~=~ \tau^{\star 2},
\end{align}
almost surely with respect to $X$ and $Y.$
Here, $(\tau^\star, \alpha^\star)$ stands for the unique solution to the following system of equations 
\begin{subequations}
    \label{eqn:fix-eqn}
        \begin{align}
        \label{eq:fix-1} \tau^2 & = \sigma^2 + \mathbb{E}\left[ \big( \eta(\Theta+\tau Z; \alpha\tau)-\Theta\big) ^2\right] , \\
        \label{eq:fix-2} \phi^{-1} & =\mathbb{P}\big( |\Theta+\tau Z| > \alpha\tau \big),
        \end{align}
\end{subequations}
where $\Theta \sim P_{\Theta}$, and $Z \sim \mathcal{N}(0,1)$ and is independent of $\Theta$.
Here, $\eta(\cdot; b)$ is the soft-thresholding function at level $b \ge 0$
that maps $x \in \RR$ to
\[
    \eta(x; b)
    = (|x| - b)_+ \sign(x).
\]
The existence and uniqueness of the equation set~\eqref{eqn:fix-eqn} is established in \cite{li_wei_2021}. 
To facilitate accurate characterization of $\tau^\star$ as a function of $\phi$, we make assumption on how the ground true is generated as follows. 
\begin{enumerate}
    \item[\rm($\ell_1$A4)]
    \label{asm:sig-sparse-mn1ls}
    Suppose that each coordinate of $\beta_0=[\beta_0^i]_{1\leq i\leq p}$ is identically and independently drawn as follows 
\begin{equation}\label{eq:theta-distribution}
    \beta_0^i \overset{\mathrm{i.i.d.}}{\sim} \epsilon \mathcal{P}_{M/\sqrt{p_m}} + (1-\epsilon)\mathcal{P}_0,
\end{equation}
where $\mathcal{P}_{c}$ corresponds to the Dirac measure at point $c\in \mathbb{R}$, and $M>0$ is some given scalar that determines the magnitude of a non-zero entry. 

\end{enumerate}

Under the above four assumptions, it is proved in Lemma 2 (p.~50) of 
\cite{li_wei_2021} 
that 
\begin{align}
\label{eqn:profile-l1-brahms}
    \lim_{\phi \to 1^+} \tau^{\star 2}(\phi) = \infty, 
\end{align}
and 
Lemma 1 (p.~51)
of \cite{li_wei_2021}
that
\begin{align*}
    \lim_{\phi \to \infty} \tau^{\star 2}(\phi) 
    =
    \sigma^2 + \mathbb{E}\|\beta_0\|^2_2 =  \sigma^2 + \epsilon M^2.
\end{align*}
We remark that the above results are stated slight differently therein due to a different scaling, where a global $1/\sqrt{k_m}$ is applied to the design matrix and $\sqrt{p_m}$ is applied to the ground truth parameter $\beta_0.$ Here, we adapt a global scaling to allow for convenient comparisons with the MN2LS estimator. 

From the discussion above, it is therefore clear that, one can set  
\begin{align}
\label{eqn:def-R-mn1ls}
R^\deter(\cdot; \tf_{\mnla}) = 
\begin{dcases}
\sigma^2 \frac{1}{1 - \phi}
& \text{ if } \phi \in (0,1)\\
\infty & \text{ if } \phi = 1 \\
\tau^{\star 2} & \text{ if }\phi \in (1,\infty)\\
\sigma^2 + \epsilon M^2 & \text{ if } \phi = \infty
\end{dcases}
\end{align}
which satisfies the conditions of \Cref{prop:asymp-verif-mn1ls}.

In order to see this, first recognizing that the convergence~\eqref{eqn:convergence-mn1ls} holds almost surely, the first condition of \Cref{prop:asymp-verif-mn1ls} is satisfied naturally.
Additionally, as established in Section~\ref{sec:verify-profile-l2}
and in \eqref{eqn:profile-l1-brahms}, one has 
\begin{align}
    \lim_{\phi \to 1^{+}} R^\deter(\phi; \tf_{\mnla}) = \infty,
    \quad
    \text{and}
    \quad
    \lim_{\phi \to 1^{-}} R^\deter(\phi; \tf_{\mnla}) = \infty,
\end{align}
which validates the second condition of \Cref{prop:asymp-verif-mn1ls}. 
Putting everything together completes the proof of \Cref{prop:asymp-verif-mn1ls}.

\section{Proofs related to risk monotonization for one-step procedure}
\label{sec:proofs-riskmonotonization-onestep}

\subsection
{Proof of \Cref{lem:deterministic-approximation-reduction-onestep}}

The idea of the proof is similar to proof of \Cref{lem:rn-deterministic-approximation-4-prop-asymptotics}.
We wish to verify that there exists a deterministic approximation $R^\deter: \RR \times \RR \to \RR$
to the conditional prediction risk of the predictor
$\tf(\cdot; \cD_\train^{\xi_{1,n},j}, \cD_\train^{\xi_{2,n}, j})$,
$1 \le j \le M$ that satisfy
\[
    \left|
    R(\tf(\cdot; \cD_\train^{\xi_{1,n}^\star, j}, \cD_\train^{\xi_{2,n}^\star, j}))
    - R^\deter\left(\frac{p_n}{n_{1, \xi_{1,n}^\star}}, \frac{p_n}{n_{2, \xi_{2,n}^\star}}; \tf\right)
    \right|
    = o_p(1)
    R^\deter\left(\frac{p_n}{n_{1, \xi_{1,n}^\star}}, \frac{p_n}{n_{2, \xi_{2,n}^\star}}; \tf\right)
\]
as $n \to \infty$ under \ref{asm:prop_asymptotics},
where $(\xi_{1,n}^\star, \xi_{2,n}^\star)$ are indices such that
\[
    (\xi_{1,n}^\star, \xi_{2,n}^\star)
    \in \argmin_{(\xi_1, \xi_2) \in \Xi_n}
    R^\deter
    \left(
        \frac{p_n}{n_{1, \xi_1}},
        \frac{p_n}{n_{2, \xi_2}};
        \tf
    \right).
\]
Following the arguments in the proof of
\Cref{lem:rn-deterministic-approximation-4-prop-asymptotics},
using the lower bound on $R(\tf(\cdot; \cD_\train^{\xi_{1,n, j}}, \cD_\train^{\xi_{2,n, j}}))$
and identical distribution across $j$,
it suffices to show that
for all $\epsilon > 0$,
\[
    \PP
    \left(
    \left|
    R(\tf(\cdot; \cD_\train^{\xi_{1,n}^\star}, \cD_\train^{\xi_{2,n}^\star}))
    - R^\deter\left(\frac{p_n}{n_{1, \xi_{1,n}^\star}}, \frac{p_n}{n_{2, \xi_{2,n}^\star}}; \tf\right)
    \right|
    \ge \epsilon
    \right)
    \to 0
\]
as $n \to \infty$ under \ref{asm:prop_asymptotics}.
Note that here we have dropped the superscript $j$ for brevity.
Now we will show that \eqref{eq:rn-deterministic-approximation-reduced-onestep-prop-asymptotics}
along with the assumed continuity behavior of $R^\deter(\cdot, \cdot; \tf)$
implies desired conclusion.
Fix $\eps > 0$
and define a sequence $h_n(\epsilon)$ as follows:
\[
    h_n(\epsilon)
    :=
    \PP
    \left(
    \left|
    R(\tf(\cdot; \cD_\train^{\xi_{1,n}^\star}, \cD_\train^{\xi_{2,n}^\star}))
    - R^\deter\left(\frac{p_n}{n_{1, \xi_{1,n}^\star}}, \frac{p_n}{n_{2, \xi_{2,n}^\star}}; \tf\right)
    \right|
    \ge \epsilon
    \right).
\]
We want to show that $h_n(\epsilon) \to \infty$ as $n \to \infty$
under \ref{asm:prop_asymptotics}.
We first note that using \Cref{lem:subsequence-to-sequence},
it suffices to show that for an arbitrary subsequence $\{ n_k \}_{k \ge 1}$,
there exists further subsequence $\{ n_{k_{l}} \}_{l \ge 1}$ such that
$h_{n_{k_{l}}} \to 0$ as $n \to \infty$.
Also, note that since $n_\train / n \to 1$,
the grid $\Xi_n$ satisfies the space-filling property
from \Cref{lem:spacefilling-grid-onestep-general}
that $\Pi_{\Xi_n}(\zeta_1, \zeta_2) \to (\zeta_1, \zeta_2)$
for any $(\zeta_1, \zeta_2)$ that satisfy $\zeta_1^{-1} + \zeta_2^{-1} \le \gamma^{-1}$
and the set of $(\zeta_1, \zeta_2)$ that satisfy this condition is compact.
Now, we apply \Cref{lem:limit-argmins-metricspace}
on the function $R^\deter(\cdot, \cdot; \tf)$
and the grid $\Xi_n$.
Let sequence $\{ x_n \}_{n \ge 1}$ be such that
$x_n := (p_n/n_{1,\xi_{1,n}^\star} , p_n/n_{2,\xi_{2,n}^\star})$ for $n \ge 1$.
\Cref{lem:limit-argmins-metricspace} guarantees that
for any arbitrary subsequence $\{ x_{n_{k}} \}_{k \ge 1}$,
there exists a further subsequence $\{ x_{n_{k_{l}}} \}_{l \ge 1}$
such that
\begin{equation}
    \label{eq:subsequence-limit-gaurantee-onestep}
    x_{n_{k_{l}}}
    \to
    (\phi_1, \phi_2)
    \in
    \argmin_{\zeta_1^{-1} + \zeta_2^{-1} \le \gamma^{-1}} R^\deter(\zeta_1, \zeta_2; \tf).
\end{equation}
We will now show that $h_{n_{k_{l}}} \to 0$ as $l \to \infty$
if assumption \eqref{eq:rn-deterministic-approximation-reduced-onestep-prop-asymptotics}
\Cref{lem:deterministic-approximation-reduction-onestep}
is satisfied.
It is easy to see that the assumption implies
\[
    R(\tf(\cdot; \cD_{\train}^{\xi_{1,n}^\star}, \cD_{\train}^{\xi_{2,n}^\star}))
    \pto R^\deter(\phi_1, \phi_2; \tf)
\]
as $n, p_n, \xi_{1,n}^\star, \xi_{2,n}^\star \to \infty$, whenever
\[
    (p_n / n_{1, \xi_{1,n}^\star}, p_n / n_{2, \xi_{2,n}^\star})
    \to (\phi_1, \phi_2) \in \argmin_{\zeta_1^{-1} + \zeta_2^{-1} \le \gamma^{-1}} R^\deter(\zeta_1, \zeta_2; \tf).
\]
But using the continuity of $R^\deter(\cdot, \cdot; \tf)$
on the set $\argmin_{\zeta_1^{-1} + \zeta_2^{-1} \le \gamma^{-1}} R^\deter(\zeta_1, \zeta_2; \tf)$
and the fact that the sequence $\{ x_{n_{k_{l}}} \}_{l \ge 1}$ converges
to a point in this minimizing set
from \eqref{eq:subsequence-limit-gaurantee-onestep},
it follows that
that $h_{n_{k_{l}}} \to 0$ as $l \to \infty$ as desired.
This finishes the proof.

\subsection
{Proof of \Cref{prop:semicontinuity-metricspace}}
    Fix $t < \infty$.
    We will verify that
    the set $C_t := \{ x : h(x) \le t \}$ is closed.
    Note that $C_t \subseteq M \setminus C$
    because $h(x) < \infty$ for $x \in C_t$.
    Now consider any converging sequence $\{ x_n \}_{n \ge 1}$
    in $C_t$ with limit point $p$.
    We will argue that $p \in C_t$.
    First note that the function $h$
    is continuous over $C_t$
    because $C_t \subseteq M \setminus C$.
    Note that $p \notin C$,
    because if it does 
    then $h(x_n) \to \infty$ as $n \to \infty$,
    which in turn implies that for infinitely
    many $k \ge 1$, $h(x_k) > t$,
    contradicting $x_n \in C_t$ for all $n \ge 1$.
    Hence, $p \in M \setminus C$
    and $x_n \in M \setminus C$ for all $n \ge 1$.
    Therefore,
    continuity of $h$ on $M \setminus C$
    yields $h(x_n) \to h(p)$.
    Moreover, $h(x_n) \le t$
    implies that $\lim_{n \to \infty} h(x_n) \le t$,
    which in turn implies that $h(p) \le t$.
    Hence $p \in C$, finishing the proof.

\subsection
{Proof of \Cref{prop:continuity-from-continuous-convergence-rdet-onestep}}

The proof uses a similar contradiction strategy
employed in the proof of
\Cref{prop:continuity-from-continuous-convergence-rdet}.
We only sketch the proof, and omit the details.

Suppose $R^\deter(\cdot, \cdot; \tf)$
is discontinuous at some point $(\phi_{1,\infty}, \phi_{2,\infty})$.
This gives us a sequence $\{ (\phi_{1,r}, \phi_{2,r}) \}_{r \ge 1}$
such that for some $\epsilon > 0$
and all $r \ge 1$,
\begin{equation}
    \label{eq:RRdiscontinuity-Rdet-sequence-existence-implication-onestep}
    R^\deter(\phi_{1,r}, \phi_{2,r}; \tf)
    \notin
    [
        R^\deter(\phi_{1, \infty}, \phi_{2, \infty}; \tf) - 2 \epsilon,
        R^\deter(\phi_{1, \infty}, \phi_{2, \infty}; \tf) + 2 \epsilon
    ],
\end{equation}
while $(\phi_{1,r}, \phi_{2,r}) \to (\phi_{1, \infty}, \phi_{2, \infty})$
as $r \to \infty$.
From the continuous convergence hypothesis,
for each $r \ge 1$,
one can then construct a sequence of datasets
$\{(\cD^{\phi_{1,r}}_{k_{1,m}}, \cD^{\phi_{2,r}}_{k_{2,m}})\}_{m \ge 1}$
with $p_m$ features and $(k_{1,m}, k_{2,m})$ observations
for which
\begin{equation}
    \label{eq:contradiction-proof-prob-converge-onestep}
    R(\tf(\cdot; \cD^{\phi_{1,r}}_{k_{1,m}}, \cD^{\phi_{2,r}}_{k_{2,m}}))
    \pto
    R^\deter(\phi_{1,r}, \phi_{2,r}; \tf)
\end{equation}
as $p_m, k_{1,m}, k_{2,m} \to \infty$
and $(p_m / k_{1,m}, p_m / k_{2,m}) \to (\phi_{1,r}, \phi_{2,r})$.
From 
\eqref{eq:RRdiscontinuity-Rdet-sequence-existence-implication-onestep} 
and 
\eqref{eq:contradiction-proof-prob-converge-onestep},
one can obtain a sequence of increasing integers $\{ m_r \}_{r \ge 1}$
such that for each $r \ge 1$,
with probability $0 < p < 1$,
\[
    |
        R(\tf(\cdot; \cD^{\phi_{1,r}}_{k_{1,m}}, \cD^{\phi_{2,r}}_{k_{2,m}}))
        - R^\deter(\phi_{1,\infty}, \phi_{2,\infty}; \tf)
    |
    > \epsilon.
\]
This then lets us construct a sequence of datasets
$\{ (\cD'_{k_{1,m}}, \cD'_{k_{2,m}}) \}_{m \ge 1}$
similar as done in the proof of \Cref{prop:continuity-from-continuous-convergence-rdet}
for which
\[
    R(\tf(\cdot; \cD'_{k_{1,m}}, \cD'_{k_{2,m}}))
    \not\pto
    R^\deter(\phi_{1,\infty}, \phi_{2,\infty}; \tf)
\]
as $p_m, k_{1,m}, k_{2,m} \to \infty$
and $(p_m / k_{1,m}, p_m/k_{2,m}) \to (\phi_{1,\infty}, \phi_{2,\infty})$.
This supplies the required contradiction
to the continuous convergence hypothesis.

\subsection
{Proof of \Cref{thm:asymptotic-risk-tuned-one-step}}
   The idea of the proof is similar to that of
   the proof of \Cref{thm:asymptotic-risk-tuned-zero-step}.
   We will break the proof in two cases.
   \paragraph{Case of $M = 1$.}
   Consider first the case when $m = 1$.
   In this case,
   $\hf^\cv = \tf_1^\xi$,
   which we denote by $\tf^\xi$ for notational simplicity.
   Bound the desired difference as
   \begin{align*}
        &\left|
            R(\hf^\cv)
            - \min_{1/\zeta_1 + 1/\zeta_2 \le n/p}
            R^\deter(\hf; \zeta_1, \zeta_2)
        \right| \\
        &\le
        \left|
            R(\hf^\cv)
            - \min_{\xi \in \Xi} R(\tf^\xi)
        \right|
        +
        \left|
            \min_{\xi \in \Xi} R(\tf^\xi)
            - \min_{\xi \in \Xi}
            R^\deter
            \left(
                \tf;
                \frac{p_n}{n - \xi_1 \lfloor n^\nu \rfloor},
                \frac{p_n}{\xi_2 \lfloor n^\nu \rfloor}
            \right)
        \right| \\
        & \quad + 
        \left|
            \min_{\xi \in \Xi}
            R^\deter
            \left(
                \tf;
                \frac{p_n}{n - \xi_1 \lfloor n^\nu \rfloor},
                \frac{p_n}{\xi_2 \lfloor n^\nu \rfloor}
            \right)
            - \min_{1/\zeta_1 + 1/\zeta_1 \le n/p}
            R^\deter(\tf; \zeta_1, \zeta_2)
        \right|
   \end{align*}
   We show below that each of the terms
   asymptotically go to zero. 
   Observe that
   \[
        \big| \Xi \big|
        = \sum_{\xi_1 = 2}^{\left\lceil n/\lfloor n^\nu \rfloor - 2 \right\rceil}
        (\xi_1 - 1)
        \le n^2.
   \]
   Since $\hsigma_\Xi = \tsigma_\Xi = o_p(\sqrt{n^\nu/\log(n)})$,
   under the setting of
   \Cref{lem:bounded-orlitz-error-control}
   or \Cref{lem:bounded-variance-error-control},
   \Cref{rem:growth-rates-probabilistic-bound} hold
   so that
   \[
        \left|
            R(\hf^\cv)
            - \min_{\xi \in \Xi} R(\tf)
        \right|
        = o_p(1).
   \]
   The assumption on the asymptotic risk profile
    \eqref{eq:rn-deterministic-approximation-onestep-prop-asymptotics}
   leads to
   \[
        \left|
            \min_{\xi \in \Xi}
            R(\tf^\xi)
            - \min_{\xi \in \Xi}  
            R^\deter
            \left(
                \tf;
                \frac{p_n}{n - \xi_1 \lfloor n^\nu \rfloor},
                \frac{p_n}{\xi_2 \lfloor n^\nu \rfloor}
            \right)
        \right|
        = o_p(1).
   \]
   Since the risk profile $R^\deter(\tf; \zeta_1, \zeta_2)$
   is assumed be continuous at its minimizer,
   applying 
    \Cref{lem:spacefilling-grid-onestep-general}
   we get
   \[
        \min_{\xi \in \Xi}     
        R^\deter
        \left(
            \tf;
            \frac{p_n}{n - \xi_1 \lfloor n^\nu \rfloor},
            \frac{p_n}{\xi_2 \lfloor n^\nu \rfloor}
        \right)
        \to
        \min_{1/\zeta_1 + 1/\zeta_2 \le n/p}
        R^\deter(\tf; \zeta_1, \zeta_2).
   \]
   Combining the above three convergences,
   we have the desired conclusion.
  
  \paragraph{Case of $M > 1$.} 
   When $m > 1$,
   we bound the desired difference as
   \begin{align*}
        &\left(
            R(\hf^\cv)
            - \min_{1/\zeta_1 + 1/\zeta_2 \le n/p} R^\deter(\tf; \zeta_1, \zeta_2)
        \right)_+ \\
        &\le
        \left(
            R(\hf^\cv)
            - \min_{\xi \in \Xi} R(\hf^\xi)
        \right)_+
        +
        \left(
            \min_{\xi \in \Xi} R(\hf^\xi)
            - \frac{1}{M} \sum_{j=1}^{M} \min_{\xi \in \Xi} R(\tf_j^\xi)
        \right)_+ \\
        & \quad +
        \left(
            \frac{1}{M} \sum_{j=1}^{M} \min_{\xi \in \Xi} R(\tf_j^\xi)
            - \min_{\xi \in \Xi}
            R^\deter
            \left(
                \tf^\xi;
                \frac{p_n}{n - \xi_1 \lfloor n^\nu \rfloor},
                \frac{p_n}{\xi_2 \lfloor n^\nu \rfloor}
            \right)
        \right)_+ \\
        & \quad +
        \left(
            \min_{\xi \in \Xi}
            R^\deter
            \left(
                \tf;
                \frac{p_n}{n - \xi_1 \lfloor n^\nu \rfloor},
                \frac{p_n}{\xi_2 \lfloor n^\nu \rfloor}
            \right)
            - \min_{1/\zeta_1 + 1/\zeta_2 \le n/p}
            R^\deter(\tf; \zeta_1, \zeta_2)
        \right)_+
   \end{align*}
   As before, we show below that each of the terms
   asymptotically vanish.
   Noting that $\hsigma_\Xi \le \tsigma_\Xi$,
   application of
   \Cref{rem:growth-rates-probabilistic-bound}
   shows that the first term is $o_p(1)$.
   The second term is $0$ exactly as argued
   in the proof of
   \Cref{thm:asymptotic-risk-tuned-zero-step}.
   The third term is $o_p(1)$
   by noting that
    \eqref{eq:rn-deterministic-approximation-onestep-prop-asymptotics}
   holds for all $j = 1, \dots, m$.
   Finally, the fourth term is $0$
   as argued for the case of $m = 1$.

\section{Proofs related to deterministic profile verification for one-step procedure}
\label{sec:verif-riskprofile-mnlsbase-mnlsonestep}

In this section,
we verify the assumption \eqref{eq:rn-deterministic-approximation-reduced-onestep-prop-asymptotics}
for the one-step procedure, 
where the base prediction procedure is
linear,
under some regularity conditions.
We also specifically 
consider the cases of MN2LS and MN1LS base prediction procedures.

\subsection
{Predictor simplifications and risk decompositions}

In this section,
we first provide preparatory lemmas
that will be useful in the proofs of
\Cref{lem:one-step-predrisk-decomposition}
and
\Cref{cor:verif-onestep-program}.

Let $\bX_1 \in \RR^{k_{1,m} \times p_m}$ and $\bY_1 \in \RR^{k_{1,m}}$
denote the feature matrix and response vector corresponding 
to the first split dataset $\cD_{k_{1,m}}$.
Similarly,
let $\bX_2 \in \RR^{k_{2,m} \times p_m}$ and $\bY_2 \in \RR^{k_{2,m}}$
denote the feature matrix and response vector corresponding to the second split
dataset $\cD_{k_{2,m}}$.

The following lemma gives an alternative representation
for the ingredient one-step predictor
assuming that the base prediction procedure is linear.
\begin{lemma}
    [Alternate representation for the ingredient one-step predictor]
    \label{lem:onestep-ingredient-simplifications}
    Suppose the base prediction procedure $\tf$ is linear
    such that $\tf(x; \cD_{k_{1,m}}) = x^\top \tbeta(\cD_{k_{1,m}})$
    for some estimator $\tbeta(\cD_{k_{1,m}})$ trained on $\cD_{k_{1,m}}$.
    Let $\tf(\cdot; \cD_{k_{1,m}}, \cD_{k_{2,m}})$ 
    denote the ingredient one-step predictor \eqref{eq:onestep-ingredient-decomp}.
    Then, $\tf(\cdot; \cD_{k_{1,m}}, \cD_{k_{2,m}})$ is a linear predictor
    such that 
    $\tf(x; \cD_{k_{1,m}}, \cD_{k_{2,m}}) = x^\top \tbeta(\cD_{k_{1,m}}, \cD_{k_{2,m}})$
    with the corresponding ingredient one-step estimator $\tbeta(\cD_{k_{1,m}}, \cD_{2,m})$
    given by
    \begin{equation}
        \label{eq:alternate-representation-onestep-ingredient-estimator}
        \tbeta(\cD_{k_{1,m}}, \cD_{k_{2,m}})
        =
        \big\{I_p - (\bX_2^T \bX_2 / k_{2,m})^{\dagger} (\bX_2^T \bX_2 / k_{2,m})\big\}
        \tbeta(\cD_{k_{1,m}})
        + \tbeta_\mnls(\cD_{k_{2,m}}),
    \end{equation}
    where $\tbeta_\mnls(\cD_{k_{2,m}})$ is the MN2LS estimator fit on $\cD_{k_{2,m}}$.
    Furthermore, suppose assumption \ref{asm:lin-mod} holds true for $\cD_{k_{2,m}}$.
    Then, the error between $\tbeta(\cD_{k_{1,m}}, \cD_{k_{2,m}})$ 
    and 
    $\beta_0$
    can be expressed as
    \begin{align}
        &\tbeta(\cD_{k_{1,m}}, \cD_{k_{2,m}}) - \beta_0 \nonumber \\
        &\quad = \big\{ I_p - (\bX_2^\top \bX_2 / k_{2,m})^\dagger (\bX_2^\top \bX_2 / k_{2,m}) \big\}
        (\tbeta(\cD_{k_{1,m}}) - \beta_0)
        + (\bX_2^\top \bX_2 / k_{2,m})^\dagger \bX_2^\top \beps_2 / k_{2,m}.
        \label{eq:error-onestep-estimator-vs-signal}
    \end{align}
\end{lemma}
\begin{proof}
    For the first part,
    start by re-arranging the ingredient one-step predictor \eqref{eq:onestep-ingredient-decomp} 
    as follows:
    \begin{align*}
        \tf(x; \cD_{k_{1,m}}, \cD_{k_{2,m}})
        &= \tf(x; \cD_{k_{1,m}})
        + x^\top (\bX_2^\top \bX_2 / k_{2,m})^{\dagger}
        \bX_2^\top (\bY_2 - \bX_2 \tbeta(\cD_{k_{1,m}})) / k_{2,m} \\
        &= x^\top \tbeta(\cD_{k_{1,m}}) 
        + x^\top (\bX_2^\top \bX_2 / k_{2,m})^{\dagger}
        \bX_2^\top (\bY_2 - \bX_2 \tbeta(\cD_{k_{1,m}})) / k_{2,m} \\
        &= x^\top \big\{ I_p - (\bX_2^\top \bX / k_{2,m})^{\dagger} (\bX_2^\top \bX_2) / k_{2,m}  \big\}
        \tbeta(\cD_{k_{1,m}}) 
        + x^\top (\bX_2^\top \bX_2 / k_{2,m})^\dagger \bX_2^\top \bY_2 / k_{2,m} \\
        &= x^\top \big\{ I_p - (\bX_2^\top \bX / k_{2,m})^{\dagger} (\bX_2^\top \bX_2) / k_{2,m}  \big\}
        \tbeta(\cD_{k_{1,m}}) + x^\top \tbeta_\mnls(\cD_{k_{2,m}}),
    \end{align*}
    where 
    $
        \tbeta_\mnls(\cD_{k_{2,m}})
        = (\bX_2^\top  \bX_2 / k_{2,m})^{\dagger} \bX_2^\top \bY_2 / k_{2,m}
    $ 
    is the MN2LS estimator fit on $\cD_{k_{2,m}}$.
    Thus, $\tf(\cdot; \cD_{k_{1,m}}, \cD_{k_{2,m}})$
    is a linear predictor with the corresponding
    ingredient one-step estimator
    $\tbeta(\cD_{k_{1,m}}, \cD_{2,m})$ given by 
    \eqref{eq:alternate-representation-onestep-ingredient-estimator}.
    This completes the proof of the first part.
    
    For the second part, note that under linear model 
    $\bY_2 = \bX_2 \beta_0 + \beps_2$
    (from \ref{asm:lin-mod} for $\cD_{k_{2,m}}$),
    the ingredient one-step estimator $\tbeta(\cD_{k_{1,m}}, \cD_{k_{2,m}})$ can be further simplified to
    \begin{align*}
        &\tbeta(\cD_{k_{1,m}}, \cD_{k_{2,m}}) \\
        & = \big\{ I_p - (\bX_2^\top \bX_2 / k_{2,m})^\dagger (\bX_2^\top \bX_2 / k_{2,m}) \big\} 
        \tbeta(\cD_{k_{1,m}})
        + (\bX_2^\top \bX_2 / k_{2,m})^\dagger (\bX_2^\top \bX_2 / k_{2,m}) \beta_0
        + (\bX_2^\top \bX_2 / k_{2,m})^\dagger \bX_2^\top \beps_2 / k_{2,m}.
    \end{align*}
    Hence, the error between $\tbeta(\cD_{k_{1,m}}, \cD_{k_{2,m}})$ 
    and $\beta_0$ 
    can be expressed as
    \begin{align*}
        &\tbeta(\cD_{k_{1,m}},\cD_{k_{2,m}}) - \beta_0 \\
        &=
        \big\{ I_p - (\bX_2^\top \bX_2 / k_{2,m})^\dagger (\bX_2^\top \bX_2 / k_{2,m}) \big\}
        \tbeta(\cD_{k_{1,m}})
        + (\bX_2^\top \bX_2 / k_{2,m})^\dagger (\bX_2^\top \bX_2 / k_{2,m}) \beta_0
        + (\bX_2^\top \bX_2 / k_{2,m})^\dagger \bX_2^\top \beps_2 / k_{2,m}
        - \beta_0 \\
        &=
        \big\{ I_p - (\bX_2^\top \bX_2 / k_{2,m})^\dagger (\bX_2^\top \bX_2 / k_{2,m}) \big\}
        \tbeta(\cD_{k_{1,m}})
        + \big\{ (\bX_2^\top \bX_2 / k_{2,m})^\dagger (\bX_2^\top \bX_2 / k_{2,m}) - I_p \big\} \beta_0
        + (\bX_2^\top \bX_2 / k_{2,m})^\dagger \bX_2^\top \beps_2 / k_{2,m} \\
        &=
        \big\{ I_p - (\bX_2^\top \bX_2 / k_{2,m})^\dagger (\bX_2^\top \bX_2 / k_{2,m}) \big\}
        (\tbeta(\cD_{k_{1,m}}) - \beta_0)
        + (\bX_2^\top \bX_2 / k_{2,m})^\dagger \bX_2^\top \beps_2 / k_{2,m}.
    \end{align*}
    This completes the proof of the second part.
\end{proof}

Recall that we are interested in the conditional squared prediction risk
of $\tf(\cdot; \cD_{k_{1,m}}, \cD_{k_{2,m}})$:
\begin{equation}
    \label{eq:squared-risk-onestep-ingredient}
    R_{\bX_1, \bY_1, \bX_2, \bY_2}(\tf(\cdot; \cD_{k_{1,m}}, \cD_{k_{2,m}}))
    = \EE[(Y_0 - \tf(X_0; \cD_{k_{1,m}}, \cD_{k_{2,m}}))^2 \mid \bX_1, \bY_1, \bX_2, \bY_2],
\end{equation}
where $(X_0, Y_0)$ is sampled independently
and from the same distribution
as the training data
$(\bX_1, \bY_1)$ and $(\bX_2, \bY_2)$.
We are being explicit about the dependence
of $R(\tf(\cdot; \cD_{k_{1,m}}, \cD_{k_{2,m}}))$
on $(\bX_1, \bY_1, \bX_2, \bY_2)$
as we will consider concentration
of $R(\tf(\cdot; \cD_{k_{1,m}}, \cD_{k_{2,m}}))$
conditional on $(\bX_1, \bY_1)$
first, followed by that on $(\bX_2, \bY_2)$.
For notational convenience,
let $\hSigma_1 := \bX_1^T \bX_1 / k_{1,m}$ and $\hSigma_2 := \bX_2^T \bX_2 / k_{2,m}$
denote the sample covariance matrices for the two data splits 
$\cD_{k_{1,m}}$ and $\cD_{k_{2,m}}$, respectively.
The next lemma gives conditional concentration of the squared prediction risk
\eqref{eq:squared-risk-onestep-ingredient}
of the one-step ingredient predictor under
the additional assumptions
\ref{asm:rmt-feat}--\ref{asm:covariance-bounded-eigvals} on $\cD_{k_{2,m}}$.

\begin{lemma}
    [Conditional concentration of squared prediction risk of one-step ingredient predictor]
    \label{lem:cond-concen-onestep-risk}
    Assume the setting of \Cref{lem:onestep-ingredient-simplifications}.
    In addition, 
    suppose 
    assumptions \ref{asm:rmt-feat}--\ref{asm:covariance-bounded-eigvals}
    hold
    for $\cD_{k_{2,m}}$.
    Let $k_{1,m}, k_{2,m}, p_{m} \to \infty$
    such that 
    $ p_m / k_{2,m} \to \phi_2 \in (0, 1) \cup (1, \infty)$
    and assume
    $\limsup \| \tbeta(\cD_{k_{1,m}}) - \beta_0 \|_2 < \infty$ almost surely.
    Then,
    we have
    \begin{align*}
        & R_{\bX_1, \bY_1, \bX_2, \bY_2}
        (\tf(\cdot; \cD_{k_{1,m}}, \cD_{k_{2,m}})) \\
        &
        \quad
        -
        (\tbeta(\cD_{k_{1,m}}) - \beta_0)^\top
        (I_p - \hSigma_2^\dagger \hSigma_2)
        \Sigma
        (I_p - \hSigma_2^\dagger \hSigma_2)
        (\tbeta(\cD_{k_{1,m}}) - \beta_0)
        - \sigma^2 \tr[\hSigma_2^\dagger \Sigma] / k_{2,m}
        - \sigma^2
        \asto 0.
    \end{align*}
\end{lemma}
\begin{proof}
    
    The proof follows similar steps as those in the proof of 
    \Cref{prop:cond-conv-mn2ls}.
    We start by decomposing the squared prediction risk:
    \begin{equation}
        \label{eq:squared-predrisk-onestep}
        R_{\bX_1, \bY_1, \bX_2, \bY_2}(\tf(\cdot; \cD_{k_{1,m}}, \cD_{2,m}))
        = (\tbeta(\cD_{k_{1,m}}, \cD_{2,m}) - \beta_0)^\top 
        \Sigma 
        (\tbeta(\cD_{k_{1,m}}, \cD_{k_{2,m}}) - \beta_0) + \sigma^2.
    \end{equation}
    Under \ref{asm:lin-mod}, from \Cref{lem:onestep-ingredient-simplifications},
    we have
    \[
        \tbeta(\cD_{k_{1,m}}, \cD_{k_{2,m}}) - \beta_0
        = 
        (I_p - \hSigma_2^\dagger \hSigma_2) (\tbeta(\cD_{k_{1,m}}) - \beta_0)
        + \hSigma_2^{\dagger} \bX_2^\top \beps_2 / k_{2,m}.
    \]
    Thus, the first term in the squared prediction risk \eqref{eq:squared-predrisk-onestep}
    of $\tf(\cdot; \cD_{k_{1,m}}, \cD_{k_{2,m}})$ can be split into:
    \[
        (\tbeta(\cD_{k_{1,m}}, \cD_{k_{2,m}}) - \beta_0)^\top
        \Sigma
        (\tbeta(\cD_{k_{1,m}}, \cD_{k_{2,m}}) - \beta_0)
        = \bB_1 + \bC_1 + \bV_1,
    \]
    where the terms $\bB_1$, $\bC_1$, and $\bV_1$ are given as follows:
    \begin{align*}
        \bB_1
        &= 
        (\tbeta(\cD_{k_{1,m}}) - \beta_0)^\top
        (I_p - \hSigma_2^\dagger \hSigma_2)
        \Sigma
        (I_p - \hSigma_2^\dagger \hSigma_2)
        (\tbeta(\cD_{k_{1,m}}) - \beta_0),
        \\
        \bC_1
        &= 
        (\tbeta(\cD_{k_{1,m}}) - \beta_0)^\top
        (I_p - \hSigma_2^\dagger \hSigma_2) \hSigma_2^\dagger \bX_2^\top \beps_2 / k_{2,m},
        \\
        \bV_1
        &= 
        \beps_2 
        (\bX_2 \hSigma_2^\dagger \Sigma \hSigma_2^\dagger \bX_2^\top / k_{2,m}) 
        \beps_2 / k_{2,m}.
    \end{align*}
    The rest of the proof shows concentration for the terms $\bC_1$ and $\bV_1$.
    
    As argued in the proof of \Cref{prop:cond-conv-mn2ls},
    appealing to \Cref{lem:concen-linform} we have that $\bC_1 \asto 0$
    as $p_m, k_m \to \infty$ 
    such that $p_m / k_{2,m} \to \phi \in (0, 1) \cup (1, \infty)$,
    assuming $\limsup \| \tbeta(\cD_{k_{1,m}}) - \beta_0 \|_2 < \infty$.
    This is because, 
    from a bounding similar to
    \eqref{eq:bound-ell2normsquare-crossterm-zerostep},
    we have
    \[
        \limsup \| \bX_2 \hSigma_2^\dagger (I_p - \hSigma_2^\dagger \hSigma_2) (\tbeta(\cD_{k_{1,m}}) - \beta_0) \|_2^2 / k_{2,m}
        \le C \limsup \| \tbeta(\cD_{k_{1,m}} - \beta_0) \|_2^2
        \le C,
    \]
    almost surely for a constant $C < \infty$.
    Similarly, for the term $\bV_1$, 
    using 
    \Cref{lem:concen-quadform}
    along with
    the bound from \eqref{eq:bound-operatornorm-varianceterm-zerostep},
    we have $\bV_1 - \sigma^2 \tr[\hSigma_2^\dagger \Sigma] / k_{2,m} \asto 0$.
    This finishes the proof.
\end{proof}

\begin{lemma}
    [Conditional deterministic approximation of squared risk of ingredient one-step predictor]
    \label{lem:conditional-deterministic-convergence-onestep}
    Assume the setting of \Cref{lem:cond-concen-onestep-risk}.
    Let $k_{1,m}, k_{2,m}, p_{m} \to \infty$ such that
    $p_m / k_{2,m} \to \phi_2 \in (0, 1) \cup (1, \infty]$.
    Then, we have
    \[
        R_{\bX_1, \bY_1, \bX_2, \bY_2}(\tf(\cdot; \cD_{k_{1,m}}, \cD_{k_{2,m}}))
        - R^{\mathrm{g}}_{\bX_1, \bY_1}(\tf(\cdot; \cD_{k_{1,m}}))
        \asto 0,
    \]
    where $R^\mathrm{g}_{\bX_1, \bY_1}(\tf(\cdot; \cD_{k_{1,m}}))$
    is a certain generalized squared prediction risk
    of the predictor $\tf(\cdot; \cD_{k_{1,m}})$,
    fit on the first split data $\cD_{k_{1,m}}$, 
    given by
    \begin{equation}
        \label{eq:generalized_prediction_risk_onestep}
        R^{\mathrm{g}}_{\bX_1, \bY_1}(\tf(\cdot; \cD_{k_{1,m}}))
        = 
        \begin{dcases}
        (\tbeta(\cD_{k_{1,m}}) - \beta_0)^\top
        \Sigma
        (\tbeta(\cD_{k_{1,m}}) - \beta_0)
        + \sigma^2
        & \text{ if } \phi_2 = \infty \\
        (\tbeta(\cD_{k_{1,m}}) - \beta_0)^\top
        g(\Sigma)
        (\tbeta(\cD_{k_{1,m}}) - \beta_0)
        + \sigma^2 \tr[h(\Sigma)] / k_{2,m}
        + \sigma^2
        & \text{ if } \phi \in (1, \infty) \\
        \sigma^2 \frac{1}{1 - \phi_2}
        & \text{ if } \phi \in (0, 1),
        \end{dcases}
    \end{equation}
    where $g(\Sigma)$ and $h(\Sigma)$
    are matrix functions of $\Sigma$
    given explicitly as follows:
    \[
        g(\Sigma)
        = 
        (1 + \tv_g(0; \phi_2))
        (v(0; \phi_2) \Sigma + I_{p_m})^{-1}
        \Sigma
        (v(0; \phi_2) \Sigma + I_{p_m})^{-1},
        \quad
        h(\Sigma)
        =
        \tv(0; \phi_2) (v(0; \phi_2) \Sigma + I)^{-2} \Sigma^2,
    \]
    and $v(0; \phi_2)$, $\tv(0; \phi_2)$, and $\tv_g(0; \phi_2)$
    are as defined in
    \eqref{eq:fixed-point-v-onestep-main-phi2},
    \eqref{eq:tv-onestep-main-phi2},
    and
    \eqref{eq:tvg-onestep-main-phi2},
    respectively.
\end{lemma}

\begin{proof}
    We will start with the functionals
    derived in \Cref{lem:cond-concen-onestep-risk}
    and obtain corresponding asymptotic deterministic equivalents
    conditioned on $\bX_1$ and $\bY_1$
    as $k_{1,m}, k_{2,m}, p_{m} \to \infty$,
    and $p_{m} / k_{2,m} \to \phi \in (0, 1) \cup (1, \infty]$.
    We will split into three cases depending on where $\phi$ falls.
  
  \begin{itemize}[leftmargin=*]
  \item
    \underline{$\phi_2 \in (0, 1)$}:
    When $k_{1,m}, k_{2,m}, p_{m} \to \infty$ 
    such that $p_m / k_{2,m} \to \phi_2 \in (0, 1)$,
    $(I_p - \hSigma_2^\dagger \hSigma) = 0$ almost surely
    and $\tr[\hSigma^\dagger \Sigma] / k_{2,m} - \phi_2 / (1 - \phi_2) \asto 0$,
    as argued in the proof of \Cref{prop:limits-risk-functionals-mn2ls}.

    \item  
    \underline{$\phi \in (1, \infty)$}:
    Next we consider the case when $k_{1,m}, k_{2,m}, p_m \to \infty$,
    such that $p_m / k_{2,m} \to \phi \in (1, \infty)$.
    Consider the bias functional
    $
        (\tbeta(\cD_{k_{1,m}}) - \beta_0)^\top 
        (I_p - \hSigma_2^\dagger \hSigma_2) 
        \Sigma 
        (I_p - \hSigma_2^\dagger \hSigma_2) 
        (\tbeta(\cD_{k_{1,m}}) - \beta_0).
    $
    Invoking Part 1 of \Cref{cor:limiting-resolvents-mn2ls}
    with $f(\Sigma) = \Sigma$,
    as $k_{2,m}, p_m \to \infty$
    such that $p_m / k_m \to \phi_2 \in (1, \infty)$,
    we have
    \[
        (I_p - \hSigma_2^\dagger \hSigma_2)
        \Sigma
        (I_p - \hSigma_2^\dagger \hSigma_2)
        \asympequi
        (1 + \tv_g(0; \phi_2))
        (v(0; \phi_2) \Sigma + I_{p_m})^{-1}
        \Sigma
        (v(0; \phi_2) \Sigma + I_{p_m})^{-1},
    \]
    where $v(0; \phi_2)$ and $\tv_g(0; \phi_2)$
    are as defined in
    \eqref{eq:fixed-point-v-onestep-main-phi2}
    and \eqref{eq:tvg-onestep-main-phi2},
    respectively.
    Now,
    note that the vector $(\tbeta(\cD_{k_{1,m}}) - \beta_0)$
    is independent of $\hSigma_2^\dagger$.
    Thus,
    from the definition of asymptotic equivalence,
    we have
    \[
        (\tbeta(\cD_{k_{1,m}}) - \beta_0)^\top
        (I_p - \hSigma_2^\dagger \hSigma_2)
        \Sigma
        (I_p - \hSigma_2^\dagger \hSigma_2)
        (\tbeta(\cD_{k_{1,m}}) - \beta_0)
        -
        (\tbeta(\cD_{k_{1,m}}) - \beta_0)^\top
        g(\Sigma)
        (\tbeta(\cD_{k_{1,m}}) - \beta_0)
        \asto 0.
    \]
    Consider now the variance resolvent $\hSigma_2^\dagger \Sigma$.
    From Part 2 of \Cref{cor:limiting-resolvents-mn2ls}
    with $f(\Sigma) = \Sigma$,
    as $k_{2,m}, p_m \to \infty$
    such that $p_m / k_{2,m} \to \phi_2 \in (1, \infty)$,
    we have
    \[
        \hSigma_2^\dagger \Sigma
        \asympequi
        \tv(0; \phi_2)
        (v(0; \phi_2) \Sigma + I_{p_m})^{-2} \Sigma^2.
    \]
    Hence, using 
    \Cref{lem:calculus-detequi}~\eqref{lem:calculus-detequi-item-trace},
    we have
    \[
        \sigma^2 \tr[\hSigma_2^\dagger \Sigma] / k_{2,m}
        - \sigma^2 \tr[\tv(0; \phi_2) (v(0; \phi_2) \Sigma + I_{p_m})^{-2} \Sigma^2] / k_{2,m}
        \asto 0.
    \]
 
     \item
    \underline{$\phi_2 = \infty$}:
    Finally, consider the case when $k_{1,m}, k_{2,m}, p_{m} \to \infty$
    and $p_m / k_{2, m} \to \infty$.
    We start by expressing the ingredient one-step estimator 
    \eqref{eq:onestep-ingredient-decomp} as
    \[
        \tbeta(\cD_{k_{1,m}}, \cD_{k_{2,m}})
        = \tbeta(\cD_{k_{1,m}})
        + (\bX_2^\top \bX_2 / k_{2,m})^{\dagger} \bX_2^\top (\bY_2 - \bX_2 \tbeta(\cD_{k_{1,m}})) / k_{2,m}.
    \]
    Using triangle inequality,
    note that
    \begin{align*}
        \| 
            \tbeta(\cD_{k_{1,m}}, \cD_{k_{2,m}})
            - \tbeta(\cD_{k_{1,m}})
        \|_2
        &= 
        \|
            (\bX_2^\top \bX_2 / k_{2,m})^{\dagger}
            \bX_2^\top (\bY_2 - \bX_2 \tbeta(\cD_{k_{1,m}})) / k_{2,m}
        \|_2 \\
        &\le 
        \| (\bX_2^\top \bX_2 / k_{2,m})^{\dagger} \bX_2 / \sqrt{k_{2,m}} \|_{\mathrm{op}}
        \| \bY_2 - \bX_2 \tbeta(\cD_{k_{1,m}}) / \sqrt{k_{2,m}} \|_{2}.
    \end{align*}
    Under the setting of \Cref{lem:cond-concen-onestep-risk},
    the second term in the display above is almost surely bounded.
    Hence, following the proof of \Cref{prop:limits-infty-mn2ls},
    it follows that
    $\| \tbeta(\cD_{k_{1,m}}, \cD_{k_{2,m}}) - \tbeta(\cD_{k_{1,m}}) \|_2 \asto 0$.
    From the analogous reasoning in the proof of \Cref{prop:limits-infty-mn2ls},
    this in turn implies that
    \[
        R_{\bX_1, \bY_1, \bX_2, \bY_2}(\tf(\cdot; \cD_{k_{1,m}}, \cD_{k_{2,m}}))
        - (\tbeta(\cD_{k_{1,m}}) - \beta_0)^\top \Sigma (\tbeta(\cD_{k_{1,m}}) - \beta_0)
        -  \sigma^2
        \asto 0.
    \]
    \end{itemize}
   
   This completes all three cases and finishes the proof. 
\end{proof}

\subsection
{Proof of \Cref{lem:one-step-predrisk-decomposition}}

The idea of the proof is to use the conditional deterministic risk approximation
derived in \Cref{lem:conditional-deterministic-convergence-onestep}
and obtain a limiting expression for the deterministic approximation
in terms of the assumed limiting distribution \eqref{eq:generalied-predrisk-distribution-baseprocedure}.

We start by noting that
\[
    \| \tbeta(\cD_{k_{1,m}}) - \beta_0 \|_2^2
    \le r_{\min}^{-1} \| \tbeta(\cD_{k_{1,m}}) - \beta_0 \|_{\Sigma}^2.
\]
Thus, under the assumption that
there exists a deterministic approximation
$R^\deter(\phi_1; \tf)$
to the conditional risk of $\tf(\cdot; \cD_{k_{1,m}})$
such that
$R(\tf(\cdot; \cD_{k_{1,m}})) \pto R^\deter(\phi_1; \tf)$
as $k_{1,m}, p_m \to \infty$
and $p_m / k_{1,m} \to \phi_1$,
for $\phi_1$ satisfying
$R^\deter(\phi_1; \tf) < \infty$,
it follows that $\limsup_{} \| \tbeta(\cD_{k_{1,m}}) - \beta_0 \|_2 < \infty$.
We can now invoke
\Cref{lem:conditional-deterministic-convergence-onestep}.
Let $k_{2,m} \to \infty$ such that 
$p_m / k_{2,m} \to \phi_2 \in (0, 1) \cup (1, \infty]$.
We will split into various cases depending on $\phi_2$.

\begin{enumerate}
    \item 
    The limit for $\phi_2 = \infty$ is clear
    from 
    the $\phi_2 = \infty$ case in
    \eqref{eq:generalized_prediction_risk_onestep}.
    \item
    When $\phi_2 \in (1, \infty)$,
    we need to obtain 
    limiting expressions for
    the quantities
    $(\tbeta(\cD_{k_{1,m}}) - \beta_0)^\top g(\Sigma) (\tbeta(\cD_{k_{1,m}}) - \beta_0)$
    and
    $\tr[h(\Sigma)] / k_{2,m}
    = \tr[\tv(0; \phi_2) \Sigma^2 (v(0; \phi_2) \Sigma + I)^{-2}] / k_{2,m}$
    in terms of the limiting distributions $Q$ and $H$.
    
    For the former, we start by expanding the quadratic form:
    \begin{align}
        &
        (\tbeta(\cD_{k_{1,m}}) - \beta_0)^\top
        g(\Sigma)
        (\tbeta(\cD_{k_{1,m}}) - \beta_0) \nonumber \\
        &=
        (\tbeta(\cD_{k_{1,m}}) - \beta_0)^\top
        W g(R) W^\top
        (\tbeta(\cD_{k_{1,m}}) - \beta_0)
        \nonumber \\
        &=
        \sum_{i=i}^{p_m}
        ((\tbeta(\cD_{k_{1,m}}) - \beta_0)^\top w_i)^2
        g(r_i) 
        \nonumber \\
        &=
        \sum_{i=1}^{p_m} ((\tbeta(\cD_{k_{1,m}}) - \beta_0)^\top w_i)^2 r_i
        \sum_{i=1}^{p_m} 
        \frac{((\tbeta(\cD_{k_{1.m}}) - \beta_0)^\top w_i)^2 r_i \cdot g(r_i) / r_i}{\sum_{i=1}^{p_m} ((\tbeta(\cD_{k_{1,m}}) - \beta_0)^\top w_i)^2 r_i} 
        \nonumber \\
        &=
        (R(\tf(\cdot; \cD_{1,m})) - \sigma^2)
        \int  \widetilde{g}(r) \, \mathrm{d}\widehat{Q}_n(r),
        \label{eq:gen_predrisk_lim_1}
    \end{align}
    where $\widetilde{g}(r)$ is given by
    \[
        \widetilde{g}(r)
        = \frac{g(r)}{r}
        = (1 + \tv_g(0; \phi_2))
        \frac{1}{(v(0; \phi_2) r + 1)^2}.
    \]
    Under the assumption that $\widehat{Q}_n \dto Q$ in probability,
    we have
    \begin{equation}
        \label{eq:gen_predrisk_lim_innerintegral}
        \int \widetilde{g}(r) \, \mathrm{d}\widehat{Q}_n(r)
        \pto
        \int \widetilde{g}(r) \, \mathrm{d}Q(r)
        =
        \int
        \frac
        {(1 + \tv_g(0; \phi_2))}
        {(v(0; \phi_2) r + 1)^2}
        \,
        \mathrm{d}Q(r).
    \end{equation}
    Observe that $\widetilde{g}$ is continuous.
    Since $R(\tf(\cdot; \cD_{k_{1,m}})) \asto R^\deter(\psi_1; \tf)$,
    from \eqref{eq:gen_predrisk_lim_1} 
    and \eqref{eq:gen_predrisk_lim_innerintegral}, 
    we have
    \begin{align}
        (\tbeta(\cD_{k_{1,m}}) - \beta_0)^\top
        g(\Sigma)
        (\tbeta(\cD_{k_{1,m}}) - \beta_0)
        &\pto
        (R^\deter(\phi_1; \tf) - \sigma^2)
        (1 + \tv_g(0; \phi_2))
        \int \frac{1}{(v(0; \phi_2) r + 1)^2} \, \mathrm{d}Q(r)
        \nonumber \\
        &= 
        R^\deter(\phi_1; \tf)
        \Upsilon_b(\phi_1, \phi_2)
        - \sigma^2 \Upsilon_b(\phi_1, \phi_2),
        \label{eq:gen_predrisk_bias_lim_onestep}
    \end{align}
    where $\Upsilon_b(\phi_1, \phi_2)$
    is as defined in \eqref{eq:def-upsilonb-upsilonv}.
    
    For the latter,
    using \Cref{lem:calculus-detequi}~\eqref{lem:calculus-detequi-item-trace}
    and noting that the integrand is continuous,
    we have
    \begin{align}
        \label{eq:gen_predrisk_var_lim_onestep}
        \tr[h(\Sigma)] / k_{2,m}
        =
        \frac{p_m}{k_{2,m}}
        \tv(0; \phi_2)
        \int
        \frac{r^2}{(1 + v(0; \phi_2) r)^2}
        \, \mathrm{d}H_{p_m}(r)
        &\asto
        \phi_2
        \tv(0; \phi_2)
        \int \frac{\rho^2}{(v(0; \phi_2) r + 1)^2} \, \mathrm{d}H(r) \nonumber \\
        &= \tv_g(0; \phi_2),
    \end{align}
    where $\tv_g(0; \phi_2)$
    is as defined in \eqref{eq:tvg-onestep-main-phi2}.
    
    Putting 
    \eqref{eq:generalized_prediction_risk_onestep},
    \eqref{eq:gen_predrisk_bias_lim_onestep},
    and 
    \eqref{eq:gen_predrisk_var_lim_onestep}
    together, 
    the result follows for $\phi_2 \in (1, \infty)$.
    \item
    The final case of $\phi_2 \in (0, 1)$
    follows analogous argument 
    as in the proof of \Cref{prop:limits-risk-functionals-mn2ls}.
\end{enumerate}

This completes the proof.

\subsection
{Proof of \Cref{cor:verif-onestep-program}}

We will show that there exists 
a deterministic risk approximation 
$R^\deter(\cdot, \cdot; \tf) 
: (0, \infty] \times (0, \infty] \to [0, \infty]$
to the conditional prediction risk $R(\tf(\cdot; \cD_{k_{1,m}}, \cD_{k_{2,m}}))$ 
of the one-step ingredient predictor $\tf(\cdot; \cD_{k_{1,m}}, \cD_{k_{2,m}})$
that satisfies the three-point program 
\ref{prog:onestep-cont-conv}--\ref{prog:onestep-divergence-closedset}.
In particular, we will show that the following $R^\deter(\cdot, \cdot; \tf)$,
that is a continuation of \eqref{eq:detapprox-onestep-general},
satisfies the required conditions:
\[
    R^\deter(\phi_1, \phi_2; \tf)
    = 
    \begin{dcases}
        R^\deter(\phi_1; \tf)
        & \text{ if } \phi_2 = \infty \\
        (R^\deter(\phi_1; \tf) - \sigma^2) 
        \Upsilon_b(\phi_1, \phi_2) 
        + \sigma^2 (1 - \Upsilon_b(\phi_1, \phi_2))
        + \sigma^2 \tv_g(0; \phi_2)
        & \text{ if } \phi_2 \in (1, \infty) \\
        \infty
        & \text{ if } \phi_2 = 1 \\
        \sigma^2 \frac{\phi_2}{1 - \phi_2}
        & \text{ if } \phi_2 \in (0, 1),
    \end{dcases}
\]
where $R^\deter(\cdot; \tf)$
is the assumed deterministic risk approximation
to the conditional prediction risk $R(\tf(\cdot; \cD_{k_{1,m}}))$
of the base predictor $\tf(\cdot; \cD_{k_{1,m}})$,
and $\Upsilon_b(\cdot; \cdot)$ and $\tv_g(0; \cdot)$ 
are as defined in \eqref{eq:def-upsilonb-upsilonv}.
Below we split the three verifications:

\begin{enumerate}
    \item 
    Let 
    $\Phi_1^\infty := \{ \phi_1 \in (0, \infty] : R^\deter(\phi_1; \tf) = \infty \}$
    denote the set of limiting aspect ratios greater than one, where the deterministic risk approximation
    to the base procedure is $\infty$.
    By the hypothesis of \Cref{lem:one-step-predrisk-decomposition},
    we have $R(\tf(\cdot; \cD_{k_{1,m}})) \pto R^\deter(\phi_1; \tf)$
    as $k_{1,m}, p_m \to \infty$ 
    and $p_m / k_{1,m} \to \phi_1 \in (0, \infty] \setminus \Phi_1^\infty$.
    Now observe that $R^\deter(\phi_1, \phi_2; \tf) = \infty$
    only at $\Phi^\infty := \{ (\phi_1, \phi_2) : \phi_1 \in \Phi_1^\infty \text{ or } \phi_2 =  1 \}$.
    This is because 
    $\Upsilon_b(\phi_1, \phi_2), \tv_g(0; \phi_2) < \infty$
    for $\phi_2 \in (1, \infty)$
    from 
    \Cref{lem:fixed-point-v-properties}~\eqref{lem:fixed-point-v-properties-Upsilonb-properties}.
    Note from the conclusion of \Cref{lem:one-step-predrisk-decomposition}
    that
    $R(\tf(\cdot; \cD_{k_{1,m}}, \cD_{k_{2,m}}))
    \pto R^\deter(\phi_1, \phi_2; \tf)$
    as $k_{1,m}, k_{2,m}, p_{m} \to \infty$
    and $(p_m / k_{1,m}, p_m / k_{2,m}) \to (\phi_1, \phi_2) \in (0, \infty] \times (0, \infty] 
    \setminus \Phi^\infty$,
    or in other words, continuous convergence 
    of the risk to the deterministic approximation 
    holds for all limiting $(\phi_1, \phi_2)$
    for which $R^\deter(\phi_1, \phi_2; \tf) < \infty$.
    This verifies \ref{prog:onestep-cont-conv}.
    \item
    From the argument above,
    we have $R^\deter(\phi_1, \phi_2; \tf) = \infty$
    over $\Phi^\infty$.
    Pick any $(\phi_1, \phi_2) \in \Phi^\infty$.
    We will show that 
    $R^\deter(\phi_1', \phi_2'; \tf) \to \infty$
    as $(\phi_1', \phi_2') \to (\phi_1, \phi_2)$.
    From the definition of $\Phi^\infty$,
    the point $(\phi_1, \phi_2)$ falls into
    either of the following two cases:
    \begin{itemize}
        \item $\phi_2 = 1$:
        In this case,
        observe that
        $R^\deter(\phi_1', \phi_2') \to \infty$
        as $(\phi_1', \phi_2') \to (\phi_1, 1^{+})$
        because $\lim_{\phi_2' \to 1^{-}} \phi_2' / (1 - \phi_2') = \infty$,
        and $R^\deter(\phi_1', \phi_2') \to \infty$
        as $(\phi_1', \phi_2') \to (\phi_1, 1^{+})$
        because, from 
        \Cref{lem:fixed-point-v-properties}~\eqref{lem:fixed-point-v-properties-Upsilonb-properties},
        $\lim_{\phi_2' \to 1^{+}} \tv_g(0; \phi_2') = \infty$.
        Thus, 
        $R^\deter(\phi_1', \phi_2') \to \infty$ 
        as $(\phi_1', \phi_2') \to (\phi_1, \phi_2)$.
        \item $\phi_1 \in \Phi_1^\infty$:
        In this case,
        $R^\deter(\phi_1') \to \infty$
        as $\phi_1' \to \phi_1$
        from the assumption that $R^\deter(\cdot; \tf)$
        satisfies \ref{prog:zerostep-cont-infty}.
        Because $\Upsilon_b(\phi_1', \phi_2'), \tv_g(0; \phi_2') > 0$
        over $(\phi_1', \phi_2') \in (0, \infty] \times (1, \infty]$
        from arguments in
        \Cref{lem:fixed-point-v-properties}~\eqref{lem:fixed-point-v-properties-item-tvg-properties}
        and
        \Cref{lem:fixed-point-v-properties}~\eqref{lem:fixed-point-v-properties-Upsilonb-properties},
        it follows that
        \[
        \lim_{(\phi_1', \phi_2') \to (\phi_1, \phi_2)} 
        R^\deter(\phi_1', \phi_2'; \tf)
        = 
        \lim_{\phi_1' \to \phi_1} 
        R^\deter(\phi_1'; \tf) 
        = \infty.
        \]
        Thus, $R^\deter(\phi_1', \phi_2') \to \infty$
        as $(\phi_1', \phi_2') \to (\phi_1, \phi_2)$.
    \end{itemize}
    Therefore,
    whenever $(\phi_1', \phi_2') \to (\phi_1, \phi_2)$,
    we have $R^\deter(\phi'_1, \phi'_2; \tf) \to \infty$,
    and thus $R^\deter(\cdot, \cdot; \tf)$
    satisfies \ref{prog:onestep-cont-infty}.
    \item
    Finally, the set of $(\phi_1, \phi_2)$
    such that $R^\deter(\phi_1, \phi_2; \tf) = \infty$
    is $\Phi^\infty$.
    Because $\Phi^\infty$
    is product of two sets each of which is closed in $\RR$,
    this set is closed in $\RR^2$.
    Therefore, $R^\deter(\cdot, \cdot; \tf)$
    satisfies \ref{prog:onestep-divergence-closedset}.
\end{enumerate}

Put together,
all of \ref{prog:onestep-cont-conv}--\ref{prog:onestep-divergence-closedset}
hold, and this in turn implies that
$\tf(\cdot; \cD_{k_{1,m}}, \cD_{k_{2,m}})$
satisfies \eqref{eq:rn-deterministic-approximation-reduced-onestep-prop-asymptotics}.
This finishes the proof.

\subsection
{Proof of \Cref{prop:verif-riskprofile-mnlsbase-mnlsonestep}}

It suffices to verify the hypothesis of
\Cref{lem:one-step-predrisk-decomposition}
and then appeal to \Cref{cor:verif-onestep-program}.
We will use \Cref{cor:limiting-resolvents-mn2ls} 
along with the
Portmanteau theorem
to certify existence of a limiting distribution $Q$
assumed in \Cref{lem:one-step-predrisk-decomposition}.
The form of $Q$ is defined through
limiting formulas for the generalized prediction risks
of the base predictor.

Let $f$ be any continuous and bounded function.
We will show that
$
    \int f(r) \, \mathrm{d}\widehat{Q}_n(r)
$
converges to a deterministic limit
that is a function of $H$ and $G$,
and show existence of $Q$ through this limit.
We start by noting that
\begin{equation}
    \label{eq:gen_predrisk_mn2ls_expression}
    \int f(r) \, \mathrm{d}\widehat{Q}_n(r)
    = 
    (\tbeta(\cD_{k_{1,m}}) - \beta_0)^\top
    f(\Sigma)
    (\tbeta(\cD_{k_{1,m}}) - \beta_0),
\end{equation}
where $f(\Sigma) = W f(R) W^\top$,
and $f(R)$ is a matrix obtained by applying $f$ component-wise to the diagonal entries
of $R$.
We will now obtain
a limiting expression for the term on the right hand side of 
\eqref{eq:gen_predrisk_mn2ls_expression},
which has the form of a generalized prediction risk of $\tbeta(\cD_{k_{1,m}})$.
Similar to the proof of
\Cref{prop:asymp-bound-ridge-main},
we will first obtain a deterministic equivalent for the generalized
prediction risk.
Following similar steps as in the proof of \Cref{prop:cond-conv-mn2ls},
we have that
\begin{equation}
    \label{eq:gen_predrisk_condconcen}
    (\tbeta(\cD_{k_{1,m}}) - \beta_0)^\top
    f(\Sigma)
    (\tbeta(\cD_{k_{1,m}}) - \beta_0)
    -
    \beta_0^\top
    (I_p - \hSigma_1^\dagger \hSigma_1)
    f(\Sigma)
    (I_p - \hSigma_1^\dagger \hSigma_1)
    \beta_0
    + \tr[\hSigma_1^\dagger f(\Sigma)] / k_{1,m}
    \asto 0.
\end{equation}
Now, using first part of \Cref{cor:limiting-resolvents-mn2ls},
we can write
\[
    (I_p - \hSigma_1^\dagger \hSigma_1)
    f(\Sigma)
    (I_p - \hSigma_1^\dagger \hSigma_1)
    \asympequi
    (1 + \tv_g(0; \phi_1))
    (v(0; \phi_1) \Sigma + I_{p_m})^{-1}
    \Sigma
    (v(0; \phi_1) \Sigma + I_{p_m})^{-1}.
\]
Using Property 4 of \Cref{sec:calculus_deterministic_equivalents},
this then yields
\begin{equation}
    \label{eq:gen_predrisk_bias_lim}
    \beta_0^\top
    (I_p - \hSigma_1^\dagger \hSigma_1)
    f(\Sigma)
    (I_p - \hSigma_1^\dagger \hSigma_1)
    \beta_0
    \asto
    (1 + \tv_g(0; \phi_1))
    \int \frac{f(r)}{(v(0; \phi_1) r + 1)^2} \, \mathrm{d}G(r).
\end{equation}
Similarly,
using second part of \Cref{cor:limiting-resolvents-mn2ls},
we have
\[
    \hSigma_1^\dagger f(\Sigma)
    \asympequi
    \tv(0; \phi_1)
    (v(0; \phi_1) \Sigma + I_{p_m})^{-2}
    \Sigma
    f(\Sigma).
\]
Hence,
appealing to Property 4 of \Cref{sec:calculus_deterministic_equivalents} again,
we have
\begin{equation}
    \label{eq:gen_predrisk_var_lim}
    \tr[\hSigma_1^\dagger f(\Sigma)] / k_{1,m}
    \asto
    \phi_1
    \tv(0; \phi_1)
    \int \frac{r f(r)}{(v(0; \phi_1) r + 1)^2} \, \mathrm{d}H(r).
\end{equation}
Therefore,
from 
\eqref{eq:gen_predrisk_mn2ls_expression}--\eqref{eq:gen_predrisk_var_lim},
it follows that
\[
    \int f(r) \, \mathrm{d}\widehat{Q}_n(r)
    \asto
    (1 + \tv_g(0; \phi_1))
    \int \frac{f(r)}{(v(0; \phi_1) r + 1)^2} \, \mathrm{d}G(r)
    +
    \phi_1
    \tv(0; \phi_1)
    \int \frac{r f(r)}{(v(0; \phi_1) r + 1)^2} \, \mathrm{d}H(r).
\]
Observe
that this defines a distribution $Q$
because
one can take $f(r) = e^{itr} = \cos(tr) + i \sin(tr)$,
which then implies
convergence of the characteristic function
at all points.
This finishes the proof.
To get more insight 
into the risk behaviour of the ingredient 
one-step predictor,
we can also write out an explicit
formula for the deterministic approximation
$R^\deter(\cdot, \cdot; \tf_{\mnls})$.
We will do so below.

For the particular functional
$R(\tf(\cdot; \cD_{k_{1,m}}, \cD_{k_{2,m}}))$,
we have a specific $f$ given by
\[
    f(r)
    = (1 + \tv_g(0; \phi_2))
    \frac{r}{(v(0; \phi_2) r + 1)^2}.
\]
Thus,
the final expression for $R^\deter(\phi_1, \phi_2)$
can be written explicitly as follows:
\[
    R^\deter(\phi_1, \phi_2)
    =
    \begin{dcases}
        R^\deter(\min\{\phi_1, \phi_2\})
        & \text{ if } 
        \phi_1 = \infty \text{ or } \phi_2 = \infty
        \\
        \rho^2
        (1 + \tv_g(0; \phi_1, \phi_2))
        (1 + \tv_g(0; \phi_2))
        \int \frac{r}{(1 + v(0; \phi_1) r)^2 (1 + v(0; \phi_2) r)^2} \, \mathrm{d}G(r) \\
        \quad + ~ 
        \sigma^2 
            (1 + \tv_g(0; \phi_2))
            \phi_1
            \tv(0; \phi_1)
            \int \frac{r}{(v(0; \phi_1) r + 1)^2 (v(0; \phi_2) r + 1)^2} \, \mathrm{d}H(r) \\
        \qquad
            + ~
            \sigma^2
            \left(
            \phi_2
            \tv(0; \phi_2)
            \int \frac{r}{(1 + v(0; \phi_2) r)^2} \, \mathrm{d}H(r)
            + 1
            \right)
        & \text{ if } (\phi_1, \phi_2) \in (1, \infty) \times (1, \infty) \\
        \sigma^2
        \left(
            \phi_2
            \tv(0; \phi_2)
            \int \frac{r}{(1 + v(0; \phi_2) r)^2}
            \, \mathrm{d}H(r)
            +
            1
        \right)
        & \text{ if } (\phi_1, \phi_2) \in (0, 1) \times (1, \infty) \\
        \sigma^2
        \frac{1}{1 - \phi_2}
        & \text{ if } (\phi_1, \phi_2) \in (0, \infty) \times (0, 1),
    \end{dcases}
    \]
where
$v(0; \phi)$ is as defined in \eqref{eq:fixed-point-v-mn2ls-1},
$\tv(0; \phi)$ is as defined in \eqref{eq:fixed-point-v'-mn2ls-1},
$\tv_g(0; \phi)$ is as defined in \eqref{eq:def-tvg-mn2ls-1},
and
$\tv_g(0; \phi_1, \phi_2)$ is as defined below:
\[
    \tv_g(0; \phi_1, \phi_2)
    =
    \ddfrac
    {(1 + \tv_g(0; \phi_2)) \phi_1 \int \frac{r^2}{(1 + v(0; \phi_2) r)^2 (1 + v(0; \phi_1) r)^2} \, \mathrm{d}H(r) }
    {\frac{1}{v(0; \phi_1)^2}  - \phi_1 \int \frac{r^2}{(1 + v(0; \phi_1) r)^2} \, \mathrm{d}H(r) }.
\]
Here, $R^\deter(\cdot)$ is $R^\deter(\cdot; \tf_{\mnls})$
as defined in \eqref{eq:detapprox-formula-zerostep-mn2ls}.

\subsection
{Proof of \Cref{prop:verif-riskprofile-mn1lsbase-mn2lsonestep}}

Verification of the hypothesis of \Cref{lem:one-step-predrisk-decomposition}
is easy in this case  because $\Sigma = I_p$.
Observe that 
under \ref{asm:iso-gau-feat-mn1ls},
the distribution $\hQ_n$
is simply a point mass at $1$.
Thus,
the hypothesis of \Cref{lem:one-step-predrisk-decomposition}
is trivially satisfied.
Moreover,
we can explicitly write
expressions for the functions
$\tv_g(0; \cdot)$
and
$\Upsilon_b(\cdot; \cdot)$.
Towards that end, we will first obtain
expressions for the ingredient functions
$v(0; \cdot)$ and $\tv(0; \cdot)$.

\begin{itemize}
    \item 
    \underline{$v(0; \phi_2)$}:
    The fixed-point equation
    \eqref{eq:fixed-point-v-onestep-main-phi2}
    can be solved explicitly since $H$ is a point mass at $1$.
    The fixed-point equation 
    in this case simplifies to
    \begin{equation}
        \label{eq:fixed-point-v-iso-phi2}
        \frac{1}{v(0; \phi_2)}
        =
        \phi_2 \frac{1}{v(0; \phi_2) + 1}.
    \end{equation}
    Solving \eqref{eq:fixed-point-v-iso-phi2}
    for $v(0; \phi_2)$, we get
    \begin{equation}
        \label{eq:v-phi2-isotropic}
        v(0; \phi_2) = \frac{1}{\phi_2 - 1},
        \quad
        \text{and}
        \quad
        1 + v(0; \phi_2)
        = \frac{\phi_2}{\phi_2 - 1}.
    \end{equation}
    \item
    \underline{$\tv(0; \phi_2)$}:
    Using \eqref{eq:v-phi2-isotropic},
    we can compute the inverse of $\tv(0; \phi_2)$
    per \eqref{eq:tv-onestep-main-phi2} as
    \[
        \tv(0; \phi_2)^{-1}
        =
        (\phi_2 - 1)^2
        -
        \phi_2
        \frac{(\phi_2 - 1)^2}{\phi_2^2}
        = 
        (\phi_2 - 1)^2
        - \frac{(\phi_2 - 1)^2}{\phi_2}
        = (\phi_2 - 1)^2 \frac{\phi_2 - 1}{\phi_2}
        = \frac{(\phi_2 - 1)^3}{\phi_2}.
    \]
    Thus, we have 
    \begin{equation}
        \label{eq:tv-phi2-isotropic}
        \tv(0; \phi_2) = \frac{\phi_2}{(\phi_2 - 1)^3},
        \quad
        \text{and}
        \quad
        \tv(0; \phi_2) \phi_2
        = \frac{\phi_2^2}{(\phi_2 - 1)^3}.
    \end{equation}
    
\end{itemize}

\noindent Using \eqref{eq:v-phi2-isotropic} and \eqref{eq:tv-phi2-isotropic},
we can explicitly write out
expressions
for 
$\Upsilon_b(\phi_1, \phi_2)$
and
$\tv_g(0; \phi_2)$.
\begin{itemize}
    \item 
    \underline{$\tv_g(0; \phi_2)$}:
    Substituting \eqref{eq:v-phi2-isotropic} and \eqref{eq:tv-phi2-isotropic}
    into
    \eqref{eq:tvg-onestep-main-phi2},
    we obtain
    \begin{equation}
        \label{eq:tvg-phi2-isotropic}
        \tv_g(0; \phi_2)
        =
        \frac{\phi_2^2}{(\phi_2 - 1)^3}
        \frac{(\phi_2 - 1)^2}{\phi_2^2}
        = \frac{1}{\phi_2 - 1},
        \quad
        \text{and}
        \quad
        (1 + \tv_g(0; \phi_2))
        = \frac{\phi_2}{\phi_2 - 1}.
    \end{equation}
    \item
    \underline{$\Upsilon_b(\phi_1, \phi_2)$}:
    Substituting \eqref{eq:v-phi2-isotropic} and \eqref{eq:tv-phi2-isotropic}
    into \eqref{eq:def-upsilonb-upsilonv},
    we get
    \begin{equation}
        \label{eq:Upsilonb-phi2-isotropic}
            \Upsilon_b(\phi_1, \phi_2)
            =
            \frac{\phi_2}{\phi_2 - 1}
            \frac{(\phi_2 - 1)^2}{\phi_2^2}
            =
            \frac{\phi_2 - 1}{\phi_2},
            \quad
            \text{and}
            \quad
            1 - \Upsilon_b(\phi_1, \phi_2)
            = \frac{1}{\phi_2}.
    \end{equation}
    Observe that since the distribution $Q$ does not depend on $\phi_1$ in this case,
    $\Upsilon_b(\phi_1, \phi_2)$ in turn also does not depend on $\phi_1$.
\end{itemize}

Therefore,
using \eqref{eq:tvg-phi2-isotropic} and \eqref{eq:Upsilonb-phi2-isotropic},
the deterministic risk approximation
from \eqref{eq:detapprox-onestep-general}
simplifies in this case as follows:
\begin{equation*}
    R^\deter(\phi_1, \phi_2; \tf)
    \to
    \begin{dcases}
         \rho^2 + \sigma^2
        & \text{ if }
        \phi_1 = \phi_2 = \infty \\
        R^\deter(\phi_1)
        & \text{ if }
        \phi_2 = \infty \\
        \rho^2 \left(1 - \frac{1}{\phi_2}\right) 
        + \sigma^2 \left(\frac{1}{\phi_2 - 1}\right)
        + \sigma^2
        & \text{ if }
        \phi_1 = \infty \\
        R^\deter(\phi_1)
        \left(1 - \frac{1}{\phi_2}\right)
        + \sigma^2 \left(\frac{1}{\phi_2 - 1}\right)
        + \sigma^2
        & \text{ if } (\phi_1, \phi_2) \in (1, \infty) \times (1, \infty) \\
        \sigma^2 \left(\frac{\phi_1}{1 - \phi_1}\right)
        \left(1 - \frac{1}{\phi_2}\right)
        + \sigma^2 \left(\frac{1}{\phi_2 - 1}\right)
        + \sigma^2
        & \text{ if } (\phi_1, \phi_2) \in (0, 1) \times (1, \infty) \\
        \sigma^2 \left(\frac{\phi_2}{1 - \phi_2}\right)
        + \sigma^2
        & \text{ if } (\phi_1, \phi_2) \in (0, \infty) \times (0, 1).
    \end{dcases}
\end{equation*}
Here, $R^\deter(\cdot)$
is $R^\deter(\cdot; \tf_{\mnla})$ as defined in \eqref{eqn:def-R-mn1ls}.

\section{Technical helper lemmas, proofs, and miscellaneous details}
\label{sec:useful-lemmas}

In this section,
we gather various technical lemmas along with their proofs,
and other miscellaneous details.
Specific pointers to which lemmas are used in which proofs
are provided at the start of each section.

\subsection
{Lemmas for verifying space-filling properties of discrete optimization grids}

In this section,
we collect supplementary lemmas that are used in the proofs of 
\Cref{thm:asymptotic-risk-tuned-zero-step,thm:asymptotic-risk-tuned-one-step}
in \Cref{sec:proofs-riskmonotonization-zerostep,sec:proofs-riskmonotonization-onestep},
respectively.

\begin{lemma}
    [Verifying space-filling property of the discrete grid used in the zero-step procedure]
    \label{lem:spacefilling-grid-zerostep-general}
    Let $\{ p_n \}$, $\{ m_{1,n} \}$,
    $\{ m_{2,n} \}$
    are three sequences of positive integers such that
    $m_{2,n} \le m_{1,n}$
    for $n \ge 1$.
    Suppose 
    \[
        \frac{p_n}{m_{1,n}} \to \gamma \in (0, \infty)
        \quad
        \text{and}
        \quad
        \frac{m_{2,n}}{m_{1,n}} \to 0
    \]
    as $n \to \infty$.
    Define
    a sequence of grids $\cG_n$
    as follows:
    \[
        \cG_n
        :=
        \left\{
            \frac{p_n}{m_{1,n} - k m_{2,n}}:
            1 \le k \le
            \left\lceil \frac{m_{1,n}}{m_{2,n}} - 2 \right\rceil
        \right\}.
    \]
    Then,
    for any $\zeta^\star \in [\gamma, \infty]$,
    $\Pi_{\cG_n}(\zeta^\star) \to \zeta^\star$
    as $n \to \infty$,
    where $\Pi_{\cG_n}(y) = \argmin_{x \in \cG_n} |y - x|$
    is the point in the grid $\cG_n$ closest to $y$.
    In particular,
    in the context of
    \Cref{alg:zero-step},
    taking $m_{1,n} = n_\train$
    and $m_{2,n} = \lfloor n^{\nu} \rfloor$ for $\nu \in (0, 1)$,
    we get
    the aspect ratios used in \Cref{alg:zero-step}
    ``converge'' to $[\gamma, \infty]$
    when $n_\train / n \to 1$ under \ref{asm:prop_asymptotics}.
\end{lemma}

\begin{proof}
    We will consider different cases depending
    on where $\zeta^\star \in [\gamma, \infty]$ lands. See \Cref{fig:zerostep_grid_projection_illustration}.

    \begin{figure}[!ht]
        \centering
        \includegraphics[width=0.9\columnwidth]{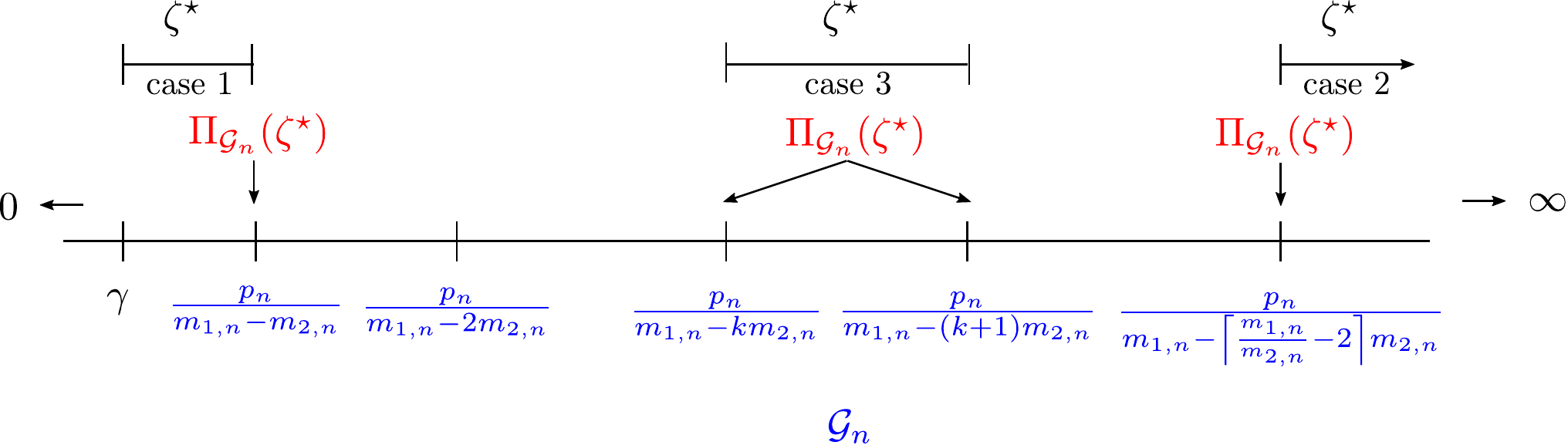}
        \caption{Illustration of different cases of 
        $\zeta \in [\gamma, \infty]$ and the corresponding projection $\Pi_{\cG_n}(\zeta^\star)$.}
        \label{fig:zerostep_grid_projection_illustration}
    \end{figure}    
    
    \begin{enumerate}
    \item
    Consider the first case when
    \[
        \gamma
        \le
        \zeta^\star
        \le
        \frac{p_n}{m_{1,n} - m_{2,n}}.
    \]
    In this case,
    $\Pi_{\cG_n}(\zeta^\star)$
    is simply the first point in the grid.
    Observe that in this case
    \[
        \Pi_{\cG_n}(\zeta^\star)
        - \zeta^\star
        \le
        \frac{p_n}{m_{1,n} - m_{2,n}}
        - \gamma
        =
        \ddfrac
        {\frac{p_n}{m_{1,n}}}
        {1 - \frac{m_{2,n}}{m_{1,n}}}
        - \gamma
        \to
        \gamma - \gamma
        = 0
    \]
    as $n \to \infty$
    under the assumptions that $p_n / m_{1,n} \to \gamma$
    and $m_{2,n} / m_{1,n} \to 0$.
    
    \item
    Consider the second case when
    \[
        \ddfrac{p_n}
        {m_{1,n} - 
        \left\lceil
            \frac{m_{1,n}}{m_{2,n}} - 2
        \right\rceil
        }
        \le
        \zeta^\star 
        \le \infty.
    \]
    In this case,
    $\Pi_{\cG_n}(\zeta^\star)$
    is simply the last point in the grid.
    We will show eventually the only $\zeta^\star$
    in this case is $\zeta^\star = \infty$.
    Note that $p_n/(m_{1,n} - k m_{2,n})$ increases with $k \ge 0$. If $\zeta^\star = \infty$, then $\Pi_{\mathcal{G}_n}(\zeta^\star) = p_n/(m_{1,n} - k^\star  m_{2,n})$ for $k^* = \lceil m_{1,n}/ m_{2,n} - 2\rceil$. Hence, it suffices to prove that $p_n/(m_{1,n} - k^{\star} m_{2,n}) \to \infty$ as $n\to\infty$. This follows from the fact that
    \[
    \frac{m_{1,n}}{ m_{2,n}} - \left\lceil \frac{m_{1,n}}{ m_{2,n}} - 2\right\rceil \le 2,\quad\mbox{and}\quad \frac{p_n}{m_{1,n} - k^{\star} m_{2,n}} = \frac{p_n}{ m_{2,n}(m_{1,n}/ m_{2,n} - \lceil m_{1,n}/ m_{2,n} - 2\rceil)} \ge \frac{p_n}{2 m_{2,n}} \to \infty = \zeta^*,
    \]
    as $n\to\infty$ 
    and
    $p_n / m_{1,n} \to \gamma \in (0, \infty)$.
    
    \item 
    Consider the third case when
    \begin{equation}
        \label{eq:spacefilling-zerostep-general-case3}
        \frac{p_n}{m_{1,n} - k m_{2,n}}
        \le
        \zeta^\star
        \le
        \frac{p_n}{m_{1,n} - (k+1) m_{2,n}}
        \quad
        \text{for some }
        1 \le k \le
        \left\lceil \frac{m_{1,n}}{m_{2,n}} - 2 \right\rceil.
    \end{equation}
    From the first inequality in 
    \eqref{eq:spacefilling-zerostep-general-case3},
    we have
    \begin{equation}
        \label{eq:spacefilling-zerostep-general-case3-lb}
        \frac{p_n}{m_{1,n} - k m_{2,n}} 
        \le \zeta^\star
        \implies
        \frac{p_n}{m_{1,n} \zeta^\star}
        \le
        1 - k \frac{m_{2,n}}{m_{1,n}}
        \implies
        k \frac{m_{2,n}}{m_{1,n}}
        \le
        1 - \frac{p_n}{m_{1,n} \zeta^\star}.
    \end{equation}
    Similarly,
    from the second inequality of 
    \eqref{eq:spacefilling-zerostep-general-case3},
    we have
    \begin{equation}
        \label{eq:spacefilling-zerostep-general-case3-ub}
        \frac{p_n}{m_{1,n} \zeta^\star}
        \ge
        1 - \frac{(k+1) m_{2,n}}{m_{1,n}}
        \implies
        k \frac{m_{2,n}}{m_{1,n}}
        \ge
        1 - \frac{p_n}{m_{1,n} \zeta^\star}
        - \frac{m_{2,n}}{m_{1,n}}.
    \end{equation}
    The upper and lower bounds from
    \eqref{eq:spacefilling-zerostep-general-case3-ub}
    and
    \eqref{eq:spacefilling-zerostep-general-case3-lb}
    together imply that
    \[
        1 - \frac{p_n}{m_{1,n} \zeta^\star} - \frac{m_{2,n}}{m_{1,n}}
        \le
        \frac{k m_{2,n}}{m_{1,n}}
        \le
        1 - \frac{p_n}{m_{1,n} \zeta^\star}.
    \]
    Because $\lim_{n \to \infty} m_{2,n} / m_{1,n} = 0$,
    we conclude that
    \begin{equation}
        \label{eq:spacefilling-zerostep-general-ratio-lim}
        \lim_{n \to \infty}
        \frac{k m_{2,n}}{m_{1,n}}
        = 1 - \frac{\gamma}{\zeta^\star}
        \in (0, 1).
    \end{equation}
    Now, note that
    since $\Pi_{\cG_n}(\zeta^\star)$
    is either of the two points of the grid partition,
    we have
    \begin{align*}
        | \Pi_{\cG_n}(\zeta^\star) - \zeta^\star |
        &\le
        \frac{p_n}{m_{1,n} - (k+1) m_{2,n}}
        - \frac{p_n}{m_{1,n} - k m_{2,n}} \\
        &=
        \frac{p_n}{m_{1,n} - (k+1) m_{2,n}}
        \frac{m_{2,n}}{m_{1,n} - k m_{2,n}} \\
        &=
        \ddfrac
        {\frac{p_n}{m_{1,n}}}
        {1 - \frac{(k+1) m_{2,n}}{m_{1,n}}}
        \ddfrac
        {\frac{m_{2,n}}{m_{1,n}}}
        {1 - \frac{k m_{2,n}}{m_{1,n}}} \\
        &\to
        \ddfrac
        {\gamma}
        {1 - \left( 1 - \frac{\gamma}{\zeta^\star} \right)}
        \ddfrac
        {0}
        {\left(1 - \left( 1 - \frac{\gamma}{\zeta^\star} \right) \right)}
        = 0,
    \end{align*}
    as $n \to \infty$
    and $p_n / m_{1,n} \to \gamma$
    and $m_{2,n} / m_{1,n} \to 0$,
    where the limiting in the convergences on the last line
    follow from \eqref{eq:spacefilling-zerostep-general-ratio-lim}.
    \end{enumerate}
    
    This completes all the cases.
    
    Finally, observe that
    for \Cref{alg:zero-step},
    when $m_{2,n} = \lfloor n^\nu \rfloor$ for some $\nu \in (0, 1)$
    and $m_{1,n} = n_\train$ such that $n_\train / n \to 1$
    as $n \to \infty$,
    $p_n / m_{1,n} \to \gamma \in (0, \infty)$,
    and $m_{2,n} / m_{1,n} \to 0$,
    and hence the statement follows.
    
\end{proof}

\begin{lemma}
    [Verifying space-filling property of the discrete grid 
    used in the one-step procedure]
    \label{lem:spacefilling-grid-onestep-general}
    Let $\{ p_n \}$, $\{ m_{1,n} \}$, $\{ m_{2,n} \}$
    are three sequences of positive integers
    such that $m_{2,n} \le m_{1,n}$
    for $n \ge 1$,
    and $n \to \infty$,
    \[
        \frac{p_n}{m_{1,n}} \to \gamma \in (0, \infty)
        \quad
        \text{and}
        \quad
        \frac{m_{2,n}}{m_{1,n}} \to 0.
    \]
    Define a sequence of grids $\cG_n$ as follows:
    \[
        \cG_n
        :=
        \left\{
            \left( 
                \frac{p_n}{m_{1,n} - k_1 m_{2,n}},
                \frac{p_n}{k_2 m_{2,n}}
            \right)
            :
            k_1 \in
            \left\{
                2, \dots,
                \left\lceil
                    \frac{m_{1,n}}{m_{2,n}}
                    - 2
                \right\rceil
            \right\},
            k_2 \in
            \{0, \dots, k_1 - 1 \}
        \right\}.
    \]
    Let $\zeta_1^\star$
    and $\zeta_2^\star$
    be two non-negative real numbers
    such that
    \[
        \frac{1}{\zeta_1^\star}
        + \frac{1}{\zeta_2^\star}
        \le \frac{1}{\gamma}.
    \]
   Let $\Pi_{\cG_n}(\zeta_1^\star, \zeta_2^\star)
   = (\pi_{1,n}, \pi_{2,n})$ 
   denote the projection of the point $(\zeta_1^\star, \zeta_2^\star)$
   on the grid $\cG_n$ with respect to the $\ell_1$ distance.
    Then,
    $\pi_{1,n} \to \zeta_1^\star$
    and $\pi_{2,n} \to \zeta_2^\star$
    as $n \to \infty$.
    In particular,
    in the context of \Cref{alg:one-step},
    taking $m_{1,n} = n_\train$,
    $m_{2,n} = \lfloor n^\nu \rfloor$ for some $\nu \in (0, 1)$,
    we get the aspect ratios 
    used in \Cref{alg:one-step}
    ``converge'' to
    the set $\{ (\zeta_1, \zeta_2) : \zeta_1^{-1} + \zeta_2^{-1} \le \gamma^{-1} \}$
    when $n_\train / n \to 1$ under \ref{asm:prop_asymptotics}.
\end{lemma}

\begin{proof}
    The proof follows the general strategy employed in the proof \Cref{lem:spacefilling-grid-zerostep-general}
    and uses the result as ingredient.
    
    Fix any point $(\zeta_1^\star, \zeta_2^\star)$
    that satisfies the constraint
    \[
        \frac{1}{\zeta_1^\star}
        + \frac{1}{\zeta_2^\star}
        \le \frac{1}{\gamma}.
    \]
    We will construct a pair $(g_1^\star, g_2^\star)$
    in the grid $\cG_n$ such that
    $(g_1^\star, g_2^\star) \to (\zeta_1^\star, \zeta_2^\star)$.
    Because
    \[
      \| \Pi_{\cG_n}(\zeta_1^\star, \zeta_2^\star)
      - (\zeta_1^\star, \zeta_2^\star) \|_{\ell_1}
     \le \| (g_1^\star, g_2^\star) - (\zeta_1^\star, \zeta_2^\star) \|_{\ell_1},
    \]
    such a choice shows the desired result. 

    Define
      \[
          (k_1^\star, k_2^\star)
          =
          \left(
                \left\lceil \frac{m_{1,n} - p_n / \zeta_1^\star}{m_{2,n}} \right\rceil,
                \left\lfloor \frac{p_n / \zeta_2^\star}{m_{2,n}} \right\rfloor 
          \right),
          \quad
          \text{and}
          \quad
            (g_1^\star, g_2^\star)
            =
            \left(
                \frac{p}{m_{1,n} - k_1^\star m_{2,n}} ,
                \frac{p}{k_2^\star m_{2,n}}
            \right).
      \]
    By appealing to 
    \Cref{lem:spacefilling-grid-zerostep-general},
    it follows that $\pi_{1,n} \to \zeta_1^\star$
    as $n \to \infty$.
    Note that the value of $k_1^\star$
    is exactly the right point of the grid interval in 
    \Cref{fig:zerostep_grid_projection_illustration} in
    the proof of
    \Cref{lem:spacefilling-grid-zerostep-general}.
    Since $\zeta_1^\star \in [\gamma, \infty]$
    and the first coordinate of the grid $\mathcal{G}_n$ is the same
    as that in 
    \Cref{lem:spacefilling-grid-zerostep-general},
    we have that $g_1^\star$ is a feasible choice and $g_1^\star \to \zeta_1^\star$.
    It remains to verify the conditions for $g_2^\star$.

    Note that when $\zeta_2^\star = \infty$, $k_2^\star = 0$,
    which satisfies the desired condition. Assume that $\zeta_2^\star < \infty$.
    We verify below that $k_2^\star < k_1^\star$ so that $k_2^\star$
    is a feasible choice and that
    \[
        \frac{k_2^\star m_{2,n}}{p_n}
        \to \frac{1}{\zeta_2^\star},
    \]
    which implies the desired convergence of the reciprocal.
    
    Observe that
    \begin{align*}
        k_2^\star
        \le \frac{p_n}{\zeta_2^\star m_{2,n}} 
        \le \frac{p_n}{m_{2,n}}
        \left( \frac{m_{1,n}}{p_n} - \frac{1}{\zeta_1^\star} \right) 
        \le \frac{m_{1,n} - p_n / \zeta_1^\star}{m_{2,n}} = k_1^\star.
    \end{align*}
    This verifies the first condition.
    For the second part, consider
    \begin{align*}
        0 \le \left|
            \frac{k_2^\star m_{2,n}}{p_n} - \frac{1}{\zeta_2^\star}
        \right|
        =
        \left|
            \left\lfloor \frac{p_n/\zeta_2^\star}{m_{2,n}} \right\rfloor
            \frac{m_{2,n}}{p_n} - \frac{1}{\zeta_2^\star}
        \right| 
        \le 
        \frac{m_{2,n}}{p_n} 
        \to 0
    \end{align*}
    under \ref{asm:prop_asymptotics} as $n \to \infty$.
    
    Finally, note that
    for \Cref{alg:one-step},
    when $m_{2,n} = \lfloor n^\nu \rfloor$ for some $\nu \in (0, 1)$
    and $m_{1,n} = n_\train$ such that $n_\train / n \to 1$
    as $n \to \infty$,
    $p_n / m_{1,n} \to \gamma \in (0, \infty)$,
    and $m_{2,n} / m_{1,n} \to 0$,
    and therefore the statement follows.
        
\end{proof}

\subsection
[Lemmas for restricting arbitrary sequences to specific convergent sequences]
{Lemmas for restricting arbitrary sequences to specific convergent sequences}

In this section,
we collect supplementary lemmas that are used in the proofs of
\Cref{lem:rn-deterministic-approximation-4-prop-asymptotics,lem:deterministic-approximation-reduction-onestep}
in \Cref{sec:proofs-riskmonotonization-zerostep,sec:proofs-riskmonotonization-onestep},
respectively.

\begin{lemma}
    [From subsequence convergence
    to sequence convergence]
    \label{lem:subsequence-to-sequence}
    Let $\{ a_m \}_{m \ge 1}$ be a sequence in $\RR$.
    Suppose for any subsequence $\{ a_{m_k} \}_{k \ge 1}$,
    there is a further subsequence $\{ a_{m_{k_l}} \}_{l \ge 1}$
    such that $\lim_{m \to \infty} a_{m_{k_l}} = 0$.
    Then $\lim_{m \to \infty} a_m = 0$.
\end{lemma}
\begin{proof}
   Let $\alpha := \limsup_{m \to \infty} a_m$
   and $\beta := \liminf_{m \to \infty} a_m$.
   This means that there is subsequence $\{ a_{m_k} \}_{k \ge 1}$
   such that $\lim_{m \to \infty }a_{m_k} = \alpha$.
   Similarly, there is a (different) subsequence $\{ a_{m_l} \}_{l \ge 1}$
   such that $\lim_{m \to \infty} a_{m_l} = \beta$.
   But since every converging sequence has a further subsequence
   that converges to the same limit,
   the lemma follows.
\end{proof}

\begin{lemma}
   [Limit of minimization over finite grids in a metric space]
   \label{lem:grid-minimization-metric-space}
   Let $(M, d)$ be a metric space,
   and $C$ be a subset of $M$.
   Suppose $h : M \to \RR$ is a function
   that attains its infimum over $C$ at $\zeta^\star$.
   Let $\cG$ be a finite set of points in $C$.
   Then,
   the following inequalities hold:
   \begin{equation}\label{eq:grid-minimization-inequalities}
        0
        \le \min_{x \in \cG} h(x) - \inf_{x \in C} h(x)
        \le h(\Pi_\cG(\zeta^\star)) - h(\zeta^\star),
   \end{equation}
   where $\Pi_\cG(y) = \argmin_{x \in \cG} d(x, y)$
   is the point in the grid closest to $y$.
   Consequently,
   if $\cG_n$ is a sequence of grids such that
   $\Pi_{\cG_n}(\zeta^\star) \to \zeta^\star$,
   and $h(\cdot)$ is continuous at $\zeta^\star$,
   then
   \begin{equation}\label{eq:grid-minimization-limit}
        \min_{x \in \cG_n} h(x)
        - \inf_{x \in C} h(x) \to 0.
   \end{equation}
\end{lemma}

\begin{proof}
    Since $\cG \subseteq C$ and $\Pi_\cG(\zeta^\star) \in \cG$,
    we have the following chain of inequalities:
    \[
        h(\zeta^\star)
        = \inf_{x \in C} h(x)
        \le \min_{x \in \cG} h(x)
        \le h(\Pi_\cG(\zeta^\star)).
    \]
    Subtracting $h(\zeta^\star)$ throughout,
    we get the desired result
    \eqref{eq:grid-minimization-inequalities}.
    In addition, if $\cG_n$ is a sequence of grids such that
    $\Pi_\cG(\zeta^\star) \to \zeta^\star$,
    then
    continuity of $h(\cdot)$ at $\zeta^\star$ implies
    $h(\Pi_\cG(\zeta^\star)) \to h(\zeta^\star)$
    leading to
    \eqref{eq:grid-minimization-limit}.
\end{proof}

\begin{lemma}
    [Limit points of argmin sequence over space-filling grids]
    \label{lem:limit-argmins-metricspace}
    Let $(M, d)$ be a metric space
    and $C$ be a compact subset of $M$.
    Let $\cG_n$ be a sequence of grids such that
    for any $\zeta \in C$, $\Pi_{\cG_n}(\zeta) \to \zeta$
    as $n \to \infty$
    where $\Pi_{\cG_n}(y) = \argmin_{x \in \cG_n} d(x, y)$
    is the point in the grid $\cG_n$ closest to $y$.
    Let $h : C \to [0, \infty]$ be a lower semicontinuous function,
    and let $x_n \in \argmin_{x \in \cG_n} h(x)$.
    Then,
    for any arbitrary subsequence $\{ x_{n_k} \}_{k \ge 1}$ of $\{ x_n \}_{n \ge 1}$,
    there exists a further subsequence $\{ x_{n_{k_l}} \}_{l \ge 1}$ 
    such that
    $x_{n_{k_l}}$ converges to a point in $\argmin_{\zeta \in C} h(\zeta)$
    as $l \to \infty$.
\end{lemma}
\begin{proof}
    Because $h$ is lower semicontinuous
    and $C$ is compact,
    $h$ attains its minimum on $C$
    (see, e.g., Section 1.6 of \cite{pedersen_2012}
    and also see Theorem 1.9 of \citet{rockafellar_wets_2009}
    with the domain $\RR^{n}$ replaced with any metric space.).
    Let $\cM = \argmin_{\zeta \in C} h(\zeta)$,
    which is non-empty.
    Because $C$ is compact,
    for any arbitrary subsequence $\{ x_{n_k} \}_{k \ge 1}$,
    there is a further subsequence $\{ x_{n_{k_{l}}} \}_{l \ge 1}$
    that converges to some point $p \in C$.
    Lower semicontinuity of $h$
    now implies that
    \begin{equation}
        \label{eq:lowersemicontinuity-liminf}
        \liminf_{l \to \infty} \,
        h(x_{n_{k_{l}}})
        \ge h(p).
    \end{equation}
    See, e.g., Section 1.5 of \cite{pedersen_2012}.
    By definition,
    $h(x_{n_{k_{l}}}) = \min_{x \in \cG_{n_{k_{l}}}} h(x)$
    and because
    $\Pi_{\cG_{n_{k_{l}}}}(\zeta) \to \zeta$
    for any $\zeta \in C$,
    \Cref{lem:grid-minimization-metric-space}
    implies that
    \[
        \lim_{l \to \infty}
        h(x_{n_{k_{l}}})
        ~=~
        \min_{\zeta \in C} h(\zeta).
    \]
    Combined with \eqref{eq:lowersemicontinuity-liminf},
    we conclude that $h(p) = \min_{\zeta \in C} h(\zeta)$,
    and hence $p \in \cM = \argmin_{\zeta \in C} h(\zeta)$.
\end{proof}

\subsection
{Lemmas for certifying continuity from continuous convergence}

In this section,
we collect supplementary lemmas that are used in the proofs of
\Cref{prop:continuity-from-continuous-convergence-rdet,prop:continuity-from-continuous-convergence-rdet-onestep}
in \Cref{sec:proofs-riskmonotonization-zerostep}
and \Cref{sec:proofs-riskmonotonization-onestep},
respectively.

\begin{lemma}
    [Deterministic functions; see, e.g., Problem 57, Chapter 4 of \cite{pugh_2002}, 
    converse of Theorem 21.3 in \cite{munkres_2000}]
    \label{lem:continuity-from-continuous-convergence-fixed-functions}
    Suppose $f_n$ and $f$ are (deterministic) functions from $I \subseteq \RR$ to $\RR$.
    For any $x \in I$
    and any arbitrary sequence $\{ x_n \}_{n \ge 1}$ in $I$
    for which $x_n \to x$,
    assume that $f_n(x_n) \to f(x)$ as $n \to \infty$.
    Then, $f$ is continuous on $I$.
\end{lemma}
\begin{proof}
    The following is a standard proof by contradiction.
    Assume $f$ is discontinuous at $a \in I$.
    Then, there exists a sequence $x_n \to a$ such that
    \[
        f(x_n) \notin [f(a) - 2 \epsilon, f(a) + 2 \epsilon]
    \]
    for some $\epsilon > 0$.
    Note that $f_n(x) \to f(x)$ for all $x \in I$.
    Now, consider another sequence $y_n$ such that
    \begin{align*}
        y_1 = y_2 = \cdots = y_{N_1} = x_1,
        & \quad
        \text{where} 
        \quad
        | f_{N_1}(x_1) - f(a) | > \epsilon \\
        y_{N_1 + 1} = y_{N_1 + 2} = \cdots = y_{N_2} = x_2,
        & \quad
        \text{where}
        \quad
        | f_{N_2}(x_2) - f(a) | > \epsilon, N_2 > N_1 \\
        \vdots
    \end{align*}
    Observe that $y_n \to a$, however $f_n(y_n) \not \to f(a)$.
    Hence, a contradiction.
\end{proof}

\begin{lemma}
    [Extension of \Cref{lem:continuity-from-continuous-convergence-fixed-functions} to random functions]
    \label{lem:continuity-from-continuous-convergence-random-functions}
    Suppose $f_n$ is a sequence of random real-valued functions from $I \subseteq \RR$ such that, for every deterministic sequence $\{ x_n \}_{n \ge 1}$  in $I$
    such that $x_n \to x \in I$,  $f_n(x_n) \to f(x)$ in probability, for a deterministic function $f$ on $I$.
    Then, $f$ is continuous on $I$.
\end{lemma}

\begin{proof}
    The idea of the proof is similar to
    that of an analogous statement for fixed functions;
    see \Cref{lem:continuity-from-continuous-convergence-fixed-functions}.
    We will use proof by contradiction.
    Assume that $f$ is discontinuous at $a \in I$.
    Then, as in the proof of \Cref{lem:continuity-from-continuous-convergence-fixed-functions} for deterministic functions,
    there exists a $\epsilon > 0$ and a sequence $\{x_n\} \subset I$ such that $x_n \to a$ and
    \begin{equation}\label{eq:gg}
        f(x_n) \notin [f(a) - 2 \epsilon, f(a) + 2 \epsilon].
    \end{equation}
    From the hypothesis, we have that, for each $x \in I$, $f_n(x) \to f(x)$ in probability. Let $p \in (0,1)$ be a fixed number. Then, there exists an integer $N_1 \geq 1$ such that the event
    \[
 \Omega_{N_1} = \left\{  |  f_{N_1}(x_1) - f(x_1) | < \epsilon \right\}
    \]
    holds with probability at least $p$. Thus, on $\Omega_{N_1}$, by the triangle inequality,
    \begin{equation}\label{eq:ff}
    |f_{N_1}(x_1) - f(a) | \geq  | f(x_1) - f(a) | - | f_{N_1}(x_1) - f(x_1) |  > \epsilon,
    \end{equation}
    where last inequality stems from \eqref{eq:gg}.
    Next, for $i=2,3,\ldots,$ let $N_i \geq  N_{i-1} + 1$ be an integer such that the event 
    \[
    \Omega_{N_i} = \left\{  |  f_{N_i}(x_i) - f(x_i) | < \epsilon \right\}
    \]
has probability at least $p$. These sequences of numbers $\{ N_i\}$ and events $\{ \Omega_{N_i} \}$ exist  because, by hypothesis, $f_{n}(x_i) \to f(x_i) $ in probability for each $i$. Furthermore $N_i \to \infty$ and, on each $\Omega_{N_i}$, $|f_{N_i}(x_i) - f(a) |  > \epsilon $ by the same argument used in \eqref{eq:ff}.

Consider the sequence $\{y_n\}$ given by
\begin{align*}
        y_1 = y_2 = \cdots = y_{N_1} = x_1
        \\
        y_{N_1 + 1} = y_{N_1 + 2} = \cdots = y_{N_2} = x_2
         \\
        \vdots
    \end{align*}
such that, by construction, $y_n \to a $. We will derive a contradiction by showing that it cannot be the case that  $f_n(y_n) \to a$ in probability, thus violating the hypothesis. Indeed, the sequence of probability values $\{ \mathbb{P} (|f_n(y_n) - f(a)  | > \epsilon) \}$ does not converge to zero since, for each $n$, there exist infinitely many  $N_i > n $ such that
\[
\mathbb{P} ( |f_{N_i}(y_{N_i}) - f(a)  |> \epsilon  ) \geq \mathbb{P} (\Omega_{N_i}) > p > 0. 
\]
Thus, it must be the case that $f$ is continuous at $a$. Continuity of $f$ over $I$ readily follows. 

\end{proof}

\subsection{A lemma for lifting $\QQ$-continuity to $\RR$-continuity}
\label{sec:rational-continuity-to-real-continuity}

The following lemma is used in the proofs of
\Cref{prop:continuity-from-continuous-convergence-rdet,prop:continuity-from-continuous-convergence-rdet-onestep}
in
\Cref{sec:proofs-riskmonotonization-zerostep,sec:proofs-riskmonotonization-onestep},
respectively.

Recall that a function $f : \RR \to \RR$
is continuous at a point $x_{\infty} \in \RR$,
if for all sequences $\{ x_n \}_{n \ge 1}$
in $\RR$
for which $x_n \to x_{\infty}$ as $n \to \infty$,
we have $f(x_n) \to f(x_{\infty})$ as $n \to \infty$.
Call this $\RR$-continuity of $f$ at the point $x_{\infty}$,
and call a function is $\RR$-continuous if it is $\RR$-continuous on its domain.
Define a variant of continuity with respect 
to rational sequences, dubbed $\QQ$-continuity,
as follows.
\begin{definition}
    [$\QQ$-continuity]
    \label{def:q-continuity}
    A function $f : \RR \to \RR$ is $\QQ$-continuous at a point $x_\infty \in \RR$,
    if for all sequences $\{ x_n \}_{n \ge 1}$
    in $\QQ$
    for which $x_n \to x_{\infty}$ as $n \to \infty$,
    we have $f(x_n) \to f(x_{\infty})$ as $n \to \infty$.
    A function is $\QQ$-continuous if it is $\QQ$-continuous over its domain.
\end{definition}

The following lemma shows that $\QQ$-continuity implies $\RR$-continuity. 

\begin{lemma}
    [$\QQ$-continuity implies $\RR$-continuity]
    \label{lem:rational-continuity-implies-real-continuity}
    Suppose $f : \RR \to \RR$ is a $\QQ$ continuous function.
    Then $f$ is $\RR$-continuous.
\end{lemma}
\begin{proof}
    To prove $\RR$-continuity of $f$,
    fix any $y_\infty \in \RR$,
    and consider 
    any arbitrary sequence $\{ y_n \}_{n \ge 1}$
    in $\RR$
    such that $y_n \to y_\infty$ as $n \to \infty$.
    For any $\epsilon > 0$,
    if we can produce $n_\epsilon$
    such that
    $| f(y_n) - f(y_\infty) | \le \epsilon$
    for all $n \ge n_\epsilon$,
    then $\RR$-continuity of $f$ follows.
    We will produce such $n_\epsilon$ below.
    
    For every $m \ge 1$, construct a sequence $\{ x_{k,m} \}_{k \ge 1}$
    in $\QQ$
    such that $x_{k,m} \to y_m$ as $k \to \infty$;
    see \Cref{fig:rationa-continuity-to-real-continuity}.
    (Note this is possible because $\QQ$ is dense in $\RR$.)
    Now, 
    for every $m \ge 1$,
    using $\QQ$-continuity of $f$ at $y_m$,
    we have $f(x_{k, m}) \to f(y_m)$ as $k \to \infty$.
    Fix $\epsilon > 0$.
    Let $k_0(\epsilon) = 1$
    and for $m \ge 1$,
    define a positive integer $k_m(\epsilon)$ by
    \[
        k_m(\epsilon) 
        = \min \{ k > k_{m-1}(\epsilon) : |f(x_{k, m}) - f(y_m)| \le \epsilon / 2\}.
    \]
    Such a $k_m(\epsilon)$ always exists because
    $x_{k, m} \to y_m$ as $k \to \infty$
    and $f$ is $\QQ$-continuous at $y_m$.
    Note that $k_m(\epsilon) > k_{m-1}(\epsilon)$,
    which in turn implies that $k_m(\epsilon) \ge m$ 
    and thus $k_{m}(\epsilon) \to \infty$ as $m \to \infty$.
    Hence, as $m \to \infty$, $x_{k_m(\epsilon), m} \to y_\infty$.
    Using the $\QQ$-continuity of $f$ at $y_\infty$,
    there exists a positive integer $m_\epsilon$
    such that
    for all $m \ge m_\epsilon$,
    we have $| f(x_{k_m(\epsilon), m}) - f(y_\infty) | \le \epsilon / 2$.
    For all $m \ge m_\epsilon$,
    by the triangle inequality,
    observe that
    \[
        | f(y_m) - f(y_\infty) |
        \le |f(y_m) - f(k_{m}(\epsilon))| + |f(k_{m}(\epsilon)) - f(y_\infty)|
        \le \epsilon.
    \]
    Therefore, choosing $n_\epsilon = m_\epsilon$ completes the proof.
\end{proof}

\begin{figure}[!ht]
    \centering
    \includegraphics[width=0.8\columnwidth]{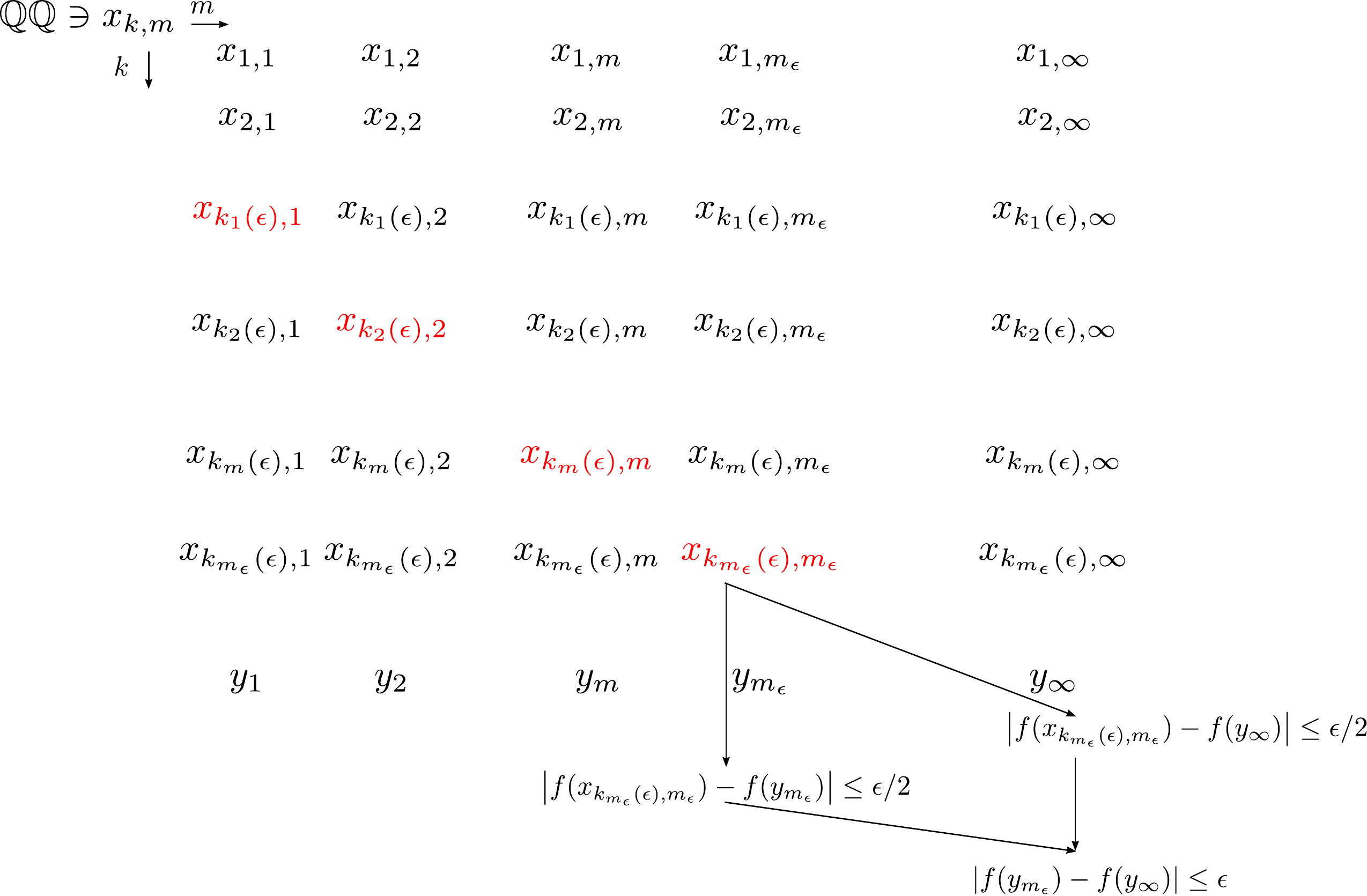}
    \caption{Illustration of the grid of rational sequences used in the proof of \Cref{lem:rational-continuity-implies-real-continuity}.}
    \label{fig:rationa-continuity-to-real-continuity}
\end{figure}

\subsection
[Lemmas on deterministic equivalents for generalized bias and variance resolvents]
{Lemmas on asymptotic deterministic equivalents for generalized bias and variance resolvents}

In this section,
we collect lemmas on asymptotic deterministic equivalents
for generalized bias and variance resolvents
associated with ridge and ridgeless regression
that are used in the proof of
\Cref{prop:asymp-bound-ridge-main}
in
\Cref{sec:verif-asymp-profile-ridge},
and
\Cref{prop:verif-riskprofile-mnlsbase-mnlsonestep}
and
\Cref{lem:one-step-predrisk-decomposition}
in
\Cref{sec:verif-riskprofile-mnlsbase-mnlsonestep}.

\begin{lemma}
    [Deterministic equivalents for generalized bias and variance ridge resolvents]
    \label{lem:deter-approx-generalized-ridge}
    Suppose $X_i \in \RR^{p}$, $1 \le i \le n$, are i.i.d.\ random vectors
    with each $X_i = Z_{i} \Sigma^{1/2}$,
    where $Z_i \in \RR^{p}$ contains i.i.d.\ random variables $Z_{ij}$, $1 \le j \le p$,
    each with $\EE[Z_{ij}] = 0$, $\EE[Z_{ij}^2] = 1$, and $\EE[|Z_{ij}|^{8+\alpha}] \le M_\alpha$
    for some constants $\alpha > 0$ and $M_\alpha < \infty$,
    and $\Sigma \in \RR^{p \times p}$ is a positive semidefinite matrix
    such that $r_{\min} I_p \preceq \Sigma \preceq r_{\max} I_p$
    for some constants $r_{\min} > 0$ and $r_{\max} < \infty$
    (independent of $p$).
    Let $\bX \in \RR^{n \times p}$ be the random matrix with $X_i$, $1 \le i \le n$, as its rows
    and let $\hSigma \in \RR^{p \times p}$ denote the $p \times p$ random matrix $\bX^\top \bX / n$.
    Let $A \in \RR^{p \times p}$ be any deterministic positive semidefinite matrix
    that commutes with $\Sigma$
    such that $ a_{\min} I_p \preceq A \preceq a_{\max} I_p$
    for some constants $a_{\min} > 0$ and $a_{\max} < \infty$
    (independent of $p$).
    Let $\gamma_n := p / n$.
    Then, for $\lambda > 0$,
    as $n, p \to \infty$
    with $0 < \liminf \gamma_n \le \limsup \gamma_n < \infty$,
    the following asymptotic deterministic equivalences hold:
    \begin{enumerate}
        \item Generalized variance of ridge regression:
        \begin{equation}
            \label{eq:detequi-ridge-genbias}
            (\hSigma + \lambda I_p)^{-2} \hSigma A
            \asympequi
            \tv(-\lambda; \gamma_n) (v(-\lambda; \gamma_n) \Sigma + I_p)^{-2} \Sigma A,
        \end{equation}
        where $v(-\lambda; \gamma_n) \ge 0$
        is the unique solution to the fixed-point equation
        \begin{equation}
            \label{eq:fixed-point-v-ridge-statement}
            v(-\lambda; \gamma_n)^{-1}
            = \lambda + \gamma_n \tr[\Sigma (v(-\lambda; \gamma_n) \Sigma + I_p)^{-1}] / p,
        \end{equation}
        and $\tv(-\lambda; \gamma_n)$ is defined via 
        $v(-\lambda; \gamma_n)$ by the equation
        \begin{equation}
            \label{eq:def-v'-ridge}
            \tv(-\lambda; \gamma_n)^{-1}
            = 
                v(-\lambda; \gamma_n)^{-2}
                - 
                \gamma_n
                \tr[\Sigma^2 (v(-\lambda; \gamma_n) \Sigma + I_p)^{-2}] / p.
        \end{equation}
        \item Generalized bias of ridge regression:
        \begin{equation}
            \label{eq:detequi-ridge-genvar}
            \lambda^2
            (\hSigma + \lambda I_p)^{-1} A (\hSigma + \lambda I_p)^{-1}
            \asympequi 
            (v(-\lambda; \gamma_n) \Sigma + I_p)^{-1}
            (\tv_g(-\lambda; \gamma_n) \Sigma + A)
            (v(-\lambda; \gamma_n) \Sigma + I_p)^{-1},
        \end{equation}
        where $v(-\lambda; \gamma_n)$ as defined in \eqref{eq:fixed-point-v-ridge},
        and $\tv_g(-\lambda; \gamma_n)$
        is defined via $v(-\lambda; \gamma_n)$ by the equation
        \begin{equation}
            \label{eq:def-vg'-ridge}
            \tv_g(-\lambda; \gamma_n)
            =
            \ddfrac
            {
                \gamma_n \tr[A \Sigma (v(-\lambda; \gamma_n) \Sigma + I_p)^{-2}] / p
            }
            {
                v(-\lambda; \gamma_n)^{-2}
                - \gamma_n \tr[\Sigma^2 (v(-\lambda; \gamma_n) \Sigma + I_p)^{-2}] / p
            }.
        \end{equation}
    \end{enumerate}
\end{lemma}
\begin{proof}
    The main idea for both the first and second parts is to use 
    \Cref{cor:basic-ridge-resolvent-equivalent-in-v}
    as the starting point,
    and apply the calculus rules for asymptotic deterministic equivalents 
    listed in \Cref{sec:calculus_deterministic_equivalents}
    to manipulate into the desired equivalents.
    
    \paragraph{Part 1.}
    For the first part,
    observe that we can express the resolvent of interest 
    (associated with the generalized variance of ridge regression)
    as a derivative (with respect to $\lambda$) of a certain resolvent:
   \begin{equation}
        \label{eq:deriv-def-genvar-ridge}
        (\hSigma + \lambda I_p)^{-2} \hSigma A
        = (\hSigma + \lambda I_p)^{-1} A - \lambda (\hSigma + \lambda I_p)^{-2} A
        = 
        \frac{\partial}{\partial \lambda}
        [\lambda (\hSigma + \lambda I_p)^{-1} A)].
   \end{equation}
   To find a deterministic equivalent for $(\hSigma + \lambda)^{-2} \hSigma A$,
   it thus suffices to obtain a deterministic equivalent for the resolvent
   $\lambda (\hSigma + \lambda I_p)^{-1} A$
   and take its derivative, 
   thanks to the differentiation rule from
    \Cref{lem:calculus-detequi}~\eqref{lem:calculus-detequi-item-differentiation}.
    Similar derivative trick is used in the proof of Theorem 2.1
    in \cite{liu_dobriban_2019}
    and Theorem 2.1 in \cite{dobriban_wager_2018}
    to compute the standard variance of ridge regression,
    by \cite{dobriban_sheng_2020} in the context of distributed ridge regression,
    and in the earlier works by
    \cite{karoui_kolsters_2011,rubio_mestre_2011,ledoit_peche_2011},
    among others,
    to compute certain limiting trace functionals.
    
    Starting with
    \Cref{cor:basic-ridge-resolvent-equivalent-in-v},
    we have
    \[
        \lambda (\hSigma + \lambda I_p)^{-1}
        \asympequi (v(-\lambda; \gamma_n) \Sigma + I_p)^{-1},
    \]
  where $v(-\lambda; \gamma_n)$ is the unique solution to the fixed point equation
  \begin{equation}
        \label{eq:fixed-point-v-ridge}
        v(-\lambda; \gamma_n)^{-1}
        = \lambda + \gamma_n \tr[\Sigma (v(-\lambda; \gamma_n) \Sigma + I_p)^{-1}] / p.
  \end{equation}
   Since $A$ has bounded operator norm (uniformly in $p$),
   from 
    \Cref{lem:calculus-detequi}~\eqref{lem:calculus-detequi-item-product},
   we have
  \begin{equation}
        \label{eq:variance-base-functional-equivalence-ridge}
        \lambda (\hSigma + \lambda I_p)^{-1} A
        \asympequi (v(-\lambda; \gamma_n) \Sigma + I_p)^{-1} A,
  \end{equation}
  where $v(-\lambda; \gamma_n)$ is as defined by \eqref{eq:fixed-point-v-ridge}.
    It now remains to take the derivative of 
    the right hand side of \eqref{eq:variance-base-functional-equivalence-ridge}
    with respect to $\lambda$.
    Before doing so,
    we will briefly argue that the differentiation rule indeed applies in this case.
    Let $T \in \RR^{p \times p}$
    be a matrix
    with trace norm uniformly bounded in $p$.
    Note that
    \begin{align*}
        \tr[T \lambda (\hSigma + \lambda I_p)^{-1} A]
        &= \tr[T (I_p - \hSigma (\hSigma + \lambda I_p)^{-1}) A] \\
        &\le \| (I_p - \hSigma (\hSigma + \lambda I_p)^{-1}) A \|_{\mathrm{op}} \tr[T] \\
        &\le 
        \| I_p - \hSigma (\hSigma + \lambda I_p)^{-1} \|_{\mathrm{op}}
        \| A \|_{\mathrm{op}}
        \tr[T] \\
        &\le \| A \|_{\mathrm{op}} \tr[T]
        \le C,
    \end{align*}
    for some constant $C < \infty$.
    Here,
    the first inequality follows from Proposition 3.4.10 of \cite{pedersen_2012}
    (see also, Problem III.6.2 of \cite{bhatia_2013}),
    and
    the second inequality follows from the submultiplicativity of the operator norm.
    Similarly,
    note that
    \begin{align*}
        \tr[T (v(-\lambda; \gamma_n) \Sigma + I_p)^{-1} A]
        \le 
        \| (v(-\lambda; \gamma_n) \Sigma + I_p)^{-1} \|_{\mathrm{op}}
        \| A \|_{\mathrm{op}} 
        \tr[T] 
        \le C,
    \end{align*}
    for some constant $C < \infty$.
    Thus, we can safely apply the differentiation rule
    from \Cref{lem:calculus-detequi}~\eqref{lem:calculus-detequi-item-differentiation}
    to get
    \[
       (\hSigma + \lambda I_p)^{-1} \hSigma A
       \asympequi 
       \frac{\partial}{\partial \lambda}
       [(v(-\lambda; \gamma_n) \Sigma + I_p)^{-1} A],
    \]
    we have
    \begin{equation}
        \label{eq:deriv-detequi-genvar-ridge}
        \frac{\partial}{\partial \lambda}
        [(v(-\lambda; \gamma_n) \Sigma + I_p)^{-1} A]
        = - \frac{\partial}{\partial \lambda}[v(-\lambda; \gamma_n)]
        (v(-\lambda; \gamma_n) \Sigma + I_p)^{-2} A.
    \end{equation}
    We can write 
    - $\partial / \partial \lambda [v(-\lambda; \gamma_n)]$
    in terms of $v(-\lambda; \gamma_n)$
    by taking derivative of \eqref{eq:fixed-point-v-ridge}
    with respect to $\lambda$
    and solving for 
    - $\partial / \partial \lambda [v(-\lambda; \gamma_n)]$.
    Taking the derivative of \eqref{eq:fixed-point-v-ridge} yields
    the following equation:
    \begin{equation}
        - \label{eq:deriv-relation-v-ridge}
        \frac{\partial}{\partial \lambda} [v(-\lambda; \gamma_n)]
        v(-\lambda; \gamma_n)^{-2}
        =  1 + \gamma_n - \frac{\partial}{\partial \lambda}[v(-\lambda; \gamma_n)] 
        \tr[\Sigma^2 (v(-\lambda; \gamma_n) \Sigma + I_p)^{-2}] / p.
    \end{equation}
    Denoting 
    - $\partial / \partial \lambda [v(-\lambda; \gamma_n)]$
    by $\tv(-\lambda; \gamma_n)$
    and solving for $\tv(-\lambda; \gamma_n)$
    in \eqref{eq:deriv-relation-v-ridge},
    we get
    \begin{equation}
        \label{eq:def-tv-ridge-in-the-proof}
        \tv(-\lambda; \gamma_n)^{-1}
        = 
        v(-\lambda; \gamma_n)^{-2}
        - \gamma_n \tr[\Sigma^2 (v(-\lambda; \gamma_n) \Sigma + I_p)^{-2}] / p.
    \end{equation}
    Combining 
    \eqref{eq:deriv-def-genvar-ridge},
    \eqref{eq:deriv-detequi-genvar-ridge},
    and 
    \eqref{eq:def-tv-ridge-in-the-proof},
    the statement follows.
    This completes the proof of the first part.

    \paragraph{Part 2.}
    For the second part,
    observe that we can express
    the resolvent of interest 
    (appearing in the generalized bias of ridge regression)
    as a derivative 
    of a certain parameterized resolvent at a fixed value
    of the parameter:
    \begin{equation}
        \label{eq:deriv-def-genbias-ridge}
        \lambda^2 
        (\hSigma + \lambda I_p)^{-1}
        A
        (\hSigma + \lambda I_p)^{-1}
        =
        \lambda^2
        (\hSigma + \lambda I_p + \lambda \rho A)^{-1} 
        A 
        (\hSigma + \lambda I_p + \lambda \rho A)^{-1} |_{\rho = 0}
        =
        -
        \frac{\partial}{\partial \rho}
        [\lambda (\hSigma + \lambda I_p + \lambda \rho A)^{-1}] \mathrel{\Big |}_{\rho = 0}.
    \end{equation}
    It is worth remarking that in contrast to Part 1, we needed
    to introduce another parameter $\rho$ for this part to appropriately
    pull out the matrix $A$ in the middle.
    This trick has been used in the proof
    of Theorem 5 in \cite{hastie_montanari_rosset_tibshirani_2019}
    in the context of standard bias calculation for ridge regression.
    Our strategy henceforth will be to
    obtain a deterministic equivalent
    for the resolvent $\lambda (\hSigma + \lambda I_p + \lambda \rho A)^{-1}$,
    take its derivative with respect to $\rho$, and set $\rho = 0$.
    Towards that end,
    we first massage it
    to make it amenable
    for application of \Cref{lem:basic-ridge-resolvent-deterministic-equivalent}
    as follows:
   \begin{align}
        \lambda
        \big(\hSigma + \lambda I_p + \lambda \rho A\big)^{-1}
        &=
        \lambda
        \big(\hSigma + \lambda (I_p + \rho A)\big)^{-1} 
        \nonumber \\
        &=
        (I_p + \rho A)^{-1/2}
        \lambda\big((I_p + \rho A)^{-1/2} \hSigma (I_p + \rho \Sigma)^{-1/2} + \lambda I_p\big)^{-1}
        (I_p + \rho A)^{-1/2} 
        \nonumber \\
        &=
        (I_p + \rho A)^{-1/2}
        \lambda
        \big(
        \hSigma_{\rho, A} + \lambda I_p
        \big)^{-1}
        (I_p + \rho A)^{-1/2}, \label{eq:resolvent-massage-genbias-ridge}
   \end{align}
   where $\hSigma_{\rho, A} := \Sigma_{\rho, A}^{1/2} (\bZ^\top \bZ / n) \Sigma_{\rho, A}^{1/2}$
   and $\Sigma_{\rho, A} := (I_p + \rho A)^{-1/2} \Sigma (I_p + \rho A)^{-1/2}$.
   We will now obtain a deterministic equivalent
   for $\lambda (\hSigma_{\rho, A} + \lambda I_p)^{-1}$,
   and use the product rule to arrive at the deterministic
   equivalent for $\lambda (\hSigma + \lambda I_p + \lambda \rho A)^{-1}$.
   
   Using 
   \Cref{cor:basic-ridge-resolvent-equivalent-in-v},
   we have
   \begin{equation}
        \label{eq:resolvent-in-v-genbias-ridge}
        \lambda (\hSigma_{\rho, A} + \lambda I_p)^{-1}
        \asympequi (v_g(-\lambda, \rho; \gamma_n) \Sigma_{\rho, A} + I_p)^{-1},
   \end{equation}
   where $v_g(-\lambda, \rho; \gamma_n)$
   is the unique solution to the fixed-point equation
   \begin{equation}
        \label{eq:fixed-point-in-v-genbias-ridge}
        v_g(-\lambda, \rho; \gamma_n)^{-1}
        = \lambda + \gamma_n 
        \tr[\Sigma_{\rho, A}  (v_g(-\lambda, \rho; \gamma_n) \Sigma_{\rho, A} + I_p)^{-1}] / p.
   \end{equation}
   Combining 
   \eqref{eq:resolvent-massage-genbias-ridge}
   with 
   \eqref{eq:resolvent-in-v-genbias-ridge},
   and using the product rule from 
    \Cref{lem:calculus-detequi}~\eqref{lem:calculus-detequi-item-product}
   (which is applicable since $(I_p + \rho A)^{-1/2}$ is a deterministic matrix),
   we get
   \begin{align*}
        \lambda (\hSigma + \lambda I_p + \lambda \rho A)^{-1}
        &= 
        (I_p + \rho A)^{-1/2}
        \lambda
        (\hSigma_{\rho, A} + \lambda I_p)^{-1}
        (I_p + \rho A)^{-1/2}
        \\
        &\asympequi
        (I_p + \rho A)^{-1/2}
        (v_g(-\lambda, \rho; \gamma_n) \Sigma_{\rho, A} + I_p)^{-1}
        (I_p + \rho A)^{-1/2} \\
        &= 
        (I_p + \rho A)^{-1/2}
        (v_g(-\lambda, \rho; \gamma_n) (I_p + \rho A)^{-1/2} \Sigma (I_p + \rho A)^{-1/2}  + I_p)^{-1}
        (I_p + \rho A)^{-1/2} \\
        &=
        (v_g(-\lambda, \rho; \gamma_n) \Sigma + I_p + \rho A)^{-1}.
   \end{align*}
   Similarly, 
   the right hand side of the fixed-point equation 
   \eqref{eq:fixed-point-in-v-genbias-ridge} 
   can be simplified
   by substituting back for $\Sigma_{\rho, A}$ to yield
   \begin{align}
       v_g(-\lambda, \rho; \gamma_n)^{-1}
       &= \lambda + \gamma_n 
       \tr[(I_p + \rho A)^{-1/2} \Sigma (I_p + \rho A)^{-1/2} 
       (v_g(-\lambda, \rho; \gamma_n) \Sigma_{\rho, A} + I_p)^{-1}] / p \nonumber \\
       &=
       \lambda + \gamma_n
       \tr[ \Sigma 
       (v_g(-\lambda, \rho; \gamma_n) 
       (I_p + \rho A)^{1/2} \Sigma_{\rho, A} (I_p + \rho A)^{1/2} 
       + (I_p + \rho A))^{-1}] / p \nonumber \\
       &=
       \lambda + \gamma_n
       \tr[\Sigma (v_g(-\lambda, \rho; \gamma_n) \Sigma + I_p + \rho A)^{-1}] / p. \label{eq:fixed-point-vg-with-rho}
   \end{align}
   Finally, we will now use
  the differentiation rule 
    from \Cref{lem:calculus-detequi}~\eqref{lem:calculus-detequi-item-differentiation}
   (with respect to $\rho$ this time).
   The applicability of the differentiation rule
   follows analogously to first part
   for $\rho > -1/a_{\min}$.
   Additionally,
    it is easy to verify that
    both sides of \eqref{eq:fixed-point-vg-with-rho}
    are analytic in $\rho$.
    Taking derivative with respect to $\rho$,
    we get
   \begin{equation}
        \label{eq:deriv-detequi-genbias-ridge}
        -
        \frac{\partial}{\partial \rho}
        [ ( v_g(-\lambda, \rho; \gamma_n) \Sigma + I_p + \rho A )^{-1} ]
        =
        (v_g(-\lambda, \rho; \gamma_n) \Sigma + I_p + \rho A)^{-1}
        \left(\frac{\partial}{\partial \rho}[v_g(-\lambda, \rho; \gamma_n)] \Sigma + A\right)
        (v_g(-\lambda, \rho; \gamma_n) \Sigma + I_p + \rho A)^{-1}.
   \end{equation}
   Setting $\rho = 0$ and observing that
   $v_g(-\lambda, 0; \gamma_n) = v(-\lambda; \gamma_n)$,
   where $v(-\lambda; \gamma_n)$ is as defined in \eqref{eq:fixed-point-v-ridge},
    we have
    \begin{equation}
        \label{eq:deriv-detequi-genbias-ridge-rho0}
        \frac{\partial}{\partial \rho}
        [
        (v_g(-\lambda, \rho; \gamma_n) \Sigma + I_p + \rho A)^{-1}
        ] \mathrel{\Big |}_{\rho = 0}
        = 
        (v(-\lambda; \gamma_n) \Sigma + I_p)^{-1}
        \left( 
            \frac{\partial}{\partial \rho}[v_g(-\lambda, \rho; \gamma_n)] 
            \mathrel{\Big|}_{\rho = 0} 
            \Sigma  + A 
        \right)
        (v(-\lambda; \gamma_n) \Sigma + I_p)^{-1}.
    \end{equation}
    To obtain an equation for 
    $\partial / \partial \rho [v_g(-\lambda, \rho; \gamma_n)] |_{\rho = 0}$,
    we can differentiate the fixed-point equation \eqref{eq:fixed-point-vg-with-rho}
    with respect to $\rho$ to yield
    \begin{multline}
        -
        \frac{\partial}{\partial \rho}
        [v_g(-\lambda, \rho; \gamma_n)] 
        v_g(-\lambda, \rho; \gamma_n)^{-2}
        = 
        - \gamma_n 
        \frac{\partial}{\partial \rho}[v_g(-\lambda, \rho; \gamma_n)]
        \tr[\Sigma^2 (v_g(-\lambda, \rho; \gamma_n) \Sigma + I_p + \rho A)^{-2}] / p \\
        - \gamma_n
        \tr[A \Sigma (v_g(-\lambda, \rho; \gamma_n) \Sigma + I_p + \rho A)^{-2}] / p.
    \end{multline}
    Setting $\rho = 0$ in the equation above,
    and using the fact that $v_g(-\lambda, 0; \gamma_n) = v(-\lambda; \gamma_n)$,
    and denoting 
    $\partial / \partial \rho[v_g(-\lambda, \rho; \gamma_n)] |_{\rho = 0}$
    by $\tv_g(-\lambda; \gamma_n)$,
    we get that
    \begin{equation}
        \label{eq:def-tvg-ridge-in-the-proof}
        \tv_g(-\lambda; \gamma_n)
        = 
        \ddfrac
        {\gamma_n \tr[A \Sigma (v(-\lambda; \gamma_n) \Sigma + I_p)^{-2}] / p}
        {v(-\lambda; \gamma_n)^{-2} - \gamma_n \tr[\Sigma^2 (v(-\lambda; \gamma_n) \Sigma + I_p)^{-2}] / p}.
    \end{equation}
    Therefore, from \eqref{eq:deriv-def-genbias-ridge} 
    and \eqref{eq:deriv-detequi-genbias-ridge-rho0},
    we finally have
   \[
        \lambda^2
        (\hSigma + \lambda I_p)^{-1}
        A
        (\hSigma + \lambda I_p)^{-1}
        \asympequi
        (v(-\lambda; \gamma_n) \Sigma + I_p)^{-1}
        (\tv_g(-\lambda; \gamma_n) \Sigma + A)
        (v(-\lambda; \gamma_n) \Sigma + I_p)^{-1},
   \]
   where $v(-\lambda; \gamma_n)$ is as defined in \eqref{eq:fixed-point-v-ridge},
   and $\tv_g(-\lambda; \gamma_n)$ is as defined in \eqref{eq:def-tvg-ridge-in-the-proof}.
   This completes the proof of the second part.

\end{proof}

\begin{lemma}
    [Deterministic equivalents for generalized bias and variance ridgeless resolvents]
    \label{lem:deter-approx-generalized-ridgeless}
    Assume the setting of \Cref{lem:deter-approx-generalized-ridge}
    with $\gamma_n \in (1, \infty)$.
    Then,
    the following deterministic equivalences hold:
    \begin{enumerate}
        \item Generalized variance of ridgeless regression:
        \begin{equation}
            \label{eq:detequi-mn2ls-genbias}
            \hSigma^{+} A
            \asympequi
            \tv(0; \gamma_n)
            (v(0; \gamma_n) \Sigma + I_p)^{-2} \Sigma A,
        \end{equation}
        where $v(0; \gamma_n)$ is the unique solution
        to the fixed-point equation
        \begin{equation}
            \label{eq:fixed-point-v-mn2ls}
            \gamma_n^{-1}
            = \tr[v(0; \gamma_n) \Sigma (v(0; \gamma_n) \Sigma + I_p)^{-1}] / p,
        \end{equation}
        and $\tv(0; \gamma_n)$ is defined through $v(0; \gamma_n)$  via
        \begin{equation}
            \label{eq:def-vg-mn2ls-tv}
            \tv(0; \gamma_n)
            = \big( v(0; \gamma_n)^{-2} - \gamma_n \tr[\Sigma^2 (v(0; \gamma_n) \Sigma + I_p)^{-2}] / p \big)^{-1}.
        \end{equation}
        \item Generalized bias of ridgeless regression:
        \begin{equation}
            (I_p - \hSigma^{+} \hSigma)
            A
            (I_p - \hSigma^{+} \hSigma)
            \asympequi
            (v(0; \gamma_n) \Sigma + I_p)^{-1}
            (\tv_g(0; \gamma_n) \Sigma + A)
            (v(0; \gamma_n) \Sigma + I_p)^{-1},
        \end{equation}
    \end{enumerate}
\end{lemma}
where $v(0; \gamma_n)$ is as defined in \eqref{eq:fixed-point-v-mn2ls},
and $\tv_g(0; \gamma_n)$ is defined via $v(0; \gamma_n)$ by
\begin{equation}
    \label{eq:def-tv-mn2ls-tvg}
    \tv_g(0; \gamma_n)
    =
    \gamma_n
    \tr[A \Sigma (v(0; \gamma_n) \Sigma + I_p)^{-2}] / p
    \cdot
    \big(v(0; \gamma_n)^{-2} - \gamma_n \tr[\Sigma^2 (v(0; \gamma_n) \Sigma + I_p)^{-2}] / p\big)^{-1}.
\end{equation}

\begin{proof}
    The proofs for both the parts use the results of \Cref{lem:deter-approx-generalized-ridge}
    and a limiting argument as $\lambda \to 0^{+}$.
    The results of \Cref{lem:deter-approx-generalized-ridge}
    are pointwise in $\lambda$,
    but can be strengthened to be uniform in $\lambda$
    over a range
    that includes $\lambda = 0$
    allowing one to take the limits of the deterministic equivalents
    obtained in \Cref{lem:deter-approx-generalized-ridge}
    as $\lambda \to 0^{+}$.
    
    \paragraph{Part 1.}
    We will use the result in Part 1 of \Cref{lem:deter-approx-generalized-ridge}
    as our starting point.
    Let $\Lambda := [0, \lambda_{\max}]$
    where $\lambda_{\max} < \infty$,
    and let $T$ be a matrix with bounded trace norm.
    Note that
    \begin{equation}
        \label{eq:var-func-bound}
        \vert \tr[(\hSigma + \lambda I_p)^{-2} \hSigma A T ] \vert
        \le \| (\hSigma + \lambda I_p)^{-2} \hSigma A \|_{\mathrm{op}} \tr[T]
        \le C \| (\hSigma + \lambda I_p)^{-2} \hSigma \|_{\mathrm{op}} 
        \| A \|_{\mathrm{op}}
        \le C
    \end{equation}
    for some constant $C < \infty$.
    Here, the last inequality follows
    because $s_i^2 / (s_i^2 + \lambda)^2 \le 1$
    where $s_i^2$, $1 \le i \le p$,
    are the eigenvalues of $\hSigma$,
    and the operator norm $A$ is assumed to be bounded.
    Consider the magnitude of the derivative (in $\lambda$) of the map
     $\lambda
     \mapsto
     \tr[(\hSigma + \lambda I_p)^{-2} \hSigma A T]$
     given by
    \[
        \left|
        \frac{\partial}{\partial \lambda}
        \tr[(\hSigma + \lambda I_p)^{-2} \hSigma A T]
        \right|
        = 2 \vert \tr[(\hSigma + \lambda I_p)^{-3} \hSigma A T] \vert.
    \]
    Following the argument in \eqref{eq:var-func-bound},
    for $\lambda \in \Lambda$,
    observe that
    \[
        \vert \tr[(\hSigma + \lambda I_p)^{-3} \hSigma A T] \vert
        \le \| (\hSigma + \lambda I_p)^{-3} \hSigma \|_{\mathrm{op}} 
        \| A \|_{\mathrm{op}} \tr[T]
        \le C
    \]
    for some constant $C < \infty$.
    Similarly,
    in the same interval
    $
        \tr
        [
        \tv(-\lambda; \gamma_n)
        (v(-\lambda; \gamma_n) \Sigma + I_p)^{-2}
        \Sigma A T
        ]
        \le C.
    $
    In addition,
    from \Cref{lem:fixed-point-v-lambda-properties},
    we have
    the map
    $
        \lambda 
        \mapsto
        \tr[\tv(-\lambda; \gamma_n) 
        (v(-\lambda; \gamma_n) \Sigma + I_p)^{-2} A T]
    $
    is differentiable in $\lambda$
    and the derivative 
    for $\lambda \in \Lambda$
    is bounded.
    Therefore,
    the family of functions
    $\tr[ (\hSigma + \lambda I_p)^{-2} \hSigma A T] 
    - \tr[\tv(-\lambda; \gamma_n) (v(-\lambda; \gamma_n) \Sigma + I_p)^{-2} \Sigma A T]$
    forms an equicontinuous family in $\lambda$
    over $\lambda \in \Lambda$.
    Thus, the convergence in Part 1 of \Cref{lem:deter-approx-generalized-ridge} 
    is uniforms in $\lambda$.
    We can now use the Moore-Osgood theorem
    to interchange the limits
    to obtain
    \begin{align*}
        &\lim_{p \to \infty}
        \tr[\hSigma^{+} A T]
        - \tr[\tv(0; \gamma_n) (v(0; \gamma_n) \Sigma + I_p)^{-2} \Sigma A T] \\
        &=
        \lim_{p \to \infty}
        \lim_{\lambda \to 0^{+}}
        \tr[(\hSigma + \lambda I_p)^{-2} \hSigma A T]
        - \tr[\tv(-\lambda; \gamma_n) (v(-\lambda; \gamma_n) \Sigma + I_p)^{-2} \Sigma A T)] \\
        &=
        \lim_{\lambda \to 0^{+}}
        \lim_{p \to \infty}
        \tr[(\hSigma + \lambda I_p)^{-2} \hSigma A T]
        - \tr[\tv(-\lambda; \gamma_n) (v(-\lambda; \gamma_n) \Sigma + I_p)^{-2} \Sigma A T)] \\
        &= 0.
    \end{align*}
    In the first equality above,
    we used the fact that 
    $\hSigma^{+} = \hSigma^{+} \hSigma \hSigma^{+} 
    = \lim_{\lambda \to 0^{+}} (\hSigma + \lambda I_p)^{-1} \hSigma (\hSigma + \lambda I_p)^{-1}$,
    and that the functions
    $v(\cdot; \gamma_n)$
    and
    $\tv(\cdot; \gamma_n)$
    are continuous 
    (which follows, from say 
    \Cref{lem:v-tv-tvderiv-bouding-in-lambda}~\eqref{lem:v-tv-tvderiv-bounding-in-lambda-item-f}).
    This
    provides the right hand side of \eqref{eq:detequi-mn2ls-genbias}.
    Similarly, the fixed-point equation \eqref{eq:fixed-point-v-ridge}
    as $\lambda \to 0^{+}$ becomes
    \[
        v(0; \gamma_n)^{-1}
        = \gamma_n \tr[\Sigma (v(0; \gamma_n) \Sigma + I_p)^{-1}] / p.
    \]
    Moving $v(0; \gamma_n)$ to the other side
    (from 
    \Cref{lem:fixed-point-v-properties}~\eqref{lem:fixed-point-v-properties-item-v-properties},
    it follows that $v(0; \gamma_n) > 0$
    for $\gamma_n \in (1, \infty)$),
    we arrive at the desired result.
    
    \paragraph{Part 2.}
    As done in Part 1,
    it is not difficult to show
    that over $\lambda \in \Lambda$
    the family of functions
    $\tr[\lambda^2 (\hSigma + \lambda I_p)^{-1} A (\hSigma + \lambda I_p)^{-1} T]
    - \tr[(v(-\lambda; \gamma_n) \Sigma + I_p)^{-1} 
    (\tv_g(-\lambda; \gamma_n) \Sigma + A) 
    (v(-\lambda; \gamma_n) \Sigma + I_p)^{-1} T]$
    form an equicontinuous family.
    Therefore,
    the convergence in Part 2 of 
    \Cref{lem:deter-approx-generalized-ridge}
    is uniform in $\lambda$ over $\Lambda$ (that includes $0$).
    Using the Moore-Osgood theorem
    to the interchange the limits,
    one has
    \begin{align*}
        &\lim_{p \to \infty}
        \tr[(I_p - \hSigma^{+} \hSigma) A (I_p - \hSigma^{+} \hSigma) T]
        -
        \tr[ 
        (v(0; \gamma_n) \Sigma + I_p)^{-1} 
        (\tv_g(0; \gamma_n) \Sigma + A) 
        (v(0; \gamma_n) \Sigma + I_p)^{-1}
        T
        ] \\
        &=
        \lim_{p \to \infty}
        \lim_{\lambda \to 0^{+}}
        \tr[\lambda^2 (\hSigma + \lambda I_p)^{-1} A (\hSigma + \lambda I_p)^{-1} T]
        - \tr[ (v(-\lambda; \gamma_n) \Sigma + I_p)^{-1} 
        (\tv_g(-\lambda; \gamma_n) \Sigma + A) 
        (v(-\lambda; \gamma_n) \Sigma + I_p)^{-1} T] \\
        &=
        \lim_{\lambda \to 0^{+}}
        \lim_{p \to \infty}
        \tr[\lambda^2 (\hSigma + \lambda I_p)^{-1} A (\hSigma + \lambda I_p)^{-1} T]
        - \tr[ (v(-\lambda; \gamma_n) \Sigma + I_p)^{-1}
        (\tv_g(-\lambda; \gamma_n) \Sigma + A) 
        (v(-\lambda; \gamma_n) \Sigma + I_p)^{-1} T] \\
        &= 0.
    \end{align*}
    Now both \eqref{eq:def-vg-mn2ls-tv} and \eqref{eq:def-tv-mn2ls-tvg}
    follow by taking $\lambda \to 0^{+}$
    in 
    \eqref{eq:detequi-ridge-genvar}
    and
    \eqref{eq:def-vg'-ridge},
    respectively.
    
    This concludes the proof.

\end{proof}

\begin{corollary}
    [Limiting deterministic equivalents for generalized bias and variance ridgeless resolvents]
    \label{cor:limiting-resolvents-mn2ls}
    Assume the setting of \Cref{lem:deter-approx-generalized-ridge}.
    Let $f : \RR_{\ge 0} \to \RR_{\ge 0}$ be a function.
    Then,
    as $n, p \to \infty$ and $p / n \to \gamma \in (1, \infty)$,
    the following equivalences hold:
    \begin{enumerate}
        \item Limiting generalized variance of ridgeless regression:
        \begin{equation}
            \hSigma^{+} f(\Sigma)
            \asympequi
            \tv(0; \gamma)
            (v(0; \gamma) \Sigma + I_p)^{-2}
            \Sigma f(\Sigma),
        \end{equation}
        where
        $v(0; \gamma)$ and $\tv(0; \gamma)$
        are defined by \eqref{eq:fixed-point-v-mn2ls}
        and \eqref{eq:def-vg-mn2ls-tv}, respectively.
        \item Limiting generalized bias of ridgeless regression:
        \begin{equation}
            (I_p - \hSigma^{+} \hSigma)
            f(\Sigma)
            (I_p - \hSigma^{+} \hSigma)
            \asympequi
            (1 + \tv_g(0; \gamma)) 
            (v(0; \gamma) \Sigma + I_p)^{-1}
            f(\Sigma)
            (v(0; \gamma) \Sigma + I_p)^{-1},
        \end{equation}
        where $v(0; \gamma)$
        is as defined in
        \eqref{eq:fixed-point-v-mn2ls}
        and $\tv_g(0; \gamma)$
        is as defined in
        \eqref{eq:def-tv-mn2ls-tvg}
        with $A$ replaced by $f(\Sigma)$.
    \end{enumerate}
\end{corollary}
\begin{proof}
    The proof follows from \Cref{lem:deter-approx-generalized-ridgeless},
    in conjunction with 
    \Cref{lem:fixed-point-v-properties}
    (\eqref{lem:fixed-point-v-properties-item-v-properties}, 
    \eqref{lem:fixed-point-v-properties-item-tv-properties},
    \eqref{lem:fixed-point-v-properties-item-tvg-properties})
    to provide continuity of the functions
    $v(0; \cdot)$, $\tv(0; \cdot)$, and $\tv_g(0; \cdot)$
    (in the aspect ratio) over $(1, \infty)$.
\end{proof}

\subsection{Lemmas on properties of solutions of certain fixed-point equations}

In this section,
we collect helper lemmas that are used in the proofs of
\Cref{prop:asymp-bound-ridge-main}
in \Cref{sec:verif-asymp-profile-ridge},
\Cref{cor:verif-onestep-program}
in \Cref{sec:verif-riskprofile-mnlsbase-mnlsonestep},
and
\Cref{lem:deter-approx-generalized-ridgeless}
and
\Cref{cor:limiting-resolvents-mn2ls}
in
\Cref{sec:useful-lemmas}.

\begin{lemma}
    [Continuity and limiting behavior of functions of 
    the solution of a fixed-point equation in the aspect ratio]
    \label{lem:fixed-point-v-properties}
    Let $a > 0$ and $b < \infty$ be real numbers.
    Let $P$ be a probability measure supported
    on $[a, b]$.
    Consider the function $v(0; \cdot) : \phi \mapsto v(0; \phi)$,
    over $(1, \infty)$,
    where $v(0; \phi) \ge 0$ is the unique solution to the fixed-point equation
    \begin{equation}
        \label{eq:fixed-point-gen-phi}
        \frac{1}{\phi}
        = \int \frac{v(0; \phi) r}{1 + v(0; \phi) r} \, \mathrm{d}P(r).
    \end{equation}
    Then,
    the following properties hold:
    \begin{enumerate}
        \item 
        \label{lem:fixed-point-v-properties-item-v-properties}
        The function $v(0; \cdot)$ is continuous 
        and strictly decreasing over $(1, \infty)$.
        Furthermore, $\lim_{\phi \to 1^{+}} v(0; \phi) = \infty$,
        and $\lim_{\phi \to \infty} v(0; \phi) = 0$.
        \item
        \label{lem:fixed-point-v-properties-item-phivinverse-properties}
        The function 
        $\phi \mapsto (\phi v(0; \phi))^{-1}$
        is strictly increasing over $(1, \infty)$.
        Furthermore,
        $\lim_{\phi \to 1^{+}} (\phi v(0; \phi))^{-1} = 0$
        and $\lim_{\phi \to \infty} (\phi v(0; \phi))^{-1} = 1$.
        \item
        \label{lem:fixed-point-v-properties-item-tv-properties}
        The function 
        $\tv(0; \cdot) : \phi \mapsto \tv(0; \phi)$,
        where
        \[
           \tv(0; \phi)
           =
           \left(
                \frac{1}{v(0; \phi)^2}
                - \phi
                \int \frac{r^2}{(1 + r v(0; \phi))^2} 
                \, \mathrm{d}P(r)
           \right)^{-1},
        \]
        is continuous over $(1, \infty)$.
        Furthermore,
        $\lim_{\phi \to 1^{+}} \tv(0; \phi) = \infty$,
        and $\lim_{\phi \to \infty} \tv(0; \phi) = 0$.
        \item
        \label{lem:fixed-point-v-properties-item-tvg-properties}
        The function 
        $\tv_g(0; \cdot) : \phi \mapsto \tv_g(0; \phi)$,
        where
        \[
            \tv_g(0; \phi)
            = \tv(0; \phi)
            \phi
            \int
            \frac{r^2}{(1 + v(0; \phi) r)^2}
            \, \mathrm{d}P(r),
        \]
        is continuous over $(1, \infty)$.
        Furthermore,
        $\lim_{\phi \to 1^{+}} \tv_g(0; \phi) = \infty$,
        and $\lim_{\phi \to \infty} \tv_g(0; \phi) = 0$.
        \item
        \label{lem:fixed-point-v-properties-Upsilonb-properties}
        Let $Q$ be a (fixed) probability distribution
        supported on $[a, b]$
        that depends on a scalar $\phi_1$.
        Then, the function
        $\Upsilon_b(\phi_1; \cdot) : \phi \mapsto \Upsilon_b(\phi_1, \phi)$,
        where
        \[
            \Upsilon_b(\phi_1, \phi)
            = (1 + \tv_g(0; \phi))
            \int \frac{1}{(1 + v(0; \phi) r)^2}
            \, \mathrm{d}Q(r),
        \]
        is continuous over $(1, \infty)$.
        Furthermore,
        $\Upsilon_b(\phi_1, \phi) < \infty$
        for $\phi \in (1, \infty)$,
        and $\lim_{\phi \to \infty} \Upsilon_b(\phi_1, \phi) = 1$.
    \end{enumerate}
\end{lemma}

\begin{proof}
    We consider the five parts separately below.
    Before doing so though,
    it is worth mentioning
    that for $\phi \in (1, \infty)$, 
    there is a unique non-negative solution $v(0; \phi)$
    to the fixed-point equation \eqref{eq:fixed-point-gen-phi}
    as stated in the statement.
    This follows from 
    \Cref{lem:v-tv-tvderiv-bouding-in-lambda}~\eqref{lem:v-tv-tvderiv-bounding-in-lambda-item-f}.
    The following properties refer to the function
    $v(0; \cdot) : \phi \mapsto v(0; \phi)$
    defined via this unique solution.
    
    \paragraph{Part 1.}
    We begin with the first part.
    Observe that the function
    \[
        t \mapsto \int \frac{1}{1 + t r} \, \mathrm{d}P(r)
    \]
    is strictly decreasing and strictly convex over $(0, \infty)$.
    Thus, the function
    \[
        T 
        :  t \mapsto 
        1 - \int \frac{1}{1 + t r} \, \mathrm{d}P(r)
        =  \int \frac{t}{1 + t r} \, \mathrm{d}P(r)
    \]
    is strictly increasing and strictly concave over $(0, \infty)$,
    with $\lim_{t \to 0} T(t) = 0$ and $\lim_{t \to \infty} T(t) = 1$.
    Since the inverse image of a strictly increasing and strictly concave
    real function is strictly increasing and strictly convex
    (see, e.g. Proposition 3 of \cite{hiriart-urruty_martinez_2003}), 
    we have that $T^{-1}$ is strictly convex and strictly increasing.
    This also implies that $T^{-1}$ is continuous.
    Note that $v(0; \phi) = T^{-1}(\phi^{-1})$.
    Since $\phi^{-1}$ is continuous, 
    it follows that $v(0; \cdot)$ is continuous. 
    In addition, since $\phi \mapsto \phi^{-1}$ is strictly decreasing,
    we have that $v(0; \cdot)$ is strictly decreasing.
    Moreover, $\lim_{\phi \to 1^{+}}T^{-1}(\phi^{-1}) = \infty$,
    and $\lim_{\phi \to \infty} T^{-1}(\phi^{-1}) = 0$.

    \paragraph{Part 2.}
    From \eqref{eq:fixed-point-gen-phi},
    we have
    \[
        \frac{1}{\phi v(0; \phi)}
        = \int \frac{r}{1 + v(0; \phi) r} \, \mathrm{d}P(r).
    \]
    Because $v(0; \phi)$ is strictly decreasing over $(1, \infty)$,
    the right side of the display above is strictly increasing.
    Furthermore,
    because $\lim_{\phi \to 1^{+}} v(0; \phi) = \infty$,
    we have $\lim_{\phi \to 1^{+}} (\phi v(0; \phi))^{-1} = 0$,
    and because $\lim_{\phi \to \infty} v(0; \phi) = 0$,
    we have $\lim_{\phi \to \infty} (\phi v(0; \phi))^{-1} = 1$.
    
    \paragraph{Part 3.}
   From Part 1, the function $1 / v(0; \cdot)^2$ is continuous.
   In addition, 
    observe that
   the function
   \[
       \phi \mapsto 
       \int \frac{r^2}{(1 + v(0; \phi) r)^2} \, \mathrm{d}P(r)
   \]
   is also continuous.
   Thus, $\tv(0; \cdot)$ is continuous.
   Furthermore,
    since $\lim_{\phi \to 1^{+}} v(0; \phi) = \infty$,
    it follows that $\lim_{\phi \to 1^{+}} \tv(0; \phi) = \infty$.
   Similarly, 
  from $\lim_{\phi \to \infty} v(0; \phi) = 0$
  and the fact that
  \[
    \lim_{\phi \to \infty}
    \int \frac{r^2}{(1 + r v(0; \phi))^2} \, \mathrm{d}P(r)
    \ge a^2 > 0,
  \]
   it follows that $\lim_{\phi \to \infty} \tv(0; \phi) = 0$.
   
   \paragraph{Part 4.}
   Similar to Part 3, continuity of $\tv_g(0; \cdot)$
   follows from the continuity of $\tv(0; \cdot)$
   and $v(0; \phi)$.
   To compute the desired limits, 
   observe that
   \begin{align*}
        1 + \tv_g(0; \phi)
        =
        \ddfrac
        {\frac{1}{v(0; \phi)^2}}
        {\frac{1}{v(0; \phi)^2} - \phi \int \frac{r^2}{(1 + v(0; \phi) r)^2} 
        \, \mathrm{d}P(r)}.
   \end{align*}
   We thus have
   \begin{align}
        (1 + \tv_g(0; \phi))^{-1}
        &=
        1
        -
        v(0; \phi)^2
        \phi
        \int \frac{r^2}{(1 + r v(0; \phi))^2}
        \, \mathrm{d}P(r) \label{eq:tvg_plus_1_inv_1} \\
        &=
        1
        -
        \phi
        \int
        \frac{r^2}{(v(0; \phi)^{-1} + r)^2}
        \, \mathrm{d}P(r). \label{eq:tvg_plus_1_inv_2}
   \end{align}
   Because $\lim_{\phi \to 1^{+}} v(0; \phi) = \infty$,
   from \eqref{eq:tvg_plus_1_inv_2},
  we have
   \[
        \lim_{\phi \to 1^{+}}
        (1 + \tv_g(0; \phi))^{-1}
        =
        1 
        - 
        \lim_{\phi \to 1^{+}}
        \phi 
        \int 
        \frac{r^2}{(v(0; \phi)^{-1} + r)^2}
        \mathrm{d}P(r)
        = 1 - 1
        = 0.
   \]
   It follows then that $\lim_{\phi \to 1^{+}} \tv_g(0; \phi) = \infty$.
   
   On the other hand,
    observe 
    from \eqref{eq:tvg_plus_1_inv_1}
    that
   \begin{equation}
        \label{eq:tvg_plus_1_inv_3}
        (1 + \tv_g(0; \phi))^{-1}
        = 
        1 - 
        \phi v(0; \phi)
        v(0; \phi)
        \int
        \frac{r^2}{(1 + r v(0; \phi))^2}
        \, \mathrm{d}P(r).
   \end{equation}
   From Part 2, we have $\lim_{\phi \to \infty} \phi v(0; \phi) = 1$,
   and from Part 1, we have $\lim_{\phi \to \infty} v(0; \phi) = 0$.
   Moreover,
   since $P$ is supported on $[a, b]$,
   and $v(0; \phi) > 0$ for $\phi \in (1, \infty)$
   from Part 1,
   for $\phi \in (1, \infty)$,
   note that
   \[
        0
        <
        \int \frac{r^2}{(1 + r v(0; \phi))^2}
        <
        b^2.
   \]
   Thus, from \eqref{eq:tvg_plus_1_inv_3},
   we obtain
   \[
        \lim_{\phi \to \infty}
        (1 + \tv_g(0; \phi))^{-1}
        =
        1 - 0
        = 1.
   \]
   We hence conclude that $\lim_{\phi \to \infty} \tv_g(0; \phi) = 0$.
  
  \paragraph{Part 5.}
  
    The continuity claim follows
    from the continuity of $v(0; \cdot)$
    and $\tv_g(0; \cdot)$
    from Parts 1 and 4, respectively.
  From calculation similar to that in Part 4,
  it follows that
  $(1 + \tv_g(0;\phi)) < \infty$
  for $\phi \in (1, \infty)$.
  Now, since $v(0; \phi) > 0$ for $\phi \in (1, \infty)$
  from Part 1,
  and $Q$ is supported on $[a, b]$,
  observe that
  \[
        \int
        \frac{1}{(1 + v(0; \phi) r)^2}
        \, \mathrm{d}Q(r)
        \le
        1
        < \infty.
  \]
  Hence, $\Upsilon_b(\phi_1, \phi) < \infty$ for $\phi \in (1, \infty)$.
  Moreover, because $\lim_{\phi \to \infty} (1 + \tv_g(0; \phi)) = 1$,
  and $\lim_{\phi \to \infty} v(0; \phi) = 0$,
  we obtain
  \[
        \lim_{\phi \to \infty}
        \Upsilon_b(\phi_1, \phi)
        =
        \lim_{\phi \to \infty}
        (1 + \tv_g(0; \phi))
        \cdot
        \lim_{\phi \to \infty}
        \int
        \frac{1}{(1 + v(0; \phi) r)^2}
        \, \mathrm{d}Q(r)
        = 1.
  \]
  Therefore, $\lim_{\phi \to \infty} \Upsilon_b(\phi_1, \phi) = 1$,
  as desired.

   This completes all the five parts,
   and finishes the proof.
\end{proof}

\begin{lemma}
    [Bounding derivatives of the solution 
    of a fixed-point equation 
    in the regularization parameter]
    \label{lem:fixed-point-v-lambda-properties}
    Let $a > 0$ and $b < 0$ be real numbers.
    Let $P$ be a probability measure supported on $[a, b]$.
    Let $\gamma \in (1, \infty)$ be a real number.
    Let $\Lambda = [0, \lambda_{\max}]$
    for some constant $\lambda_{\max} < \infty$.
    For $\lambda \in \Lambda$,
    let $v(-\lambda; \gamma) \ge 0$
    denote the solution
    to the fixed-point equation
    \[
        \frac{1}{v(-\lambda; \gamma)}
        =  \lambda
        + \gamma \int \frac{r}{v(-\lambda; \gamma) r + 1} \, \mathrm{d}P(r).
    \]
    Then, the function 
    $\lambda \mapsto v(-\lambda; \gamma)$
    is 
    twice differentiable over $\Lambda$.
    Furthermore, 
    over $\Lambda$,
    $v(-\lambda; \gamma)$,
    $\partial / \partial \lambda [v(-\lambda; \gamma)]$,
    and
    $\partial^2 / \partial \lambda^2 [v(-\lambda; \gamma)]$
    are bounded above.
\end{lemma}
\begin{proof}
    Start by re-writing the fixed-point equation as
    \[
        \lambda
        = \frac{1}{v(-\lambda; \gamma)}
        - \gamma  \int \frac{r}{v(-\lambda; \gamma) r  + 1} \, \mathrm{d}P(r).
    \]
    Define a function $f$ by
    \[
        f(x) = \frac{1}{x} - \gamma \int \frac{r}{x r + 1} \, \mathrm{d}P(r).
    \]
    Observe that $v(-\lambda; \gamma) = f^{-1}(\lambda)$.
    The claim of twice differentiability
    of the function $\lambda \mapsto v(-\lambda; \gamma_n)$
    follows from 
    \Cref{lem:v-tv-tvderiv-bouding-in-lambda}~\eqref{lem:v-tv-tvderiv-bounding-in-lambda-item-finv}.
    The claim of boundedness of the function
    and its first derivatives (with respect to $\lambda$)
    follows from 
    \Cref{lem:v-tv-tvderiv-bouding-in-lambda}
    (\eqref{lem:v-tv-tvderiv-bounding-in-lambda-item-finv},
    \eqref{lem:v-tv-tvderiv-bounding-in-lambda-item-finv'},
    \eqref{lem:v-tv-tvderiv-bounding-in-lambda-item-finv''}).
    
\end{proof}

\begin{lemma}
    [Bounding derivatives of the solution of a fixed-point equation]
    \label{lem:v-tv-tvderiv-bouding-in-lambda}
    Let $a > 0$ and $b < \infty$ be two real numbers.
    Let $P$ be a probability distribution supported on $[a, b]$.
    Let $\gamma \in (1, \infty)$ be a real number.
    Define a function $f$ by
    \begin{equation}
        \label{eq:fixed-point-lambda-generic}
        f(x)
        = \frac{1}{x}
        - \gamma \int \frac{r}{x r + 1} \, \mathrm{d}P(r).
    \end{equation}
    Then, the following properties hold:
    \begin{enumerate}
        \item
        \label{lem:v-tv-tvderiv-bounding-in-lambda-item-f}
        There is a unique $0 < x_0 < \infty$ such that $f(x_0) = 0$.
        The function $f$ is twice differentiable
        and strictly decreasing
        over $(0, x_0)$,
        with $\lim_{x \to 0^{+}} f(x) = \infty$
        and $f(x_0) = 0$.
        \item
        \label{lem:v-tv-tvderiv-bounding-in-lambda-item-f'}
        The derivative $f'$
        is strictly increasing over $(0, x_0)$,
        with $\lim_{x \to 0^{+}} f'(x) = - \infty$
        and $f'(x_0) < 0$.
        \item
        \label{lem:v-tv-tvderiv-bounding-in-lambda-item-f''}
        The second derivative $f''$
        is strictly decreasing over $(0, x_0)$,
        with $\lim_{x \to 0^{+}} f''(x) = \infty$
        and
        $f''(x_0) > 0$.
        \item
        \label{lem:v-tv-tvderiv-bounding-in-lambda-item-finv}
        The inverse function $f^{-1}$
        is twice differentiable,
        bounded over $[0, \infty)$
        by $x_0 < \infty$,
        and strictly decreasing over 
        $(0, \infty)$,
        with
        $f^{-1}(0) = x_0$
        and
        $\lim_{y \to \infty} f^{-1}(y) = 0$.
        \item
        \label{lem:v-tv-tvderiv-bounding-in-lambda-item-finv'}
        The derivative of the inverse function
        $(f^{-1})'$
        is bounded over $[0, \infty)$
        by
        \[
            \ddfrac
            {x_0^2}
            {1 - \gamma \int \left(\frac{x_0 r}{x_0 r + 1}\right)^2
            \, \mathrm{d}P(r)}
            < \infty.
        \]
        \item
        \label{lem:v-tv-tvderiv-bounding-in-lambda-item-finv''}
        The second derivative of the inverse function
        $(f^{-1})''$ is bounded over 
        $[0, \infty)$
        by
        \[
            \ddfrac
            {2 x_0^3}
            {\left( 1 - \gamma \int \left(\frac{x_0 r}{x_0 r + 1}\right)^2 \, \mathrm{d}P(r) \right)^3}
            < \infty.
        \]
    \end{enumerate}
\end{lemma}

\begin{proof}
    We consider different parts separately below.
    
    \paragraph{Part 1.}
        Observe that
        \[
            f(x)
            =
            \frac{1}{x}
            - \gamma \int \frac{r}{xr + 1}
            \, \mathrm{d}P(r)
            =
            \frac{1}{x}
            \left(
                1 - \gamma \int \frac{xr}{xr + 1}
                \, \mathrm{d}P(r)
            \right).
        \]
        The function $g: x \mapsto 1/x$ is positive and strictly decreasing
        over $(0, \infty)$
        with $\lim_{x \to 0^{+}} g(x) = \infty$ and $\lim_{x \to \infty} g(x) = 0$,
        while
        the function
        \[
            h: x \mapsto
            1 - \gamma \int \frac{xr}{xr + 1}
            \, \mathrm{d}P(r)
        \]
        is strictly decreasing
        over $(0, \infty)$
        with $h(0) = 1$ and $\lim_{x \to \infty} h(x) = 1 - \gamma < 0$.
        Thus, there is a unique $0 < x_0 < \infty$
        such that $h(x_0) = 0$,
        and
        consequently
        $f(x_0) = 0$.
        Because
        $h$ is positive over
        $[0, x_0]$,
        $f$, a product of two positive strictly decreasing functions,
        is strictly decreasing over $(0, x_0)$,
        with $\lim_{x \to 0^{+}} f(x) = \infty$
        and $f(x_0) = 0$.
   
   \paragraph{Part 2.}
    The derivative $f'$ at $x$ is given by
    \[
        f'(x)
        = - \frac{1}{x^2}
        + \gamma \int \frac{r^2}{(x r + 1)^2}
        \, \mathrm{d}P(r)
        = 
        -
        \frac{1}{x^2}
        \left(
            1 - \gamma \int \left(\frac{xr}{xr + 1}\right)^2
            \, \mathrm{d}P(r)
        \right).
    \]
    The function
    $g: x \mapsto 1/x^2$
    is positive and strictly decreasing over $(0, \infty)$
    with $\lim_{x \to 0^{+}} g(x) = \infty$
    and $\lim_{x \to \infty} g(x) = 0$.
    On the other hand, 
    the function
    \[
        h: 
        x \mapsto
        1 - \gamma \int \left( \frac{xr}{xr + 1} \right)^2
        \, \mathrm{d}P(r)
    \]
    strictly decreasing over $(0, \infty)$
    with $h(0) = 1$
    and $h(x_0) > 0$.
    This follows because
    for $x \in [0, x_0]$,
    \begin{equation}
        \begin{split}
            \label{eq:bound-deriv-v-in-lambda-part-2}
            \gamma \int \left( \frac{xr}{xr + 1} \right)^2 \, \mathrm{d}P(r)
            &\le
            \left( \frac{x_0 b}{x_0 b + 1} \right)
            \gamma 
            \int \left( \frac{xr}{xr + 1} \right) \, \mathrm{d}P(r) \\
            &<
            \gamma \int \frac{xr}{xr + 1} \, \mathrm{d}P(r)
            \le
            \gamma \int \frac{x_0 r}{x_0 r + 1} \, \mathrm{d}P(r)
            =
            1,
        \end{split}
    \end{equation}
    where the first inequality in the chain above
    follows as the support of $P$ is $[a, b]$,
    and the last inequality follows since $f(x_0) = 0$
    and $x_0 > 0$,
    which implies that
    \[
        \frac{1}{x_0}
        = \gamma \int \frac{r}{x_0 r + 1} \, \mathrm{d}P(r),
        \quad
        \text{ or equivalently that}
        \quad
        1
        = \gamma \int \frac{x_0 r}{x_0 r + 1} \, \mathrm{d}P(r).
    \]
    Thus, $-f'$, a product of two positive strictly decreasing functions,
    is strictly decreasing, and in turn, $f'$ 
    is strictly increasing.
    Moreover, $\lim_{x \to 0^{+}} f'(x) = -\infty$ and $f'(x_0) < 0$.
    
    \paragraph{Part 3.}
    The second derivative $f''$ at $x$
    is given by
    \[
        f''(x)
        = \frac{2}{x^3}
        - 2 \gamma \int \frac{r^3}{(x r + 1)^3}
        \, \mathrm{d}P(r)
        =
        \frac{2}{x^3}
        \left(
            1 - \gamma \int \left(\frac{xr}{xr + 1}\right)^3
            \, \mathrm{d}P(r)
        \right).
    \]
    The rest of the arguments are similar to those in Part 2.
    The function $g : x \mapsto 1/x^3$ is positive and strictly decreasing
    over $(0, \infty)$ with $\lim_{x \to 0^{+}} g(x) = \infty$
    and $\lim_{x \to \infty} g(x) = 0$,
    while the function
    \[
        h:
        x \mapsto
        1 - \gamma \int \left( \frac{xr}{xr + 1} \right)^3
        \, \mathrm{d}P(r)
    \]
    is strictly decreasing over $(0, \infty)$
    with $h(0) = 1$ and $h(x_0) > 0$ as
    \begin{equation}
        \begin{split}
        \label{eq:bound-deriv-v-in-lambda-part-3}
            \gamma \int \left( \frac{xr}{xr + 1} \right)^3 \, \mathrm{d}P(r)
            &\le
            \left( \frac{x_0 b}{x_0 b + 1} \right)^2
            \gamma 
            \int \left( \frac{xr}{xr + 1} \right) \, \mathrm{d}P(r) \\
            &<
            \gamma \int \frac{xr}{xr + 1} \, \mathrm{d}P(r)
            \le
            \gamma \int \frac{x_0 r}{x_0 r + 1} \, \mathrm{d}P(r)
            =
            1.
        \end{split}
    \end{equation}
    It then follows that $f''$ is strictly decreasing,
    with $\lim_{x \to 0^{+}} f''(x) = \infty$
    and $f''(x_0) > 0$.
    
    \paragraph{Part 4.}
   Because 
   $f$ is twice differentiable and strictly monotonic over $(0, x_0)$,
   $f^{-1}$ is twice differentiable and strictly monotonic
   (see, e.g., Problem 2, Chapter 5 of \cite{rudin_1976}).
   Since $f(x_0) = 0$, $f^{-1}(0) = x_0$,
   and since $\lim_{x \to 0^{+}} f(x) = \infty$,
   $\lim_{y \to \infty} f^{-1}(y) = 0$.
   Hence, $f^{-1}$ is bounded above over $[0, \infty)$ by $x_0 < \infty$.
    
    \paragraph{Part 5.}
    Because $f'(x) \neq 0$ over $(0, x_0)$,
    by the inverse function theorem,
    we have
    \[
        \left| (f^{-1})'(f(x)) \right|
        = \left| \frac{1}{f'(x)} \right|
        <
        \left| \frac{1}{f'(x_0)} \right|
        =
        \ddfrac
        {1}
        {\frac{1}{x_0^2} \left( 1 - \gamma \int \left( \frac{xr}{xr + 1} \right)^2 \, \mathrm{d}P(r) \right)}
        < \infty,
    \]
    where the first inequality
    uses the fact that 
    $|f'(x_0)| < |f'(x)|$ for $x \in (0, x_0]$
    from Part 2,
    and the last inequality uses the bound from \eqref{eq:bound-deriv-v-in-lambda-part-2}.
    
    \paragraph{Part 6.}
    Similar to Part 5,
    by inverse function theorem,
    we have
    \[
        \left|
            (f^{-1})''(f(x))
        \right|
        =
        \left|
            \frac{f''(x)}{f'(x)^3}
        \right|
        =
        \ddfrac
        {\frac{2}{x^3} \left( 1 - \gamma \int \left( \frac{xr}{xr + 1} \right)^3 \, \mathrm{d}P(r)  \right)}
        {\frac{1}{x^6} \left( 1 - \gamma \int \left( \frac{xr}{xr + 1} \right)^2 \, \mathrm{d}P(r) \right)^3}
        \le
        \ddfrac
        {2 x_0^3}
        {\left( 1 - \gamma \int \left( \frac{xr}{xr + 1} \right)^2 \, \mathrm{d}P(r) \right)^3}
        < \infty,
    \]
    where the first inequality uses
    the bound from \eqref{eq:bound-deriv-v-in-lambda-part-3},
    and the second inequality uses
    the bound from \eqref{eq:bound-deriv-v-in-lambda-part-2}.
    
    This finishes all the six parts,
    and concludes the proof.
    
\end{proof}

We remark that
the technique of Lemma A.2
of \cite{hastie_montanari_rosset_tibshirani_2019}
can be applied to obtain similar conclusions
as those in \Cref{lem:fixed-point-v-lambda-properties,lem:v-tv-tvderiv-bouding-in-lambda}.
However, since our parameterization
is slightly different,
we make use of the inverse function theorem instead
of the implicit function theorem
employed in \cite{hastie_montanari_rosset_tibshirani_2019}.

\subsection
[Risk characterization of one-step procedure with ridgeless regression]
{Proof of \Cref{thm:esterrlim_optimized_onestep} 
(Risk characterization of one-step procedure with ridgeless regression)
}
\label{sec:onestep-risk-analysis-isotropic-features}

The following theorem characterizes the risk of the one-step procedure
starting with MN2LS base procedure for isotropic features under square error.
Let $R^\deter(\gamma; \tf^\onestep)$
denote the risk of
the one-step predictor
starting with the MN2LS base predictor
on i.i.d.\ data with limiting aspect ratio $\gamma$.

\begin{theorem}
    [Limiting risk of one-step procedure with ridgeless regression]
    \label{thm:esterrlim_optimized_onestep}
    Suppose assumptions
    \ref{asm:lin-mod},
    \ref{asm:rmt-feat} with $\Sigma = I$,
    \ref{asm:signal-bounded-norm}
    hold true.
    Let $\mathrm{SNR} := {\rho^2}/{\sigma^2}$.
    Then, 
    the limiting risk of the one-step predictor
    starting with the MN2LS base predictor
    under \ref{asm:prop_asymptotics}
    is given as follows:
    \begin{itemize}[leftmargin=*]
        \item \underline{When $\mathrm{SNR} \le 1$}:
        \begin{equation*}
        \frac{R^\deter(\gamma; \hf^\onestep)}{\sigma^2}
        - 1
        =
        \begin{dcases}
        \frac{\gamma}{1 - \gamma} 
        & \text{ if } \gamma \le \frac{\mathrm{SNR}}{\mathrm{SNR} + 1}  < 1\\
        \mathrm{SNR} & \text{ otherwise.}
        \end{dcases}
        \end{equation*}
        \item \underline{When $1 < \mathrm{SNR} \le \mathrm{SNR}^{\star} (\approx 10.7041)$}:
        \begin{align*}
            &
            \frac{R^\deter(\gamma; \hf^\onestep)}{\sigma^2}
            -
            1
            = 
            \\
            &\begin{dcases}
            \frac{\gamma}{1 - \gamma}
            & \text{ if } \gamma \le 1 - \frac{1}{2 \sqrt{2 \sqrt{\mathrm{SNR}} - 1}} < 1 \\
            2 \sqrt{2 \sqrt{\mathrm{SNR}} - 1} - 1
            & \text{ if } 1 - \frac{1}{2 \sqrt{2 \sqrt{\mathrm{SNR}} - 1}} < \gamma
            \le \left(2 - \frac{1}{\sqrt{\mathrm{SNR}}} - \frac{1}{\sqrt{2 \sqrt{\mathrm{SNR}} - 1}}\right)^{-1} \\
            \left\{ \mathrm{SNR} \left( 1 - \frac{1}{\zeta_1} \right) + \frac{1}{\zeta_1 - 1} \right\}
            \left( 1 - \frac{1}{\zeta_2} \right) + \frac{1}{\zeta_2 - 1}
            & \text{ otherwise},
            \end{dcases}
        \end{align*}
        where $\mathrm{SNR}^{\star}$ (which is approximately 10.7041)
        is value of $x > 1$ that solves
        \begin{equation}
            \label{eq:def-snr*}
            1
            -
            \frac{1}{2 \sqrt{2 \sqrt{x} - 1}}
            =
            \left(
                2 - \frac{1}{x} - \frac{1}{\sqrt{2 \sqrt{x} - 1}}
            \right)^{-1},
        \end{equation}
     and $\zeta_1, \zeta_2 \ge 1$ 
    are solutions to the equations
    \begin{gather}
        \mathrm{SNR} \left( \frac{1}{\zeta_1} - \frac{1}{\zeta_2} \right)
        = \frac{\zeta_1^2}{(\zeta_1 - 1)^2} - \frac{\zeta_2^2}{(\zeta_2 - 1)^2}
        + \frac{1}{\zeta_1 - 1} \left( 1 - \frac{\zeta_1}{\zeta_2} \frac{\zeta_1}{(\zeta_1 - 1)} \right) 
        \label{eq:onestep-iso-riskopt-lagrance-equations-1}
        \\
        \frac{1}{\zeta_1} + \frac{1}{\zeta_2} = \frac{1}{\gamma}.
        \label{eq:onestep-iso-riskopt-lagrance-equations-2}
    \end{gather}
        \item \underline{When $\mathrm{SNR} > \mathrm{SNR}^{\star}$}:
        \begin{equation*}
            \frac{R^\deter(\gamma; \hf^\onestep)}{\sigma^2}
            -
            1
            =
            \begin{dcases}
            \frac{\gamma}{1 - \gamma}
            & \text{ if } \gamma \le \gamma^\star < 1 \\
            \left\{ \mathrm{SNR} \left( 1 - \frac{1}{\zeta_1} \right) + \frac{1}{\zeta_1 - 1} \right\}
            \left( 1 - \frac{1}{\zeta_2} \right) + \frac{1}{\zeta_2 - 1}
            & \text{ otherwise,}
            \end{dcases}
        \end{equation*}
        where 
        $\mathrm{SNR}^{\star}$ is as defined in \eqref{eq:def-snr*},
        $\gamma^\star$ is given by
        \[
            1 - 
            \left(
            1 + 
            \min_{\gamma \le 1}
            \left\{ \mathrm{SNR} \left( 1 - \frac{1}{\zeta_1} \right) + \frac{1}{\zeta_1 - 1} \right\}
            \left( 1 - \frac{1}{\zeta_2} \right) + \frac{1}{\zeta_2 - 1}
            \right)^{-1},
        \]
        and $\zeta_1, \zeta_2 \ge 1$
        are solutions to the set of equations
        \eqref{eq:onestep-iso-riskopt-lagrance-equations-1}
        and
        \eqref{eq:onestep-iso-riskopt-lagrance-equations-2}.
    \end{itemize}
    Furthermore, in each case, the limiting risk is a non-decreasing function of $\gamma$.
\end{theorem}

\begin{proof}
From \Cref{prop:verif-riskprofile-mnlsbase-mnlsonestep},
it follows that
that the limiting risk
of the ingredient one-step predictor
for various limiting split proportions $(\zeta_1, \zeta_2)$ 
under isotropic features
is given by
\begin{equation*}
    R^\deter(\zeta_1, \zeta_2; \tf) - 1 = 
    \begin{cases}
        \left\{
        \rho^2 \left(1 - \frac{1}{\zeta_1}\right)
        + \sigma^2 \left(\frac{1}{\zeta_1 - 1}\right)
        \right\}
        \left(1 - \frac{1}{\zeta_2}\right)
        + \sigma^2 \left(\frac{1}{\zeta_2 - 1}\right)
        & \text{ when } \zeta_1 > 1, \zeta_2 > 1 \\
        \left\{
        \sigma^2 \left(\frac{\zeta_1}{1 - \zeta_1}\right)
        \right\}
        \left(1 - \frac{1}{\zeta_2}\right)
        + \sigma^2 \left(\frac{1}{\zeta_2 - 1}\right)
        & \text{ when } \zeta_1 < 1, \zeta_2 > 1 \\
        \sigma^2 \left(\frac{\zeta_2}{1 - \zeta_2}\right)
        & \text{ when } \zeta_2 < 1.
    \end{cases}
\end{equation*}
Note that the last case covers both $\zeta_1 > 1$ and $\zeta_1 < 1$.
Given a fixed $\gamma$, our goal is to minimize $R^\deter(\zeta_1, \zeta_2; \tf)$
with the constraint $\frac{1}{\zeta_1} + \frac{1}{\zeta_2} \le \frac{1}{\gamma}$.

To simplify the calculations below,
we first scale out the factor of $\sigma^2$
and express the risk in terms of $\mathrm{SNR} := \frac{\rho^2}{\sigma^2}$
to write
    \begin{equation*}
    \frac{R^\deter(\zeta_1, \zeta_2; \tf)}{\sigma^2} - 1
    = 
    \begin{cases}
        \left\{
        \mathrm{SNR} \left(1 - \frac{1}{\zeta_1}\right)
        + \left(\frac{1}{\zeta_1 - 1}\right)
        \right\}
        \left(1 - \frac{1}{\zeta_2}\right)
        + \left(\frac{1}{\zeta_2 - 1}\right)
        & \text{ when } \zeta_1 > 1, \zeta_2 > 1 \\
        \left\{
        \frac{\zeta_1}{1 - \zeta_1}
        \right\}
        \left(1 - \frac{1}{\zeta_2}\right)
        + \left(\frac{1}{\zeta_2 - 1}\right)
        & \text{ when } \zeta_1 < 1, \zeta_2 > 1 \\
        \left(\frac{\zeta_2}{1 - \zeta_2}\right)
        & \text{ when } \zeta_2 < 1.
    \end{cases}
\end{equation*}
The problem of minimizing $R(\hbeta^{\onestep})$ can now be broken into
three separate minimization problems, one for each of the cases above.
The final allocation is then the one that gives the minimum among the three cases.

We next notice a simple observation that lets us eliminate the third case.
Any feasible allocation of $\zeta_1$ and $\zeta_2$ in the third case
is also a feasible allocation for the second case.
This can be seen by making $\zeta_1$ for the second case
equal to $\zeta_2$ in the third case
and letting $\zeta_2$ for the second case tend to $\infty$.
Moreover, this gives the same objective value for both the cases.
Hence, the minimum of the second case is no larger
than the minimum of the third case
and we can ignore the minimization of the third case.

Overall we are thus left with two minimization problems:
\begin{equation}
    \label{prob:overparam}
    \begin{array}{ll}
    \mbox{minimize} 
            &
            \left\{
            \mathrm{SNR}
            \left(1 - \frac{1}{\zeta_1}\right)
            + \left(\frac{1}{\zeta_1 - 1}\right)
            \right\}
            \left(1 - \frac{1}{\zeta_2}\right)
            + \left(\frac{1}{\zeta_2 - 1}\right)  \\
    \mbox{subject to} & \frac{1}{\zeta_1} + \frac{1}{\zeta_2} \le \frac{1}{\gamma} \\
    & \zeta_1 > 1 \\
    & \zeta_2 > 1
    \end{array}
\end{equation}
from the first case, and
\begin{equation}
    \label{prob:underparam}
    \begin{array}{ll}
    \mbox{minimize} 
            &
            \left\{
            \frac{\zeta_1}{1 - \zeta_1}
            \right\}
            \left(1 - \frac{1}{\zeta_2}\right)
            + \left(\frac{1}{\zeta_2 - 1}\right)  \\
    \mbox{subject to} & \frac{1}{\zeta_1} + \frac{1}{\zeta_2} \le \frac{1}{\gamma} \\
    & \zeta_1 < 1 \\
    & \zeta_2 > 1
    \end{array}
\end{equation}
from the second case.
We now in turn analyze both of these optimization problems.

\subsubsection*{Optimization problem~\eqref{prob:underparam}}

Let's start with the problem~\eqref{prob:underparam}.
Note that the objective function of the optimization problem \eqref{prob:underparam}
does not depend on $\mathrm{SNR}$.
Hence the optimal value will only be a function of $\gamma$.
In addition, the constraint $\zeta_1 < 1$ is only satisfied when $\gamma < 1$.
Thus, when $\gamma > 1$, the problem is infeasible.
We divide the remaining range of $\gamma$ into two main cases of
$0 < \gamma < 0.5$ and $0.5 < \gamma < 1$.
In each of the cases,
we show that the minimum value of the problem is
$\frac{\gamma}{1 - \gamma}$, which is achieved by setting
$\zeta_1 = \gamma$ and $\zeta_2 = \infty$.

\paragraph{When $\gamma \le 0.5$.}
We first note that any allocation $\zeta_1 > 0.5$ is suboptimal
because when $\zeta_1 > 0.5$, we have $\frac{\zeta_1}{1 - \zeta_1} > 1$
by \Cref{lem:riskprofile_underparam}~\eqref{lem:riskprofile_underparam-item--greater-than-one}.
Thus using 
\Cref{lem:riskprofile_overparam}~\eqref{lem:riskprofile_overparam-item-snr-greater-than-one-risk},
the objective function in this case is always larger than $1$ for such $\zeta_1$.
However, we can achieve $1$ by setting $\zeta_1 = 0.5$ and $\zeta_2 \to \infty$.
Therefore we only need to consider $\zeta_1 \le 0.5$.
For such $\zeta_1$, we have $\frac{\zeta_1}{1 - \zeta_1} \le 1$
by \Cref{lem:riskprofile_underparam}~\eqref{lem:riskprofile_underparam-item-increasing}.
Now using 
\Cref{lem:riskprofile_overparam}~\eqref{lem:riskprofile_overparam-item-snr-less-than-one},
the optimal allocation is obtained by setting $\zeta_2 \to \infty$
and choosing the least $\zeta_1$, which is $\gamma$,
and the corresponding optimal value is $\frac{\gamma}{1 - \gamma}$.

\paragraph{When $0.5 < \gamma < 1$.}
We claim that the optimum value is still $\frac{\gamma}{1 - \gamma}$,
which is achieved by setting $\zeta_1 = \gamma$ and $\zeta_2 \to \infty$.
This is a slightly more involved argument than the previous case
because now $\frac{\zeta_1}{1 - \zeta_1}$ will be larger than $1$
since $\zeta_1 > \gamma > 0.5$,
and hence there is a possibility of optimal allocation other than
$\zeta_1 = \gamma$ and $\zeta_2 = \infty$.
We proceed as follows.

Consider any feasible $\zeta_1 < 1$.
On one hand, using 
\Cref{lem:riskprofile_overparam}~\eqref{lem:riskprofile_overparam-item-snr-greater-than-one-min},
we note that the unconstrained optimal $\zeta_2^\star$ for this $\zeta_1$ is
$\frac{\sqrt{\frac{\zeta_1}{1 - \zeta_1}}}{\sqrt{\frac{\zeta_1}{1 - \zeta_1}} - 1}$.
On the other hand, from the constraint $\frac{1}{\zeta_2} \le \frac{1}{\gamma} - \frac{1}{\zeta_1}$,
we know that we need to satisfy $\zeta_2 \ge \frac{1}{\frac{1}{\gamma} - \frac{1}{\zeta_1}}$.
There are now two possible scenarios.

\begin{itemize}
    \item When $\frac{4}{7} < \gamma < 1$.
    
    In this case, we verify that any feasible $\zeta_1$ (such that $\gamma \le \zeta_1 < 1$) satisfies
    \begin{equation*}
    \frac{\sqrt{\frac{\zeta_1}{1 - \zeta_1}}}{\sqrt{\frac{\zeta_1}{1 - \zeta_1}} - 1}
    <
    \frac{1}{\frac{1}{\gamma} - \frac{1}{\zeta_1}}.
    \end{equation*}
    To see this, the above inequality after separating components of $\gamma$ and $\zeta_1$ reads
    \begin{equation*}
        \frac{1}{\gamma} < \frac{1}{\zeta_1} + 1 - \sqrt{\frac{1}{\zeta_1} - 1}.
    \end{equation*}
    It is easy to check
    that the function $x \mapsto 1 + \frac{1}{x} - \sqrt{\frac{1}{x} - 1}$
    attains minimum value of $\frac{7}{4}$ (at $x = \frac{4}{5}$)
    on the interval $0.5 < x < 1$.
    Thus whenever $\gamma > \frac{4}{7}$, this condition will be satisfied
    for all feasible $\zeta_1$.
    In this case, from 
    \Cref{lem:riskprofile_overparam}~\eqref{lem:riskprofile_overparam-item-snr-greater-than-one-min},
    the optimal $\zeta_2$ that satisfy the constraint is $\frac{1}{\frac{1}{\gamma} - \frac{1}{\zeta_1}}$.
    Plugging this value into the objective function,
    we arrive at the objective function
    \begin{equation*}
        \left\{ \frac{\zeta_1}{1 - \zeta_1} \right\}
        \left( 1 - \frac{1}{\gamma} + \frac{1}{\zeta_1} \right)
        +
        \frac{\frac{1}{\gamma} - \frac{1}{\zeta_1}}{1 - \frac{1}{\gamma} + \frac{1}{\zeta_1}}
    \end{equation*}
    and the overall optimization problem reduces to
    \begin{equation}
        \begin{array}{ll}
        \mbox{minimize} 
        &
        \left\{ \frac{\zeta_1}{1 - \zeta_1} \right\}
        \left( 1 - \frac{1}{\gamma} + \frac{1}{\zeta_1} \right)
        +
        \frac{\frac{1}{\gamma} - \frac{1}{\zeta_1}}{1 - \frac{1}{\gamma} + \frac{1}{\zeta_1}}  \\
        \mbox{subject to} & \zeta_1 \ge \gamma \ge \frac{4}{7} \\
        & \zeta_1 < 1.
        \end{array}
    \end{equation}
    We can verify that the objective function is increasing in the constraint set
    and achieves the minimum at $\zeta_1 = \gamma$.
    The corresponding $\zeta_2$ then tends to $\infty$ as desired.
    \item When $0.5 < \gamma < \frac{4}{7}$,
    or equivalently $\frac{7}{4} < \frac{1}{\gamma} < 2$.
    
    In this case, we can check that when
    \begin{equation}
        \label{eq:gamma1_interval_lower}
        \frac{\frac{2}{\gamma} - \sqrt{\frac{4}{\gamma} - 7} - 1}
        {2 \left( \frac{1}{\gamma^2} - \frac{2}{\gamma} + 2 \right)}
        \le
        \zeta_1
        \le
        \frac{\frac{2}{\gamma} + \sqrt{\frac{4}{\gamma} - 7} - 1}
        {2 \left( \frac{1}{\gamma^2} - \frac{2}{\gamma} + 2 \right)},
    \end{equation}
    we have
    \begin{equation*}
        \frac{1}{\gamma} >
        \frac{1}{\zeta_1} + 1 - \sqrt{\frac{1}{\zeta_1} - 1}
    \end{equation*}
    which leads to
    \begin{equation*}
      \frac{1}{\frac{1}{\gamma} - \frac{1}{\zeta_1}}
        <
        \frac{\sqrt{\frac{\zeta_1}{1 - \zeta_1}}}{\sqrt{\frac{\zeta_1}{1 - \zeta_1}} - 1}
    \end{equation*}
    Thus
    $\zeta_2^\star = \frac{\sqrt{\frac{\zeta_1}{1 - \zeta_1}}}{\sqrt{\frac{\zeta_1}{1 - \zeta_1}} - 1}$
    is feasible.
    The objective at this $\zeta_2$ is $2 \sqrt{\frac{\zeta_1}{1 - \zeta_1}} - 1$.
    Now note that the function $x \mapsto 2 \sqrt{\frac{x}{1 - x}} - 1$
    is increasing for $0 < x < 1$
    and thus the optimal $\zeta_1$ in this case is the lower point of the above interval \eqref{eq:gamma1_interval_lower}.
    The optimal value for this case is thus given by 
    \begin{equation*}
      2 \sqrt{\frac{\frac{2}{\gamma} - \sqrt{\frac{4}{\gamma} - 7} - 1}{\frac{2}{\gamma^2} - \frac{4}{\gamma} + 4 - \frac{2}{\gamma} + \sqrt{\frac{4}{\gamma} - 7} + 1}} - 1.
    \end{equation*} 
    While when
    \begin{equation*}
        \gamma
        <
        \zeta_1
        <
        \frac{\frac{2}{\gamma} - \sqrt{\frac{4}{\gamma} - 7} - 1}
        {2\left( \frac{1}{\gamma^2} - \frac{2}{\gamma} + 2 \right)},
        \quad
        \text{ or }
        \quad
        \frac{\frac{2}{\gamma} + \sqrt{\frac{4}{\gamma} - 7} - 1}
        {2\left( \frac{1}{\gamma^2} - \frac{2}{\gamma} + 2 \right)}
        <
        \zeta_1
        <
        1,
    \end{equation*}
    we have
    \begin{equation*}
        \frac{1}{\gamma} <
        \frac{1}{\zeta_1} + 1 - \sqrt{\frac{1}{\zeta_1} - 1}.
    \end{equation*}
    As argued before, in this case, the optimal $\zeta_2$ is
    $\frac{1}{\frac{1}{\gamma} - \frac{1}{\zeta_1}}$
    and the objective function at this value is given by
    \begin{equation*}
        \left\{ \frac{\zeta_1}{1 - \zeta_1} \right\}
        \left( 1 - \frac{1}{\gamma} + \frac{1}{\zeta_1} \right)
        +
        \frac{\frac{1}{\gamma} - \frac{1}{\zeta_1}}{1 - \frac{1}{\gamma} + \frac{1}{\zeta_1}}.
    \end{equation*}
    This function is again increasing in $\zeta_1$ in the constrained set
    and hence the optimal value of $\zeta_1$ is the lower point when $\zeta_1 = \gamma$
    leading to the optimal value $\frac{\gamma}{1 - \gamma}$.
    Now, we have
    \begin{equation*}
        \frac{\gamma}{1 - \gamma}
        <
      2 \sqrt{\frac{\frac{2}{\gamma} - \sqrt{\frac{4}{\gamma} - 7} - 1}{\frac{2}{\gamma^2} - \frac{4}{\gamma} + 4 - \frac{2}{\gamma} + \sqrt{\frac{4}{\gamma} - 7} + 1}} - 1
    \end{equation*}
    for $0.5 < \gamma < \frac{4}{7}$.
    Thus overall, even in this case,
    the optimal allocation is $\zeta_1 = \gamma$ and $\zeta_2 \to \infty$.
\end{itemize}

\subsubsection*{Optimization problem~\eqref{prob:overparam}}

We now turn to problem~\eqref{prob:overparam}.
In this case, the solution depends on both $\mathrm{SNR}$ and $\gamma$. 
Note that the objective function can be written more compactly as
$h(\zeta_2; h(\zeta_1; \mathrm{SNR}))$
where $h(\gamma; \mathrm{SNR})$ is defined as
\begin{equation*}
    h(\gamma; \mathrm{SNR})
    =
    \mathrm{SNR} \left(1 - \frac{1}{\gamma}\right) + \frac{1}{\gamma - 1}.
\end{equation*}

We first consider the case when $\mathrm{SNR} \le 1$.
We argue that the optimum value in this case is $\mathrm{SNR}$ itself
and it is achieved by setting both $\zeta_1 \to \infty$ and $\zeta_2 \to \infty$.
This can be seen as follows.
For any feasible $\zeta_1 > 1$,
the minimum value of $h(\gamma; \mathrm{SNR})$ is $\mathrm{SNR}$
and it is achieved as $\zeta_1 \to \infty$ from 
\Cref{lem:riskprofile_overparam}~\eqref{lem:riskprofile_overparam-item-snr-less-than-one}.
Since this minimum value is less than $1$,
$h(\zeta_2; \mathrm{SNR})$ is again minimized as $\zeta_2 \to \infty$
and overall minimum is $\mathrm{SNR}$.

Let us consider the case when $\mathrm{SNR} > 1$.
For ease of notation, we denote $\mathrm{SNR}$ by $s$.

We first claim that we 
can restrict to
$\zeta_1 \ge \frac{\sqrt{s}}{\sqrt{s} - 1}$
without loss of generality.
This is because for any $1 < \zeta_1 < \frac{\sqrt{s}}{\sqrt{s} - 1}$,
there is a corresponding $\zeta_1 \ge \frac{\sqrt{s}}{\sqrt{s} - 1}$
that gives either the same or smaller objective value
while enlarging the constraint set for $\zeta_2$.
This claim follows from 
\Cref{lem:riskprofile-onestep-ridgless-overparam}~\eqref{lem:riskprofile-onestep-ridgless-overparam-fixed-first-ratio}.

Next observe that the minimum without the constraint
$\frac{1}{\zeta_1} + \frac{1}{\zeta_2} \le \frac{1}{\gamma}$ is
\begin{equation*}
    2 \sqrt{2 \sqrt{s} - 1} - 1,
\end{equation*}
which is achieved by setting $\zeta_1 = \frac{\sqrt{s}}{\sqrt{s} - 1}$
and $\zeta_2 = \frac{\sqrt{2 \sqrt{s} - 1}}{\sqrt{2 \sqrt{s} - 1} - 1}$.
The values of $\gamma$ for which this value is achievable are:
\begin{equation}
    \label{eq:gammaopt_unconstrained_overparam}
    \gamma
    \le
    \left(
    1 - \frac{1}{\sqrt{s}}
    +
    1 - \frac{1}{\sqrt{2 \sqrt{s} - 1}}
    \right)^{-1}.
\end{equation}
In other words,
the optimum value of problem \eqref{prob:overparam} is
$2 \sqrt{2 \sqrt{s} - 1} - 1$ for $\gamma$ satisfying \eqref{eq:gammaopt_unconstrained_overparam}
achieved by setting
$\zeta_1 = \frac{\sqrt{s}}{\sqrt{s} - 1}$ and
$\zeta_2 = \frac{\sqrt{2 \sqrt{s} - 1}}{\sqrt{2 \sqrt{s} - 1} - 1}$.

Now we consider $\gamma$ bigger than \eqref{eq:gammaopt_unconstrained_overparam}.
For such $\gamma$, we need to move either (or both)
of $\zeta_1$ and $\zeta_2$ from their unconstrained optimum values above.
We claim that the constraint 
$\frac{1}{\zeta_1} + \frac{1}{\zeta_2} \le \frac{1}{\gamma}$
need to be satisfied with equality in this case.
This can be seen as follows.
By way of contradiction,
suppose the optimal allocation is $(\zeta_1^\star, \zeta_2^\star)$,
and $\smash{\frac{1}{\zeta_1^\star} + \frac{1}{\zeta_2^\star} < \frac{1}{\gamma}}$.
We now argue that we can strictly decrease the objective function
while satisfying the constraint
by producing a feasible allocation $(\zeta_1^{\star \star}, \zeta_2^{\star \star})$
that strictly dominates the assumed allocation.
We have two cases to consider.

\begin{enumerate}
    \item $\zeta_1^\star \ge \frac{\sqrt{s}}{\sqrt{s} - 1}$ 
    and 
    $\zeta_2^\star > \frac{\sqrt{2 \sqrt{s} - 1}}{\sqrt{2 \sqrt{s} - 1} - 1}$.
    In this case,
    observe that we can keep $\zeta_1^{\star \star} = \zeta_1^\star$ 
    and decrease $\zeta_2^\star$
    so that $\zeta_2^{\star \star} = \frac{1}{\gamma} - \frac{1}{\zeta_1^\star}$.
    This is feasible.
    Now note that
    \[
        h(\zeta_2^{\star \star}; h(\zeta_1^{\star \star}; s))
        =
        h(\zeta_2^{\star \star}; h(\zeta_1^{\star}; s))
        <
        h(\zeta_2^{\star}; h(\zeta_1^{\star}; s))
    \]
    where the inequality follows from 
    \Cref{lem:riskprofile-onestep-ridgless-overparam}~\eqref{lem:riskprofile-onestep-ridgless-overparam-fixed-second-ratio}.
    Thus, the new allocation strictly decreases the objective value.
    
    \item 
    $\zeta_1^\star > \frac{\sqrt{s}}{\sqrt{s} - 1}$
    and
    $\zeta_2^\star = \frac{\sqrt{2 \sqrt{s} - 1}}{\sqrt{2 \sqrt{s} - 1} - 1}$.
    In this case,
    we can decrease $\zeta_1^\star$ first 
    so that $\zeta_1^{\star \star} = \frac{1}{\gamma} - \frac{1}{\zeta_2^\star}$,
    and keep $\zeta_2^{\star \star} = \zeta_2^{\star}$.
    Observe that this modification keeps us in the feasible region.
    Now note that
    \[
        h(\zeta_2^{\star \star}; h(\zeta_1^{\star \star}; s))
        =
        h(\zeta_2^\star; h(\zeta_1^{\star \star}; s))
        <
        h(\zeta_2^\star; h(\zeta_1^\star; s))
    \]
    where the inequality follows from
    \Cref{lem:riskprofile-onestep-ridgless-overparam}~\eqref{lem:riskprofile-onestep-ridgless-overparam-fixed-first-ratio}.
    Thus, the objective value is again strictly smaller.
\end{enumerate}

Hence, in both the cases,
the objective value can be strictly improved
while staying within the feasible constraint.
Therefore, we must hit the constraint with equality.

With the equality constraint, we can now use the method of Lagrange multipliers.
The Lagrangian is given by
\begin{equation*}
    \cL(\zeta_1, \zeta_2, \mu)
    = h(\zeta_2; h(\zeta_1; s))
    + \mu \left( \frac{1}{\zeta_1} +\frac{1}{\zeta_2} - \frac{1}{\gamma} \right).
\end{equation*}
The optimality conditions are given by the following system of equations in $(\zeta_1, \zeta_2, \mu)$
\begin{gather*}
    \left\{
    s \left( 1 - \frac{1}{\zeta_1} \right) + \frac{1}{\zeta_1 - 1}
    \right\}
    \frac{1}{\zeta_2^2}
    - \frac{1}{(\zeta_2 - 1)^2}
    - \frac{\mu}{\zeta_2^2}
    = 0 \\
    \left( 1 - \frac{1}{\zeta_2} \right)
    \left\{
    \frac{s}{\zeta_1^2} - \frac{1}{(\zeta_1 - 1)^2}
    \right\}
    - \frac{\mu}{\zeta_1^2}
    = 0 \\
    \frac{1}{\zeta_1} + \frac{1}{\zeta_2} = \frac{1}{\gamma}.
\end{gather*}
After minor simplifications, these lead to
\begin{gather*}
    s \left( 1 - \frac{1}{\zeta_1} \right) - \mu
    = \frac{\zeta_1^2}{(\zeta_1 - 1)^2} - \frac{1}{\zeta_1 - 1} \\
    s \left( 1 - \frac{1}{\zeta_2} \right) - \mu
    = \frac{\zeta_1^2}{(\zeta_1 - 1)^2} \left( 1 - \frac{1}{\zeta_2} \right) \\
    \frac{1}{\zeta_1} + \frac{1}{\zeta_2} = \frac{1}{\gamma}.
\end{gather*}
Eliminating $\mu$, we get two equations in two unknowns $(\zeta_1, \zeta_2)$:
\begin{gather*}
    s \left( \frac{1}{\zeta_1} - \frac{1}{\zeta_2} \right)
    = \frac{\zeta_1^2}{(\zeta_1 - 1)^2} - \frac{\zeta_2^2}{(\zeta_2 - 1)^2}
    + \frac{1}{\zeta_1 - 1} \left( 1 - \frac{\zeta_1}{\zeta_2} \frac{\zeta_1}{(\zeta_1 - 1)} \right) \\
    \frac{1}{\zeta_1} + \frac{1}{\zeta_2} = \frac{1}{\gamma},
\end{gather*}
as claimed.

Finally, to obtain various boundary cutoff points for $\gamma$
and $\mathrm{SNR}$
in each of the cases, note that:
\begin{itemize}
    \item 
    When $x = \frac{\mathrm{SNR}}{\mathrm{SNR} + 1}$,
    we have $\frac{x}{1 - x} = \mathrm{SNR}$.
    \item
    When $x = 1 - \frac{1}{2 \sqrt{2 \sqrt{\mathrm{SNR}} - 1}}$,
    we have $\frac{x}{x - \gamma} = 2 \sqrt{2 \sqrt{\mathrm{SNR}} - 1} - 1$.
    In addition, from a short calculation it follows that,
    when $\mathrm{SNR} \approx 10.704$,
    we have
    $1 - \frac{1}{2 \sqrt{2 \sqrt{\mathrm{SNR}} - 1}}
    = \left(2 - \frac{1}{\sqrt{\mathrm{SNR}}} - \frac{1}{\sqrt{2 \sqrt{\mathrm{SNR}} - 1}} \right)^{-1}$.
    \item
    When $x = \gamma^\star$,
    we have $\frac{x}{1 - x}
    = \min_{\gamma \le 1} h(\gamma_2; h(\gamma_1; \mathrm{SNR}))$.
\end{itemize}

This finishes the proof.
See
\Cref{fig:risk-monotonization-onestep-isotropic-risk-optimization-split-illustration}
for an illustration of the optimal splitting of the aspect ratios
$(\zeta_1^\star(\gamma), \zeta_2^\star(\gamma))$
for a given $\gamma$ for two different $\mathrm{SNR}$ values.
\end{proof}

\begin{figure}[!ht]
    \centering
    \includegraphics[width=0.45\columnwidth]{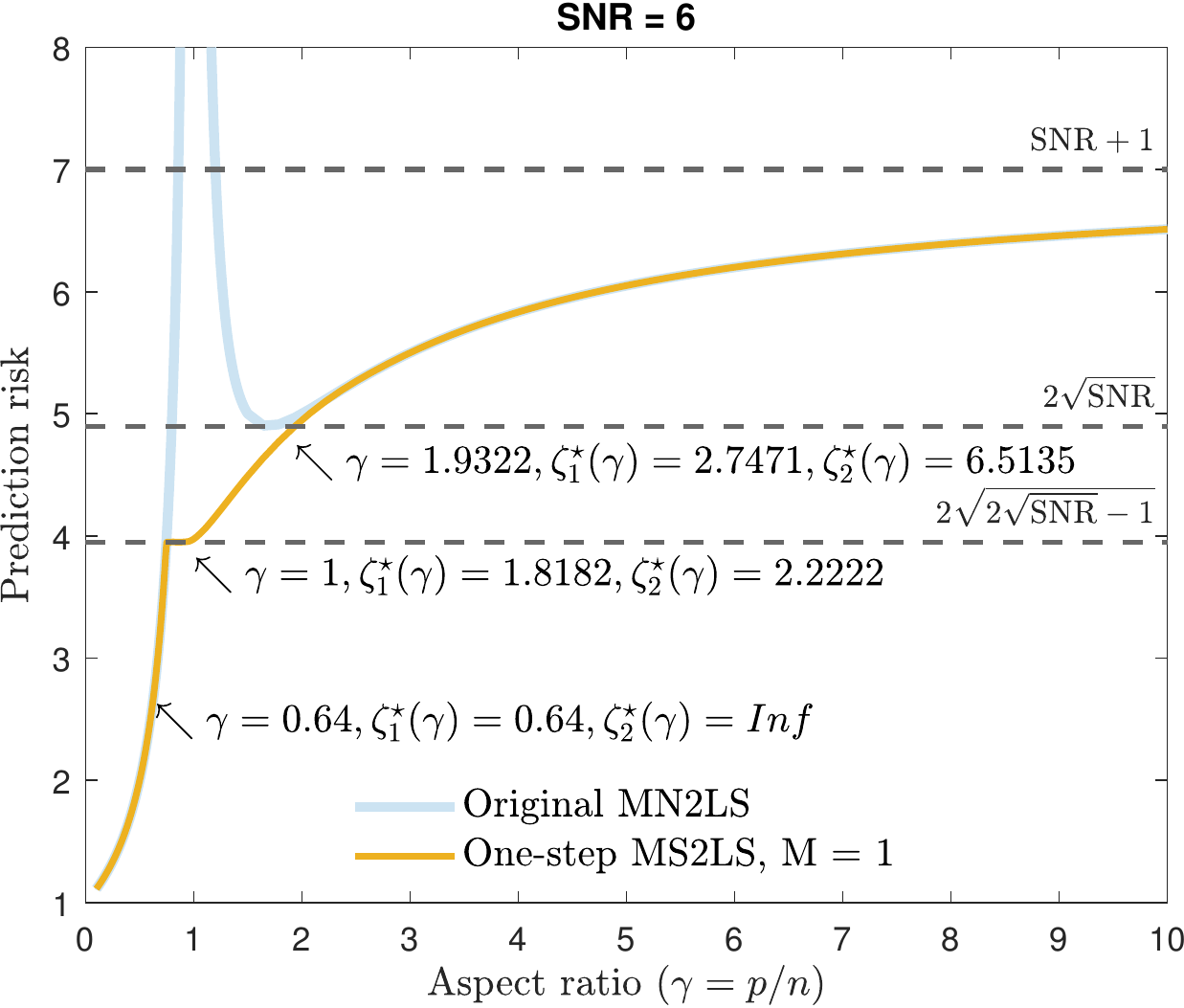}
    \quad
    \includegraphics[width=0.45\columnwidth]{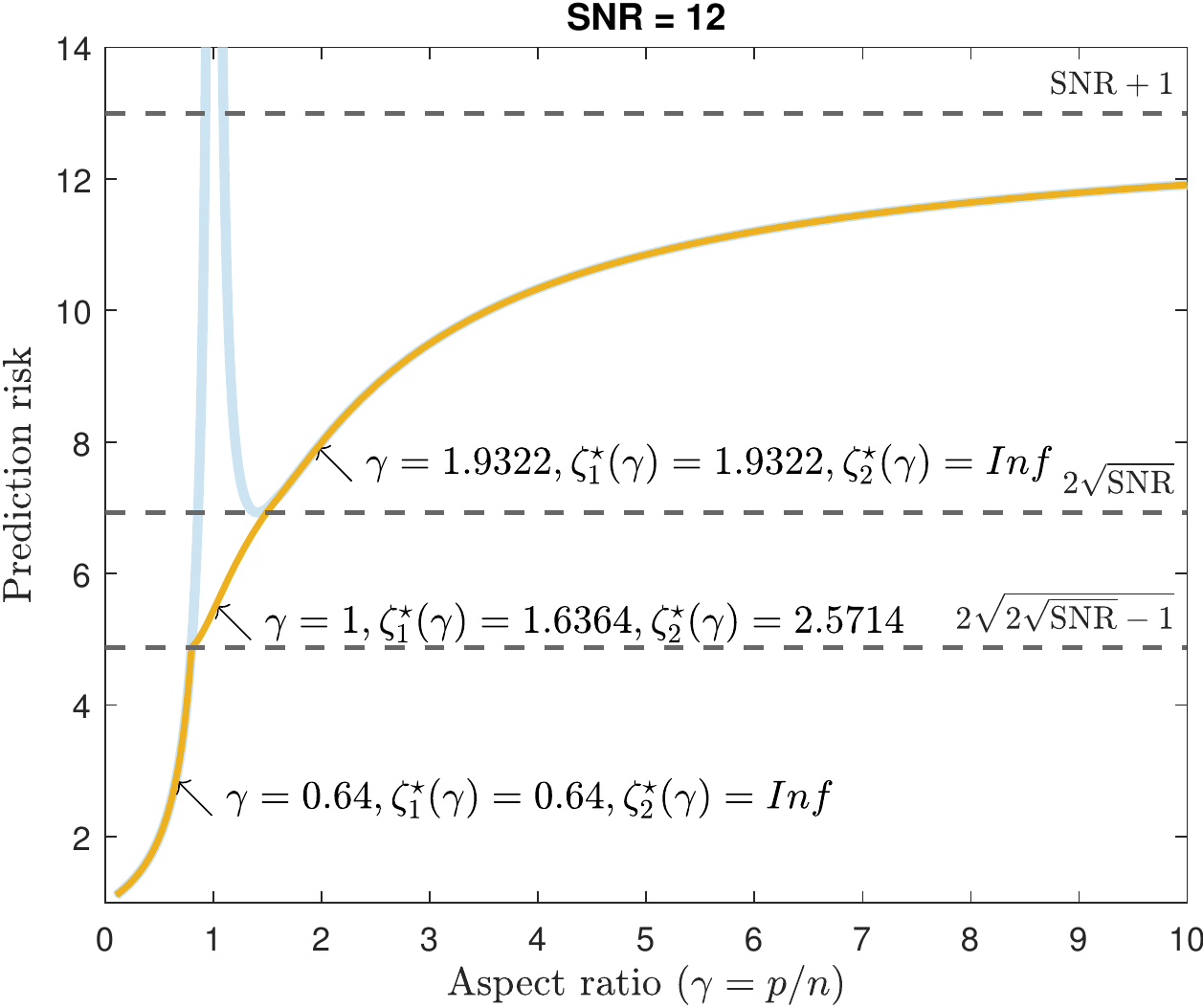}
    \caption{Illustration of the optimal splitting of the aspect ratios
    for the one-step optimization with MN2LS base prediction procedure.
    Here, $(\zeta_1^\star(\gamma), \zeta_2^\star(\gamma))$
    indicates the optimal splitting of the aspect ratio $\gamma$
    for the first and second splits.}
    \label{fig:risk-monotonization-onestep-isotropic-risk-optimization-split-illustration}
\end{figure}

\subsection{Lemmas on properties of risk profile of ridgeless regression}

In this section,
we collect helper lemmas used in the proof of
\Cref{thm:esterrlim_optimized_onestep}.
All the lemmas in this section
are quite elementary,
and only abstracted out for ease
of repeated use in the proof of \Cref{thm:esterrlim_optimized_onestep}. 

\begin{lemma}
    [Properties of ridgeless risk profile in the underparameterized regime]
    \label{lem:riskprofile_underparam}
    The function $g:  x \mapsto \frac{x}{1 - x}$ 
    over the domain 
    $(0, 1)$
    has the following properties:
    \begin{enumerate}
        \item 
        \label{lem:riskprofile_underparam-item-increasing}
        The function $g$ is increasing in $x$.
        \item 
        \label{lem:riskprofile_underparam-item-less-than-one}
        When $x \le 0.5$, $g(x) \le 1$.
        \item 
        \label{lem:riskprofile_underparam-item--greater-than-one}
        When $x > 0.5$, $g(x) > 1$.
    \end{enumerate}
\end{lemma}
\begin{proof}
    The claims are easy to check.
    See \Cref{fig:isotropic_risk} 
    (the $x < 1$ segment)
    for illustration.
\end{proof}

\begin{lemma}
    [Properties of ridgeless risk profile in the overparameterized regime]
    \label{lem:riskprofile_overparam}
    Let $h(\cdot; s) : x \mapsto s \left(1 - \frac{1}{x} \right) + \frac{1}{x - 1}$
    be a function defined on the domain $x > 1$, parametrized by $s \ge 0$.
    The function $h$ has the following properties:
    \begin{enumerate}
        \item 
        \label{lem:riskprofile_overparam-item-snr-less-than-one}
        When $s \le 1$, the function is decreasing in $x$
        and approaches the minimum value of $s$ as $x \to \infty$.
        \item 
        \label{lem:riskprofile_overparam-item-snr-greater-than-one-min}
        When $s > 1$, the function attains
        the minimum value of $2 \sqrt{s} - 1$ at $x = \frac{\sqrt{s}}{\sqrt{s} - 1}$.
        \item 
        \label{lem:riskprofile_overparam-item-snr-greater-than-one-risk}
        When $s > 1$, $h(x; s) > 1$ for all $x > 1$.
        \item 
        \label{lem:riskprofile_overparam-item-after-min-increasing}
        For $x > \frac{\sqrt{s}}{\sqrt{s} - 1}$,
        the function is increasing in $x$.
        \item
        \label{lem:riskprofile_overparam-item-fixed-gamma-increasing-in-snr}
        The function $s \mapsto h(x; s)$
        is increasing in $s$ for $s \ge 0$
        for any fixed $x > 1$.
    \end{enumerate}
\end{lemma}
\begin{proof}
    The first property is easy to check.
    The second property follows elementary calculus.
    The third property follows from the second property.
    The fourth property follows by inspecting
    the derivative of $h(\cdot; s)$ for $x > \frac{\sqrt{s}}{\sqrt{s} - 1}$.
    The fifth property is easy to check.
    See Figure \ref{fig:isotropic_risk} (the $x > 1$ segment) for illustration.
    
    \begin{figure}[!ht]
        \centering
        \includegraphics[width=0.5\columnwidth]{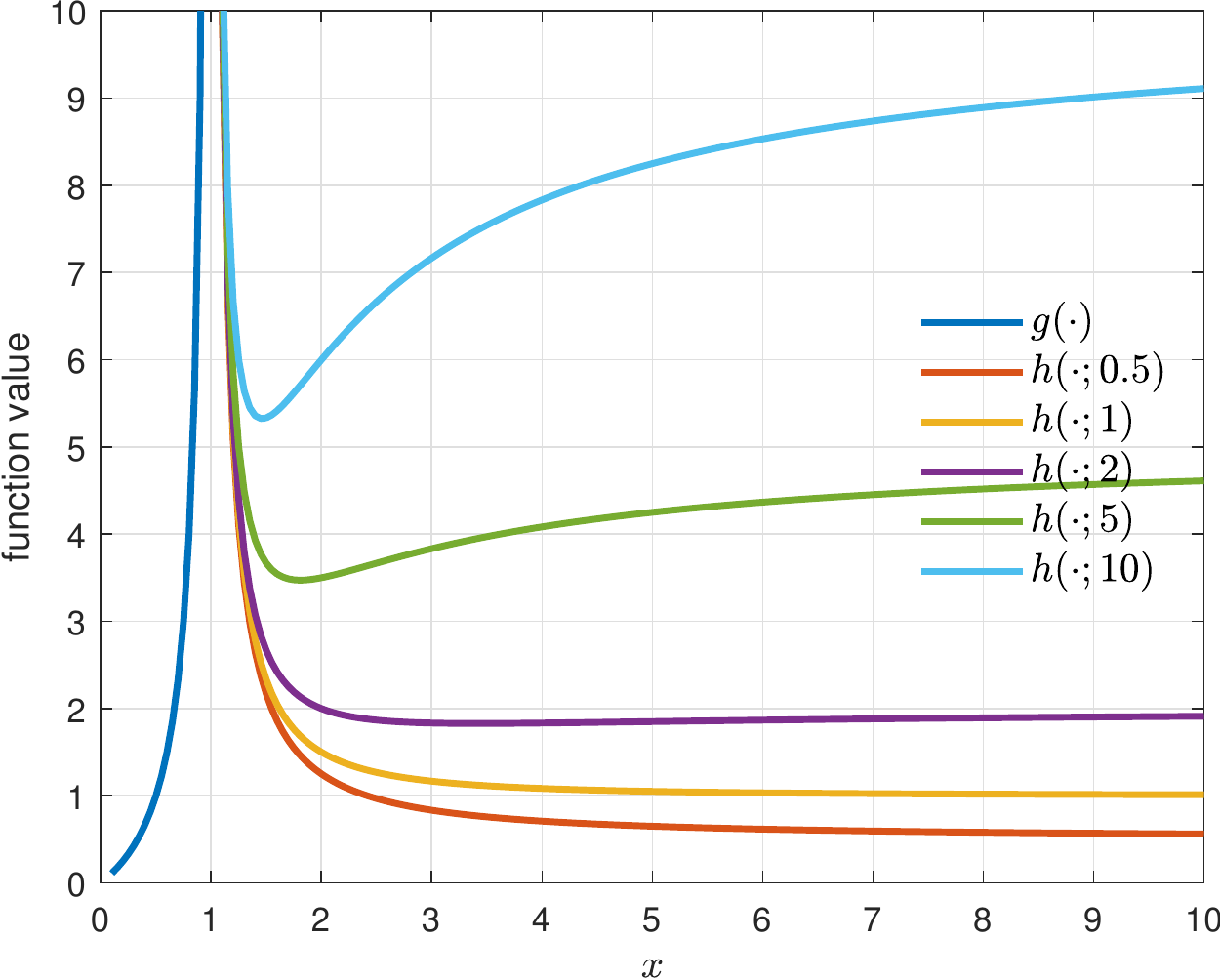}
        \caption{Illustration of ridgeless risk profile with varying SNR.}
        \label{fig:isotropic_risk}
    \end{figure}
\end{proof}

\begin{lemma}
    [Properties of ridgeless one-step ingredient risk profile
    in the overparameterized regime]
    \label{lem:riskprofile-onestep-ridgless-overparam}
    Let $h(x; s) : x \mapsto s \left(1 - \frac{1}{x}\right) + \frac{1}{x - 1}$
    be a function defined on the domain $x > 1$,
    parameterized by $s \ge 1$.
    Let $g: (x, y) \mapsto h(y; h(x; s))$
    be a function defined on the domain $x > 1$ and $y > 1$,
    parameterized by $s \ge 1$.
    The function $g$ has the following properties:
    \begin{enumerate}
        \item
        \label{lem:riskprofile-onestep-ridgless-overparam-fixed-first-ratio}
        For any fixed $y > 1$,
        the function $g$
        is minimized at $x = \frac{\sqrt{s}}{\sqrt{s} - 1}$
        and increasing in $x$ for $x \ge \frac{\sqrt{s}}{\sqrt{s} - 1}$.
        \item 
        \label{lem:riskprofile-onestep-ridgless-overparam-fixed-second-ratio}
        For any fixed $x > 1$,
        $g(x, y)$
        is increasing over
        $y \ge \frac{\sqrt{h(x; s)}}{\sqrt{h(x; s)} - 1}$.
    \end{enumerate}
\end{lemma}
\begin{proof}
    The first claim follows from
    \Cref{lem:riskprofile_overparam}~\eqref{lem:riskprofile_overparam-item-snr-greater-than-one-min},
    \eqref{lem:riskprofile_overparam-item-after-min-increasing},
    \eqref{lem:riskprofile_overparam-item-fixed-gamma-increasing-in-snr}.
    The second claim follows
    from 
    \Cref{lem:riskprofile_overparam}~\eqref{lem:riskprofile_overparam-item-after-min-increasing}.
\end{proof}

\subsection{Control of additive error term in expectation}
\label{sec:growth-rates-expectation-bound}

The following remark 
complements \Cref{rem:growth-rates-probabilistic-bound}
and specifies
the growth allowed conditions on $\hsigma_{\Xi}$
to ensure that $\EE[\Delta_n^\add] = o(1)$.

\begin{remark}
    [Tolerable growth rates on $\hsigma_{\Xi}$
    for $\mathbb{E}\Delta_n^{\add} = o(1)$]\label{rem:growth-rates-expectation-bound}
    Suppose $|\Xi| \le n^{S}$ for some $S < \infty$.
    Under the setting of
    \Cref{lem:bounded-orlitz-error-control},
    if for some $t \ge 1$,
    \[
        \max_{\xi \in \Xi}\,
        \| \hsigma_\xi \|_{L_t}
        = o\left(\frac{n_\test^{1/2}}{n^{-A  + (A + S)/t}}\right),
    \]
    then $\EE[\Delta_n^{\add}] = o(1)$.
    On the other hand,
    under the setting of
    \Cref{lem:bounded-variance-error-control},
    if
    \[
        \max_{\xi \in \Xi}\,
        \| \hsigma_\xi \|_{L_2}
        = o\left(\frac{n_\test^{1/2}}{n^{(S-A)/2}}\right) 
    \]
    then $\EE[\Delta_n^{\add}] = o(1)$.
    The remark follows simply by observing that
    the first term in the expectation bounds
    \eqref{eq:bounded-orlicz-expectation-bound}
    and \eqref{eq:bounded-variance-expectation-bound}
    for both
    \Cref{lem:bounded-orlitz-error-control,lem:bounded-variance-error-control}
    are $o(1)$,
    while the second term in
    \Cref{lem:bounded-orlitz-error-control}
    is of order
    \[
      O\left(\frac{n^{-A/r + S/t}}{n_\test^{1/2}}\right)
      \max_{\xi \in \Xi} \| \hsigma_\xi \|_{L_t},
    \]
    for $r, t \ge 1$ and $1/r + 1/t = 1$,
    and the second term in
    \Cref{lem:bounded-variance-error-control}
    is of order
    \[
        O\left(\frac{n^{-A/2  + S/2}}{n_\test^{1/2}}\right)
        \max_{\xi \in \Xi} \| \hsigma_\xi \|_{L_2}.
    \]
\end{remark}

It is worth mentioning that
one can also derive suitable
growth rates on $\hkappa_{\Xi}$
that yield conditions for $\EE[\Delta_n^\mul] = o(1)$.
However, this does not directly
lead to control of $\EE[R(\hf^\cv(\cdot; \cD_n))]$
in the multiplicative form \eqref{eq:oracle-risk-inequality-multiplicative-form}.
This is because of the denominator $(1 - \Delta_n^\mul)_{+}$
appearing in \eqref{eq:oracle-risk-inequality-multiplicative-form}.
For every $n$,
there is a non-zero probability that
the denominator $(1 - \Delta_n^\mul)_+$ is zero. 
Hence,
the right hand side of \eqref{eq:oracle-risk-inequality-multiplicative-form}
may not have a finite expectation in general.
However,
assuming
$\EE[R(\hf^\xi(\cdot; \cD_n))] < C$ for some $C < \infty$
for all $\xi \in \Xi$,
one can control $\EE[R(\hf^\cv(\cdot; \cD_n))]$
by explicitly analyzing
$\PP(\Delta^\mul_n > 1/2)$,
and using the bound
\begin{align*}
    R(\hf^\cv(\cdot; \cD_n))
    &\le
    \frac{1 + \Delta_n^\mul}{(1 - \Delta_n^\mul)_{+}}
    \cdot \min_{\xi \in \Xi} R(\hf^\xi(\cdot; \cD_\train).
    \1_{\Delta_n^\mul \le 1/2}
    +
    \sum_{\xi \in \Xi}
    R(\hf^\xi(\cdot; \cD_n))
    \1_{\Delta_n^\mul > 1/2}.
\end{align*}

\subsection{A lemma on norm equivalence implications}
\label{sec:norm-equivalence-implications}

The following lemma formalizes
various norm equivalence implications
mentioned in \Cref{rem:psi2l2-l4l2,rem:psi1l1-l2l1}.

\begin{proposition}
    [Norm equivalence implications]
    \label{prop:norm-equivalences}
    The following statements hold.
    \begin{enumerate}
        \item
    Suppose a random $X$ satisfies $L_4-L_2$ equivalence,
    i.e., there exists a constant $C$ such that $\EE[X^4] \le C \EE[X^2]$,
    then the random variable satisfies $L_2-L_1$ equivalence,
    i.e., there exists a constant $C$ such that $\EE[X^2] \le C \EE[| X |]$.
    \item A random variable $W$ satisfying $\psi_2-L_2$ equivalence
    also satisfies $\psi_1-L_1$ equivalence.
    \end{enumerate}
\end{proposition}
\begin{proof}
    We will use the fact 
    that the map
    $p \mapsto \log \EE[|X|^p]$ ($p \ge 1$) is convex.
    In other words,
    for $\lambda \in (0, 1)$,
    we have
    \begin{equation}
        \label{eq:liapunov-convexity-condition}
        \log \EE[|X|^{\lambda r + (1 - \lambda) s}]
        \le \lambda \log \EE[|X|^r]
        + (1 - \lambda) \log \EE[|X|^s]
    \end{equation}
    We now use $r = 4$ and $s = 1$,
    and $\lambda = 1/3$ so that $\lambda r + (1 - \lambda) s = 2$.
    Plugging these choices in \eqref{eq:liapunov-convexity-condition} yields
    \[
        \log \EE[X^2]
        \le 
        \frac{1}{3} \log \EE[X^4] + \frac{2}{3} \log \EE[|X|].
    \]
    In terms of norms the inequality then becomes
    \[
        2 \log \| X \|_{L_2}
        \le \frac{4}{3} \log \| X \|_{L_4} + \frac{2}{3} \log \| X \|_{L_1}.
    \]
    This yields
    \[
        \frac{2}{3} \log \frac{\| X \|_{L_2}}{\| X \|_{L_1}}
        \le \frac{4}{3} \log \frac{\| X \|_{L_4}}{\| X \|_{L_2}}.
    \]
    Manipulating both sides, we end up with
    \[
        \frac{\| X \|_{L_2}}{\| X \|_{L_1}}
        \le \left( \frac{\| X \|_{L_4}}{\| X \|_{L_2}} \right)^2
    \]
    as desired.
    
    The second facts follows
    because $\psi_2-L_2$ equivalence
    implies $L_p-L_2$ equivalence for each $p \ge 1$,
    i.e., for each $p \ge 1$,
    we have that
    \[
        \| W \|_{L_p}
        \le C \sqrt{p} \| W \|_{L_2},
    \]
    for an universal constant $C$;
    see \citet[Proposition 2.5.2]{vershynin_2018}, for example.
    This in particular implies,
    $L_4-L_2$ equivalence,
    and by the first fact implies $L_2-L_1$.
    Thus, there exists a universal constant $C$
    such that
    \[
        \| W \|_{L_2} \le \| W \|_{L_1}.
    \]
    Combining with the inequality above,
    we then get for $p \ge 1$,
    \[
        \| W \|_{L_p}
        \le C \sqrt{p} \| W \|_{L_1}
        \le C p \| W \|_{L_1}.
    \]
    Now, using \citet[Proposition 2.7.1]{vershynin_2018},
    this implies $\psi_1-L_1$ equivalence.
   
   Alternatively,
   assuming $\psi_2-L_2$ equivalence,
   observe the following chain of inequalities:
   \[
        C \| X \|_{L_4}
        \overset{(a)}{\le} \| X \|_{\psi_1}
        \overset{(b)}{\le} (\log 2)^{1/2} \| X \|_{\psi_2} 
        \overset{(c)}{\le} C \| X \|_{L_2}
   \]
   where $(a)$ follows from \citet[Proposition 2.5.2]{vershynin_2018},
   $(b)$ follows from \citet[Problem 2.2.5]{wellner_vandervaart_2013},
   $(c)$ follows from the assumed $\psi_2-L_2$ equivalence.
   Finally, since $\psi_2-L_2$ equivalence 
   implies $L_4-L_2$ equivalence,
   and from the fact this implies $L_2-L_1$
   equivalence concludes the proof.
   
   \Cref{fig:norm-equivalences-implications}
   visually summarizes the norm equivalence implications.
\end{proof}

\begin{figure}[!ht]
    \centering
    \includegraphics[width=0.5\columnwidth]{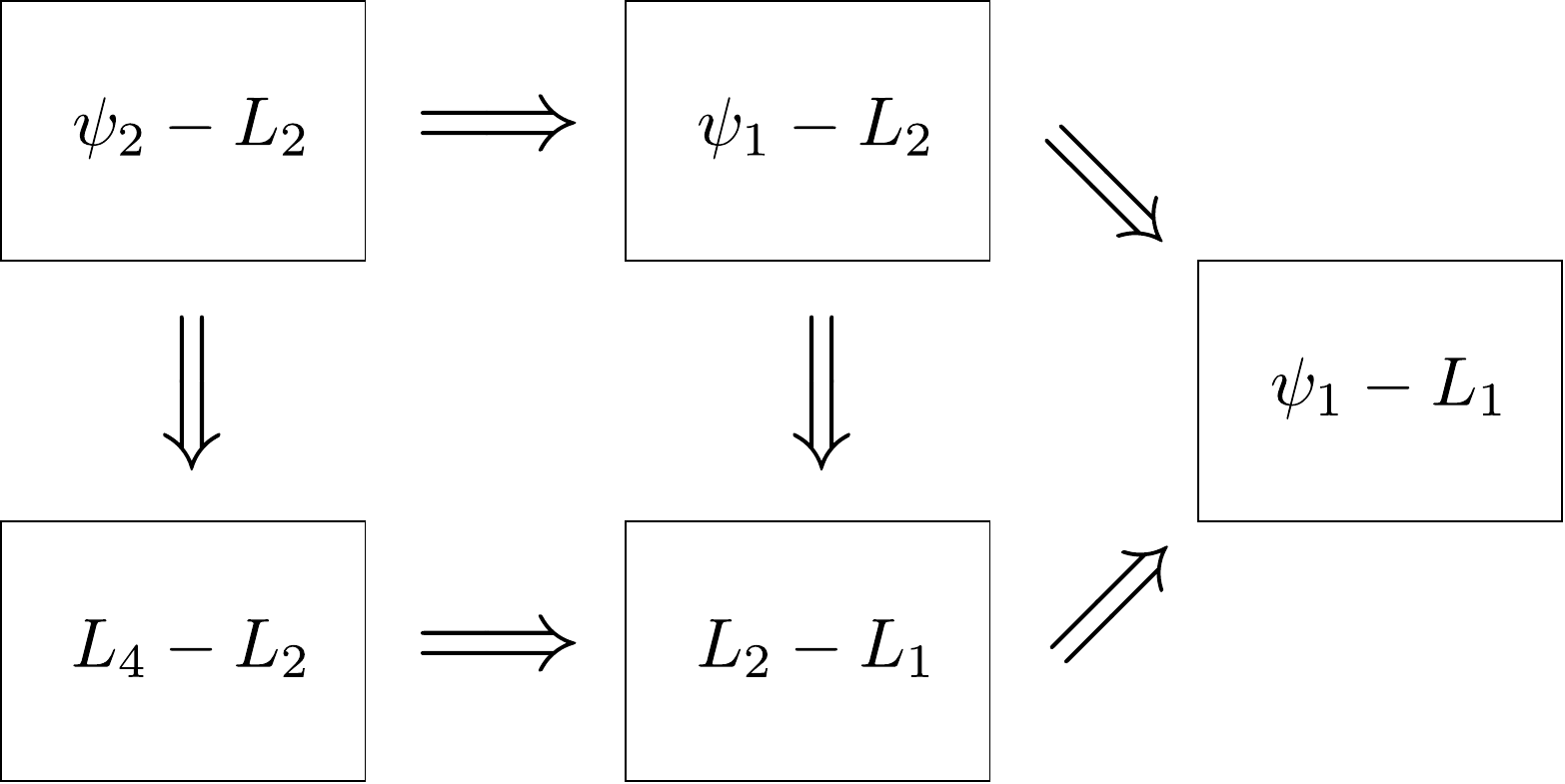}
    \caption{Visual illustration of norm equivalence implications
    discussed in \Cref{rem:psi2l2-l4l2,rem:psi1l1-l2l1},
    and in the proof of \Cref{prop:norm-equivalences}.
    In the figure, 
    $\fbox{A} \Rightarrow \fbox{B}$ indicates
    that equivalence $A$ implies equivalence $B$.
    }
    \label{fig:norm-equivalences-implications}
\end{figure}

\subsection
[Squared risk decomposition for ingredient zero-step predictor]
{Proof of \eqref{eq:sqauredrisk_decomposition_zerostep_bagged}}
\label{sec:squaredrisk_decomposition_zerostep_bagged}

Below we prove the risk decomposition \eqref{eq:sqauredrisk_decomposition_zerostep_bagged} 
for the ingredient zero-step predictor under squared error loss.
The proof follows from the following iterated bias-variance decomposition.

\begin{align*}
    &\EE
    \big[ 
        (Y_0 - \tf_M(X_0; \cD_\train))^2 
        \mathrel{| } 
        \cD_\train 
    \big] \\
    &=  
    \EE
    \Big[ 
        \EE
        \big[ 
            (Y_0 - \hf_M(X_0; \cD_\train))^2 
            \mathrel{| } 
            \cD_\train, (X_0, Y_0) 
        \big] 
    \mid \cD_\train 
    \Big] 
    \\
    &= 
    \EE 
    \Big[ 
        \Big(
            Y_0
            -
            \EE
            \big[
                \tf_M(X_0; \cD_\train)
                \mathrel{| }
                \cD_\train, (X_0, Y_0)
            \big]
        \Big)^2
        \mathrel{\big |}
        \cD_\train
    \Big]
    + 
    \EE 
    \Big[ 
        \mathrm{Var}
        \big(
            \tf_M(X_0; \cD_\train)
            \mathrel{| }
            \cD_\train, (X_0, Y_0)
        \big)
        \mathrel{\big |}
        \cD_\train
    \Big] \\
    &= 
    \EE
    \left[ 
    \left(
        Y_0
        -
        \frac{1}{\binom{n}{k_n}}
        \sum_{i_1, \dots, i_{k_n}}
        \tf\big(X_0; \{ (X_{i_j}, Y_{i_j}) : 1 \le j \le k_n \}\big)
    \right)^2
    \mathrel{\Bigg |}
    \cD_\train
    \right]
    + 
    \EE 
    \left[ 
    \frac{1}{M}
    \mathrm{Var}
    \left(
        \tf(X_0; \cD_{\train, 1})
        \mathrel{\big |}
        \cD_\train, (X_0, Y_0)
    \right)
    \mathrel{\bigg |}
    \cD_\train
    \right] 
    \\
    &= R(\tf_\infty(\cdot; \cD_\train))
    + 
    \frac{1}{M} 
    \EE 
    \left[ 
        \frac{1}{\binom{n}{k_n}}
        \sum_{i_1, \dots, i_{k_n}}
        \left(
            \tf\big(X_0; \{ (X_0, Y_0) : 1 \le j \le k_n \}\big)
            -
            \tf_\infty(X_0; \cD_\train)
        \right)^2
        \mathrel{\bigg |}
        \cD_\train
    \right],
\end{align*}
where in the last line $f_\infty(\cdot; \cD_\train) : \RR^{p} \to \RR$ is defined
such that for any $x \in \RR^{p}$
\[
\tf_{\infty}(x; \mathcal{D}_{\train}) 
= \frac{1}{\binom{n}{k_n}}\sum_{1\le i_1 < \ldots < i_{k_n} \le n_{\train}} \tf(x; \{(X_{i_j}, Y_{i_j}): 1\le j\le k_n\}).
\]

\section{Calculus of deterministic equivalents}
\label{sec:calculus_deterministic_equivalents}

We use the language of deterministic equivalents
in the proofs of
\Cref{prop:asymp-bound-ridge-main} 
and \Cref{prop:verif-riskprofile-mn1lsbase-mn2lsonestep}
in 
\Cref{sec:verif-asymp-profile-ridge}
and
\Cref{sec:verif-riskprofile-mnlsbase-mnlsonestep},
respectively.
In this section, we provide a basic review of the definitions
and useful calculus rules.
For more details, see \cite{dobriban_sheng_2021}.

\begin{definition}
    \label{def:deterministic-equivalence}
    Consider sequences $\{ A_p \}_{p \ge 1}$ and $\{ B_p \}_{p \ge 1}$
    of (random or deterministic) matrices of growing dimension.
    We say that $A_p$ and $B_p$ are equivalent and write
    $A_p \asympequi B_p$ if
    $\lim_{p \to \infty} | \tr[C_p (A_p - B_p)] | = 0$ almost surely
    for any sequence $C_p$ matrices with bounded trace norm
    such that $\limsup \| C_p \|_{\mathrm{tr}} < \infty$
    as $p \to \infty$.
\end{definition}

An observant reader will notice that
\cite{dobriban_sheng_2021}
use the notation $A_p \asymp B_p$
to denote deterministic asymptotic equivalence.
In this paper, we instead prefer to use the notation
$A_p \asympequi B_p$ for such equivalence
to stress the fact that this equivalence
is exact in the limit rather than up to constants
as the ``standard'' use of the asymptotic notation
$\asymp$ would hint at.

\begin{lemma}
    [Calculus of deterministic equivalents, \cite{dobriban_wager_2018}, \cite{dobriban_sheng_2021}]
    \label{lem:calculus-detequi}
    Let $A_p$, $B_p$, and $C_p$ be sequences of (random or deterministic) matrices.
    The calculus of deterministic equivalents satisfy the following properties:
    \begin{enumerate}
        \item 
        \label{lem:calculus-detequi-item-equivalence}
        Equivalence:
        The relation $\asympequi$ is an equivalence relation.
        \item 
        \label{lem:calculus-detequi-item-sum}
        Sum:
        If $A_p \asympequi B_p$ and $C_p \asympequi D_p$, then $A_p + C_p \asympequi B_p + D_p$.
        \item 
        \label{lem:calculus-detequi-item-product}
        Product:
        If $A_p$ a sequence of matrices with bounded operator norms, i.e., $\| A_p \|_{\op} < \infty$,
        and $B_p \asympequi C_p$, then $A_p B_p \asympequi A_p C_p$.
        \item 
        \label{lem:calculus-detequi-item-trace}
        Trace:
        If $A_p \asympequi B_p$, then $\tr[A_p] / p - \tr[B_p] / p \to 0$ almost surely.
        \item 
        \label{lem:calculus-detequi-item-differentiation}
        Differentiation:
        Suppose $f(z, A_p) \asympequi g(z, B_p)$ where the entries of $f$ and $g$
        are analytic functions in $z \in S$ and $S$ is an open connected subset of $\CC$.
        Suppose for any sequence $C_p$ of deterministic matrices with bounded trace norm
        we have $| \tr[C_p (f(z, A_p) - g(z, B_p))] | \le M$ for every $p$ and $z \in S$.
        Then we have $D_{z}(f(z, A_p)) \asympequi D_{z}(g(z, B_p))$ for every $z \in S$,
        where the derivatives are taken entry-wise with respect to $z$.
    \end{enumerate}
\end{lemma}

We record deterministic equivalent for the standard ridge resolvent.

\begin{lemma}
    [Deterministic equivalent for basic ridge resolvent,
    adapted from 
    Theorem 1 of \cite{rubio_mestre_2011};
    see also Theorem 3.1 of \cite{dobriban_sheng_2021}]
    \label{lem:basic-ridge-resolvent-deterministic-equivalent}
    Suppose $X_i \in \RR^{p}$, $1 \le i \le n$, are i.i.d.\ random vectors
    where each $X_i = Z_{i} \Sigma^{1/2}$,
    where $Z_i$ contains i.i.d.\ entries $Z_{ij}$, $1 \le j \le p$,
    with $\EE[Z_{ij}] = 0$, $\EE[Z_{ij}^2] = 1$, and $\EE[|Z_{ij}|^{8+\alpha}] \le M_\alpha$
    for some $\alpha > 0$ and $M_\alpha < \infty$,
    and $\Sigma \in \RR^{p \times p}$ is a positive semidefinite matrix
    such that $0 \preceq \Sigma \preceq r_{\max} I_p$
    for some constant
    (independent of $p$)
    $r_{\max} < \infty$.
    Let $\bX \in \RR^{n \times p}$ the matrix with $X_i$, $1 \le i \le n$ as rows
    and $\hSigma \in \RR^{p \times p}$ denote the random matrix $\bX^\top \bX / n$.
    Define $\gamma_n = p / n$.
    Then, 
    for $z \in \CC^{> 0}$,
    as $n, p \to \infty$ such that $0 < \liminf \gamma_n \le \limsup \gamma_n < \infty$,
    we have
    \begin{equation}
        (\hSigma - z I_p)^{-1}
        \asympequi
        (c(e(z; \gamma_n)) \Sigma - z I_p)^{-1},
    \end{equation}
    where
    $c(e(z; \gamma_n))$ is defined as
    \begin{equation}
        \label{eq:basic-ridge-equivalence-c-e-relation}
        c(e(z; \gamma_n))
        = \frac{1}{ 1 + \gamma_n e(z; \gamma_n)},
    \end{equation}
    and $e(z; \gamma_n)$ is the unique solution in $\CC^{> 0}$ to the fixed-point equation
    \begin{equation}
        \label{eq:basic-ridge-equivalence-e-fixed-point}
        e(z; \gamma_n)
        = 
        \tr[ \Sigma (c(e(z; \gamma_n)) \Sigma  - z I_p)^{-1} ] / p.
    \end{equation}
    Furthermore,
    $e(z; \gamma_n)$
    is the Stieltjes transform
    of a certain positive measure on $\RR_{\ge 0}$
    with total mass $\tr[\Sigma] / p$.
\end{lemma}

We note that in defining $e(\lambda; \gamma_n)$,
it is also implicitly a parameterized by $\Sigma$.
We suppress this dependence for notational simplicity,
and only explicitly indicate dependence on 
$z$
and $\gamma_n$
that will be useful for our purposes.

\begin{corollary}
    \label{cor:basic-ridge-resolvent-equivalent-in-v}
    Assume the setting of 
    \Cref{lem:basic-ridge-resolvent-deterministic-equivalent}.
    For $\lambda > 0$,
    we have
    \[
        \lambda (\hSigma + \lambda I_p)^{-1}
        \asympequi
        (v(-\lambda; \gamma_n) \Sigma + I_p)^{-1},
    \]
    where $v(-\lambda; \gamma_n)$
    is the unique solution to the fixed-point equation
    \[
        \frac{1}{v(-\lambda; \gamma_n)}
        =
        \lambda
        + \gamma_n \tr[\Sigma (v(-\lambda; \gamma_n) \Sigma + I_p)^{-1}] / p.
    \]
\end{corollary}
\begin{proof}
    From \Cref{lem:basic-ridge-resolvent-deterministic-equivalent},
    for $z \in \CC^{> 0}$,
    we have
    the basic equivalence for ridge resolvent
    \begin{equation}
        \label{eq:basic-ridge-equivalence-in-z-v1}
        (\hSigma - z I_p)^{-1}
        \asympequi
        (c(e(z; \gamma_n)) \Sigma  - z I_p)^{-1},
    \end{equation}
    where $c(e(z; \gamma_n))$
    is defined by 
    \eqref{eq:basic-ridge-equivalence-c-e-relation}
    and 
    and $e(z; \gamma_n)$
    is the unqiue solution in $\CC^{> 0}$
    to the fixed-point equation
    \eqref{eq:basic-ridge-equivalence-e-fixed-point}.
    Substituting for $e(z; \gamma_n)$
    from
    \eqref{eq:basic-ridge-equivalence-c-e-relation}
    into
    \eqref{eq:basic-ridge-equivalence-e-fixed-point},
    we can write
    the fixed-point equation for $c(e(z; \gamma_n))$
    as
    \begin{equation}
        \label{eq:fixed-point-for-c-v1}
        \frac{1}{c(e(z; \gamma_n)) \gamma_n}
        - \frac{1}{\gamma_n}
        = \tr[\Sigma (c(e(z; \gamma_n)) \Sigma - z I_p)^{-1}] / p.
    \end{equation}
    Manipulating
    \eqref{eq:fixed-point-for-c-v1},
    we can write
    \begin{equation}
        \label{eq:fixed-point-for-c-v2}
        \frac{1}{c(e(z; \gamma_n))}
        - 1
        = \gamma_n \tr[\Sigma (c(e(z; \gamma_n)) \Sigma - z I_p)^{-1}] / p
        = \frac{\gamma_n}{(-z)}
        \tr[\Sigma ( c(e(z; \gamma_n)) / (-z) \Sigma + I_p)^{-1}] / p.
    \end{equation}
    Moving $(-z)$ across in \eqref{eq:fixed-point-for-c-v2},
    we have equivalently
    the following equation for $c(e(z; \gamma_n))$:
    \begin{equation}
        \label{eq:fixed-point-for-c-v3}
        \frac{(-z)}{c(e(z; \gamma_n))}
        + z
        = 
        \gamma_n
        \tr[\Sigma (c(e(z; \gamma_n)) / (-z))^{-1}] / p.
    \end{equation}
    Now defining $c(e(z; \gamma_n)) / (-z)$
    by $v(z; \gamma_n)$,
    the fixed-point equation \eqref{eq:fixed-point-for-c-v3} becomes
    \begin{equation}
        \label{eq:fixed-point-for-v-in-z}
        \frac{1}{v(z; \gamma_n)}
        = - z + \gamma_n \tr[\Sigma (v(z; \gamma_n) \Sigma + I_p)^{-1}] / p.
    \end{equation}
    Note that \eqref{eq:fixed-point-for-v-in-z}
    is also known as the Silverstein equation \citep{silverstein_1995},
    and $v(z; \gamma_n)$ as the companion Stieltjes transform.
    Along the same lines,
    from \eqref{eq:basic-ridge-equivalence-in-z-v1},
    we have
    \begin{equation}
        \label{eq:basic-ridge-equivalence-in-z-v2}
        (-z) (\hSigma - z I_{p})^{-1}
        \asympequi (-z) (c(e(z; \gamma_n)) \Sigma - z I_p)^{-1}
        = (c(e(z; \gamma_n))/(-z) \Sigma + I_p)^{-1}.
    \end{equation}
    Substituting for $v(z; \gamma_n)$,
    we can thus write
    \begin{equation}
        \label{eq:basic-ridge-equivalence-in-v}
        (-z) (\hSigma - z I_p)^{-1}
        \asympequi
        (v(z; \gamma_n) \Sigma + I_p)^{-1}.
    \end{equation}
    Now, taking $z = -\lambda$
    in 
    \eqref{eq:fixed-point-for-v-in-z}
    and
    \eqref{eq:basic-ridge-equivalence-in-v}
    yields
    the equivalence
    \[
        \lambda (\hSigma + \lambda I_p)^{-1}
        \asympequi (v(-\lambda; \gamma_n) \Sigma + I_p)^{-1},
    \]
    where $v(-\lambda; \gamma_n)$ 
    is the unique solution to the fixed point equation
    \[
        \frac{1}{v(-\lambda; \gamma_n)}
        = \lambda 
        + \gamma_n \tr[\Sigma (v(-\lambda; \gamma_n) \Sigma + I_p)^{-1}] / p.
    \]
    Finally, since $v(-\lambda; \gamma_n)$
    is a Stieltjes transform of a probability measure 
    (with support on $\RR_{\ge 0}$),
    we have that
    for $\Re(\lambda) > 0$, by taking $\Im(\lambda) \to 0$,
    we have that $\Im(v(-\lambda; \gamma_n)) \to 0$,
    and thus the statement 
    follows.
\end{proof}

We remark that we will directly apply 
\Cref{cor:basic-ridge-resolvent-equivalent-in-v}
for a real $\lambda > 0$
(in particular, in \Cref{lem:deter-approx-generalized-ridge}).
The limiting argument to go from
a complex $\lambda$ to a real $\lambda$
follow as done in the proof of 
\Cref{cor:basic-ridge-resolvent-equivalent-in-v}.
See, for example, proof of Theorem 5 in \cite{hastie_montanari_rosset_tibshirani_2019}
(that uses
Lemma 2.2 of
\cite{knowles_yin_2017})
for more details.

\section{Useful concentration results}
\label{sec:useful-concentration-results}

In this section,
we gather statements of concentration results
available in the literature
that are used 
in the proofs in 
\Cref{sec:proofs-crossvalidation-modelselection,sec:verif-asymp-profile-ridge,sec:verif-riskprofile-mnlsbase-mnlsonestep}.

\subsection*{Non-asymptotic statements}

\paragraph{Tail bounds.}

The following two tail bounds
are used in the proofs of
\Cref{lem:bounded-orlitz-error-control,lem:bounded-variance-error-control,lem:bounded-orlitz-error-control-mul-form,lem:bounded-variance-error-control-mul-form}
in \Cref{sec:proofs-crossvalidation-modelselection}.

\begin{lemma}
    [Bernstein's inequality, adapted from Theorem 2.8.1 of \cite{vershynin_2018}]
    \label{lem:bernstein-ineq}
    Let $Z_1, \dots, Z_n$ be independent mean-zero sub-exponential random variables.
    Then, for every $t \ge 0$, we have
    \[
    \PP\left\{ \left| \sum_{i=1}^{n} Z_i \right| \ge t \right\}
    \le
    2 \exp
    \left(
        -c \min
        \left\{
            \frac{t^2}{\sum_{i=1}^{n}\|Z_i\|_{\psi_1}^2},
            \frac{t}{\max_{1 \le i \le n} \| Z_i \|_{\psi_1}}
        \right\}
    \right),
    \]
    where $c > 0$ is an absolute constant.
    In other words,
    with probability at least $1 - \eta$,
    we have
    \[
    \left| \sum_{i=1}^{n} Z_i \right|
    \le
    \max
    \left\{
    \sqrt{\frac{1}{c} \sum_{i=1}^{n} \| Z_i \|_{\psi_1}^2 \log\left(\frac{2}{\eta}\right)},
    \frac{1}{c} \max_{1 \le i \le n} \| Z_i \|_{\psi_1} \log\left(\frac{2}{\eta}\right)
    \right\}.
    \]
\end{lemma}

\begin{lemma}
   [Concentration for median-of-means (MOM) estimator,
   adapted from Theorem 2 of~\cite{lugosi2019mean}]
   \label{lem:mom-concentration}
   Let $W_1, \dots, W_n$ be i.i.d.\ random variables
   with mean $\mu$ and variance bounded by $\sigma^2$.
   Suppose we split the data $\{ W_1, \dots, W_n \}$
   into $B$ batches $\cT_1, \dots, \cT_B$.
   Let $\hmu_b$ be sample mean computed on $\cT_b$ for $b = 1, \dots, B$.
   Define
   \[ \hmu_B^\MOM := \mathrm{median}(\hmu_1, \dots, \hmu_B). \]
   Then, we have
   \[
        \PP
        \left\{
            \left| \hmu_B^\MOM - \mu \right|
            > \sigma \sqrt{4 B / n } 
        \right\}
        \le
        \exp(- B / 8).
   \]
   Thus, letting $0 < \eta < 1$ be a real number,
   $B = \lceil 8 \log(1/\eta) \rceil$,
   with probability at least $1 - \eta$,
   \[
       \left| \hmu_B^\MOM - \mu \right|
       \le \sigma\sqrt{\frac{32 \log(1/\eta)}{n}}.
   \]
\end{lemma}
With $B = \lceil 8\log(1/\eta)\rceil$, we use the notation $\MOM(\{W_1, \ldots, W_n\}, \eta)$ for $\hmu_B^\MOM$, that is,
\begin{equation}\label{eq:definition-MOM-rvs-eta}
\MOM(\{W_1, \ldots, W_n\}, \eta) ~:=~ \widehat{\mu}^{\MOM}_{\lceil 8\log(1/\eta)\rceil}.
\end{equation}

\paragraph{Moment bounds.}

The following two moment bounds imply
\Cref{lem:concen-linform,lem:concen-quadform}
that are used in the proofs of
\Cref{prop:asymp-bound-ridge-main}
and \Cref{cor:verif-onestep-program}
in \Cref{sec:verif-asymp-profile-ridge} and \Cref{sec:verif-riskprofile-mnlsbase-mnlsonestep},
respectively.

\begin{lemma}
    [Moment bound on centered linear form,
    adapted from Lemma 7.8 of \cite{erdos_yau_2017}]
    \label{lem:mom_linform}
    Let $\bZ \in \RR^{p}$ be a random vector containing
    i.i.d.\ entries $Z_i$, $i = 1, \dots, n$, such that
    for each i, $\EE[Z_i] = 0$, $\EE[Z_i^2] = 1$,
    and $\EE[|Z_i|^k] \le M_k$.
    Let $a \in \RR^{p}$ be a deterministic vector.
    Then,
    \[
        \EE[|a^\top \bZ|^q]
        \le C_q M_q \| a \|_2^q
    \]
    for a constant $C_q$ that only depends on $q$.
\end{lemma}

\begin{lemma}
    [Moment bound on centered quadratic form,
    adapted from Lemma B.26 of \cite{bai_silverstein_2010}]
    \label{lem:mom_quadform}
    Let $\bZ \in \RR^{n}$ be a random vector with i.i.d.\ entries
    $Z_i$, $i = 1, \dots, n$, such that for each $i$,
    $\EE[Z_i] = 0$, $\EE[Z_i^2] = 1$, and $\EE[|Z_i|^k] \le M_k$ for $k > 2$ and some constant $M_k$.
    Let $A \in \RR^{p \times p}$ be a deterministic matrix.
    Then, for $q \ge 1$,
    \[
        \EE
        \big[
            |\bZ^\top A \bZ - \tr[A]|^{q}
        \big]
        \le
        C_q
        \big\{
            (M_4 \tr[A A^\top])^{q/2}
            + M_{2q} \tr[(A A^\top)^{q/2}]
        \big\}
    \]
    for a constant $C_q$ that only depends on $q$.
\end{lemma}

\subsection*{Asymptotic statements}

As a consequence of \Cref{lem:mom_linform} and \Cref{lem:borel-cantelli-moment},
we have the following concentration of a linear form with independent components.
\begin{lemma}
    [Concentration of linear form with independent components]
    \label{lem:concen-linform}
    Let $\bZ \in \RR^{p}$ be a random vector with i.i.d.\ entries $Z_i$, $i = 1, \dots, p$
    such that for each $i$, $\EE[Z_i] = 0$, $\EE[|Z_i|^{4+\alpha}] \le M_\alpha$
    for some constant $M_\alpha < \infty$.
    Let $\bA \in \RR^{p}$ be a random vector independent of $\bZ$
    such that $\limsup_{p} \| \bA_p \|^2 / p \le M_n$ almost surely
    for a constant $M_n < \infty$.
    Then, $\bA^\top \bZ / p \to 0$ almost surely as $p \to \infty$.
\end{lemma}

\noindent
As a consequence of \Cref{lem:mom_quadform} and \Cref{lem:borel-cantelli-moment},
we have the following concentration of a quadratic form with independent components.
\begin{lemma}
    [Concentration of quadratic form with independent components]
    \label{lem:concen-quadform}
    Let $\bZ \in \RR^{p}$ be a random vector with i.i.d.\ entries $Z_i$, $i = 1, \dots, p$
    such that for each i, $\EE[Z_i] = 0$, $\EE[Z_i^2] = 1$, $\EE[|Z_i|^{4+\alpha}] \le M_\alpha$
    for some $\alpha > 0$ and constant $M_\alpha < \infty$.
    Let $\bD \in \RR^{p \times p}$ be a random matrix such that $\limsup \| \bD \|_{\textrm{op}} \le M_o$
    almost surely as $p \to \infty$ for some constant $M_o < \infty$.
    Then, $\bZ^\top \bD \bZ / p - \tr[\bD] / p \to 0$
    almost surely as $p \to \infty$.
\end{lemma}

\begin{lemma}
    [Moment version of the Borel-Cantelli lemma]
    \label{lem:borel-cantelli-moment}
    Let $\{ Z_n \}_{n \ge 1}$ be a sequence of real-valued random variables such that
    the sequence $\{ \EE|Z_n|^{q} \}_{n \ge 1}$ is summable for some $q > 0$.
    Then, $Z_n \to 0$ almost surely as $n \to \infty$.
\end{lemma}

\section{Notation}
\label{sec:notation}

Below we list general notation used in this paper. 
\Cref{tab:notation-summary-1} 
at the end of the manuscript provides 
a comprehensive list of some of the specific notation used throughout.

\begin{itemize}
\item
We denote scalar random variables in regular upper case (e.g., $X$),
and vector and matrix random variables in bold upper case (e.g., $\bX$).
We use calligraphic letters to denote sets (e.g., $\cD$),
and blackboard letters to denote some specials sets listed next.
\item 
We use $\NN$ to denote the set of natural numbers.
We use 
$\QQ$ to denote the set of rational numbers,
$\QQ_{> 0}$ to denote the set of positive rational numbers;
$\RR$ to denote the set of real numbers,
$\RR_{\ge 0}$ to denote the set of non-negative real numbers,
$\RR_{> 0}$ to denote the set of positive real numbers;
$\CC$ to denote the set of complex numbers,
$\CC^{> 0}$ to denote the upper half of the complex plane,
i.e., $\CC^{> 0} = \{ z \in \CC : \Im(z) > 0 \}$.

\item For a real number $a$,
$(a)_+$ denotes its positive part,
$\lfloor a \rfloor$ denotes its floor, 
$\lceil a \rceil$ denotes its ceiling,
$\sign(a)$ denotes its sign.
For a complex number $z$,
$\Re(z)$ denotes its real part,
$\Im(z)$ denotes its imaginary part,
$\overline{z}$ denote its conjugate,
$| z |$ denotes its absolute value.

\item For a set $\cA$,
$|\cA|$ denotes its cardinality,
$\cA^{\complement}$ denotes its complement,
$\1_{\cA}$ denotes its indicator function.
For a function $f$,
$\partial / \partial x [f]$
denotes
its partial derivative 
with respect to variable $x$.
We also use $f'$ to denote
derivative of $f$ when it is clear from the context.

\item 
For an event $A$,
$\PP(A)$ denotes its probability,
and $\1_{A}$ its indicator random variable.
For a random variable $X$,
$\EE[X]$ denotes its expectation,
$\Var(X) = \EE[(X - \EE[X])^2]$ denotes its variance;
$\EE[X^r]$ denotes its $r$-th moment,
$\EE[|X|^r]$ denotes its $r$-th absolute moment,
$\| X \|_{L_r} = (\EE[|X|^r])^{1/r}$ denotes its $L_r$ norm,
for a real number $r \ge 1$;
$\| X \|_{\psi}$ denotes its $\psi$ norm
for an Orlicz function $\psi$;
see \Cref{subsec:control-of-error-terms} for more details.

\item For a vector $a \in \RR^{p}$, 
$\| a \|_{r}$ denotes its $\ell_r$ norm for $r \ge 1$,
$\| a \|_A = \sqrt{a^\top A a}$ denotes its norm with respect
to a positive semidefinite matrix $A \in \RR^{p \times p}$.

\item For a matrix $A \in \RR^{n \times p}$,
$A^\top \in \RR^{p \times n}$ denote its transpose,
$A^\dagger \in \RR^{p \times n}$ denotes the its Moore-Penrose inverse,
$\| A \|_{\op}$ denotes its operator norm,
$\| A \|_{\tr}$ denotes its trace norm or nuclear norm
($\| A \|_{\tr} = \tr[(A^\top A)^{1/2}] = \sum_i \sigma_i(A))$,
where $\sigma_1(A) \ge \sigma_2(A) \ge \dots $
denote its singular values in non-increasing order.
For a square matrix $A \in \RR^{p \times p}$,
$\tr[A] = \sum_{i=1}^{p} A_{ii}$ denotes its trace.
A $p$-dimensional identity matrix is denoted as $I_p$
or simply $I$ when it is clear from the context.

\item For a $p \times p$ positive semidefinite matrix $A$
with eigenvalue decomposition $A = V R V^\top$
for an orthonormal matrix $V$ and a diagonal matrix $R$,
and a function $f : \RR_{\ge 0} \to \RR_{\ge 0}$,
we denote by $f(A)$ the $p \times p$ positive semidefinite matrix $V f(R) V^\top$,
where $f(R)$ is a $p \times p$ diagonal matrix
obtained by applying the function $f$ to each diagonal entry of $R$.

\item For two sequences of matrices $A_n$ and $B_n$,
we use the notation $A_n \asympequi B_n$
to denote a certain notion of asymptotic equivalence;
see \Cref{sec:calculus_deterministic_equivalents} for more details.
For symmetric matrices $A$ and $B$, 
$A \preceq B$ denotes the Loewner ordering
to mean that the matrix $B - A$ is positive semidefinite.

\item 
We write $a \asymp b$ when there exist absolute constants
$C_l$ and $C_u$ such that $C_l \le a / b \le C_u$.
We write $a \lesssim b$ when there exists an absolute constant
$C$ such that $a \le C b$.

\item 
We use $O$ and $o$ to denote the big-$O$ and little-$o$ asymptotic notation, respectively.
We use $O_p$ and $o_p$ to denote the probabilistic big-$O$ and little-$o$ asymptotic notation, respectively.
We denote 
convergence in probability by $\pto$, 
almost sure convergence by $\asto$,
weak convergence by $\dto$.

\item Finally, we use generic letters $C, C_1, C_2, \dots$ to denote constants
whose value may change from line to line.
\end{itemize}

\newgeometry{left=0.25cm, right=0.25cm,top=0.5cm,bottom=0.5cm,includefoot}

\begin{table}[!ht]
    \centering
    \begin{tabular}{ll}
        \toprule
        Notation & Meaning (Location in the paper) \\
        \midrule
        $(X, Y)$
        & feature vector $X \in \RR^{p}$
        and response variable $Y \in \RR$ (\Cref{sec:oracle-risk-inequalities}) \\
        $\cD_n = \{ (X_i, Y_i) \}_{i=1}^{n}$
        & dataset with $n$ observations
        $(X_i, Y_i)$, $1 \le i \le n$ (\Cref{sec:oracle-risk-inequalities}) \\
        $\hf(\cdot; \cD_n) : \RR^{p} \to \RR$
        & predictor fitted on dataset $\cD_n$ using prediction procedure $\hf$ 
        (\Cref{sec:oracle-risk-inequalities})
        \\
        $\ell : \RR \times \RR \to \RR_{\ge 0}$
        & non-negative loss function 
        (\Cref{sec:oracle-risk-inequalities}) \\
        $\ell(Y_0, \hf(X_0; \cD_n))$
        & prediction loss of 
        predictor $\hf(\cdot; \cD_n)$
        evaluated at test point $(X_0, Y_0)$ 
        (\Cref{sec:oracle-risk-inequalities}) \\
        $R(\hf(\cdot; \cD_n))$
        & prediction risk of predictor $\hf(\cdot; \cD_n)$ 
        \eqref{eq:prediction-risk}
        \\
        $\hR(\hf(\cdot; \cD_n))$
        & estimator of prediction risk of $\hf(\cdot; \cD_n)$ 
        (\Cref{sec:oracle-risk-inequalities})
        \\
        \midrule
        $\hf^\cv(\cdot; \cD_n)$
        & cross-validated predictor fitted using dataset $\cD_n$ 
        (\Cref{alg:general-cross-validation-model-selection})
        \\
        $\hf^\xi$, $\xi \in \Xi$
        & collection of prediction procedures indexed by set $\Xi$ 
        (\Cref{alg:general-cross-validation-model-selection})
        \\
        $n_\train$, $n_\test$
        & number of train and test observations 
        (\Cref{alg:general-cross-validation-model-selection})
        \\
        $\cD_\train, \cD_\test$
        & random split of $\cD_n$ into train and test datasets
        with $n_\train$ and $n_\test$ observations 
        (\Cref{alg:general-cross-validation-model-selection})
        \\
        $\cI_\train, \cI_\test$
        & disjoint subsets of $\cI_n := \{ 1, \dots, n \}$ that are index sets for $\cD_\train$ and $\cD_\test$ 
        (\Cref{alg:general-cross-validation-model-selection})
        \\
        $\CEN$, $\AVG$, $\MOM$
        & centering procedure, averaging, median-of-means
        (\ref{eqn:avg-risk-pre}, \ref{eqn:mom-risk-pre})
        \\
        $\eta$ 
        & parameter in median-of-means
        \eqref{eq:definition-MOM-rvs-eta}
        \\
        \midrule
        $\Delta_n^\add$, $\Delta_n^\mul$
        & error terms in the additive and multiplicative oracle risk inequalities 
        (\ref{eq:Delta_n_add}, \ref{eq:Delta_n_mul}) \\
        $\hsigma_\xi$, $\hsigma_\Xi$ 
        & conditional second moment of loss and their max over $\Xi$ 
        (\Cref{lem:bounded-orlitz-error-control,lem:bounded-variance-error-control})
        \\
        $\hkappa_\xi$, $\hkappa_\Xi$
        & conditional kurtosis-like parameter of loss 
        and their max over $\Xi$ 
        (\Cref{lem:bounded-orlitz-error-control-mul-form,lem:bounded-variance-error-control-mul-form})
        \\
        $\| \ell(Y_0, \hf(X_0; \cD_n)) \|_{\psi_1 \mid \cD_n}$
        & conditional $\psi_1$ norm of prediction loss \eqref{eq:conditional-Orlicz} \\
        $\| \ell(Y_0, \hf(X_0; \cD_n)) \|_{L_r \mid \cD_n}$
        & conditional $L_r$ norm of prediction loss ($r \ge 1)$ \eqref{eq:conditional-Lp} \\
        \midrule
        $\tbeta_\ridge$, $\tbeta_\lasso$, $\tbeta_\mnls$, $\tbeta_\mnla$
        & ridge, lasso, min $\ell_2$, $\ell_1$-norm least squares estimation procedures 
        (\ref{eq:mn2ls}--\ref{eq:lasso}) 
        \\
        $\tf_\mnls$, $\tf_\mnla$
        & 
        min $\ell_2$, $\ell_1$-norm least squares prediction procedures 
        (\ref{eq:mn2ls-predictor}, \ref{eq:mn1ls-predictor}) \\
        \midrule
        $\hf^\zerostep(\cdot; \cD_n)$
        & zero-step predictor fitted on dataset $\cD_n$ (\Cref{alg:zero-step}) \\
        $\nu \in (0, 1)$
        & exponent for block sizes $\lfloor n^\nu \rfloor$ in zero-step prediction procedure (\Cref{alg:zero-step}) \\
        $n_{\xi}$ & 
        $n - \xi \lfloor n^\nu \rfloor$ (\Cref{alg:zero-step}) \\
        $M$
        & number of sub-samples for averaging for zero-step ingredient predictor \eqref{eq:zerostep-averaged} \\
        $\cD_{\train}^{\xi, j}$, $1 \le j \le M$
        & random subset of $\cD_\train$ of size $n_{\xi}$ (\Cref{alg:zero-step}) \\
        $\tf(\cdot; \cD_{\train}^{\xi, j})$
        & zero-step ingredient predictor fitted on dataset $\cD_{\train}^{\xi, j}$ 
        using base prediction procedure $\tf$
        \eqref{eq:zerostep-averaged}
        \\
        $R^\deter(m; \tf)$
        & deterministic approximation to $R(\tf(\cdot; \cD_m))$ 
        (\Cref{def:rn_nonstochastic-def}) \\
        $R^\deter_\nearrow(n; \tf)$
        & monotonized deterministic approximation at sample size $n$ 
        under general asymptotics \eqref{eq:minimizer-rdeter-def}  \\
        PA($\gamma$)
        & proportional asymptotics regime \ref{asm:prop_asymptotics} \\
        DETPA-0
        & assumption of deterministic risk approximation to conditional risk 
        under PA \eqref{eq:rn-deterministic-approximation-2-prop-asymptotics} \\
        DETPAR-0
        & reduction of assumption 
        \ref{eq:rn-deterministic-approximation-2-prop-asymptotics}
        (\Cref{lem:rn-deterministic-approximation-4-prop-asymptotics}, 
        \ref{tag:detpar-0})\\
        $R^\deter(p_m/m; \tf)$ & 
        deterministic risk 
        approximation 
        at aspect ratio $p_m / m$
        under 
        PA
        (\Cref{sec:zerostep-overparameterized})
        \\
        $\xi_n^\star$ & 
        optimal sequence of $\xi$ for zero-step monotonized risk approximation 
        (\ref{eq:minimizer-rdeter-def}, 
        \ref{eq:rn-deterministic-approximation-2-prop-asymptotics})
        \\
        PRG-0-C1,C2
        & deterministic risk approximation program for zero-step
        \ref{prog:zerostep-cont-conv}--\ref{prog:zerostep-cont-infty} \\
        $k_m, p_m$ &
        sample size and feature size when verifying zero-step profile assumption 
        (\Cref{lem:rn-deterministic-approximation-4-prop-asymptotics})
        \\
        $\rho^2$, $\sigma^2$, $\mathrm{SNR}$
        & signal energy, noise energy, signal-to-noise ratio ($\rho^2$/$\sigma^2$)
        (\Cref{sec:zerostep-illustration}) \\
        $R_{\mnls}^\deter(\phi; \rho^2, \sigma^2)$
        & MN2LS risk approximation at aspect ratio $\phi$,
        signal energy $\rho^2$, noise energy $\sigma^2$ 
        \eqref{eq:rdet_mnls_onestep_iterated_formula}
        \\
        $\tf_{\infty}(\cdot; \cD_{\train})$
        & zero-step ingredient predictor fitted on $\cD_n$ 
        with $M = \infty$ \eqref{eq:predictor_M_infty} \\
        \midrule
        $\hf^\onestep(\cdot; \cD_n)$
        & one-step predictor fitted on dataset $\cD_n$ (\Cref{alg:one-step}) \\
        $(n_{1, \xi_1}, n_{2, \xi_2})$
        & $(n - \xi_1 \lfloor n^\nu \rfloor, \xi_2 \lfloor n^\nu \rfloor)$ 
        (\Cref{alg:one-step}) \\
        $(\cD_\train^{\xi_1, j}, \cD_\train^{\xi_2, j})$, $1 \le j \le M$
        & random pairs of disjoint subsets of $\cD_\train$ of sizes 
        $(n_{1, \xi_1}, n_{2, \xi_2})$
        (\Cref{alg:one-step}) \\
        $\tf(\cdot; \cD_{\train}^{\xi_1, j}, \cD_{\train}^{\xi_2, j})$
        & one-step ingredient predictor fitted on datasets 
        $(\cD_{\train}^{\xi_1, j}, \cD_{\train}^{\xi_2, j})$ 
        \eqref{eq:onestep-average}
        \\
        DETPA-1, DETPA-1*
        & assumption of deterministic risk approximation to conditional risk under PA 
        (\ref{eq:rn-deterministic-approximation-onestep-prop-asymptotics}) \\
        DETPAR-1
        & reduction of assumption
        \ref{eq:rn-deterministic-approximation-onestep-prop-asymptotics}
        (\Cref{lem:deterministic-approximation-reduction-onestep},
        \ref{eq:rn-deterministic-approximation-reduced-onestep-prop-asymptotics})\\
        $R^\deter(p/n_1, p/n_2; \tf)$ & 
        risk approximation
        of ingredient one-step predictor
        at aspect ratios $(p / n_1, p / n_2)$
        (\Cref{sec:onestep-overparameterized}) \\
        $(\xi_{1,n}^\star, \xi_{2,n}^\star)$
        & optimal pair of sequence of $\xi$ 
        for one-step monotonized risk approximation 
        \eqref{eq:minimizer-sequence-onestep}
        \\
        PRG-1-C1,C2,C3
        & deterministic risk approximation program for one-step
        \ref{prog:onestep-cont-conv}--\ref{prog:onestep-divergence-closedset} \\
        $k_{1,m}, k_{2,m}, p_m$ &
        sample size and feature sizes when verifying one-step profile assumption 
        (\Cref{lem:deterministic-approximation-reduction-onestep})
        \\
        $w_i$, $r_i$, $1 \le i \le p_m$ 
        & eigenvectors and eigenvalues of feature covariance matrix 
        $\Sigma \in \RR^{p_m \times p_m}$ 
        (\Cref{sec:verifying-deterministicprofiles-onestep})
        \\
        $\widehat{Q}_n$, $Q$
        & a certain random distribution and its weak limit \eqref{eq:generalized_prediction_risk_onestep} \\
        $H_{p_m}$, $H$
        & empirical distribution of eigenvalues of $\Sigma$ and limiting spectral distribution 
        \eqref{eq:empirical_distribution_eigenvalues_Sigma}
        \\
        $v(0; \phi_2), \tv(0; \phi_2), \tv_g(0; \phi_2), \Upsilon_b(\phi_1, \phi_2)$
        & scalars in risk approximation of one-step procedure with linear base procedure 
        (\ref{eq:fixed-point-v-onestep-main-phi2}--\ref{eq:def-upsilonb-upsilonv}) \\
        $R^\deter_{\mnls}(\phi_1, \phi_2; \rho^2, \sigma^2)$
        & MN2LS one-step risk approx at aspect ratios $(\phi_1, \phi_2)$,
        signal energy $\rho^2$, noise energy $\sigma^2$ 
        \eqref{eq:rdet_mnls_onestep_iterated_formula} \\
        \bottomrule
    \end{tabular}
    \caption{Summary of some of the main notation used in the paper.}
    \label{tab:notation-summary-1}
\end{table}

\restoregeometry

\end{document}